\documentclass[12pt,a4paper]{article}

\usepackage{color}
\usepackage{amsmath,amssymb,amsthm,amsfonts}
\renewcommand\eqref[1]{(\ref{#1})} %Need with hyperref

\usepackage[mathscr]{eucal}
\usepackage{enumitem} %control on enumerate and itemize commands

\setitemize[1]{label=$\bullet$,leftmargin=*}

\newcommand{\abs}[1]{\lvert#1\rvert}
\newcommand{\norm}[1]{\lVert#1\rVert}
\newcommand{\absbig}[1]{\bigl\lvert#1\bigr\rvert}
\newcommand{\normbig}[1]{\bigl\lVert#1\bigr\rVert}
\newcommand{\absBig}[1]{\Bigl\lvert#1\Bigr\rvert}
\newcommand{\normBig}[1]{\Bigl\lVert#1\Bigr\rVert}

\newcommand{\abslr}[1]{\left\lvert#1\right\rvert}
\newcommand{\normlr}[1]{\left\lVert#1\right\rVert}
\newcommand{\set}[1]{\left\{#1\right\}}
\newcommand{\bra}[1]{\langle#1\rangle}
\newcommand{\brac}[1]{(1+\abs{#1})}%{\langle#1\rangle}%
\newcommand{\bract}[1]{(1+#1)}%{\langle#1\rangle}%
\def\p#1{{\left({#1}\right)}}
\def\jp#1{{\left\langle{#1}\right\rangle}}

\newcommand\R{{\mathbb R}}
\newcommand\C{{\mathbb C}}
\newcommand\N{{\mathbb N}}
\newcommand\NUO{{\mathbb N}\cup\{0\}}

\newcommand\Rn{{{\mathbb R}^n}}
\newcommand\Sn{{{\mathbb S}^n}}
\newcommand\Snm{{{\mathbb S}^{n-1}}}
\newcommand\SNm{{{\mathbb S}^{N-1}}}

\newcommand\al{\alpha}
\newcommand\be{\beta}
\newcommand\ga{\gamma}
\newcommand\de{\delta}
\newcommand\De{\Delta}
\newcommand\ep{\varepsilon}
\newcommand\ka{\kappa}
\newcommand\la{\lambda}
\newcommand\va{\varphi}
\newcommand\si{\sigma}
\newcommand\Si{\Sigma}
\newcommand\om{\omega}
\newcommand\Om{\Omega}

\newcommand\lap{\Delta}
\newcommand\grad{\nabla}
\newcommand\FT{{\mathscr F}}
\newcommand\pa{\partial}
\newcommand{\clos}[1]{\overline{#1}}
\newcommand\conv{\ast}

\newcommand\intbar{-\!\!\!\!\!\!\int}

\definecolor{lightgrey}{gray}{0.8}

\renewcommand\S{{\mathcal S}} %Schwarz space
\renewcommand\L{{\mathscr L}}
\renewcommand*{\Re}{\operatorname*{Re}}
\renewcommand*{\Im}{\operatorname*{Im}}

\DeclareMathOperator{\meas}{meas}%
\DeclareMathOperator{\supp}{supp}%
\DeclareMathOperator{\codim}{codim}%
\DeclareMathOperator{\Hess}{Hess}%
\DeclareMathOperator{\rank}{rank}%
\DeclareMathOperator{\diag}{diag}%
\DeclareMathOperator{\dist}{dist}%
\newcommand\Fvar{\upsilon} %In Randol Lemma, need variable
\newcommand\FVar{\Upsilon} %Parameter space for this variable
\newcommand\Ivar{\nu} %parameter in oscillatory integral theorem
\newcommand\IVar{{\mathcal N}}
\newcommand\gauss{\mathbf{\underline{n}}} %Gauss map
\newcommand\xtil{\widetilde{x}} %In change of variables, to eliminate t
\newcommand\latil{\widetilde{\lambda}}%As above, but in the lambda variable
\newcommand\pat{D_t}%
\newcommand\pax{D_x}%
\newcommand\paxj{D_{x_j}}%
\newcommand\cutoffN{\chi} %Cut-off function for \abs{\xi}>N
\newcommand\cutoffbes{\Phi} %cut-off for Besov spaces
\newcommand\cutoffWF{\ka} %cut-off for wave front
\newcommand\cutoffsing{\psi} %cut-off around singularities
\newcommand\cutoffcone{\Psi} %cut-off around into cones
\newcommand\cutoffadjust{G} %function to adapt for tau
\newcommand\cutoffSP{\chi} %cut-off in the stationary phase part
\newcommand\cutoffM{\chi} %cut-off in neighbourhood of multiplicities
\newcommand\phasefn{\Phi} %phase function in Sugimoto theorem
\newcommand\indicat{\bf{1}} %characteristic function on set
\newcommand\mass{\mu} %For wave equation with lower order terms
\newcommand\diss{\delta} %For wave equation with lower order terms
\newcommand\curlyR{\mathcal R}
\newcommand\curlyM{\mathcal M}

\newcommand{\psq}[1]{\left[#1\right]}

\def\aside#1{}
\def\aside#1{\footnote{#1}}

\numberwithin{equation}{section}
\theoremstyle{plain}
\newtheorem{thm}{Theorem}[section]
\newtheorem{lem}[thm]{Lemma}
\newtheorem{cor}[thm]{Corollary}
\newtheorem{prop}[thm]{Proposition}
\newtheorem{defn}[thm]{Definition}

\newtheorem{rem}[thm]{Remark}

\newtheorem{examples}[thm]{Examples}

\newtheorem{rem*}{Remark}

\title{Dispersive and Strichartz estimates for
hyperbolic equations with constant coefficients}
\author{Michael Ruzhansky
\footnote{The first author was supported by EPSRC grants
EP/E062873/01 and EP/G007233/1.}
\ \  and James Smith}
\date{\today}

\begin{document}
\maketitle

\begin{abstract}
Dispersive and Strichartz 
estimates for solutions to general strictly hyperbolic partial 
differential equations with constant coefficients 
are considered. 
The global time decay estimates of $L^p-L^q$ 
norms of propagators is discussed,
and it is shown how
the time decay rates depend on the geometry of the problem.
The frequency space is separated in several zones each
giving a certain decay rate. Geometric conditions on
characteristics responsible for the particular decay
are investigated. Thus, a comprehensive
analysis is carried out for strictly hyperbolic equations of
high orders with lower order terms of a general form.
Results are applied to time decay estimates
for the Fokker--Planck equation and for semilinear 
hyperbolic equations.
\end{abstract}

\tableofcontents

\section{Introduction}

These notes are devoted to the investigation of dispersive and
Strichartz estimates for general hyperbolic equations with
constant coefficients. The analysis that we carry out is also
applicable to hyperbolic systems either by looking at 
characteristics of the system directly, or first taking the
determinant of the system (the dispersion relation). 

There are several important motivations for the analysis.
First, while hyperbolic equations of the second order
(such as the wave equation, dissipative wave equation,
Klein--Gordon equation, etc.) are very well studied, 
relatively little is known about equations of higher orders.
At the same time, equations or systems of high orders naturally
arise in applications. For example, Grad systems of 
non-equilibrium gas dynamics, when linearised near an equilibrium
point, are examples of large hyperbolic systems with
constant coefficients (see e.g. \cite{Rad03}, \cite{Rad05}).
Here one has to deal with hyperbolic equations of orders 13, 20,
etc., depending on the number of moments in the Grad system.
Moreover, there are important families of systems of size going to
infinity, or even of infinite hyperbolic systems. For
example, the Hermite--Grad method for the analysis of the
Fokker--Planck equation for the distribution function for
particles for the Brownian motion produces an infinite
hyperbolic system with constant coefficients. Indeed,
making the decomposition in the space of velocities into
the Hermite basis, and writing equations for the space-time
coefficients produces a hyperbolic system for infinitely
many coefficients (see e.g. \cite{vole+radk03}, \cite{vole+radk04},
\cite{ZR04}, and Section \ref{SEC:FokkerPlanck}). 
The Galerkin approximation of this system
leads to a family of systems with sizes increasing to infinity.
Although explicit calculations are difficult in these situations,
the time decay rate of the solution can still be calculated
(\cite{Ruzh06}).

One of the main difficulties when dealing with large systems
is that unlike in the case of the second order equations, 
in general characteristics can not be calculated explicitly. 
This raises a natural problem to look for properties of
the equation that determine the decay rates for solutions.
On one hand, it becomes clear that one has to look for geometric
properties of characteristics that may be responsible for
such decay rates. On the other hand, a subsequent problem
arises to be able to reduce these properties from some
properties of coefficients of the equation. 

One encounters several difficulties on this path. One 
difficulty lies in the
absence of general formulae for characteristic roots.
For large frequencies one can use perturbation methods to
deduce the necessary asymptotic properties of characteristics.
However, 
this approach can not be used for small frequencies, where the
situation becomes more subtle. For example, for small
frequencies characteristics may become multiple, causing
them to become irregular. This means that if we use the
usual representation of solutions in terms of Fourier multipliers,
phases become irregular, while amplitudes are irregular and
blow up. Thus, we will need to carry out the detailed
analysis of sets of possible multiplicities using the fact
that they are solutions of parameter dependent
polynomial equations. Another difficulty for small frequencies
is that there exists a genuine interaction between time and 
frequencies. In the case of homogeneous symbols it can be shown
(see e.g. Section \ref{SEC:homogoperators}) that time can be
taken out of the estimates, after which low frequencies
can be ignored since the corresponding operators are smoothing
and their estimates are independent of time. In the case of
the presence of lower order terms, the time can no longer
be eliminated from the estimates, so even small frequencies
become large for large times and may influence the resulting
estimates.

The purpose of this work is to present a comprehensive
analysis of such problems. Despite the difficulties
described above, we will be able to determine what geometric
properties of characteristic roots are responsible for
qualitatively different time decay rates for solutions.
Moreover, we will calculate these rates and relate them
to geometric properties of equations. This will lead to a
comprehensive picture of decay rates and orders in dispersive
estimates for hyperbolic equations with constant coefficients.
Such estimates lead to Strichartz estimates, for which
our analysis will be applied, with further implications for
the corresponding semilinear problems.

Thus, in this paper we consider a problem of 
determining dispersive and
Strichartz estimates for general hyperbolic equations with
lower order terms.
Therefore, we consider the Cauchy problem for general 
$m^{{\text{th}}}$ order
constant coefficient linear strictly hyperbolic equation with
solution $u=u(t,x)$:
\begin{equation}\label{EQ:standardCauchyproblem}
\left\{\begin{aligned}& \overbrace{\pat^m
u+\sum_{j=1}^{m}P_{j}(\pax)\pat^{m-j}u}^{
\textrm{homogeneous principal part}}+\overbrace{
\sum_{l=0}^{m-1}\sum_{\abs{\al}+r=l}
c_{\al,r}\pax^\al\pat^ru}^{
\textrm{general lower order terms}}=0,\quad t>0,\\
&\pat^lu(0,x)=f_l(x)\in C_0^{\infty}(\R^n),\quad l=0,\dots,m-1,\;
x\in\R^n\,,
\end{aligned}\right.
\end{equation}
where $P_j(\xi)$, the polynomial obtained from the operator
$P_j(\pax)$ by replacing each~$D_{x_k}$ by~$\xi_k$, is a constant
coefficient homogeneous polynomial of order~$j$, and the $c_{\al,r}$
are (complex) constants. 
Here, as usual, $\alpha=(\alpha_1,\ldots,\alpha_n)$,
$D_x^\alpha=D_{x_1}^{\alpha_1}\cdots
D_{x_n}^{\alpha_n}$, $D_{x_k}=\frac{1}{i}\partial_{x_k}$ and
$D_t=\frac{1}{i}\partial_t$.
The full symbol of the operator in
\eqref{EQ:standardCauchyproblem} will be denoted by
$$ L(\tau,\xi)=\tau^m
+\sum_{j=1}^{m}P_{j}(\xi)\tau^{m-j}+
\sum_{l=0}^{m-1}\sum_{\abs{\al}+r=l}
c_{\al,r}\xi^\al\tau^r,$$
where $\xi^\alpha=\xi_1^{\alpha_1}\cdots\xi_n^{\alpha_n}.$
We will always assume that
the differential operator in \eqref{EQ:standardCauchyproblem}
is hyperbolic, that is 
for each $\xi\in\R^n$\textup{,} the symbol
of the principal part\textup{,} 
$$L_m(\tau,\xi)=\tau^m+\sum_{j=1}^{m}P_{j}(\xi)\tau^{m-j},$$
has $m$ real roots with respect to $\tau$. 
For simplicity, unless explicitly stated otherwise,
we will also assume that
the operator in \eqref{EQ:standardCauchyproblem}
is \emph{strictly hyperbolic}, that is at each
$\xi\in\R^n\backslash\{0\}$, these roots are pairwise distinct. We
denote the roots of $L_m(\tau,\xi)$ with respect to $\tau$ by
$\varphi_1(\xi)\le\dots\le \varphi_m(\xi)$\textup{,} and if $L$ is
strictly hyperbolic the above inequalities are strict for $\xi\ne0$.

The condition of hyperbolicity arises naturally 
in the study of the
Cauchy problem for linear partial differential operators and it can
be shown that it is a necessary condition for $C^\infty$
well-posedness of the problem;
this is discussed in~\cite{egor+shub92} and
\cite{horm83II}, for example. Strict hyperbolicity 
is sufficient for
$C^\infty$ well-posedness of the Cauchy problem for such an operator
with any lower order terms; if the operator is only hyperbolic
(sometimes called \emph{weakly hyperbolic}) the lower order terms
must satisfy additional conditions for $C^\infty$ well-posedness,
the so-called \emph{Levi conditions}. For
this reason, we only consider strictly hyperbolic operators with
lower order terms, 
since our main interest is to understand the influence
of lower order terms on the decay properties of solutions.

The roots of the
associated full
characteristic polynomial $L(\tau,\xi)$ with respect to 
$\tau$ will be denoted
by $\tau_1(\xi),\dots,\tau_m(\xi)$ and referred to
as the
\emph{characteristic roots} of the full operator.
Clearly, if $L$ is a homogeneous operator then the characteristic
roots $\tau_k(\xi)$, $k=1,\dots,m$, coincide, possibly after
reordering, with the roots $\varphi_k(\xi)$, $k=1,\dots,m$, of the
operator $L_m$. However,
in general there is no natural ordering on the roots $\tau_k(\xi)$
as they may be complex-valued or may intersect.

The analysis here will be based on the properties of
characteristic roots $\tau_k(\xi)$. If the problem
\eqref{EQ:standardCauchyproblem} is strictly hyperbolic,
we can derive their asymptotic properties in a general
situation, necessary for our analysis. However, if the
problem is only hyperbolic, functions $\tau_k(\xi)$ may
develop singularities for large $\xi$. If this does not happen
and we have the necessary information about them, we may 
drop the strict hyperbolicity assumption. This may be
the case in some applications, for example in those 
arising in the analysis of
the Fokker--Planck equation.

We seek \textit{a priori} estimates for the solution 
$u(t,x)$ to the
Cauchy problem \eqref{EQ:standardCauchyproblem}, of the type
\begin{equation}\label{EQ:idealLpLqest}
\norm{\pax^\al\pat^r u(t,\cdot)}_{L^q}\le
K(t)\sum_{l=0}^{m-1}\norm{f_l}_{W^{N_p-l}_p}\,,
\end{equation}
where $1\le p\le2$, $\frac{1}{p}+\frac{1}{q}=1$, 
$N_p=N_p(\al,r)$ is
a constant depending on $p,\al$ and $r$, 
and $K(t)$ is a function to be determined. Here
$W_p^{N_p-l}$ is the Sobolev space over $L^p$ with
$N_p-l$ (fractional) derivatives.

We note that sometimes, 
for example in~\cite{trev802}, in the definition of a
hyperbolic operator the polynomial $L(i\tau,\xi)$ is used as it is
better suited to taking the partial Fourier transform in $x$,
corresponding as it does to $L(\pa_t,\pax)$; in this case, one
requires the roots with respect to $\tau$ to be purely imaginary
(in the cases when we will require them to be real).
However, the definition that we give above is perhaps
more standard, and thus
adopted here throughout.

For a hyperbolic equation with real coefficients we note that the
constants $c_{\al,r}$ satisfy $i^{m-\abs{\al}-l}c_{\al,r}\in\R$;
the equation is written in the form above since our results may be
used to study hyperbolic systems, which can be reduced to an
$m^\text{th}$ order equation with complex coefficients.

Most results presented here will apply to operators which
are pseudo-differential in $x$ and to hyperbolic systems
via their dispersion equation. Moreover, most of results
in this paper are in general sharp.

In this work, we place
the priority on obtaining a comprehensive collection of
estimates for hyperbolic equations with constant coefficients.
The case of variable coefficients is also of great interest,
but we leave some extensions of our analysis to this case outside the
scope of this paper. Let us mention that already in the case of
coefficients depending on time, some unpleasant
phenomena may happen. For example, already for the second
order equations the oscillations in time dependent
coefficients may change the time decay rates for solutions to
the corresponding Cauchy problem. For example, equations with
very fast oscillations, or with increasing coefficients, have been
analysed in \cite{reis+yagd99, reis+yagd00nach},
to mention only a few references. 
Results even for the wave equations with bounded 
coefficients may depend on the oscillations in coefficients
(see e.g. \cite{ReS05}).
At the same time, many 
results of this paper are stable under time perturbations of
coefficients. For example, in the case of equations with
homogeneous symbols with time-dependent coefficients with
integrable derivative, a comprehensive analysis has been 
carried out in \cite{MR07}. We will not deal with such
questions in this paper. Let us also mention that while
dispersive estimates are devoted to $L^p-L^q$ estimates for
solutions, $L^p-L^p$ estimates are also of interest.
A survey of $L^p$ estimates for general
non-degenerate Fourier integral operators and their
dependence on the geometry can be found in \cite{Ruzh00}
in the case of real-valued phase functions, while
operators with complex-valued phase functions have been
analysed in \cite{Ruzh01}. $L^p$--estimates for solutions
to some classes of hyperbolic systems with 
variable multiplicities appeared in \cite{KR07}.

Let us now explain the organisation of these notes.
In the following parts of the introduction we will review
results for second order equations and for equations with
homogeneous symbols, as well as give several more motivations for
the comprehensive analysis of this paper.
In Section \ref{section:ests} we will present results for
different types of behaviour of characteristic roots, and
also of corresponding phase functions in cases where we can 
represent solutions in terms of Fourier multipliers.
Thus, in Section \ref{SEC:21} we will present results
without and with multiplicities, when roots are separated
from the real axis, in which cases we can get exponential
decay of solutions. In Section \ref{SEC:22} we present
results for roots with non-degeneracies, in which case
we have a variety of conclusions depending on geometric
properties of roots. In Section \ref{SEC:23} we
present results for complex roots that become real on
some set. A version of this type of statements
(although not in the microlocal form used here) 
partly appeared in \cite{RS05}, and those are improved here. 
In Section \ref{SEC:24} we summarise the microlocal results
and formulate the main theorem on dispersive estimates for
general hyperbolic equations with constant coefficients.
Theorem \ref{THM:overallmainthm} is the main theorem
containing a table of results,
and the rest of this section is devoted to the explanation
and further remarks about this table.
In Section
\ref{SEC:outline} we will outline our approach, indicating
the relations between frequency regions and statements.
In Section \ref{SEC:nonlinear} we present results
for non-homogeneous equations, as well as formulate
corresponding Strichartz estimates with further applications
to semilinear equations. In general, we leave such developments
outside the scope of this paper since they are quite
well understood (see e.g. \cite{KT98}), once the
time decay rates are determined (as we will do in
Theorem \ref{THM:overallmainthm}).

The subsequent chapters contain the detailed analysis and
proofs. In Section \ref{CHAP3} we establish necessary
properties of roots of hyperbolic polynomials,
as well as carry out the perturbation analysis for
large frequencies. In Section \ref{SEC:ch-convexity}
we investigate estimates for oscillatory integrals
under certain convexity assumptions on the level
sets of the phase function. In Section
\ref{SEC:ch-nonconvex} we analyse the corresponding
oscillatory integrals without convexity assumption. 
Section \ref{SEC:Cauchy} is devoted to dispersive
estimates for solutions to the general Cauchy problem,
and here we prove various parts of Theorem
\ref{THM:overallmainthm}.
Section \ref{SEC:bddxiaroundmults} deals with multiple
characteristics. Here we present a procedure for the
resolution of multiplicities in the representation
of solutions, enabling us to obtain estimates in these
cases as well. Section \ref{SEC:phasefnliesonaxisbddxi}
is devoted to multiple roots on the real axis.
Here, we investigate solutions for frequencies very close
to multiplicities (in some shrinking neighborhoods) as well
as for larger, but still bounded, frequencies. Here we present
several different versions of 
results dependent on possibly different assumptions.
Finally, Section \ref{CHAP5} is devoted
to examples of the presented analysis with further applications.
Thus, in Section \ref{SEC:analysisofm=2} we deal with second
order equations and give examples of
how our results can be applied to investigate the interplay
between mass, dissipation, and frequencies. Further,
in Section \ref{SEC:condonlowerorderterms} we discuss
some conditions on coefficients of equations, and in
Section \ref{SEC:hyptriples} we give examples of
non-homogeneous roots in terms of hyperbolic triples and
Hermite's theorem. In Section \ref{SEC:systems} we
show briefly how the results can be applied for
strictly hyperbolic systems. And finally, in
Section \ref{SEC:FokkerPlanck} we give an application
to the Fokker--Planck equations.

The authors are grateful to Jens Wirth for remarks about
the preliminary version of the manuscript and to Tokio
Matsuyama and Mitsuru Sugimoto for discussions.

We will denote various constants throughout the paper by the
same letter $C$. Balls with radius $R$ centred at $\xi\in\Rn$
will be denoted by $B_R(\xi)$. We will use the notation
$\jp{\xi}=\sqrt{1+|\xi|^2}, \jp{D}=\sqrt{1-\Delta}$
and $|D|=|-\Delta|^{1/2}.$ The Sobolev space $W_p^l$ is then
defined as the space of measurable functions for which
$\jp{D}^l f\in L^p(\R^n_x)$.

We will also use the standard notation for
the symbol class $S^\mu=S^{\mu}_{1,0}$, as a space of
smooth functions $a=a(x,\xi)\in C^\infty(\Rn\times\Rn)$
satisfying symbolic estimates
$|\pa_x^\beta\pa_\xi^\alpha a(x,\xi)|\leq C_{\alpha\beta}
(1+|\xi|)^{\mu-|\alpha|}$, for all $x,\xi\in\Rn$, and
all multi-indices $\alpha, \beta$.

If function $a=a(\xi)$ is independent of $x$, we will sometimes
also write $a\in S^{\mu}_{1,0}(U)$ for an open set
$U\subset\Rn$, if $a=a(\xi)\in C^\infty(U)$ satisfies
$|\pa_\xi^\alpha a(\xi)|\leq C_{\alpha}
(1+|\xi|)^{\mu-|\alpha|}$, for all $\xi\in U$, and
all multi-indices $\alpha$.

\subsection{Background}
The study of $L^p-L^q$ decay estimates, or {Strichartz
estimates}, for linear evolution equations began in~1970 when
Robert Strichartz published two papers,~\cite{stri70tran} and
\cite{stri70func}. He proved that if $u=u(t,x)$ satisfies the
Cauchy problem (that is, the initial value problem) for the
homogeneous linear wave equation
\begin{equation}\label{EQ:waveCauchyprob}
\left\{
\begin{aligned}
&\pa_t^2u(t,x)-\lap_x u(t,x)=0,\quad(t,x)\in\R^n\times(0,\infty)\,,\\
&u(0,x)=\phi(x),\;\pa_t u(0,x)=\psi(x),\quad x\in\R^n\,,
\end{aligned}\right.
\end{equation}
where the initial data $\phi$ and $\psi$ lie in suitable function
spaces such as~$C_0^\infty(\R^n)$, then the \textit{a priori}
estimate
\begin{equation}\label{EQ:strichartzest}
\norm{(u_t(t,\cdot),\nabla_xu(t,\cdot))}_{L^q}\le
C\bract{t}^{-\frac{n-1}{2}\big(\frac{1}{p}-\frac{1}{q}\big)}
\norm{(\nabla_x\phi,\psi)}_{W^{N_p}_p}
\end{equation}
holds when $n\ge2$, $\frac1p+\frac1q=1$, $1<p\le2$ and $N_p\ge
n(\frac1p-\frac1q)$. Using this estimate, Strichartz proved global
existence and uniqueness of solutions to the Cauchy problem for
nonlinear wave equations with suitable (``small'') initial data.
This procedure of proving an \textit{a priori} estimate for a
linear equation and using it, together with local existence of a
nonlinear equation, to prove global existence and uniqueness for a
variety of nonlinear evolution equations is now standard; a
systematic overview, with examples including the equations of
elasticity, Schr\"odinger equations and heat equations, can be
found in~\cite{rack92}, or in many other more recent books.

There are two main approaches used in order to
prove~\eqref{EQ:strichartzest}; firstly, one may write the solution
to~\eqref{EQ:waveCauchyprob} using the d'Alembert ($n=1$), Poisson
($n=2$) or Kirchhoff ($n=3$) formulae, and their generalisation to
large~$n$,
\begin{equation*}
u(t,x)=\begin{cases}
\begin{aligned} \frac{1}{\prod_{j=1}^{\frac{n-1}{2}}(2j-1)}
\Big[&\pa_t(t^{-1}\pa_t)^{\frac{n-3}{2}}
\Big(t^{n-1}\intbar_{\pa B_t(x)}\phi\,dS\Big)\\
+&(t^{-1}\pa_t)^{\frac{n-3}{2}} \Big(t^{n-1}\intbar_{\pa
B_t(x)}\psi\,dS\Big)\Big]\quad\text{(odd $n\ge3$)}\end{aligned}\\
\begin{aligned} \frac{1}{\prod_{j=1}^{n/2}2j}
&\Big[\pa_t(t^{-1}\pa_t)^{\frac{n-2}{2}} \Big(t^{n}\intbar_{
B_t(x)}\frac{\phi(y)}{\sqrt{t^2-\abs{y-x}^2}}\,dy\Big)\\
+&(t^{-1}\pa_t)^{\frac{n-2}{2}} \Big(t^{n}\intbar_{
B_t(x)}\frac{\psi(y)}{\sqrt{t^2-\abs{y-x}^2}}\,dy\Big)\Big]
\quad\text{(even $n$)}\,,\end{aligned}
\end{cases}
\end{equation*}
(here $\intbar$ stands for the averaged integral;
for the derivation of these formulae see, for example,
\cite{evan98}), as is done in~\cite{vonW71} and~\cite{rack92}.
Alternatively, one may write the solution as a sum of Fourier
integral operators:
\begin{equation*}\label{EQ:FMrepforwaveeqn}
u(t,x)=\FT^{-1}
\Big(\frac{e^{it\abs{\xi}}+e^{-it\abs{\xi}}}{2}\,\widehat{\phi}(\xi)
+\frac{e^{it\abs{\xi}}-e^{-it\abs{\xi}}}{2\abs{\xi}}\,\widehat{\psi}(\xi)
\Big)\,.
\end{equation*}
This is done in~\cite{stri70tran},~\cite{bren75}
and~\cite{pech76}, for example. Using one of these representations
for the solution and techniques from either the theory of Fourier
integral operators (\cite{pech76}), Bessel functions
(\cite{stri70tran}), or standard analysis (\cite{vonW71}), the
estimate~\eqref{EQ:strichartzest} may be obtained. 

Let us now compare the time decay rate for the wave equation 
with equations with lower order terms.
An important example is the
Klein--Gordon equation, where 
$u=u(t,x)$ satisfies the
initial value problem
\begin{equation}\label{EQ:KGCP}
\left\{\begin{aligned} &\partial_t^2 
u(t,x)-\lap_xu(t,x)+\mu^2u(t,x)=0,\quad
(t,x)\in\R^n\times(0,\infty)\,,\\
&u(0,x)=\phi(x),\;u_t(0,x)=\psi(x),\quad x\in\R^n\,,
\end{aligned}\right.
\end{equation}
where $\phi,\psi\in C_0^{\infty}(\R^n)$, say, and 
$\mu\not=0$ is a constant
(representing a \emph{mass term}); then
\begin{equation}\label{EQ:KGest}
\norm{(u(t,\cdot),u_t(t,\cdot),\nabla_xu(t,\cdot))}_{L^q}\le
C\bract{t}^{-\frac{n}{2}\big(\frac{1}{p}-\frac{1}{q}\big)}
\norm{(\nabla_x\phi,\psi)}_{W^{N_p}_p},
\end{equation}
where $p,q,N_p$ are as before. Comparing~\eqref{EQ:strichartzest}
to~\eqref{EQ:KGest}, we see that the estimate for the solution to
the Klein--Gordon equation decays more rapidly. 
The estimate is proved
in~\cite{vonW71},~\cite{pech76} and~\cite{horm97} in different ways,
each suggesting reasons for this improvement: in~\cite{vonW71}, the
function
\begin{equation*}
v=v(x,x_{n+1},t):=e^{-i\mu x_{n+1}}u(t,x)\,,\quad x_{n+1}\in\R\,,
\end{equation*}
is defined; using~\eqref{EQ:KGCP}, it is simple to show that $v$
satisfies the wave equation in $\R^{n+1}$, and thus the Strichartz
estimate~\eqref{EQ:strichartzest} holds for~$v$, yielding the
desired estimate for~$u$. This is elegant, but cannot easily be
adapted to other situations due to the importance of the
structures of the Klein--Gordon and wave equations for this proof.
In~\cite{pech76} and~\cite{horm97}, a representation of the
solution via Fourier integral operators is used and the stationary
phase method then applied in order to obtain
estimate~\eqref{EQ:KGest}.

Another second order problem of interest is the Cauchy problem for
the dissipative wave equation,
\begin{equation}\label{EQ:dwe}
\left\{\begin{aligned}&\partial_t^2 
u(t,x)-\lap_x u(t,x)+u_t(t,x)=0\,,
\quad(t,x)\in\R^n\times(0,\infty),
\\ &u(0,x)=\phi(x),\;u_t(0,x)=\psi(x),\quad x\in\R^n\,,
\end{aligned}\right.
\end{equation}
where $\psi,\phi\in C_0^{\infty}(\R^n)$, say. In this case,
\begin{equation}\label{EQ:dwe2}
\norm{\pa_t^r\pa_x^\al u(t,\cdot)}_{L^q}\le C\bract{t}^{-\frac{n}{2}
(\frac{1}{p}-\frac{1}{q})-r-\frac{\abs{\al}}{2}}
\norm{(\phi,\grad\psi)}_{W_p^{N_p}}\,,
\end{equation}
with some $N_p=N_p(n,\alpha,r).$
This is proved in~\cite{mats76} with a view to showing
well-posedness of related semilinear equations. Once again, this
estimate (for the solution $u(t,x)$ itself) is better than that
for the solution to the wave equation; there is an even greater
improvement for higher derivatives of the solution. As before, the
proof of this may be done via a representation of the solution
using the Fourier transform:
\begin{equation*}
u(t,x)=\!\!\begin{cases}
\begin{aligned}\FT^{-1}\!\Big(\Big[\frac{e^{-t/2}
\sinh\big(\frac{t}{2}\sqrt{1-4\abs{\xi}^2}\big)}{\sqrt{1-4\abs{\xi}^2}}
+\!e^{-t/2}\!
\cosh\big(\textstyle\frac{t}{2}\sqrt{1-4\abs{\xi}^2}\big)\Big]\!
\widehat{\phi}(\xi)\\
+\frac{2e^{-t/2}
\sinh\big(\frac{t}{2}\sqrt{1-4\abs{\xi}^2}\big)}{\sqrt{1-4\abs{\xi}^2}}
\widehat{\psi}(\xi)\Big)\,,\quad\abs{\xi}\le1/2,
\end{aligned}\\
\begin{aligned}\FT^{-1}\Big(\Big[\frac{e^{-t/2}
\sin\big(\frac{t}{2}\sqrt{4\abs{\xi}^2-1}\big)}{\sqrt{4\abs{\xi}^2-1}}
+e^{-t/2}
\cos\big(\textstyle\frac{t}{2}\sqrt{4\abs{\xi}^2-1}\big)\Big]
\widehat{\phi}(\xi)\\
+\frac{2e^{-t/2}
\sin\big(\textstyle\frac{t}{2}\sqrt{4\abs{\xi}^2-1}\big)}
{\sqrt{4\abs{\xi}^2-1}} \widehat{\psi}(\xi)\Big)\,,\quad\abs{\xi}>1/2.
\end{aligned}
\end{cases}
\end{equation*}
Matsumura divides the phase space into the regions where the
solution has different properties and then uses standard techniques
from analysis.

It is, therefore, motivating
to ask why the addition of lower order
terms improves the rate of decay of the solution to the equation;
furthermore, in the first instance,
we would like to understand why the improvement in the
decay is the same for both the addition of a mass term and for the
addition of a dissipative term. 
It will follow from the analysis of the paper that the quantities
responsible for the decay rates for the Klein-Gordon and 
dissipative equations are of completely different nature. 
In the first instance the characteristic roots are real and
lie on the real axis for all frequencies, while for the
latter equation they are in the upper complex half-plane,
intersect at a point, and one of them comes to the origin.
From this point of view, the same decay rates in the dispersive
estimate for these two equations is quite a coincidence.
On the example of the dissipative equation we can see another
difficulty for the analysis, namely the appearance of the
multiple roots. This may lead to the loss of regularity in
roots and blow-ups in the amplitudes of a representation, so
we need to develop some techniques to deal with this type
of situations.

These questions are even more important for equations of
higher orders. Let us mention briefly an example of a
system that arises as the linearisation of the 13--moment
Grad system of non-equilibrium gas dynamics in two dimensions
(other Grad systems are similar).
The dispersion relation (the determinant) of this system
is a polynomial of $9^{th}$ order that can be written as
$$P=Q_9-iQ_8-Q_7+iQ_6+Q_5-iQ_4,$$
with polynomials $Q_j(\omega,\xi)$ defined by
\begin{align*} Q_9(\omega,\xi) = & |\xi|^9\omega^3\psq{
  \omega^6-\frac{103}{25}\omega^4+\frac{21}{5}\omega^2
  \p{1-\frac{912}{2625}\alpha\beta}-\frac{27}{25}
  \p{1-\frac{432}{675}\alpha\beta}},      \\
   Q_8(\omega,\xi) = & |\xi|^8\omega^2\psq{\frac{13}{3}\omega^6-
   \frac{1094}{75}\omega^4+\frac{1381}{125}\omega^2
   \p{1-\frac{2032}{6905}\alpha\beta}-
   \frac{264}{125}\p{1-\frac{143}{330}\alpha\beta}}, \\
   Q_7(\omega,\xi) = & |\xi|^7\omega\psq{\frac{67}{9}\omega^6-
   \frac{497}{25}\omega^4+\frac{3943}{375}\omega^2
   \p{1-\frac{832}{3943}\alpha\beta}-\frac{159}{125}
   \p{1-\frac{48}{159}\alpha\beta}}, \\
   Q_6(\omega,\xi) = & |\xi|^6\psq{\frac{19}{3}\omega^6-
   \frac{2908}{225}\omega^4+\frac{13}{3}\omega^2
   \p{1-\frac{32}{325}\alpha\beta}-\frac{6}{25}}, \\
   Q_5(\omega,\xi) = & |\xi|^5\omega\psq{\frac{8}{3}\omega^4-
   \frac{178}{45}\omega^2+\frac{2}{3}}, \\
   Q_4(\omega,\xi) = & \frac{4}{9} |\xi|^4\omega^2 \p{\omega^2-1},
\end{align*}
where  $$\omega(\xi)=\frac{\tau(\xi)}{|\xi|},\;
\alpha=\frac{\xi_1^2}{|\xi|^2},\; 
\beta=\frac{\xi_2^2}{|\xi|^2}.$$ 
A natural question of finding dispersive (and subsequent
Strichartz) estimates for the Cauchy problem for operator
$P(D_t,D_x)$ with symbol $P(\tau,\xi)$
becomes calculationally complicated.
Clearly, in this situation it is hard to find the
roots explicitly, and, therefore, we need some procedure of
determining what are the general properties of the 
characteristics roots, and how to derive the time decay
rate from these properties. Thus, in \cite{Rad03} and
\cite{vole+radk04} it is discussed when such polynomials
are stable. In this case, the analysis of this paper
will guarantee the decay rate, e.g. by applying
Theorem \ref{THM:dissipative} for frequencies near
the origin, Theorem \ref{THM:expdecay2} for bounded
frequencies near possible multiplicities (independent
of the structure of such multiplicities), and
Theorem \ref{THM:expdecay} for large frequencies.
In fact, once the behavior of the characteristic roots
is understood, Theorem \ref{THM:overallmainthm}
will immediately show that the overall time decay rate
here is the same as for the dissipative wave equation.

\subsection{Homogeneous symbols}\label{SEC:homogoperators}
The case where the operator in \eqref{EQ:standardCauchyproblem}
has homogeneous symbol has been studied extensively:
\begin{equation}\label{EQ:CPhigherorder}
\left\{\begin{aligned} &L_m(\pax,\pat)u=0,\quad
(t,x)\in\R^n\times(0,\infty),\\
&\pat^lu(0,x)=f_l(x),\quad l=0,\dots,m-1,\; x\in\R^n\,,
\end{aligned}\right.
\end{equation}
where~$L_m$ is a homogeneous $m^{\text{th}}$ order constant
coefficient strictly hyperbolic differential operator; the symbol
of~$L_m$ may be written in the form
\begin{equation*}
L_m(\tau,\xi)=(\tau-\va_1(\xi))\dots(\tau-\va_m(\xi)),\text{ with }
\va_1(\xi)<\cdots<\va_m(\xi)\quad(\xi\ne0).
\end{equation*}
In a series of papers, \cite{sugi94},~\cite{sugi96}
and~\cite{sugi98}, Sugimoto showed how the geometric
properties of the characteristic roots
$\va_1(\xi),\dots,\va_m(\xi)$ affect the $L^p-L^q$ estimate. To
understand this, let us summarise the method of approach.

Firstly, the solution can be written as the sum of Fourier
multipliers:
\begin{equation*}
u(t,x)=\sum_{l=0}^{m-1}[E_l(t)f_l](x),\quad\text{where }
E_l(t)=\sum_{k=1}^m \FT^{-1}e^{it\va_k(\xi)}a_{k,l}(\xi)\FT,
\end{equation*}
and $a_{k,l}(\xi)$ is homogeneous of order $-l$. Now, the problem
of finding an $L^p-L^q$ decay estimate for the solution is reduced
to showing that operators of the form
\begin{equation*}
M_r(D):=\FT^{-1}e^{i\va(\xi)}\abs{\xi}^{-r}\chi(\xi)\FT\,,
\end{equation*}
where $\va(\xi)\in C^\om(\R^n\setminus\set{0})$ is homogeneous of
order~$1$ and $\chi\in C^\infty(\R^n)$ is equal to~$1$ for
large~$\xi$ and zero near the origin, are $L^p-L^q$ bounded for
suitably large~$r\ge l$. In particular, this means that, for
such~$r$, we have
\begin{equation*}
\norm{E_l(1)f}_{L^q}\le C\norm{f}_{W_p^{r-l}}\,.
\end{equation*}
Then it may be assumed, without loss of generality, that $t=1$.
Indeed, it can be readily checked that
for $t>0$ and $f\in C_0^\infty(\R^n)$, we have the equality
\begin{equation*}
[E_l(t)f](x)=t^l[E_l(1)f(t\cdot)](t^{-1}x)\,.
\end{equation*}

Using this identity and denoting $f_t(\cdot)=f(t\cdot)$, we have
\begin{align*}
\norm{E_l(t)f}_{L^q}^q&=
t^{lq}\norm{[E_l(1)f_t](t^{-1}\cdot)}_{L^q}^q
=t^{lq}\int_{\R^n}\abs{[E_l(1)f_t](t^{-1}x)}^q\,dx\\
\stackrel{(x=tx')}{=}&
t^{lq}\int_{\R^n}t^n\abs{[E_l(1)f_t](x')}^q\,dx'
=t^{lq+n}\norm{E_l(1)f_t}_{L^q}^q\,.
\end{align*}
Then, noting that a simple change of variables yields
\begin{equation*}
\norm{f_t}_{W_p^k}^p\le Ct^{kp-n}\norm{f}_{W_p^k}^p\,,
\end{equation*}
we have,
\begin{equation*}
\norm{E_l(t)f}_{L^q} \le Ct^{l+\frac{n}{q}}\norm{f_t}_{W_p^{r-l}}
\le Ct^{r-n(\frac{1}{p}-\frac{1}{q})}\norm{f}_{W_p^{r-l}}\,;
\end{equation*}
hence,
\begin{equation*}
\norm{u(t,\cdot)}_{L^q}\le Ct^{r-n(\frac{1}{p}-\frac{1}{q})}
\sum_{l=0}^{m-1}\norm{f_l}_{W_p^{r-l}}\,.
\end{equation*}

It has long been known that the values of~$r$ for which $M_r(D)$
is $L^p-L^q$ bounded depend on the geometry of the level set
\begin{equation*}
\Si_{\va}=\set{\xi\in\R^n\setminus\set{0}:\va(\xi)=1}\,.
\end{equation*}
In~\cite{litt73},~\cite{bren75}, it is shown that if the Gaussian
curvature of $\Si_\va$ is never zero then $M_r(D)$ is $L^p-L^q$
bounded when $r\ge\frac{n+1}{2}\big(\frac{1}{p}-\frac{1}{q}\big)$.
This is extended in~\cite{bren77} where it is proven that $M_r(D)$
is $L^p-L^q$ bounded provided
$r\ge\frac{2n-\rho}{2}\big(\frac{1}{p}-\frac{1}{q}\big)$, where
$\rho=\min_{\xi\ne0}\rank\Hess\va(\xi)$.

Sugimoto extended this further in~\cite{sugi94}, where he showed
that if $\Si_\va$ is convex then $M_r(D)$ is $L^p-L^q$ bounded
when $r\ge
\big(n-\frac{n-1}{\ga(\Si_\va)}\big)\big(\frac{1}{p}-\frac{1}{q}\big)$;
here,
\begin{equation*}
\ga(\Si):=\sup_{\si\in\Si}\sup_P\ga(\Si;\si,P)\,,\quad
\Si\subset\R^n\text{ a hypersurface}\,,
\end{equation*}
where $P$ is a plane containing the normal to~$\Si$ at~$\si$ and
$\ga(\Si;\si,P)$ denotes the order of the contact between the line
$T_\si\cap P$, $T_\si$ is the tangent plane at~$\si$, and the
curve $\Si\cap P$. See Section~\ref{SUBSEC:convcondnreal} for more
on this maximal order of contact.

In order to apply this result to the solution
of~\eqref{EQ:CPhigherorder}, it is necessary to find a condition
under which the level sets of the characteristic roots are convex.
The following notion is the one that is sufficient:
\begin{defn}\label{DEF:sugimotoconvexitycond}
Let $L=L(\pat,\pax)$ be a homogeneous $m^{\text{th}}$ order constant
coefficient partial differential operator. It is said to satisfy the
\emph{convexity condition} if the matrix of
the second order derivatives, $\Hess\varphi_k(\xi)$,
corresponding to each of its characteristic roots
$\varphi_1(\xi),\dots,\varphi_m(\xi)$, is semi-definite for
$\xi\ne0$.
\end{defn}
It can be shown that if an operator~$L$ does satisfy this
convexity condition, then the above results can be applied to the
solution and thus an estimate of the form \eqref{EQ:idealLpLqest}
holds with
\begin{equation}\label{EQ:decayhom}
K(t)=\bract{t}^{-\frac{n-1}{\ga}\big(\frac{1}{p}-\frac{1}{q}\big)}\,,
\quad\text{with some }\ga\le m\,,
\end{equation}
where $\gamma$ can be related to the convex indices of the
level sets of characteristics. Indeed, under the convexity
condition one can show that $\phi_k$ can be made always
positive or negative by adding an affine function,
the corresponding level sets
$\Sigma_{\phi_k}=\{\xi\in\Rn:\phi_k(\xi)=1\}$ are convex 
for each $k=1,\dots,m$, and that 
$\gamma(\Sigma_{\phi_k})\le 2[m/2]$. 
So the decay in \eqref{EQ:decayhom} is guaranteed with 
$\gamma=2[m/2]$. 
%In fact, it is easy to see that since
%$L$ is a homogeneous differential operator, the union of the sets
%$\Sigma_{\phi_k}$ over $k=1,\ldots,m$ is the set of solutions
%to the polynomial equation $L(1,\xi)=0$, and so every line
%intersects it in at most $m$ points. 

Finally, if this convexity condition does not hold the
estimate fails; 
in~\cite{sugi96} and~\cite{sugi98} it is shown that in
general, $M_r(D)$ is $L^p-L^q$ bounded when $r\ge
\big(n-\frac{1}{\ga_0(\Si_\va)}\big)\big(\frac{1}{p}-\frac{1}{q}\big)$,
where
\begin{equation*}
\ga_0(\Si):=\sup_{\si\in\Si}\inf_P\ga(\Si;\si,P)\le \ga(\Si).
\end{equation*}
For $n=2$, $\ga_0(\Si)=\ga(\Si)$, so, the convexity condition may
be lifted in that case. However, in~\cite{sugi96}, examples are
given when $n\ge3$, $p=1,2$ where this lower bound for~$r$ is the
best possible and, thus, the convexity condition is necessary for
the above estimate. It turns out that the case $n\ge3$, $1<p<2$ is
more interesting and is studied in greater depth in~\cite{sugi98},
where microlocal geometric properties must be looked at in order
to obtain an optimal result.

Two remarks are worth making; firstly, the convexity condition
result recovers the Strichartz decay estimate for the wave
equation, since that clearly satisfies such a condition. Secondly,
the convexity condition is an important restriction on the
geometry of the characteristic roots that affects the $L^p-L^q$
decay rate; hence, in the case of an $m^{\text{th}}$ order
operator with lower order terms we must expect some geometrical
conditions on the characteristic roots to affect the decay
rate of solutions.

\section{Main results}\label{section:ests}
We will now turn to
analysing the conditions under which we can
obtain $L^p-L^q$ decay estimates for the general $m^{\text{th}}$
order linear, constant coefficient, strictly hyperbolic Cauchy
problem
\begin{equation}\label{EQ:standardCP(repeat)}
\left\{\begin{aligned}& L(D_t,D_x)\equiv \pat^m
u+\sum_{j=1}^{m}P_{j}(\pax)\pat^{m-j}u+
\sum_{l=0}^{m-1}\sum_{\abs{\al}+r=l}
c_{\al,r}\pax^\al\pat^ru=0,\quad t>0,\\
&\pat^lu(0,x)=f_l(x)\in C_0^{\infty}(\R^n),\quad l=0,\dots,m-1,\;
x\in\R^n\,.
\end{aligned}\right.
\end{equation}
Results of this section will
show how different behaviours of the
characteristic roots $\tau_1(\xi),\dots,\tau_m(\xi)$ affect the
rate of decay that can be obtained. 
As in the introduction, 
the symbol $P_j(\xi)$ of $P_j(\pax)$ is a  
homogeneous polynomial of order~$j$, and the $c_{\al,r}$
are constants. The differential operator in
the first line of
\eqref{EQ:standardCP(repeat)} will be denoted by $L(D_t,D_x)$
and its symbol by $L(\tau,\xi)$. The principal part of $L$
is denoted by $L_m$. Thus, $L_m(\tau,\xi)$ is a 
homogeneous polynomial of order $m$.
In the subsequent analysis, ideally, of course, we would
like to have conditions on the lower order terms for different
rates of decay; in Section~\ref{CHAP5} we shall give some results
in this direction. For now, though, we concentrate on conditions
on the characteristic roots.

First of all, it is natural to impose the stability condition,
namely that for all $\xi\in\Rn$ we have 
\begin{equation}\label{EQ:imtau>=0}
\Im\tau_k(\xi)\ge0\quad\text{for }k=1,\dots,m\,;
\end{equation}
this is equivalent to requiring the characteristic polynomial of
the operator to be stable at all points $\xi\in\R^n$, and thus
cannot be expected to be lifted. In fact, certain microlocal
decay estimates are possible even without this condition
if the supports of the Fourier transforms of the Cauchy data
are contained in the set where condition
\eqref{EQ:imtau>=0} holds. However, this restriction is
only technical so we may assume \eqref{EQ:imtau>=0}
without great loss of generality since otherwise
no time decay of solution can be expected.

Also, it is sensible to divide the
considerations of how characteristic roots behave into two parts:
their behaviour for large values of~$\abs{\xi}$  and for bounded
values of~$\abs{\xi}$. These two cases are then subdivided
further; in particular the following are the key properties to
consider:
\begin{itemize}
\item multiplicities of roots (this only occurs in the case of
bounded frequencies~$\abs{\xi}$);
\item whether roots lie on the real axis or are separated
from it;
\item behaviour as $\abs{\xi}\to\infty$ (only in the case of
large~$\abs{\xi}$);
\item how roots meet the real axis (if they do);
\item properties of the Hessian of the root, $\Hess\tau_k(\xi)$;
\item a convexity-type condition, as in the case of homogeneous
roots (Section~\ref{SEC:homogoperators}).
\end{itemize}

For some frequencies away from multiplicities we can 
actually establish independently interesting
estimates for the corresponding 
oscillatory integrals that contribute to
the solution. Around multiplicities we need to 
take extra care of the structure of solutions. This
will be done by dividing the frequencies into zones each of
which will give a certain decay rate. Combined together
they will yield the total decay rate for solution to
\eqref{EQ:standardCP(repeat)}. Several theorems below will
deal with integrals of the form
\begin{equation}\label{EQ:term}
\int_\Rn e^{i(x\cdot\xi+\tau(\xi)t)} a(\xi)\chi(\xi) d\xi,
\end{equation}
which appear in representations of solutions to Cauchy
problem \eqref{EQ:standardCP(repeat)} as kernels
of propagators, where $a(\xi)$ is a suitable amplitude and
$\chi(\xi)$ is a cut-off to a corresponding zone,
which may be bounded or unbounded.
Solution to the Cauchy problem \eqref{EQ:standardCP(repeat)}
can be written in the form
$$
u(t,x)=\sum_{j=0}^{m-1} E_j(t) f_j(x),
$$
where propagators $E_j(t)$ are defined by
\begin{equation}\label{EQ:soln}
E_j(t)f(x)=\int_{\Rn}e^{ix\cdot\xi}\Big(\sum_{k=1}^m
e^{i\tau_k(\xi)t}A_j^k(t,\xi)\Big)
\cutoffM(\xi)\widehat{f}(\xi)\,d\xi\,,
\end{equation}
with suitable amplitudes $A_j^k(t,\xi)$. 
In the areas where roots are simple, phases and
amplitudes are smooth, and
we can analyse the sum \eqref{EQ:soln} termwise,
reducing the analysis to integrals of the form
\eqref{EQ:term}. In the case
of multiple characteristics we will group terms in
\eqref{EQ:soln} in a special way to obtain suitable
decay estimates. Below we will give results for
decay rates dependent on the different qualitative behaviours of the
characteristic roots.

\subsection{Away from the real axis: exponential decay}
\label{SEC:21}
We begin by looking at the zone where roots are
separated from the real axis. If the roots are smooth,
we can analyse solution \eqref{EQ:soln} termwise:

\begin{thm}\label{THM:expdecay}
Let $\tau:U\to\C$ be a smooth function, $U\subset\R^n$ open.
Let $a\in S^{-\mu}_{1,0}(U)$, i.e. assume that 
$a=a(\xi)\in C^\infty(U)$ satisfies
$|\partial_\xi^\alpha a(\xi)|\leq C_\alpha (1+|\xi|)^
{-\mu-|\alpha|},$ for all $\xi\in U$ and all multi-indices $\alpha$.
Let $\chi\in S^0_{1,0}(\Rn)$ be such that
$\chi=0$ outside $U$.
Assume further that\textup{:}
\begin{enumerate}[leftmargin=*,label=\textup{(\roman*)}]
\item there exists $\de>0$ such that $\Im\tau(\xi)\ge\de$ for all
$\xi\in U$\textup{;}
\item $\abs{\tau(\xi)}\le C\brac{\xi}$ for all $\xi\in U$.
\end{enumerate}
Then for all $t\geq 0$ we have
\begin{equation}\label{EQ:expdecay}
\normBig{D^r_tD_x^\al\Big(\int_{\R^n}e^{i(x\cdot\xi+\tau(\xi)t)}
a(\xi)\chi(\xi)\widehat{f}(\xi)\,d\xi\Big)}_{L^q
(\R_x^n)} \le Ce^{-\de
t}\norm{f}_{W_p^{N_p+\abs{\al}+r-\mu}}\,,
\end{equation}
where $\frac{1}{p}+\frac{1}{q}=1$, $1<p\le 2$, $N_p\ge
n\big(\frac{1}{p}-\frac{1}{q}\big)$, $r\ge0$,~$\al$ a multi-index
and $f\in C_0^\infty(\R^n)$. If $p=1$, we take $N_1>n$.

Moreover, let us assume that equation $L(\tau,\xi)=0$
has only simple roots $\tau_k(\xi)$ which satisfy
condition {\rm (i)} above, in the open set 
$U\subset\Rn$,
for all $k=1,\ldots,m.$
Then solution $u$ to \eqref{EQ:standardCP(repeat)} 
satisfies
\begin{equation}\label{EQ:expdecayu}
||D_t^r D_x^\alpha  \chi(D)u(t,\cdot)||_{L^q(\R_x^n)}
\leq Ce^{-\delta t}\sum_{l=0}^{m-1} 
||f_l||_{W_p^{N_p+|\alpha|+r-l}},
\end{equation}
where $1\leq p\leq 2$, $\frac1p+\frac1q=1$, and 
$N_p, r, \alpha$ are as above.
\end{thm}
The proof of Theorem \ref{THM:expdecay} will be given
in Sections \ref{SEC:pfth21} and \ref{SEC:bddxiawayfromaxis}.
Note also that if we omit assumption (ii) in Theorem
\ref{THM:expdecay}, estimate \eqref{EQ:expdecay} with
$r=0$ still holds.  In the case of \eqref{EQ:expdecayu},
it can be shown
(see Proposition \ref{prop:roots-r}) that characteristic
roots of operator $L(D_t,D_x)$ in 
\eqref{EQ:standardCP(repeat)} satisfy (ii).

We also note, that we may have different
norms on the right hand side of \eqref{EQ:expdecayu}.
For example, we will show in Section \ref{SEC:pfth21},
that under conditions of Theorem \ref{THM:expdecay}
we also have the following estimate:
\begin{equation}\label{EQ:expdecayu-l2}
||D_t^r D_x^\alpha  \chi(D)u(t,\cdot)||_{L^q(\R_x^n)}
\leq Ce^{-\delta t}\sum_{l=0}^{m-1} 
||f_l||_{W_2^{N_q^\prime+|\alpha|+r-l}},
\end{equation}
where $1< p\leq 2$, $\frac1p+\frac1q=1$,
$N_q^\prime\geq \frac{n}{2}(\frac1p-\frac1q)$, and
$N_\infty^\prime>\frac{n}{2}$ for $p=1$.
Estimate \eqref{EQ:expdecayu-l2} will follow from 
\eqref{eq:Matsumura} and Proposition 
\ref{PROP:rootsawayfromaxis} by interpolation.
In turn, interpolating between \eqref{EQ:expdecayu}
and \eqref{EQ:expdecayu-l2}, we can obtain similar
$L^p-L^q$ estimates for all 
intermediate $p$ and $q$.

To be able to derive
time decay in the case of multiple roots,
we will group terms in
\eqref{EQ:soln} in the following way.
%For $k\not=l$ and $\epsilon>0$, we define 
%$\curlyM_{kl},\curlyM_{kl}^\ep \subset\R^n$ by 
%\begin{equation*}
%\curlyM_{kl}:=\set{\xi\in\R^n: \tau_k(\xi)=\tau_l(\xi)},\;\;
%\curlyM_{kl}^\ep:=\set{\xi\in\R^n:\dist(\xi,\curlyM_{kl})<\ep}.
%\end{equation*}
Assume that roots $\tau_1(\xi),\dots,\tau_L(\xi)$ 
coincide on a set contained in some~$\curlyM$, that is 
$\curlyM\supset \set{\tau_1(\xi)=\dots=\tau_L(\xi)}.$
For $\ep>0$, we define
$\curlyM^\ep:=\set{\xi\in\R^n:\dist(\xi,\curlyM)<\ep}.$
Choose $\ep>0$ so that these roots 
$\tau_1(\xi),\dots,\tau_L(\xi)$ do not
intersect with any of the other roots 
$\tau_{L+1}(\xi),\dots,\tau_m(\xi)$
in $\curlyM^\ep$.
If different numbers of roots intersect in different sets,
we can apply the following theorem to such sets one by one.
We note that by the strict hyperbolicity $\curlyM^\ep$ is bounded.
Here we will estimate the sum
\begin{equation}\label{EQ:solnpart}
\int_{\curlyM^\ep}e^{ix\cdot\xi}\Big(\sum_{k=1}^L
e^{i\tau_k(\xi)t}A_j^k(t,\xi)\Big)
\cutoffM(\xi)\widehat{f}(\xi)\,d\xi\,.
\end{equation}

\begin{thm}\label{THM:expdecay2}
Let the sum \eqref{EQ:soln} be the solution to the Cauchy
problem \eqref{EQ:standardCP(repeat)}.
Assume that roots $\tau_1(\xi),\dots,\tau_L(\xi)$ 
coincide in a set contained in~$\curlyM$ 
and do not intersect other roots
in the set $\curlyM^\ep$.
Let $\chi\in C_0^\infty(\curlyM^\ep)$. 
Assume that there exists $\de>0$ such that
$\Im\tau_k(\xi)\ge\de$ for all $\xi\in \curlyM^\ep$ and
$k=1,\ldots,L.$

Then for all $t\geq 0$ we have
%\begin{multline*}
$$
\normBig{D^r_tD_x^\al\Big(\int_{\curlyM^\ep} e^{ix\cdot\xi}
\Big(\sum_{k=1}^L e^{i\tau_k(\xi)t}A_j^k(t,\xi)\Big)
\cutoffM(\xi)\widehat{f}(\xi)\,dx\Big)}_{L^q(\R^n_x)}\\ \le
C(1+t)^{L-1}e^{-\de t}\norm{f}_{L^p}\,,
$$
%\end{multline*}
where $\frac{1}{p}+\frac{1}{q}=1$, $1\le p\le2$. 
\end{thm}
Thus, if characteristic roots are separated from the real
axis on the support of some $\chi\in C_0^\infty(\R^n)$,
we can separate the solution \eqref{EQ:soln} into
groups of multiple roots for which the $L^p-L^q$ norms still decay
exponentially as stated in Theorem \ref{THM:expdecay2}.
We also note that since $\curlyM^\epsilon$ is bounded,
assumption (ii) of Theorem \ref{THM:expdecay} is automatically
satisfied and, therefore, it is omitted in the formulation 
of Theorem \ref{THM:expdecay2}. Theorem \ref{THM:expdecay2}
will be proved in Section \ref{SEC:phasefnsepfromaxis}.

\subsection{Roots with non-degeneracies}
\label{SEC:22}
The following case that we consider is the one of roots satisfying
certain non-degeneracy conditions.
These may be conditions on the Hessian, convexity
conditions, or simply the information on the index
of the corresponding level surfaces. In this section
we will give the corresponding statements. We always
assume the stability condition
\eqref{EQ:imtau>=0} but no longer assume
that roots are separated from the real axis.

First we state the result for phases with the non-degenerate
Hessian. The behavior depends on critical
points $\xi^0$ with $\nabla\tau(\xi^0)=0$ and the
behavior of the Hessian at such points. As usual, we say that
the critical point $\xi^0$ is non-degenerate if the Hessian
$\Hess\tau(\xi^0)$ is non-degenerate.
\begin{thm}\label{THM:nondeghess}
Let $U\subset\Rn$ be a bounded open set, and let $\tau:U\to\C$ be
smooth and such that $\Im\tau(\xi)\geq 0$ for all $\xi\in U$.
Assume that 
there are some constants $C_0$ and $M$ such that
$|\det\Hess \tau(\xi)|\geq C_0(1+|\xi|)^{-M}$ for 
all $\xi\in U$.
Let $\chi\in S^0_{1,0}(\Rn)$ be such that $\chi=0$ outside $U$
and let $a\in S^{-\mu}_{1,0}(U)$.

Assume that $\tau$ has
only one non-degenerate
critical point in $U$, and that $U$ is sufficiently small. Then
there is a constant $C>0$ independent of the position of $U$
such that 
for all $t\geq 0$ we have
\begin{equation}\label{EQ:nondeghess}
\left|\left| \int_\Rn e^{i(x\cdot\xi+\tau(\xi)t)}
a(\xi)\chi(\xi) \widehat{f}(\xi) d\xi\right|\right|_{L^q(\R^n_x)}
\leq C(1+t)^{-\frac{n}{2}(\frac1p-\frac1q)}
||f||_{W_p^{N_p}},
\end{equation}
with $1\leq p\leq 2$, $\frac1p+\frac1q=1$, 
$N_p=\frac{M}{2}(\frac1p-\frac1q)-\mu.$
\end{thm}
For example, the case of the Klein--Gordon equation 
corresponds to $M=n+2$ in this theorem.
If we work with a fixed bounded set $U$, the
$||f||_{W_p^{N_p}}$ norm on the right hand side of
\eqref{EQ:nondeghess} can be replaced by $||f||_{L^p}$.
However, since we may also want to have estimate
\eqref{EQ:nondeghess} uniform over such $U$ (allowing it
to move to infinity while remaining to be of the same size),
we have the Sobolev norm in \eqref{EQ:nondeghess}.
From this point of view, we assume that $a$ behaves as a 
symbol in $U$ -- the meaning is that if the symbolic constants
here are uniform over the position of $U$, then also the
constant in \eqref{EQ:nondeghess} is uniform over such 
$a$ and $U$.

The condition that critical points are isolated and therefore
can be localised by different sets $U$ may follow
from certain properties of $\tau$ and will be discussed
in Section \ref{SEC:mainstatphasesection}, in particular see
Lemma \ref{lemma:critical-points} and remarks after it.
If, in addition, we take the size of $U$ uniform,
say of volume bounded by one, then constant $C$ in
\eqref{EQ:nondeghess} is also uniform over all such sets $U$.
We may also assume that if $\xi^0$ is a critical point of
$\tau$, then $\Im\tau(\xi^0)=0$. Otherwise we would have
$\Im\tau(\xi^0)>0$ and so Theorem \ref{THM:expdecay}
would actually give the exponential decay rate.
The proof of this theorem is based on the stationary phase 
method and will be given in Section \ref{SEC:mainstatphasesection}. 
If we apply different versions of the stationary
phase method under different conditions, 
we can reach different conclusions here. For example, we also
have:
\begin{thm}\label{THM:nondeghess2}
Let $U\subset\Rn$ be a bounded open and let $\tau:U\to\C$ be
smooth and such that $\Im\tau(\xi)\geq 0$ for all $\xi\in U$.
Let $\chi\in S^0_{1,0}(\Rn)$ be such that $\chi=0$ outside $U$
and let $a\in S^{-\mu}_{1,0}(U)$.
Assume that $\tau$ has only one
critical point $\xi^0$ in $U$, and that $U$ is sufficiently
small.

Suppose that there are constants $C_0,M>0$ independent of
the size and position
of $U$ and of $\xi^0$, with the following conditions.
Suppose that
$\rank \Hess\tau(\xi^0)=k$, that this rank is attained
on an $k\times k$ submatrix $A(\xi^0)$ and 
that $|\det A(\xi^0)|\geq C_0(1+|\xi^0|)^{-M}$.
Then
for all $t\geq 0$ we have
\begin{equation*}\label{EQ:nondeghess2}
\left|\left| \int_\Rn e^{i(x\cdot\xi+\tau(\xi)t)}
a(\xi)\chi(\xi) \widehat{f}(\xi) d\xi\right|\right|_{L^q(\R^n_x)}
\leq C(1+t)^{-\frac{k}{2}(\frac1p-\frac1q)}
||f||_{W_p^{N_p}},
\end{equation*}
with $1\leq p\leq 2$, $\frac1p+\frac1q=1$, 
$N_p=\frac{M}{2}(\frac1p-\frac1q)-\mu.$
\end{thm}
The proof of this theorem is similar to
the proof of Theorem \ref{THM:nondeghess} once we
restrict to the set of $k$ variables (possibly after a 
suitable change) on which the rank of the Hessian is
attained on $A(\xi^0)$. 

This result can be improved dependent on further properties
of $A(\xi^0)$. For example, if $\rank A(\xi^0)=n-1$
and this is attained on variables $\xi_1,\ldots,\xi_{n-1}$,
the analysis reduces to the behaviour of the oscillatory
integral with respect to $\xi_n$. If the $l$-th derivative
of the phase with respect to $\xi_n$ is non-zero, we get
an additional decay by $t^{-1/l}.$ This follows from
the stationary phase method, see, for example
H\"ormander \cite[Section 7.7]{horm83I}, or from an
appropriate use of van der Corput lemma. We will not
formulate further statements here since they are quite
straightforward.

\bigskip
The next theorem is an estimate of oscillatory integrals
with real-valued phases under convexity condition.  
It will be shown in Proposition \ref{prop:roots-r}
(see also Proposition \ref{Prop:asymptotic})
that for large $\xi$ characteristic roots of
the Cauchy problem \eqref{EQ:standardCP(repeat)}
satisfy assumptions of these theorems given below,
if the homogeneous roots of the principal part satisfy
them. The convexity
condition is weaker than (but does not contain)
the condition that the
Hessian of $\tau$ is positive definite and the result
can be compared with Theorem \ref{THM:nondeghess},
dependent on suitable properties of roots.

Let us first give
the necessary definitions.
Given a smooth function $\tau:\R^n\to\R$ and $\la\in\R$, set
\begin{equation*}
\Si_\la\equiv\Si_\la(\tau):=\set{\xi\in\R^n:\tau(\xi)=\la}\,.
\end{equation*}
In the case where $\tau(\xi)$ is homogeneous of order $1$
and $\tau\in C^\infty(\Rn\backslash 0)$, 
we will also write
$\Si_\tau:=\Si_1(\tau)$---for such $\tau$, we then have
$\Si_\la(\tau)=\la\Si_\tau$.
There should be no confusion in this notation since
we always reserve letters $\phi, \tau$ for phases and
$\lambda$ for the real number.
\begin{defn}\label{DEF:convexitycondition}
A smooth function $\tau:\R^n\to\R$ is said to satisfy the
\emph{convexity condition} if surface~$\Si_\la$ is convex for each
$\la\in\R$. Note that the empty set and the point set
are considered to be convex.
\end{defn}
If the Gaussian curvatures of $\Sigma_\lambda$ never vanish,
$\Sigma_\lambda$ is automatically convex (the converse is
not true). This curvature condition corresponds
to the case $k=n-1$ in Theorem \ref{THM:nondeghess2}.
Another important notion is that of the \emph{maximal order of
contact} of a hypersurface, similar to the one in
Section \ref{SEC:homogoperators}:
\begin{defn}\label{DEF:ga(Si)}
Let $\Si$ be a hypersurface in~$\R^n$ \textup{(}i.e.\ 
a manifold of
dimension $n-1$\textup{)}\textup{;} let $\si\in\Si$\textup{,} and
denote the tangent plane at~$\si$ by~$T_\si$. Now let~$P$ be a 
2--dimensional plane
containing the normal to~$\Si$ at~$\si$ and denote the order of the
contact between the line $T_\si\cap P$ 
and the curve $\Si\cap P$ by
$\ga(\Si;\si,P)$. Then set
\begin{equation*}
\ga(\Si):=\sup_{\si\in\Si}\sup_P\ga(\Si;\si,P)\,.
\end{equation*}
\end{defn}
\begin{examples}\mbox{}
\begin{enumerate}[label=(\alph*)]
\item $\ga(\Sn)=2$, as $\ga(\Sn;\si,P)=2$ for all $\si\in
\Sn$ and all planes $P$ containing $\si$ and the origin.
\item If $\va_l(\xi)$ is a characteristic root of an $m^{\text{th}}$
order homogeneous strictly hyperbolic constant coefficient operator,
then $\ga(\Si_{\va_l})\le m$; see~\cite{sugi96} for a proof of
this.
\end{enumerate}
\end{examples}
Now we can formulate the corresponding theorem.
\begin{thm}\label{THM:convexsp}
Suppose $\tau:\R^n\to\R$ satisfies the convexity condition
and let $\chi\in C^\infty(\Rn)$ \textup{;}
furthermore\textup{,} on $\supp\chi$, we assume\textup{:}
\begin{enumerate}
[label=\textup{(}\textup{\roman*}\textup{)},leftmargin=*]
\item\label{HYP:realtauisasymbol} for all multi-indices $\al$ there
exists a constant $C_\al>0$ such that
\begin{equation*}
\abs{\pa_\xi^\al\tau(\xi)}\le
C_\al\brac{\xi}^{1-\abs{\al}};
\end{equation*}
\item\label{HYP:realtauboundedbelow} there exist constants $M,C>0$
such that for all $\abs{\xi}\ge M$ we have $\abs{\tau(\xi)}\ge
C\abs{\xi}$\textup{;}
\item\label{HYP:realderivativeoftaunonzero} there exists a constant
$C>0$ such that $\abs{\pa_\om\tau(\la\om)}\ge C$ for all $\om\in
\Snm$\textup{,} $\la>0$\textup{;} in particular,
$\abs{\grad\tau(\xi)}\ge C$ for all
$\xi\in\R^n\setminus\set{0}$\textup{;}
\item \label{HYP:reallimitofSi_la} there exists a constant $R_1>0$
such that\textup{,} for all $\la>0$\textup{,}
\[
%\Si'_\la:=
\frac{1}{\la}\Si_\la(\tau)
\equiv \frac{1}{\la}\set{\xi\in\Rn: \tau(\xi)=\la}
\subset B_{R_1}(0)\,.\]
\end{enumerate}
Also\textup{,} set $\ga:=\sup_{\la>0}\ga(\Si_\la(\tau))$ and assume
this is finite. Let 
$a_j=a_j(\xi)\in S^{-j}_{1,0}$ be a symbol of order
$-j$ of type $(1,0)$ on $\R^n$. 
Then for all $t\geq 0$ we have the estimate
\begin{equation}\label{EQ:resultforconvonaxis-res}
\normBig{\int_{\R^n}e^{i(x\cdot\xi+\tau(\xi)t)}
a_j(\xi)\chi(\xi)\widehat{f}(\xi)\,d\xi}_{L^q(\R^n_x)}\le
C(1+t)^{-\frac{n-1}{\ga}\big(\frac{1}{p}-\frac{1}{q}\big)}
\norm{f}_{W_p^{N_{p,j,t}}}\,,
\end{equation}
where $\frac1p+\frac1q=1$, $1<p\le 2$, 
and the Sobolev order satisfies
$N_{p,j,t}\ge n(\frac1p-\frac1q)-j$ for
$0\leq t<1$, and
$N_{p,j,t}\ge \left(n-\frac{n-1}{\gamma}\right)
(\frac1p-\frac1q)-j$ for
$t\geq 1$.
\end{thm}
Theorem \ref{THM:convexsp} will be proved in
Section \ref{SEC:convexity}, where 
estimate \eqref{EQ:resultforconvonaxis-res} will
follow by interpolation from the $L^2-L^2$ estimate combined
with $L^1-L^\infty$ cases given in
\eqref{EQ:resultforconvonaxis} for small $t$,
and in \eqref{EQ:convestL1Linftylarget} for large $t$.
See those estimates also for the case of $p=1$ in estimate
\eqref{EQ:resultforconvonaxis-res}.
The estimate for large times will follow from
Theorem \ref{THM:sugimoto/randolargument},
which gives the $L^\infty$-estimate for the kernel of
\eqref{EQ:resultforconvonaxis-res}.
As another consequence of 
Theorem \ref{THM:sugimoto/randolargument}, 
we will also have the following 
estimate:

\begin{cor}\label{cor:convex}
Under conditions of Theorem \ref{THM:convexsp}
with $\chi\equiv 1$, assume that $a\in C_0^\infty(\Rn)$. 
Then for all 
$x\in\R^n$ and $t\geq 0$ we have the estimate
\begin{equation}\label{EQ:sugimotothmest0}
\absBig{\int_{\R^n}e^{i(x\cdot\xi+ \tau(\xi)t)}a(\xi)
\,d\xi} \le C(1+t)^{-\frac{n-1}{\ga}}\,.
\end{equation}
\end{cor}
In Proposition \ref{prop:roots-r} we show that properties
(i)--(iv) of Theorem \ref{THM:convexsp} are satisfied
for characteristic roots of $L(D_t,D_x)$ in
\eqref{EQ:standardCP(repeat)}, while in Lemma
\ref{LEM:convexityconstantconv} we will show that
the index $\gamma$ is also finite, both for large frequencies.

Now we turn to the case without convexity.
As in the case of the homogeneous operators (see
Introduction, Section~\ref{SEC:homogoperators}) we introduce
an analog of the order of contact also in the case where 
the convexity condition does not hold.
\begin{defn}
Let~$\Si$ be a hypersurface in~$\R^n$; set
\begin{equation*}
\ga_0(\Si):=\sup_{\si\in\Si}\inf_P\ga(\Si;\si,P)\le \ga(\Si),\,
\end{equation*}
where $\ga(\Si;\si,P)$ is as in Definition~\ref{DEF:ga(Si)}.
\end{defn}
\begin{rem}\mbox{}
\begin{enumerate}[label=(\alph*),leftmargin=*]
\item When $n=2$, $\ga_0(\Si)=\ga(\Si)$;
\item If $p(\xi)$ is a polynomial of order $m$,
$\Si=\set{\xi\in\R^n:p(\xi)=0}$ is compact, and
$\nabla p(\xi)\not=0$ on $\Sigma$, then
$\ga_0(\Si)\le\ga(\Si)\le m$; this is useful when applying the
result below to hyperbolic differential equations and is proved in
\cite{sugi96}.
\end{enumerate}
\end{rem}

%In the following theorem we will also allow $\tau$ to be 
%complex valued.
%\begin{thm}\label{THM:noncovexargument-sp}
%Suppose $\tau:\R^n\to\C$ is a smooth function with
%$\Im\tau(\xi)\ge 0$. Let $\chi\in C_0^\infty(\Rn)$;
%furthermore, on $\supp\chi$, we assume:
%\begin{enumerate}[label=\textup{(}\textup{\roman*}\textup{)}]
%\item\label{HYP:nonconcplxtauisasymbol} for all multi-indices $\al$
%there exist constants $C_\al, C_\al'>0$ such that
%\begin{equation*}
%\abs{\pa_\xi^\al\Re\tau(\xi)}\le C_\al\brac{\xi}^{1-\abs{\al}}
%\end{equation*}
%and
%\begin{equation*}
%\abs{\pa_\xi^\al\Im\tau(\xi)}\le C_\al'\brac{\xi}^{-\abs{\al}};
%\end{equation*}
%\item\label{HYP:nonconcplxtauboundedbelow} there exist constants
%$M,C>0$ such that for all $\abs{\xi}\ge M$ we have
%$\abs{\Re\tau(\xi)}\ge C\abs{\xi}$\textup{;}
%\item\label{HYP:nonconcplxderivativeoftaunonzero} there exists a
%constant $C_0>0$ such that 
%$\abs{\pa_\om\Re\tau(\la\om)}\ge C_0$ for
%all $\om\in \Snm$ and $\la>0$\textup{;}
%\item\label{HYP:nonconcplxlimitofSi_la} there exists a constant
%$R_1>0$ such that\textup{,} for all $\la>0$\textup{,}
%\begin{equation*}
%\frac{1}{\la}\set{\xi\in\R^n:\Re\tau(\xi)=\la}\subset
%B_{R_1}(0)\,.
%\end{equation*}
%\end{enumerate}
%Set $\ga_0:=\sup_{\la>0}\ga_0(\Si_\la(\Re\tau))$ and assume it is
%finite and let~$a(\xi)$ be a symbol of order $\frac{1}{\ga_0}-n$
%of type $(1,0)$ on~$\R^n$. Then\textup{,} the following estimate
%holds for all  $x\in\R^n$\textup{,}
%$t>0$\textup{:}
%\begin{equation*}
%\absBig{\int_{\R^n}e^{i(x\cdot\xi+ \tau(\xi)t)}a(\xi)
%\chi(\xi)\,d\xi} \le Ct^{-1/\ga_0}\,.
%\end{equation*}
%\end{thm}

\begin{thm}\label{THM:noncovexargument-sp}
Suppose $\tau:\R^n\to\R$ is a smooth function. 
Let $\chi\in C^\infty(\Rn)$;
furthermore, on $\supp\chi$, we assume:
\begin{enumerate}[label=\textup{(}\textup{\roman*}\textup{)}]
\item\label{HYP:nonconcplxtauisasymbol} for all multi-indices $\al$
there exist constants $C_\al>0$ such that
\begin{equation*}
\abs{\pa_\xi^\al\tau(\xi)}\le C_\al\brac{\xi}^{1-\abs{\al}};
\end{equation*}
\item\label{HYP:nonconcplxtauboundedbelow} there exist constants
$M,C>0$ such that for all $\abs{\xi}\ge M$ we have
$\abs{\tau(\xi)}\ge C\abs{\xi}$\textup{;}
\item\label{HYP:nonconcplxderivativeoftaunonzero} there exists a
constant $C>0$ such that 
$\abs{\pa_\om\tau(\la\om)}\ge C$ for
all $\om\in \Snm$ and $\la>0$\textup{;}
\item\label{HYP:nonconcplxlimitofSi_la} there exists a constant
$R_1>0$ such that\textup{,} for all $\la>0$\textup{,}
\begin{equation*}
\frac{1}{\la}\set{\xi\in\R^n:\tau(\xi)=\la}\subset
B_{R_1}(0)\,.
\end{equation*}
\end{enumerate}
Set $\ga_0:=\sup_{\la>0}\ga_0(\Si_\la(\tau))$ and assume it is
finite. 
Let $a_j=a_j(\xi)\in S^{-j}_{1,0}$ be a symbol of order
$-j$ of type $(1,0)$ on $\R^n$. 
Then for all $t\geq 0$ we have the estimate
\begin{equation*}\label{EQ:resultfornonconvonaxis-res}
\normBig{\int_{\R^n}e^{i(x\cdot\xi+\tau(\xi)t)}
a_j(\xi)\chi(\xi)\widehat{f}(\xi)\,d\xi}_{L^q(\R^n_x)}\le
C(1+t)^{-\frac{1}{\ga_0}\big(\frac{1}{p}-\frac{1}{q}\big)}
\norm{f}_{W_p^{N_{p,j,t}}}\,,
\end{equation*}
where $\frac1p+\frac1q=1$, $1<p\le 2$, 
and the Sobolev order satisfies
$N_{p,j,t}\ge n(\frac1p-\frac1q)-j$ for
$0\leq t<1$, and
$N_{p,j,t}\ge \left(n-\frac{1}{\gamma_0}\right)
(\frac1p-\frac1q)-j$ for
$t\geq 1$.
\end{thm}
The proof of Theorem \ref{THM:noncovexargument-sp}
will be given in Section \ref{SEC:nonconvex}. As in the convex
case, as a consequence of 
estimates for the kernel on Theorem \ref{THM:noncovexargument},
we also have the following statement:

\begin{cor}\label{cor:nonconvex}
Under conditions of Theorem \ref{THM:noncovexargument-sp}
with $\chi\equiv 1$,
assume that $a\in C_0^\infty(\Rn)$. Then
for all 
$x\in\R^n$ and $t\geq 0$ we have the estimate for the kernel:
\begin{equation*}
\absBig{\int_{\R^n}e^{i(x\cdot\xi+ \tau(\xi)t)}a(\xi)
\,d\xi} \le C(1+t)^{-\frac{1}{\ga_0}}\,.
\end{equation*}
\end{cor}
Again, in Proposition \ref{prop:roots-r} we show that properties
(i)--(iv) of Theorem \ref{THM:noncovexargument-sp} are satisfied
for characteristic roots of $L(D_t,D_x)$ in
\eqref{EQ:standardCP(repeat)}, while in Lemma
\ref{LEM:ga0conv} we will show that
the index $\gamma_0$ is also finite, both for large frequencies.

As a corollary and an example of these theorems, 
we get the following possibilities
of decay for parts of solutions with roots on the axis.
We can use a cut-off function $\chi$ to microlocalise
around points with different qualitative behaviour
(hence we also do not have to worry about Sobolev orders).

\begin{cor}\label{COR:realaxis}
Let $\Omega\subset\Rn$ be an open set and let
$\tau:\Omega\to\R$ be a smooth real valued function. 
Let $\chi\in C_0^\infty(\Omega).$
Let us make the following choices of $K(t)$, depending on
which of the following conditions are satisfied on $\supp\chi$.
\begin{itemize}
\item[{\rm (1)}] If $\det\Hess \tau(\xi)\not=0$ for all 
$\xi\in\Omega$, we set 
$K(t)=(1+t)^{-\frac{n}{2}(\frac1p-\frac1q)}.$
\item[{\rm (2)}] If $\rank\Hess \tau(\xi)=n-1$ for all 
$\xi\in\Omega$, we set 
$K(t)=(1+t)^{-\frac{n-1}{2}(\frac1p-\frac1q)}.$
\item[{\rm (3)}] If $\tau$ satisfies the convexity condition
with index $\gamma$, we set 
$K(t)=(1+t)^{-\frac{n-1}{\gamma}(\frac1p-\frac1q)}.$
\item[{\rm (4)}] If $\tau$ does not satisfy the convexity
condition but has non-convex index $\gamma_0$, we set 
$K(t)=(1+t)^{-\frac{1}{\gamma_0}(\frac1p-\frac1q)}.$
\end{itemize}
Assume in each case that other assumptions of the corresponding
Theorems {\rm \ref{THM:nondeghess}--\ref{THM:noncovexargument-sp}}
are satisfied. Let $1\leq p\leq 2, \frac{1}{p}+\frac{1}{q}=1$.
Then for all $t\geq 0$ we have
\begin{equation*}\label{EQ:realaxis}
\left|\left|\int_\Rn e^{i(x\cdot\xi+\tau(\xi)t)} a(\xi) \chi(\xi) 
\widehat{f}(\xi) d\xi
\right|\right|_{L^q(\R^n_x)}
\leq C K(t)||f||_{L^p(\Rn)}.
\end{equation*}
\end{cor}
We note that no derivatives appear in the $L^p$--norm of $f$
because the support of $\chi$ is bounded. In general, there
are different ways to ensure the convexity condition for 
$\tau$. Thus, we can say that the principal part $L_m$
of operator $L(D_t,D_x)$ in \eqref{EQ:standardCP(repeat)}
satisfies the convexity condition if all Hessians
$\va^{\prime\prime}_l(\xi)$, $l=1,\ldots,m$, are
semi--definite for all $\xi\not=0$. In this case it was
shown by Sugimoto in \cite{sugi94} that there exists a 
linear function $\alpha(\xi)$ such that
$\widetilde{\va_l}=\va_l+\alpha$ have convex level
sets $\Sigma(\widetilde{\va_l})$, and we have
$\gamma(\Sigma(\widetilde{\va_l}))\leq 2\left[
\frac{m}{2}\right].$ For large frequencies,
perturbation arguments imply that the same must be
true for $\Sigma_\lambda(\tau_l)$, for sufficiently
large $\lambda$. If we now assume that
$\Sigma_\lambda(\tau_l)$ are also convex for small $\lambda$,
then $\tau_l$ will satisfy the convexity conditions.
Alternatively, if they do not satisfy the convexity
condition for small $\lambda$, we can cut-off this regions
and analyse the decay rates by other methods developed in
this paper.

\subsection{Roots meeting the real axis}
\label{SEC:23}
In this section we will present the results for characteristic
roots (or phase functions) in the upper complex plane
near the real axis, that become real at some point or in
some set. 

For $\curlyM\subset\R^n$, denote
$\curlyM^\ep=\set{\xi\in\R^n: \dist (\xi,\curlyM)<\ep}$
as before.
The largest number $\nu\in\N$ such that
$\meas(\curlyM^{\ep})\le C\ep^{\nu}$ 
for all sufficiently small
$\ep>0$, will be denoted by $\codim\curlyM$, and 
we will call it the codimension of $\curlyM$.

We will say that {\em the root $\tau_k$ meets the real axis
at $\xi^0$ with order $s_k$} if $\Im\tau_k(\xi^0)=0$ and if
there exists a constant
$c_0>0$ such that
\begin{equation*}
c_{0}\abs{\xi-\xi^0}^{s_k}\le{\Im\tau_k(\xi)}\,,
\end{equation*}
for all~$\xi$ sufficiently near~$\xi^0$. Here we may recall
that in \eqref{EQ:imtau>=0}
we already assumed $\Im\tau_k(\xi)\geq 0$ for all $\xi$.

More generally, if the root $\tau_k$ meets the axis on the set
$Z_k=\set{\xi\in\Rn: \Im\tau_k(\xi)=0}$, we will say that 
{\em it meets the axis with order $s$} if
\begin{equation*}
c_{0}\dist (\xi, Z_k)^s\le{\Im\tau_k(\xi)}\,.
\end{equation*}
We will localise around each connected component of
$Z_k$, e.g. around each point of $Z_k$,
if it is a union of isolated points. As usual, when we talk
about multiple roots intersecting in a set $\curlyM$,
we adopt the terminology introduced in Section
\ref{SEC:21}.
Since we are dealing with strictly hyperbolic equations,
roots can meet each other only for bounded frequencies,
so we may assume that set $\curlyM$ is
bounded.

\begin{thm}\label{TH:multonaxis}
Assume that the characteristic roots
$\tau_1(\xi),\dots,\tau_L(\xi)$ intersect in the 
$C^1$ set~$\curlyM$
of codimension~$\ell$. Assume also that they 
meet the real axis in~$\curlyM$ with the finite
orders $\leq s$, i.e. that
\begin{equation*}
c_{0}\dist (\xi, \curlyM)^s\le{\Im\tau_k(\xi)}\,,
\end{equation*}
for some $c_0>0$ and all $k=1,\ldots,L$.
Assume that \eqref{EQ:soln} is the solution of the 
Cauchy problem \eqref{EQ:standardCP(repeat)} and we look
at its part \eqref{EQ:solnpart}.
Let $\chi\in C_0^\infty(\curlyM^\ep)$ for sufficiently
small $\ep>0$.
Then for all $t\geq 0$ we have
\begin{multline}\label{EQ:est-mult}
\normBig{D^r_tD_x^\al\Big(\int_{\curlyM^\ep} e^{ix\cdot\xi}
\Big(\sum_{k=1}^L e^{i\tau_k(\xi)t}A_j^k(t,\xi)\Big)
\cutoffM(\xi)\widehat{f}(\xi)\,d\xi\Big)}_{L^q(\R^n_x)}\\ \le
C\bract{t}^{-\frac{\ell}{s}\big(\frac{1}{p}-\frac{1}{q}\big)
+L-1}\norm{f}_{L^p}\,,
\end{multline}
where $\frac{1}{p}+\frac{1}{q}=1$, $1\le p\le2$. 
\end{thm}
We assume $\ep>0$ to be small enough to make sure that the
type of behaviour assumed in the theorem is the only one
that takes place
in $\curlyM^\ep$. In the complement of $\curlyM^\ep$
we may use other theorems to analyse the decay rate.
Moreover, we assume that set $\curlyM$ is $C^1$. In fact,
it is usually Lipschitz, so in order to avoid to go into
depth about its structure and existence of almost everywhere
differentiable coordinate systems, we make the technical
$C^1$ assumption. The proof of
Theorem \ref{TH:multonaxis} will be given in Section
\ref{SEC:phasefnmeetsfinorder}.

Let us now give a special case of this theorem where simple
roots meet the axis at a point, so that we have $L=1$ and
$\ell=n$. The following statement is also global in frequency,
so we have the result in Sobolev spaces.

\begin{thm}
\label{THM:dissipative}
Consider the $m^{\text{th}}$ order strictly hyperbolic Cauchy
problem~\eqref{EQ:standardCP(repeat)} for operator
$L(D_t,D_x)$, with initial data 
%$f_j\in
%L^p\cap W^{\left[\frac{n}{2}\right]+1+\abs{\al}-j+r}_2$ 
$f_j\in W^{N_p+\abs{\al}+r-j}_p$,
for
$j=0,\dots,m-1$, where $1\le p\le2$ and $2\leq q\leq\infty$
are such that $\frac{1}{p}+\frac{1}{q}=1$,
$r\geq 0$ and
$\alpha$ is a multi-index. We assume that the Sobolev
index $N_p$ satisfies
$N_p\geq n(\frac1p-\frac1q)$ for $1<p\leq 2$ and
$N_1>n$ for $p=1$.

Assume that the
characteristic roots $\tau_1(\xi),\dots,\tau_m(\xi)$ 
of $L(\tau,\xi)=0$
satisfy $\Im\tau_k\geq 0$ for all $k$, 
and also the following conditions:
\begin{enumerate}[leftmargin=*,label=\textup{(H\arabic*)}]
\item for all $k=1,\ldots,m$, we have
$$\liminf_{\abs{\xi}\to\infty}\Im\tau_k(\xi)>0\,;$$
%\item there exists $R>0$ such that
%$\Im\tau_k(\xi)\ne0$ for all $\abs{\xi}\ge R$ and
%$k\in\set{1,\dots,m}$;
\item for each
$\xi^0\in\R^n$ there is at most one index $k$
for which $\Im\tau_k(\xi^0)=0$ and there exists
a constant $c>0$ such that
\begin{equation*}
\abs{\xi-\xi^0}^s\le c{\Im\tau_k(\xi)},
\end{equation*}
for $\xi$ in some neighbourhood of $\xi^0$.
Assume also that there are finitely many points
$\xi^0$ with $\Im\tau_k(\xi^0)=0$.
\end{enumerate}
Then the solution $u=u(t,x)$ to 
Cauchy problem \eqref{EQ:standardCP(repeat)} satisfies 
the following estimate for all $t\geq 0$:
%\begin{gather}
%\norm{\pat^r\pax^\al u(t,\cdot)}_{L^\infty}\le
%C_{\al,r}\bract{t}^{m-1-\frac{n}{2p}}\sum_{j=0}^{m-1}\big(
%\norm{f_j}_{W^{\left[\frac{n}{2}\right]+1+\abs{\al}-j+r}_2}
%+\norm{f_j}_{L^p}\big)\label{mainLinftyest}\\
%\norm{\pat^r\pax^\al u(t,\cdot)}_{L^2}\le
%C_{\al,r}\bract{t}^{m-1+\frac{n}{4}-\frac{n}{2p}}\sum_{j=0}^{m-1}\big(
%\norm{f_j}_{W^{\abs{\al}-j+r}_2}+\norm{f_j}_{L^p}\big)\label{mainL2est}
%\end{gather}%
%
%Moreover,
%for each $2<q<\infty$ and $1<p<2$ such that
%$\frac{1}{p}+\frac{1}{q}=1$, the following estimates holds for the
%solution $u=u(t,x)\,:$
\begin{equation}\label{EQ:LpLqest}
\norm{\pat^r\pax^\al u(t,\cdot)}_{L^q}\le
C_{\al,r}\bract{t}^{-\frac{n}{s}
(\frac1p-\frac1q)}\sum_{j=0}^{m-1}
\norm{f_j}_{W^{N_p+\abs{\al}+r-j}_p}.
\end{equation}
\end{thm}
Theorem \ref{THM:dissipative} is proved in Section
\ref{SEC:rootmeetingreal-sec}, where we will also
give microlocal versions of this result around points
$\xi^0$ from hypothesis (H2). In the complement of such
points, we have roots separated from the real axis,
so we get the exponential decay from Theorems
\ref{THM:expdecay} and \ref{THM:expdecay2}. Moreover,
in the exponential decay zone we may have different
versions of the estimate, for example we can use 
estimate \eqref{EQ:expdecayu-l2} there instead of
\eqref{EQ:expdecayu}. As a special case,
such estimate together with \eqref{EQ:bestdispest} below
(used with $s=s_1=2$),
we improve the indices in Sobolev spaces over $L^2$
for the dissipative wave equation in
\eqref{EQ:dwe} and \eqref{EQ:dwe2} compared to \cite{mats76}.

If conditions of Theorem \ref{THM:dissipative} hold only with
$\xi^0=0$, namely if $\Im\tau_k(\xi^0)=0$
implies $\xi^0=0$,
we will call the polynomial $L(\tau,\xi)$
{\em strongly stable}. Such polynomials will be discussed in more
detail in applications in
Section \ref{SEC:FokkerPlanck}. Now we will give some
improvements of \eqref{EQ:LpLqest} under additional
assumptions on the roots:

\begin{rem}\label{REM:dismore}
The order of time decay in Theorem \ref{THM:dissipative}
may be improved in the following cases, if we make
additional assumptions.
If, in addition, we assume that $\Im\tau_k(\xi^0)=0$ in (H2) 
implies that $\xi^0=0$, then we actually get the estimate
\begin{equation*}
\normBig{D^r_tD^\al_x
u(t,\cdot)}_{L^q(\R^n_x)}\le
C\bract{t}^{-\frac{n}{s}\big(\frac{1}{p}-\frac{1}{q}\big)
-\frac{\abs{\al}}{2}}
\sum_{j=0}^{m-1}
\norm{f_j}_{W^{N_p+\abs{\al}+r-j}_p}
\,,
\end{equation*}
where here and further in this remark $N_p$ is as in
Theorem \ref{THM:dissipative}.

Now, assume further that for all $\xi^0$ in (H2) we also have
the estimate
\begin{equation}\label{EQ:imest}
|\tau_k(\xi)|\leq
c_1\abs{\xi-\xi^0}^{s_1},
\end{equation}
with some constant $c_1>0$, for all $\xi$ sufficiently
close to $\xi^0$. 

If we have that $\Im\tau_k(\xi^0)=0$ in (H2) 
implies that we have \eqref{EQ:imest} around
such $\xi^0$, then we actually get
\begin{equation*}
\normBig{D^r_tD^\al_x
u(t,\cdot)}_{L^q(\R^n_x)} \le
C\bract{t}^{-\big(\frac{n}{s}\big)\big(\frac{1}{p}-\frac{1}{q}\big)
-\frac{r s_1}{s}}
\sum_{j=0}^{m-1}
\norm{f_j}_{W^{N_p+\abs{\al}+r-j}_p}
\,.
\end{equation*}

And finally, assume that for all $\xi^0$ such that
$\Im\tau_k(\xi^0)=0$ in (H2), we also have
$\xi^0=0$ and \eqref{EQ:imest} around
such $\xi^0$.
Then we actually get
\begin{equation}\label{EQ:bestdispest}
\normBig{D^r_tD^\al_x
u(t,\cdot)}_{L^q(\R^n_x)}\le
C\bract{t}^{-\frac{n}{s}\big(\frac{1}{p}-\frac{1}{q}\big)-
\frac{\abs{\al}}{s}-\frac{r s_1}{s}}
\sum_{j=0}^{m-1}
\norm{f_j}_{W^{N_p+\abs{\al}+r-j}_p}
\,.
\end{equation}
\end{rem}
Estimate \eqref{EQ:bestdispest} with $s=s_1=2$ gives the
decay estimate for the dissipative wave equation in
\eqref{EQ:dwe}.
The proof of this remark is given in
Remark \ref{rem:mats}. 

\bigskip
Moreover, there are other possibilities of {\em multiple roots
intersecting each other
while lying entirely on the real axis}. For example,
this is the case for the wave equation or for more general
equations with homogeneous symbols, when several roots meet
at the origin. In this case roots always lie on the real axis,
but they become irregular at the point of multiplicity, which 
is the origin for homogeneous roots. In the case
when lower order terms are presents, characteristics roots
are not homogeneous in general, so we can not eliminate
time from the estimates as was done 
in Section \ref{SEC:homogoperators}. It means that we have to
look at the structure of such multiple points by making cut-offs
around them and studying their structure in more detail. 
In particular, there is an interaction between low frequencies
and large times, which does not take place for homogeneous
symbols.
The detailed discussion of this topic and corresponding
decay rates will be determined in Section
\ref{SEC:phasefnliesonaxisbddxi}.

\subsection{Application to the Cauchy problem}
\label{SEC:24}
Putting together theorems from previous sections
we obtain the following conclusion about solutions to the
Cauchy problem \eqref{EQ:standardCP(repeat)}. We will first
formulate the following general result collecting statements
of previous sections, and then will explain how this result
can be used.

\begin{thm}\label{THM:overallmainthm}
Suppose $u=u(t,x)$ is the solution of the $m^{\text{th}}$ order
linear\textup{,} constant coefficient\textup{,} strictly hyperbolic
Cauchy problem~\eqref{EQ:standardCP(repeat)}. Denote the
characteristic roots of the operator by
$\tau_1(\xi),\dots,\tau_m(\xi)$, and assume that
$\Im\tau_k(\xi)\geq 0$ for all $k=1,\ldots,n$, and all
$\xi\in\Rn$.

We introduce two functions\textup{,} $K^{(\text{l})}(t)$ and
$K^{(\text{b})}(t)$, which take values as follows\textup{:}
\begin{enumerate}[leftmargin=*,label=\textup{\Roman*.},ref=\Roman*]
\item\label{RESULT:mainthmlargexi}
Consider the behaviour of~each characteristic root\textup{,}
$\tau_k(\xi)$\textup{,} in the region $\abs{\xi}\ge M$\textup{,}
where~$M$ is a large enough 
real number. The following
table gives values for the function $K_k^{(\text{l})}(t)$
corresponding to possible properties of~$\tau_k(\xi)$\textup{;} if
$\tau_k(\xi)$ satisfies more than one\textup{,} then take
$K_k^{(\text{l})}(t)$ to be function that decays the slowest as
$t\to\infty$.
\end{enumerate}
\begin{center}\begin{upshape}
\begin{tabular}[4]{|c|c|c|}\hline
Location of $\tau_k(\xi)$ & Additional Property &
$K_k^{\text(l)}(t)$\\
\hline\hline%
away from real axis && $e^{-\de t}$, some $\de>0$\\\hline
&$\det\Hess\tau_k(\xi)\ne0$ &
$\bract{t}^{-\frac{n}{2}(\frac{1}{p}-\frac{1}{q})}$\\
on real axis &$\rank\Hess\tau_k(\xi)=n-1$ &
$\bract{t}^{-\frac{n-1}{2}(\frac{1}{p}-\frac{1}{q})}$\\
& convexity condition $\ga$ &
$\bract{t}^{-\frac{n-1}{\ga}(\frac{1}{p}-\frac{1}{q})}$\\ & no
convexity condition, $\ga_0$ &
$\bract{t}^{-\frac{1}{\ga_0}}$\\\hline 
%&$\det\Hess\tau_k(\xi)\ne0$ &
%$\bract{t}^{-\frac{n}{2}(\frac{1}{p}-\frac{1}{q})}$
%\\
%asymptotic to real axis &$\rank\Hess\tau_k(\xi)=n-1$ &
%$\bract{t}^{-\frac{n-1}{2}(\frac{1}{p}-\frac{1}{q})}$\\
%& asymptotic convexity condition & $\bract{t}^{-\frac{n-1}{\ga}}$\\
%& no convexity condition, $\ga_0$ &
%$\bract{t}^{-\frac{1}{\ga_0}(\frac{1}{p}-\frac{1}{q})}$\\ 
%\hline
\end{tabular}\end{upshape}
\end{center}
Then take $K^{(\text{l})}(t)=\max_{k=1\,\dots,n}K_k^{\text(l)}(t)$.

\begin{enumerate}[leftmargin=*,label=\textup{\Roman*.},resume,ref=\Roman*]
\item\label{RESULT:mainthmbddxi}
Consider the behaviour of the characteristic roots in the bounded
region~$\abs{\xi}\le M$\textup{;} again\textup{,} take
$K^{(\text{b})}(t)$ to be the maximum \textup{(}slowest
decaying\textup{)} function for which there are roots satisfying the
conditions in the following table\textup{:}
\end{enumerate}
\begin{center}\begin{upshape}
\begin{tabular}[4]{|c|c|c|}\hline
Location of Root(s)& Properties & $K^{(\text{b})}(t)$\\
\hline\hline%
away from axis & no multiplicities &$e^{-\de t}$, some $\de>0$\\
&$L$ roots coinciding & $\bract{t}^L e^{-\de t}$\\
\hline on axis,&$\det\Hess\tau_k(\xi)\ne0$ &
$\bract{t}^{-\frac{n}{2}(\frac{1}{p}-\frac{1}{q})}$\\
%&$\rank\Hess\tau_k(\xi)=n-1$ &
%$\bract{t}^{-\frac{n-1}{2}(\frac{1}{p}-\frac{1}{q})}$\\
no multiplicities
$^\ast$ 
& convexity condition $\ga$ &
$\bract{t}^{-\frac{n-1}{\ga}(\frac{1}{p}-\frac{1}{q})}$\\
& no convexity condition, $\ga_0$ &
$\bract{t}^{-\frac{1}{\ga_0}(\frac{1}{p}-\frac{1}{q})}$\\
\hline on axis,&$L$ roots coincide&  \\
multiplicities$^\ast, ^{**}$&on set of codimension $\ell$&$\bract{t}^{L-1-\ell}$\\
\hline meeting axis & $L$ roots coincide&\\
with finite order $s$& on set of codimension $\ell$&
$\bract{t}^{L-1-\frac{\ell}{s}(\frac{1}{p}-\frac{1}{q})}$\\
\hline
\end{tabular}\end{upshape}
\end{center}
\begin{small}$^\ast$ 
These two cases of roots lying on the real axis
require some additional regularity assumptions; see 
corresponding microlocal statements for details. \\
\noindent
$^{**}$ This is the $L^1-L^\infty$
rate in a shrinking region; see
Proposition \ref{EQ:estaroundmult} for details. For
different types of $L^2$ estimates see Section
\ref{SEC:phasefnliesonaxisbddxi}, and then
interpolate.\end{small}

\medskip
Then, with
$K(t)=\max\big(K^{\text{(b)}}(t),K^{\text{(l)}}(t)\big)$\textup{,}
the following estimate holds\textup{:}
\begin{equation*}
\norm{\pax^\al\pat^r u(t,\cdot)}_{L^q}\le C_{\alpha,r} K(t)
\sum_{l=0}^{m-1}\norm{f_l}_{W^{N_p-l}_p}\,,
\end{equation*}
where $1\le p\le2$, $\frac{1}{p}+\frac{1}{q}=1$, and
$N_p=N_p(\al,r)$ is a constant depending on~$p,\al$ and~$r$.
\end{thm}
The scheme of the proof of this theorem and precise
relations to microlocal theorems of previous sections
will be given in Section \ref{SEC:outline}.
However, let us now briefly explain how to understand this theorem.
Since the decay rates do depend on the behaviour of characteristic
roots in different regions and theorems from previous sections
determine the corresponding rates, in Theorem 
\ref{THM:overallmainthm} 
we single out properties which determine
the final decay rate. Since the same characteristic root, 
say $\tau_k$,
may exhibit different properties in different regions, we look
at the corresponding rates 
$K^{\text{(b)}}(t),K^{\text{(l)}}(t)$ under each possible condition
and then take the slowest one for the final answer. 
The value of the Sobolev index $N_p=N_p(\alpha,r)$ 
depends on the regions
as well, and it can be found from microlocal statements of
previous sections for each region.

In conditions of Part I of the theorem, it can be
shown by the perturbation arguments that only three
cases are possible for large $\xi$, namely, the characteristic root
may be uniformly separated from the real axis, it may lie on the
axis, or it may converge to the real axis at infinity. If,
for example, the
root lies on the axis and, in addition, it satisfies the convexity
condition with index $\gamma$, we get the corresponding decay rate
$K^{\text{(l)}}(t)=\bract{t}^{-\frac{n-1}{\ga}(\frac{1}{p}-\frac{1}{q})}$.
Indices $\gamma$ and $\gamma_0$ in the tables are defined
as the maximum of the corresponding indices $\gamma(\Sigma_\lambda)$
and $\gamma_0(\Sigma_\lambda)$, respectively, where 
$\Sigma_\lambda=\{\xi:\tau_k(\xi)=\lambda\}$, over all $k$ and over
all $\lambda$, for which $\xi$ lies in the corresponding region.
At present, we do not have examples of characteristic roots
tending to the real axis for large frequencies while remaining
in the open upper half of the complex plane, so we do not
give any estimates for this case in 
Theorem \ref{THM:overallmainthm}. However, in Section
\ref{SEC:asymptotic} we will still
discuss what happens in this case.

The statement in Part II is more involved since we may have multiple
roots intersecting on rather irregular sets. 
The number $L$ of coinciding roots corresponds to the number of
roots which actually contribute to the loss of regularity.
For example, operator $(\partial_t^2-\Delta)(\partial_t^2-2\Delta)$
would have $L=2$ for both pairs of roots 
$\pm|\xi|$ and $\pm\sqrt{2}|\xi|$,
intersecting at the
origin.  Meeting the axis with finite order $s$ 
means that we have the estimate 
\begin{equation}\label{EQ:disttau}
\dist(\xi,Z_k)^s\leq c|\Im\tau_k(\xi)|
\end{equation}
for all the intersecting roots, where $Z_k=\{\xi: \Im\tau_k(\xi)=0\}.$
In Part II of Theorem \ref{THM:overallmainthm}, 
the condition that 
$L$ roots meet the axis with finite order $s$ on a set of codimension
$\ell$ means that all these estimates hold
and that there is a ($C^1$) set $\curlyM$ 
of codimension $\ell$ such that
$Z_k\subset \curlyM$ for all corresponding $k$
(see Theorem \ref{TH:multonaxis} for details). 
In Theorem \ref{THM:dissipative} we discuss the special case
of a single root $\tau_k$
meeting the axis at a point $\xi_0$ with order $s$,
which means that $\Im\tau_k(\xi_0)=0$ and 
that we have the estimate 
$\abs{\xi-\xi_0}^s\le c\abs{\Im\tau_k(\xi)}$.
In fact, under certain conditions an improvement in this part
of the estimates is possible, see Theorem \ref{THM:dissipative}
and Remark \ref{REM:dismore}. 

In Part II of the theorem, condition $^{**}$ is formulated in
the region of the size decreasing with time:
if we have $L$ multiple roots 
which coincide on the real axis on a set $\curlyM$ 
of codimension $\ell$,
we have an estimate  
\begin{equation}\label{EQ:around}
|u(t,x)|\leq C(1+t)^{L-1-\ell}
\sum_{l=0}^{m-1}\norm{f_l}_{L^1},
\end{equation}
if we cut off the Fourier transforms of the Cauchy data
to the $\epsilon$-neighbourhood $\curlyM^\epsilon$ of $\curlyM$
with $\epsilon=1/t$. 
Here we may  relax the definition of the intersection
above and say that if $L$ roots coincide
in a set $\curlyM$, then 
they coincide on a set of codimension $\ell$
if the measure of the $\epsilon$-neighborhood $\curlyM^\epsilon$ 
of $\curlyM$
satisfies $|\curlyM^\epsilon|\leq C\epsilon^{\ell}$ 
for small $\epsilon>0$;
here $\curlyM^\epsilon=\{\xi\in\R^n: 
\dist (\xi,\curlyM)\leq\epsilon\}.$
The estimate \eqref{EQ:around}
follows from the procedure
described in Section \ref{SEC:resofroots} of the resolution
of multiple roots, and details and proof of
estimate \eqref{EQ:around} are
given in Section \ref{SEC:phasefnliesonaxisbddxi}, especially in
Proposition \ref{EQ:estaroundmult}.

We can then combine this with the remaining cases outside
of this neighborhood, where it is possible to establish
decay by different arguments. In particular, this is the case
of homogeneous equations with roots intersecting at the
origin. However, one sometimes needs to introduce special
norms to handle $L^2$-estimates around the multiplicities.
Details of this are given in the $L^2$ part of 
Section \ref{sec:shrinking}, in particular in
Proposition \ref{Prop:aroundmultl2}.
Finally, in the case of a simple root we may set $L=1$,
and $\ell=n$, if it meets the axis at a point.

\subsection{Schematic of method}
\label{SEC:outline}
Let us briefly explain some ideas behind the reduction of
Theorem \ref{THM:overallmainthm} to the proceeding theorems.
The realisation of the steps below will be done in Sections
\ref{SEC:Cauchy} and \ref{SEC:bddxiaroundmults}.

\begin{enumerate}[leftmargin=*,label=\textbf{Step \arabic*:}]
\item Representation of the solution.
\end{enumerate}
Using the Fourier transform in $x$, 
this reduces the problem to studying
time-dependent oscillatory integrals, at least
for frequencies with no multiplicities. In the case near
multiplicities we will introduce a special procedure to
deal with them in Section \ref{SEC:bddxiaroundmults}.
\begin{enumerate}[leftmargin=*,label=\textbf{Step \arabic*:},resume]
\item Division of the integral.
\end{enumerate}

We reduce the problem to several 
microlocal cases using suitable cut-off
functions. The problem is
divided into studying the behaviour of the characteristic roots in
three regions of the phase space---large $\abs{\xi}$, bounded
$\abs{\xi}$ away from multiplicities of roots and bounded
$\abs{\xi}$ in a neighbourhood of multiplicities.

\begin{enumerate}[leftmargin=*,label=\textbf{Step \arabic*:},resume]
\item Interpolation reduces problem to finding $L^1-L^\infty$ and
$L^2-L^2$ estimates.
\end{enumerate}
\begin{enumerate}[leftmargin=*,label=\textbf{Step \arabic*:},resume]
\item Large $\abs{\xi}$:
\begin{itemize}[leftmargin=*]
\item root separated from the real axis 
(Theorem \ref{THM:expdecay});
%\item root asymptotic to the real axis
%(Theorem \ref{THM:nonconvexcplxargument-sp});
\item root lying on the real axis
(Theorems \ref{THM:nondeghess2}--\ref{THM:noncovexargument-sp}).
\end{itemize}

\item Bounded $\abs{\xi}$, away from multiplicities:
\begin{itemize}[leftmargin=*]
\item root away from the real axis
(Theorem \ref{THM:expdecay});
\item root meeting the real axis with finite order
(Theorem \ref{THM:dissipative});
\item root lying on the real axis
(Theorems \ref{THM:nondeghess2}--\ref{THM:noncovexargument-sp}).
\end{itemize}

\item Bounded $\abs{\xi}$, around multiplicities of roots:
\begin{itemize}[leftmargin=*]
\item all intersecting roots away from the real axis
(Theorem \ref{THM:expdecay2});
\item all intersecting roots lie on the real axis around the
multiplicity
(Section \ref{SEC:phasefnliesonaxisbddxi});
\item all intersecting roots meet the real axis with finite order
(Theorem \ref{TH:multonaxis});
\item one or more of the roots meets the real axis with infinite
order
(similar to Theorems 
\ref{THM:nondeghess2}--\ref{THM:noncovexargument-sp}).
\end{itemize}
\end{enumerate}

\subsection{Strichartz estimates and 
nonlinear problems}\label{SEC:nonlinear}

Let us denote by $\kappa_{p,q}(L(D_t,D_x))$ 
the time decay rate for the
Cauchy problem \eqref{EQ:standardCP(repeat)}, so that
function $K(t)$ from Theorem \ref{THM:overallmainthm} satisfies
$K(t)\simeq t^{-\kappa_{p,q}(L)}$ for large $t$.
Thus, for polynomial decay rates, we have 
\begin{equation}\label{EQ:index}
\kappa_{p,q}(L)=-\lim_{t\to\infty}\frac{\ln K(t)}{\ln t}.
\end{equation}
We will also abbreviate the important case
$\kappa(L)=\kappa_{1,\infty}(L)$ since by interpolation we have
$\kappa_{p,p^\prime}=\kappa_{2,2}\frac{2}{p^\prime}+
\kappa_{1,\infty}
(\frac1p-\frac{1}{p^\prime})$, $1\leq p\leq 2$. 
These indices $\kappa(L)$ and $\kappa_{p,p^\prime}(L)$ of
operator $L(D_t,D_x)$ will be responsible for the decay
rate in the Strichartz estimates for solutions to
\eqref{EQ:standardCP(repeat)}, and for the subsequent 
well-posedness properties of the corresponding semilinear
equation which are discussed below.

In order to present an application to nonlinear problems let
us first consider the inhomogeneous equation
\begin{equation}\label{EQ:CPnonhom}
\left\{\begin{aligned}& L(D_t,D_x)u=f,\quad t>0,\\
&\pat^lu(0,x)=0,\quad l=0,\dots,m-1,\;
x\in\R^n\,,
\end{aligned}\right.
\end{equation}
with $L(D_t,D_x)$ as in \eqref{EQ:standardCauchyproblem}.
By the Duhamel's formula the solution can be expressed as
\begin{equation}\label{EQ:solnonhom}
u(t)=\int_0^t  E_{m-1}(t-s) f(s) ds,
\end{equation}
where $E_{m-1}$ is given in \eqref{EQ:soln}.
Let $\kappa=\kappa_{p,p^\prime}(L)$ 
be the time decay rate of 
operator $L$, determined by Theorem \ref{THM:overallmainthm}
and given in \eqref{EQ:index}.
Then
Theorem \ref{THM:overallmainthm} implies that we have estimate
\begin{equation*}\label{EQ:nonhom1}
|| E_{m-1}(t) g||_{W^s_{p^\prime}}\leq C(1+t)^{-\kappa}||g||_{W^s_p}.
\end{equation*}
Together with \eqref{EQ:solnonhom} this implies
$$||u(t)||_{W^s_{p^\prime}(\R^n_x)}\leq
C\int_0^t (t-s)^{-\kappa} ||f(s)||_{W^s_p} ds\leq
C |t|^{-\kappa} \conv ||f(t)||_{W_p^s}.$$
By the Hardy--Littlewood--Sobolev theorem this 
is $L^q(\R)-L^{q^\prime}(\R)$
bounded if $1<q<2$ and $1-\kappa=\frac1q-\frac{1}{q^\prime}.$
Therefore, this implies the following Strichartz estimate:

\begin{thm}\label{THM:nonhom}
Let $\kappa_{p,p^\prime}$ be the time decay rate of the operator
$L(D_t,D_x)$ in the Cauchy problem
\eqref{EQ:CPnonhom}. Let $1<p,q<2$ be such that
$1/p+1/p^\prime=1/q+1/q^\prime=1$ and
$1/q-1/q^\prime=1-\kappa_{p,p^\prime}$. 
Let $s\in\R$. Then there is a
constant $C$ such that the solution $u$ to the Cauchy problem
\eqref{EQ:CPnonhom} satisfies
$$
||u||_{L^{q^\prime}(\R_t,W^s_{p^\prime}(\R^n_x))}\leq
C||f||_{L^{q}(\R_t,W^s_{p}(\R^n_x))},
$$
for all data right hand side $f=f(t,x)$.
\end{thm}

By the standard iteration method we obtain the well-posedness
result for the following semilinear equation
\begin{equation}\label{EQ:CPsemi}
\left\{\begin{aligned}& L(D_t,D_x)u=F(t,x,u),\quad t>0,\\
&\pat^lu(0,x)=f_l(x),\quad l=0,\dots,m-1,\;
x\in\R^n\,.
\end{aligned}\right.
\end{equation}

\begin{thm}\label{THM:semi}
Let $\kappa_{p,p^\prime}$ be the time decay index of the operator
$L(D_t,D_x)$ in the Cauchy problem
\eqref{EQ:CPsemi}. Let $p,q$ be 
such that
$1/p+1/p^\prime=1/q+1/q^\prime=1$ and
$1/q-1/q^\prime=1-\kappa_{p,p^\prime}$. Let $s\in\R$. 

Assume that for any 
$v\in{L^{q^\prime}(\R_t,W^s_{p^\prime}(\R^n_x))}$, 
the nonlinear term satisfies 
$F(t,x,v)\in{L^{q}(\R_t,W^s_{p}(\R^n_x))}.$ Moreover,
assume that for every $\ep>0$ there exists a decomposition
$-\infty=t_0<t_1<\cdots<t_k=+\infty$ such that the
estimates
$$ ||F(t,x,u)-F(t,x,v)||_{L^{q}(I_j,W^s_{p}(\R^n_x))}\leq
\ep ||u-v||_{L^{q^\prime}(I_j,W^s_{p^\prime}(\R^n_x))}$$
hold for the intervals
$I_j=(t_j,t_{j+1})$, $j=0,\ldots,k-1.$

Finally, assume that the solution of the corresponding
homogeneous Cauchy problem is in the space
${L^{q^\prime}(\R_t,W^s_{p^\prime}(\R^n_x))}$.

Then the semilinear Cauchy problem 
\eqref{EQ:CPsemi} has a unique solution in
the space ${L^{q^\prime}(\R_t,W^s_{p^\prime}(\R^n_x))}$.
\end{thm}

\section{Properties of hyperbolic polynomials}
\label{CHAP3}

In order to study the solution~$u(t,x)$ to
\eqref{EQ:standardCauchyproblem}, we must first know some properties
of the characteristic roots $\tau_1(\xi),\dots,\tau_m(\xi)$.
Naturally, we do not have explicit formulae for the roots, 
unlike in
the cases of the dissipative wave equation and the Klein--Gordon
equation (i.e. for second order equations), 
but we do know some properties for the roots of the principal
symbol. For general hyperbolic operators, the roots
$\va_1(\xi),\dots,\va_m(\xi)$ 
of the characteristic polynomial of
the \emph{principal part} are homogeneous functions of order~$1$
since the principal part is homogeneous. Furthermore, for strictly
hyperbolic polynomials these roots are distinct when $\xi\ne0$.
Since these two properties are very useful when studying homogeneous
(strictly) hyperbolic equations, it is useful to know whether the
characteristic roots of the full equation,
$\tau_1(\xi),\dots,\tau_m(\xi)$, have similar properties. 
Indeed, if
we regard the full equation as a perturbation of the principal part
by lower order terms,
we can show that similar properties hold for large~$\abs{\xi}$;
these results are the focus of this section. 
In the outline of the method in
Section \ref{SEC:outline}, we subdivided the
phase space into large~$\abs{\xi}$ and bounded~$\abs{\xi}$, and it
is these properties that motivate this step.

\subsection{General properties}
First, we give some properties of general polynomials which are
useful to us. For constant coefficient polynomials, the following
result holds:
\begin{lem}\label{LEM:auxbdd}
Consider the polynomial over $\C$ with complex coefficients
\begin{equation*}
z^m+c_1z^{m-1}+\dots+c_{m-1}z+c_m=\prod_{k=1}^m(z-z_k).
\end{equation*}
If there exists $M>0$ such that $\abs{c_j}\le M^j$ for each
$j=1,\dots,m$\textup{,} then $\abs{z_k}\le2M$ for all $k=1,\dots,m$.

\begin{proof} Assume that $\abs{z}>2M$. Then
\begin{align*}
&\abs{z^m+c_1z^{m-1}+\dots+c_{m-1}z+c_m}
\ge\abs{z}^m\bigg(1-\frac{\abs{c_1}}{\abs{z}}-\dots
-\frac{\abs{c_{m-1}}}{\abs{z}^{m-1}}-
\frac{\abs{c_m}}{\abs{z}^m}\bigg)\\
&\mspace{220mu}\ge(2M)^m(1-2^{-1}-\dots-2^{-(m-1)}-2^{-m})>0.
\end{align*}
That is, no zero of the polynomial lies outside of the ball about
the origin of radius $2M$; hence $\abs{z_k}\le2M$ for each
$k=1,\dots,m$.
\end{proof}
\end{lem}
\begin{rem}\label{REM:bddcoeffsgivebddroots}
If we replace the hypothesis $\abs{c_j}\le M^j$ by 
$\abs{c_j}\le M$
for each $j=1,\dots,m$, then by a similar argument we obtain that
$\abs{z_k}\le\max\{2,2M\}$. The quantity
$\max\{2,2M\}$ appears because
we need $M\ge1$ for the sum on the right hand side to be positive.
\end{rem}
For general polynomials with variable coefficients, we have
continuous dependence of roots on coefficients (we give an
independent proof of this result here for the sake of completeness
and for referencing, but analogue of this result can be found
in many monographs dealing with hyperbolic polynomials).
\begin{lem}\label{LEM:ctscoeffsimplyctsroots}
Consider the $m^{\text{th}}$ order polynomial with coefficients
depending on $\xi\in\R^n$
\begin{equation*}
p(\tau,\xi)=\tau^m+a_{1}(\xi)\tau^{m-1}+\dots+a_m(\xi).
\end{equation*}
If each of the coefficient functions $a_j(\xi)$\textup{,}
$j=1,\dots,m$\textup{,} is continuous in $\R^n$ then each of the
roots $\tau_1(\xi),\dots,\tau_m(\xi)$ with respect to $\tau$ of
$p(\tau,\xi)=0$ is also continuous in $\R^n$.

\begin{proof} Define $\rho:\C^m\to\C^m$ by
$\rho(z_1,\dots,z_m)=(c_1,\dots,c_m)$ where the $c_j$ satisfy
\begin{equation*}
z^m+c_1z^{m-1}+\dots+c_m=\prod_{j=1}^m(z-z_j).
\end{equation*}
By the fundamental theorem of algebra $\rho$ is invertible
(but the inverse is not unique modulo permutation of roots),
and, moreover, $\rho$ is:
\begin{enumerate}[label=(\alph*),leftmargin=*]
\item surjective by the Fundamental Theorem of Algebra;
\item\label{ITEM:rhocts} continuous since each of the~$c_j$ may be
written as polynomials of the~$z_j$ (by the Vi\`{e}ta formulae);
\item\label{ITEM:rhoproper} proper (that is, the preimage of each
compact set is compact) by Remark~\ref{REM:bddcoeffsgivebddroots};
\end{enumerate}
properties~\ref{ITEM:rhocts} and~\ref{ITEM:rhoproper} imply
that~$\rho$ is a closed mapping.

Now, fix $\xi^0\in\R^n$. For any given $\ep>0$, consider the set
\begin{equation*}
U=\bigcup_{\al\in S_m}
\bigcap_{k=1}^m\set{\zeta=(\zeta_1,\dots,\zeta_m)\in\C^m:
\abs{\zeta_{\al_k}-\tau_k(\xi^0)}<\ep}\,,
\end{equation*}
where $\al=(\al_1,\dots,\al_m)\in S_m$ denotes the set of
permutations of $\{1,\dots,m\}$ (see Fig.~\ref{FIG:setU} for a
diagram of this).
\begin{figure}[htb]
\begin{center}
\setlength{\unitlength}{1mm}
\begin{picture}(35,35)(-5,0)
\linethickness{1pt}%
\put(-3,0){\vector(1,0){30}}\put(-6,27){$\C_2$}%
\put(0,-3){\vector(0,1){30}}\put(28,-4){$\C_1$}%%
\linethickness{0.5pt}%
\multiput(-3,18)(0,4){2}{\line(1,0){30}}%
\multiput(-3,8)(0,4){2}{\line(1,0){30}}%
\multiput(18,-3)(4,0){2}{\line(0,1){30}}%
\multiput(8,-3)(4,0){2}{\line(0,1){30}}%
\multiput(10,0)(10,0){2}{\circle*{1}}%
\put(8,-3){\tiny$\tau_1(\xi^0)$}%
\put(-8,9){\tiny$\tau_1(\xi^0)$}%
\put(18,-3){\tiny$\tau_2(\xi^0)$}%
\put(-8,19){\tiny$\tau_2(\xi^0)$}
\multiput(0,10)(0,10){2}{\circle*{1}}%
\put(8,19){\colorbox{lightgrey}{\tiny\makebox(2,2){$U_1$}}}%
\put(18,9){\colorbox{lightgrey}{\tiny\makebox(2,2){$U_2$}}}%
\multiput(29,8)(0,10){2}{\vector(0,1){4}}%
\multiput(29,12)(0,10){2}{\vector(0,-1){4}}%
\multiput(30,9)(0,10){2}{\tiny$2\ep$}%
\multiput(8,29)(10,0){2}{\vector(1,0){4}}%
\multiput(12,29)(10,0){2}{\vector(-1,0){4}}%
\multiput(9,30)(10,0){2}{\tiny$2\ep$}
\end{picture}
\end{center}\caption{$U=U_1\cup U_2$}\label{FIG:setU}
\end{figure}
Note that $U$ is, by construction, symmetric, i.e.\ if
$(z_1,\dots,z_m)\in U$ then $(z_{\al_1},\dots,z_{\al_m})\in U$ for
all $(\al_1,\dots,\al_m)\in S_m$. Let~$F$ denote the complement
to~$U$:
\begin{equation*}
F=\bigcap_{\al\in S_m}\set{\zeta=(\zeta_1,\dots,\zeta_m)\in\C^m:
\abs{\zeta_{\al_k}-\tau_k(\xi^0)}\ge\ep\;\exists\, k=1,\dots,m}.
\end{equation*}
We need to show that there exists $\de>0$ such that
$(\tau_1(\xi),\dots,\tau_m(\xi))\in U$ whenever
$\abs{\xi-\xi^0}<\de$; note:
\begin{itemize}
\item $\rho^{-1}(\rho(F))=F$ by
construction---if $\rho(w)=\rho(w')$ then both $w$ and $w'$ give
rise to the same polynomial, and hence their entries are
permutations of each other, and so either both or neither lie in
$F$;
\item by the surjectivity of $\rho$,
\[\rho(U)=\rho(F^c)=\rho([\rho^{-1}(\rho(F))]^c)=
\rho(\rho^{-1}(\rho(F)^c))=\rho(F)^c\,;\]
\item $\rho(F)$ is closed since $F$ a closed set and $\rho$ is a
closed mapping;
\end{itemize}
therefore, $\rho(U)$ is open. Thus, there exists an open ball
in~$\rho(U)$ of radius $\de'$ (for some $\de'>0$) about
$a(\xi^0)\equiv(a_1(\xi^0),\dots,a_m(\xi^0))
=\rho(\tau_1(\xi^0),\dots,\tau_m(\xi^0))$:
\begin{equation*}
B_{\de'}(a(\xi^0))=\set{(c_1,\dots,c_m)\in\C^m:
\abs{c_j-a_j(\xi^0)}<\de'\;\forall\, j=1,\dots,m}\subset\rho(U).
\end{equation*}
By the continuity of the $a_j(\xi)$, there exists $\de>0$ such that
\begin{equation*}
\abs{\xi-\xi^0}<\de\implies \abs{a_j(\xi)-a_j(\xi^0)}<\de'\text{ for
all }j=1,\dots,m\,;
\end{equation*}
hence,
\begin{equation*}
\abs{\xi-\xi^0}<\de\implies (a_1(\xi),\dots,a_m(\xi))\in
B_{\de'}(a(\xi^0))\subset \rho(U)\,.
\end{equation*}
Finally, since
$\rho(\tau_1(\xi),\dots,\tau_m(\xi))=(a_1(\xi),\dots,a_m(\xi))$ and
$U$ is symmetric (this is needed as different root orderings give
the same coefficients), we find that we have 
$(\tau_1(\xi),\dots,\tau_m(\xi))\in
U$ when $\abs{\xi-\xi^0}<\de$ as required; this completes the proof
of the lemma.
\end{proof}
\end{lem}

Now, let us turn to proving properties of the characteristic roots.
\begin{prop}\label{PROP:ctyofroots}
Let $L=L(\pat,\pax)$ be a linear $m^{\text{th}}$ order constant
coefficient
differential operator in $D_t$
with coefficients that are pseudo-differential operators
in~$x$, with symbol 
\begin{equation*}\label{EQ:symbolL}
L(\tau,\xi)=\tau^m+\sum_{j=1}^m P_j(\xi)\tau^{m-j}+
\sum_{j=1}^m a_j(\xi)\tau^{m-j},
\end{equation*}
where $P_j(\lambda\xi)=\lambda^j P_j(\xi)$ for all 
$\lambda>>1$, $|\xi|>>1$, and $a_j\in S^{j-\epsilon}$,
for some $\epsilon>0$.

Then each of the
characteristic roots of $L$\textup{,} denoted
$\tau_1(\xi),\dots,\tau_m(\xi)$\textup{,} is continuous in
$\R^n$\textup{;} 
furthermore\textup{,} for each
$k=1,\dots,m$\textup{,} the characteristic root $\tau_k(\xi)$ is
smooth away from multiplicities, and
analytic if the operator $L(D_t,D_x)$ is differential.

If operator $L(D_t,D_x)$ is strictly hyperbolic, then
there exists a constant $M$ such that\textup{,} 
if $\abs{\xi}\geq M$
then the characteristic roots 
$\tau_1(\xi),\dots,\tau_m(\xi)$ of $L$
are pairwise distinct.
\end{prop}
\begin{proof} 
The first part of Proposition is simple.
Let us now investigate the structure of the characteristic
determinant.
We use the notation and results from Chapter 12 of
\cite{gelf+kapr+zele94} concerning the discriminant $\De_p$ of the
polynomial $p(x)=p_mx^m+\dots+p_1x+p_0$,
\begin{equation*}
\De_p\equiv\De(p_0,\dots,p_m):=
(-1)^{\frac{m(m-1)}{2}}p_m^{2m-2}\prod_{i<j}(x_i-x_j)^2\,,
\end{equation*}
where the $x_j$ ($j=1,\dots,m$) are the roots of $p(x)$; that is,
the irreducible polynomial in the coefficients of the polynomial
which vanishes when the polynomial has multiple roots. We note that
$\De_p$ is a continuous function of the coefficients $p_0,\dots,p_m$
of $p(x)$ and it is a homogeneous function of degree $2m-2$ in them;
in addition, it satisfies the quasi-homogeneity property:
\begin{equation*}
\De(p_0,\la
p_1,\la^2p_2,\dots,\la^mp_m)=\la^{m(m-1)}\De(p_0,\dots,p_m).
\end{equation*}
Furthermore, $\De_p=0$ if and only if $p(x)$ has a double root.

We write $L(\tau,\xi)$ in the form
\begin{equation*}
L(\tau,\xi)= L_m(\tau,\xi)+a_{1}(\xi)\tau^{m-1}+a_{2}(\xi)\tau^{m-2}
+\dots +a_{m-1}(\xi)\tau+a_m(\xi),
\end{equation*}
where
\begin{equation*}
L_m(\tau,\xi)=\tau^m+\sum_{j=1}^mP_j(\xi)\tau^{m-j}
\end{equation*}
is the principal part of $L(\tau,\xi)$; 
note that the $P_j(\xi)$ are
homogeneous of degree~$j$ and the $a_j(\xi)$ are
symbols of degree $< j$. By the homogeneity and
quasi-homogeneity properties of $\De_L$, we have, for $\la\ne0$,
\begin{align*}
\De_L(\la\xi)&=\De(P_m(\la\xi)+a_m(\la\xi),
\dots,P_{1}(\la\xi)+a_{1}(\la\xi),1)\\
=&\;\De(\la^m[P_m(\xi)+\tfrac{a_m(\la\xi)}{\la^m}],
\dots,\la[P_1(\xi)+\tfrac{a_{1}(\la\xi)}{\la}],1)\\
=&\;\la^{m(2m-2)}\De(P_m(\xi)+\tfrac{a_m(\la\xi)}{\la^m},
\dots,\la^{-(m-1)}[P_1(\xi)+\tfrac{a_{1}
(\la\xi)}{\la}],\la^{-m})\\
&\qquad\qquad\qquad\qquad\text{(using that $\De$ is homogenous of
degree
$2m-2$)}\\
=&\;\la^{m(m-1)}\De(P_m(\xi)+\tfrac{a_m(\la\xi)}{\la^m},
\dots,P_1(\xi)+\tfrac{a_{1}(\la\xi)}{\la},1)\\
&\qquad\qquad\qquad\qquad\qquad\qquad\qquad\qquad\qquad\text{ (by
quasi-homogeneity)}.
\end{align*}
Now, since $L$ is strictly hyperbolic, the characteristic roots
$\varphi_1(\xi),\dots,\varphi_m(\xi)$ of $L_m$ are pairwise distinct
for $\xi\ne0$, so
\[\De_{L_m}(\xi)=\De(P_m(\xi),\dots,P_1(\xi),1)\ne0\text{
for }\xi\ne0.\] Since the discriminant is continuous in each
argument, there exists $\de>0$ such that if
$\absbig{\tfrac{a_j(\la\xi)}{\la^{j}}}<\de$ for all $j=1,\dots,m$
then
\[\absbig{\De(P_m(\xi)+\tfrac{a_m(\la\xi)}{\la^m}\,,
\dots,P_1(\xi)+\tfrac{a_{1}(\la\xi)}{\la}\,,1)}
\not=0,\] and hence the
roots of the associated polynomial are pairwise distinct. So, fix
$\xi\in\set{\xi\in\R^n:\abs{\xi}=1}$ and let $\la\to\infty$. Since
the $a_j(\xi)$ are polynomials of degree $< j$ it follows that
when $\abs{\xi}\geq M$, the characteristic roots of $L$ are
pairwise distinct.
\end{proof}

\subsection{Symbolic properties}
In this section we will establish a number of useful properties
of characteristic roots which will be important for the
subsequent analysis. In particular, we will show that
asymptotically roots behave like symbols, and we will show
the relation between roots of the full symbol of a
strictly hyperbolic operator with
homogeneous roots of the principal part. 
\begin{prop}[Symbolic properties of
roots]\label{PROP:perturbationresults} Let $L=L(\pat,\pax)$ 
be a hyperbolic operator of the following form
$$L(D_t,D_x)=D_t^m+\sum_{j=1}^m P_j(D_x) D_t^{m-j}+
\sum_{j=1}^m\sum_{|\alpha|+m-j=K} c_{\alpha,j}(D_x) D_t^{m-j},$$
where $P_j(\lambda\xi)=\lambda^j P_j(\xi)$ for 
$\lambda>>1$, $|\xi|>>1$, and $c_{\alpha,j}\in S^{|\alpha|}.$
Here 
$0\le K\le m-1$ is the maximum order of the lower order
terms of $L$.
Let $\tau_1(\xi),\dots,\tau_m(\xi)$ denote its
characteristic roots\textup{;} then
\begin{enumerate}[label=\upshape{\Roman*.},ref={Part~\Roman*}]
\item\label{LEM:orderroot} for each $k=1,\dots,m$\textup{,} there
exists a constant $C>0$ such that
\begin{equation*}\label{EQ:tauisoderxi}
\abs{\tau_k(\xi)}\le C\brac{\xi}\quad\text{for all }\xi\in\R^n\,.
\end{equation*}
\end{enumerate}
Furthermore\textup{,} if we insist that~$L$ is \emph{strictly}
hyperbolic\textup{,} and denote the roots of the principal part
$L_m(\tau,\xi)$ by $\varphi_1(\xi),\dots,\varphi_m(\xi)$\textup{,}
then we have the following properties as well\textup{:}
\begin{enumerate}[label=\upshape{\Roman*.},ref={Part~\Roman*},resume]
\item\label{LEM:tau-vabounds} 
%Suppose that the maximum order of
%the lower order terms is $0\le K\le m-1$. 
For each
$\tau_k(\xi)$\textup{,} $k=1,\dots,m$\textup{,} there exists a
corresponding root of the principal symbol $\va_k(\xi)$
\textup{(}possibly after reordering\textup{)} such that
\begin{equation}\label{EQ:tau-vaboundkthorder}
\abs{\tau_k(\xi)-\va_k(\xi)}\le C\brac{\xi}^{K+1-m} \quad\text{for
all }\xi\in\R^n\,.
\end{equation}
In particular\textup{,} for arbitrary lower terms\textup{,} we have
\begin{equation}\label{EQ:tau-vabound}
\abs{\tau_k(\xi)-\va_k(\xi)}\le C\quad\text{for all }\xi\in\R^n\,.
\end{equation}

\item\label{PROP:boundsonderivsoftau} There exists $M>0$
such that\textup{}, for each characteristic root of~$L$ and for each
multi-index $\al$\textup{,} we can find constants $C=C_{k,\al}>0$
such that
\begin{equation}\label{EQ:rootisassymbol}
\absbig{\pa^\al_\xi\tau_k(\xi)}\le
C\abs{\xi}^{1-\abs{\al}}\,\quad\text{for all }\abs{\xi}\ge M\,,
\end{equation}
In particular\textup{,} there exists a constant $C>0$ such that
\begin{equation}\label{EQ:gradtauisbounded}
\abs{\grad\tau_k(\xi)}\le C\quad\text{for all }\abs{\xi}\ge M\,.
\end{equation}

\item \label{PROP:derivoftau-derivofphi} There exists $M>0$ such
that\textup{,} for each $\tau_k(\xi)$ a corresponding root of the
principal symbol~$\va_k(\xi)$ can be found \textup{(}possibly after
reordering\textup{)} which satisfies\textup{,} for each multi-index
$\al$ and $k=1,\dots,m$\textup{,}
\begin{equation}\label{EQ:derivsoftau-phiforfewerlot}
\absbig{\pa^\al_\xi\tau_k(\xi)-\pa_\xi^\al\va_k(\xi)} \le
C\abs{\xi}^{K+1-m-\abs{\al}}\quad\text{for all }\abs{\xi}\ge M
\end{equation}
In particular, since $K\leq m-1$, we have
\begin{equation}\label{EQ:derivsoftauandphisymbols}
\absbig{\pa^\al_\xi\tau_k(\xi)-\pa_\xi^\al\va_k(\xi)} \le
C\abs{\xi}^{-\abs{\al}}\quad\text{for all }\abs{\xi}\ge M\,,
\end{equation}
for each multi-index $\al$ and $k=1,\dots,m$.
\end{enumerate}
\end{prop}

\newcommand\x{z}
First, we need the following lemma about perturbation 
properties of
general smooth functions. Clearly, we do not need to require that
functions are smooth, but this will be the case in our
application.
\begin{lem}\label{LEM:perturbationofpolys}
Let $p, q:\C\to\C$ be smooth functions and
suppose~$\x_0$ is a simple zero of $p(\x)$ \textup{(}i.e.\
$p(\x_0)=0$, $p'(\x_0)\ne0$\textup{)}. Consider\textup{,} for each
$\ep>0$\textup{,} the following \textup{``}perturbation\textup{''}
of $p(\x)$\textup{:}
\begin{equation*}
p_\ep(\x):=p(\x)+\ep q(\x)\,,
\end{equation*}
and suppose $\x_\ep$ is a root of $p_\ep(\x)$\textup{;}
then\textup{,} for all sufficiently small $\ep>0$\textup{,}
we have
\begin{equation}\label{EQ:x_ep-x_0}
\abs{\x_\ep-\x_0}\le C\ep\absBig{\frac{q(\x_0)}{p'(\x_0)}}\,.
\end{equation}
\end{lem}
\begin{proof}
By Taylor's theorem, we have, near~$\x_0$,
\begin{align*}
p_\ep(\x) & = p_\ep(\x_0)+p_\ep'(\x_0)(\x-\x_0)+O(\abs{\x-\x_0}^2)
\\ & = \ep q(\x_0) + (p'(\x_0)+\ep q'(x_0))(\x-\x_0)
+O(\abs{\x-\x_0}^2)\,.
\end{align*}
Thus, setting $\x=\x_\ep$, we get
\begin{equation}\label{EQ:p_eptexpatx_ep}
0=\ep q(\x_0)+(p'(\x_0)+\ep q'(\x_0))(\x_\ep-\x_0)
+O(\abs{\x_\ep-\x_0}^2)\,.
\end{equation}
Now, consider the function of $\ep$, $\x(\ep):=\x_\ep$; this is
clearly smooth since $p$ and $q$ are smooth and $\x_0$ is a simple
zero of $p(\x)$. Indeed,
$p_\ep^\prime(z_\ep)\approx p^\prime(z_0)\not=0$ for small
$\ep$, hence $z_\ep$ is a simple root of $p_\ep$.
Thus, near the origin,
\begin{equation}\label{EQ:x_eptaylorexp}
  \x(\ep)=\x(0)+\ep \x'(0)+O(\ep^2)\,.
\end{equation}
Combining~\eqref{EQ:p_eptexpatx_ep} and~\eqref{EQ:x_eptaylorexp}, we
get
\begin{equation*}
0=\ep q(\x_0)+(p'(\x_0)+\ep q'(\x_0))(\ep
\x'(0)+O(\ep^2))+O(\ep^2)\,,
\end{equation*}
or,
\begin{equation*}
0=q(\x_0)+p'(\x_0)\x'(0)+O(\ep)\,,
\end{equation*}
for small $\ep$.
Therefore, by the triangle inequality, for each $\ep>0$ small
enough,
\begin{equation*}
\abs{\x'(0)}\le
\frac{C\ep}{\abs{p'(\x_0)}}+\absBig{\frac{q(\x_0)}{p'(\x_0)}}\,,
\end{equation*}
and, thus,
\begin{equation}\label{EQ:x'(0)bound}
\abs{\x'(0)}\le C\absBig{\frac{q(\x_0)}{p'(\x_0)}}\,.
\end{equation}
Finally, combining~\eqref{EQ:x'(0)bound} with
\eqref{EQ:x_eptaylorexp}, we obtain~\eqref{EQ:x_ep-x_0} as required.
\end{proof}

\proof[Proof of Proposition~\ref{PROP:perturbationresults}]\mbox{}
\paragraph*{\ref{LEM:orderroot}:} We may
write~$L(\tau,\xi)$ in the form
\begin{equation*}
L(\tau,\xi)=\tau^m+a_1(\xi)\tau^{m-1}+\dots+
a_{m-1}(\xi)\tau+a_m(\xi),
\end{equation*}
where $|a_j(\xi)|\leq C\bra{\xi}^j$. 
Hence for all $k$ we have $|\tau_k(\xi)|\leq C\bra{\xi}$ by
Lemma~\ref{LEM:auxbdd}.

\paragraph{\ref{LEM:tau-vabounds}:} In the proof of this part, let
us write $L(\tau,\xi)$ in the form
\begin{equation*}
L(\tau,\xi)=\sum_{i=0}^RL_{m-r_i}(\tau,\xi)\,,
\end{equation*}
where $r_0=0$, $m-r_1=K$ (the maximum order of the lower order
terms), $1\le r_1<\dots<r_R\le m$,
\begin{gather*}
L_m(\tau,\xi)=\tau^m+\sum_{j=1}^mP_j(\xi)\tau^{m-j}\\\text{ and }
L_{m-r_i}(\tau,\xi)=\sum_{\abs{\al}+j=m-r_i}
c_{\al,j}(\xi)\tau^j\;
\text{ for }1\le i\le R;
\end{gather*}
here, as usual, the $P_j(\xi)$ are homogeneous in~$\xi$
of order~$j$.

Denote the roots of
\begin{equation*}
\L_{l}(\tau,\xi):=\sum_{i=0}^{l}L_{m-r_i}(\tau,\xi)\,, \quad 0\le
l\le R\,,
\end{equation*}
with respect to~$\tau$ by $\tau_1^{l}(\xi),\dots,\tau_m^{l}(\xi)$.
Note that $\L_0(\tau,\xi)=L_m(\tau,\xi)$, 
i.e.\ $\L_0(\tau,\xi)$ is
the principal symbol with no lower order terms.
Since $\L_{l}(\tau,\xi)$ are strictly hyperbolic, we will 
look at $|\xi|\geq M_0$, where all
$\tau_1^{l}(\xi),\dots,\tau_m^{l}(\xi)$ are distinct,
for all $l$.

We shall show that there exists $M\ge M_0$ so that, possibly
after reordering the roots, for all $k=1,\dots,m$,
\begin{equation}\label{EQ:taul+1-taulbound}
\abs{\tau_k^{l+1}(\xi)-\tau_k^{l}(\xi)}\le
C\abs{\xi}^{-r_{l+1}+1}\,\text{ for all }
l=0,\dots,R-1\text{ and }
\abs{\xi}\geq M\,.
\end{equation}
Assuming this, and noting that $\tau_k^0(\xi)=\va_k(\xi)$  and
$\tau_k^R(\xi)=\tau_k(\xi)$ for each $k=1,\dots,m$ (possibly after
reordering), we obtain
\begin{equation*}
\abs{\tau_k(\xi)-\va_k(\xi)}\le
\sum_{l=0}^{R-1}\abs{\tau_k^{l+1}(\xi)-\tau_k^{l}(\xi)} \le
C\abs{\xi}^{-r_1+1}\,\text{ when }\abs{\xi}\geq M;
\end{equation*}
this, together with the continuity of the $\tau_k(\xi)$ and
$\va_k(\xi)$---and thus the boundedness of
$\abs{\tau_k(\xi)-\va_k(\xi)}$ in $B_M(0)$,
gives~\eqref{EQ:tau-vaboundkthorder}. Then,~\eqref{EQ:tau-vabound}
follows by setting $K=m-1$. Here we also used $r_1=m-K.$

So, with the aim of proving~\eqref{EQ:taul+1-taulbound}, we first
introduce some notation: set
\begin{align*}
\widetilde{L}_{m-r_i}:\C\times\Snm\to\C\,:&\quad
\widetilde{L}_{m-r_i}(\tau,\om)=L_{m-r_i}(\tau,\om)\,,&i=0,\dots,R,\\
\widetilde{\L}_l:(M_0,\infty)\times \C\times\Snm\to\C\,:&
\quad\widetilde{\L}_l(\rho,\tau,\om)=
\rho^{-m}\L_l(\rho\tau,\rho\om),&l=0,\dots,R;
\end{align*}
observe that $\widetilde{L}_{m-r_i}$ is just the restriction of
$L_{m-r_i}(\tau,\xi)$ to $\C\times\Snm$. Denote by
$\widetilde\va_1(\om),\widetilde\va_2(\om),\dots,\widetilde\va_m(\om)$ the roots
of $\widetilde{L}_m(\tau,\om)=\widetilde{\L}_0(\rho,\tau,\om)$ with respect
to $\tau$, and by
$\widetilde\tau^k_1(\rho,\om),\widetilde\tau^k_2(\rho,\om),\dots,
\widetilde\tau^k_m(\rho,\om)$ those of $\widetilde{\L}_k(\rho,\tau,\om)$.

%Now, for $1\le l\le R$,
%\begin{align}
%\abs{\xi}^{-m}\L_L(\tau,\xi)&=\abs{\xi}^{-m}\Big(\tau^m+
%\sum_{j=1}^mP_j(\xi)\tau^{m-j}
%+\sum_{i=1}^l\sum_{\abs{\al}+j=m-r_i}c_{\al,j}
%\xi^\al\tau^j\Big)\notag\\
%=&\widetilde\tau^m+
%\sum_{j=1}^mP_j\big(\textstyle\frac{\xi}{\abs{\xi}}\big)
%\widetilde\tau^{m-j} +\displaystyle\sum_{i=1}^l\abs{\xi}^{-r_i}
%\sum_{\abs{\al}+j=m-r_i}c_{\al,l}
%\big(\textstyle\frac{\xi}{\abs{\xi}}\big)^\al \widetilde\tau^j\notag\\
%=&\widetilde{L}_m\big(\textstyle\frac{\xi}{\abs{\xi}},\widetilde\tau\big)+
%\displaystyle\sum_{i=1}^l\abs{\xi}^{-r_i}
%\widetilde{L}_{m-r_i}\big(\textstyle
%\frac{\xi}{\abs{\xi}},\widetilde\tau\big) \,, \label{EQ:Llisalmosthomog}
%\end{align}
%where $\widetilde\tau=\frac{\tau}{\abs{\xi}}$. 

We denote $\widetilde\tau=\frac{\tau}{\abs{\xi}}$. Since,
\begin{equation*}
\widetilde{L}_m\big(\widetilde\tau,\textstyle\frac{\xi}{\abs{\xi}}\big)=
L_m\big(\widetilde\tau,\textstyle\frac{\xi}{\abs{\xi}}\big)=
\abs{\xi}^{-m}L_m(\tau,\xi) = \abs{\xi}^{-m}\L_{0}(\tau,\xi)=
\widetilde{\L}_0\big(\abs{\xi},\widetilde\tau,\frac{\xi}{\abs{\xi}}\big)
\end{equation*}
for $\xi\in\R^n$, $\tau\in\C$, and
\begin{align*}
\widetilde{\L}_{l+1}(\rho,\tau,\om)&
=\rho^{-m}\L_{l+1}(\rho\tau,\rho\om)
=\rho^{-m}\sum_{i=0}^{l+1}L_{m-r_i}(\rho\tau,\rho\om)
\nonumber\\
=&\rho^{-m}\sum_{i=0}^{l}L_{m-r_i}(\rho\tau,\rho\om)+ \rho^{-m}
\sum_{\abs{\al}+j=m-r_{l+1}}c_{\al,j}(\rho\om)(\rho\tau)^j
\nonumber\\ =&\widetilde{\L}_{l}(\rho,\tau,\om)+
\rho^{-r_{l+1}}\sum_{\abs{\al}+j=m-r_{l+1}}
\frac{c_{\al,j}(\rho\om)}{\rho^{|\alpha|}}\tau^j
\nonumber\\ =&\widetilde{\L}_{l}(\rho,\tau,\om)
+\rho^{-r_{l+1}}L^0_{m-r_{l+1}}(\rho,\tau,\om)
\label{EQ:Ltill+1=Ltill+L}
\end{align*}
for $\om\in S^{n-1}$, $\rho>M_0$, $\tau\in\C$, $l=0,\dots,R-1$.
Here $$L^0_{m-r_{l+1}}(\rho,\tau,\om)=
\sum_{\abs{\al}+j=m-r_{l+1}}
\frac{c_{\al,j}(\rho\om)}{\rho^{|\alpha|}}\tau^j.$$
We also have
\begin{equation*}\label{EQ:Llishomog}
\abs{\xi}^{-m}\L_L(\tau,\xi)= \widetilde{\L}_{l}\big(\abs{\xi},
\textstyle\frac{\xi}{\abs{\xi}},\widetilde\tau\big)\,.
\end{equation*}
As the left-hand side of this is zero when $\tau=\tau_k^l(\xi)$,
$k=1,\dots,m$, and the right-hand side is zero when
$\widetilde\tau=\widetilde\tau_k^l(\abs{\xi},\frac{\xi}{\abs{\xi}})$,
$k=1,\dots,m$, we see that
$\abs{\xi}\widetilde\tau_k^l(\abs{\xi},\frac{\xi}{\abs{\xi}})
=\tau_k^l(\xi)$ for each $k=1,\dots,m$ (possibly after reordering).
Hence, for all $\abs{\xi}\geq M_0$, $k=1,\dots,m$ and
$l=0,\dots,R-1$, we have
\begin{equation*}
\abs{\tau_k^{l+1}(\xi)-\tau^l_k(\xi)}=
\abs{\widetilde\tau_k^{l+1}\big(\abs{\xi},
\textstyle\frac{\xi}{\abs{\xi}} \big)
-\widetilde\tau^l_k\big(\abs{\xi},\frac{\xi}{\abs{\xi}}\big)}
\abs{\xi}\,.
\end{equation*}

Next, observe that applying Lemma~\ref{LEM:perturbationofpolys} with
$\ep=\rho^{-r_{l+1}}$ to
\begin{equation*}
\widetilde{\L}_{l}(\rho,\tau,\om)
+\rho^{-r_{l+1}}L^0_{m-r_{l+1}}(\rho,\tau,\om)
\end{equation*}
yields, for all $\om\in\Snm$ and $k=1,\dots,m$,
\begin{equation*}
\abs{\widetilde\tau_k^{l+1}(\rho,\om)-\widetilde\tau^l_k(\rho,\om)}\le
C\rho^{-r_{l+1}}
\abslr{\frac{L^0_{m-r_{l+1}}(\rho,\widetilde\tau^l_k(\rho,\om),\om)}
{\pa_\tau \widetilde{\L}_l(\rho,\widetilde\tau^l_k(\rho,\om),\om)}}\,.
\end{equation*}
provided we take $\rho\geq M'$ 
for a sufficiently large constant $M'\ge
M_0$. Therefore, for all $\abs{\xi}\geq M'$, $k=1,\dots,m$ and
$l=0,\dots,R-1$, we have
\begin{equation}\label{EQ:taul+1-taulbound1}
\abs{\tau_k^{l+1}(\xi)-\tau^l_k(\xi)}\le C\abs{\xi}^{-r_{l+1}+1}
\abslr{\frac{L^0_{m-r_{l+1}}\big(|\xi|,
\frac{\tau^l_k(\xi)}{\abs{\xi}},\frac{\xi}{\abs{\xi}}\big)} 
{\pa_\tau
\widetilde{\L}_l\big(\abs{\xi},
\frac{\tau^l_k(\xi)}{\abs{\xi}},\frac{\xi}{\abs{\xi}}\big)}}\,.
\end{equation}
Thus, it suffices to show the following two inequalities when
$\abs{\xi}\geq M$ for some $M\ge M'$:
\begin{itemize}[leftmargin=*]
\item there exists a constant~$C_1$ so that, for all $1\le i\le R$,
\begin{equation}\label{EQ:Lmboundfrombelow}
\absBig{L^0_{m-r_i}\big(|\xi|,\textstyle
\frac{\tau^l_k(\xi)}{\abs{\xi}},\frac{\xi}{\abs{\xi}}\big)}=
\abslr{\displaystyle\sum_{\abs{\al}+j=m-r_i}
\frac{c_{\al,j}(\xi)}{|\xi|^{|\alpha|}}
\left(\frac{\tau^l_k(\xi)}{\abs{\xi}}\right)^j} \le C_1;
\end{equation}
\end{itemize}
and
\begin{itemize}
\item there exists a constant~$C_2>0$ so that, for all $0\le l\le
R-1$,
\begin{equation}\label{EQ:dLmboundfromabove}
\absbig{\pa_\tau
\widetilde{\L}_l\big(\abs{\xi},\textstyle
\frac{\tau^l_k(\xi)}{\abs{\xi}},\frac{\xi}{\abs{\xi}}
\big)}= \abs{\xi}^{-m+1}
\abs{\pa_\tau \L_l(\tau^l_k(\xi),\xi)} \ge C_2.
\end{equation}
\end{itemize}
Then, combining~\eqref{EQ:taul+1-taulbound1},
\eqref{EQ:Lmboundfrombelow} and~\eqref{EQ:dLmboundfromabove} gives
\eqref{EQ:taul+1-taulbound}.

%Let us now show~\eqref{EQ:Lmboundfrombelow}
%and~\eqref{EQ:dLmboundfromabove}:

The first estimate~\eqref{EQ:Lmboundfrombelow} 
follows immediately from
\ref{LEM:orderroot} 
since the $\tau^l_k(\xi)$ are roots of strictly
hyperbolic equations, and from the fact that 
$c_{\alpha,j}\in S^{|\alpha|}.$

The second,~\eqref{EQ:dLmboundfromabove}, in the case $l=0$ is
clear: the homogeneity of $L_m(\tau,\xi)$ and its roots give
\begin{equation*}
\abs{\xi}^{-m+1} \abs{\pa_\tau \L_0(\tau^0_k(\xi),\xi)}
=\absBig{\pa_\tau L_m\big(\textstyle
\va_k\big(\frac{\xi}{\abs{\xi}}\big),\frac{\xi}{\abs{\xi}}\big)}\,,
\end{equation*}
which is never zero due to the strict hyperbolicity of $L_m$ and
hence (using that the sphere $S^{n-1}$ is compact and
$L_m(\tau,\xi)$ is continuous and thus achieves its minimum) is
bounded below by some positive constant as required.

For $1\le l\le R-1$, we know that $\tau^l_k(\xi)$, $k=1,\dots,m$,
are simple zeros of $\L_L(\tau,\xi)$ for $\abs{\xi}\geq M_0$ by the
earlier choice of~$M_0$. Observe,
\begin{equation*}
\frac{(\pa_\tau \L_l)(\tau^l_k(\xi),\xi)}{\abs{\xi}^{m-1}}
=\frac{(\pa_\tau L_m)(\tau^l_k(\xi),\xi)}{\abs{\xi}^{m-1}}
+\sum_{i=1}^{l}\frac{(\pa_\tau
L_{m-r_i})(\tau^l_k(\xi),\xi)}{\abs{\xi}^{m-1}}\,.
\end{equation*}
Now,
\begin{equation*}
\frac{(\pa_\tau L_{m-r_i})(\tau^l_k(\xi),\xi)}{\abs{\xi}^{m-1}}
 \to 0\text{ as }\abs{\xi}\to\infty
\end{equation*}
for $i=1,\dots,l$, because $\pa_\tau L_{m-r_i}(\tau,\xi)$ is
a symbol of order $m-r_i-1$. Also,
using the Mean Value Theorem,
\begin{align*}
(\pa_\tau L_m)(\tau^l_k(\xi),\xi)&= (\pa_\tau
L_m)(\va_k(\xi),\xi)+[(\pa_\tau L_m)(\tau^l_k(\xi),\xi)-(\pa_\tau
L_m)(\va_k(\xi),\xi)]\\
=&(\pa_\tau L_m)(\va_k(\xi),\xi)+(\pa_\tau^2
L_m)(\bar\tau^l_k(\xi),\xi)\,,
\end{align*}
where $\bar\tau^l_k(\xi)$ lies on the line connecting
$\va_k(\xi)$ and $\tau^l_k(\xi)$ for
each $\xi\in\R^n$, $k=1,\dots,m$ and $l=1,\dots,R-1$, and
\begin{equation*}
\frac{\absbig{(\pa_\tau^2 L_m)(\bar\tau^l_k
(\xi),\xi)}}{\abs{\xi}^{m-1}}
\le C\abs{\xi}^{-1}\to 0\text{ as }\abs{\xi}\to\infty\,.
\end{equation*}
Therefore, for a sufficiently large constant~$M\ge M'$, there exists
a constant $C_2>0$ such that
\begin{equation*}
\frac{\absbig{\pa_\tau L_m(\tau^l_k(\xi),\xi)}}
{\abs{\xi}^{m-1}} \ge
C\frac{\abs{\pa_\tau L_m(\va_k(\xi),\xi)}}{\abs{\xi}^{m-1}} \ge
C_2\,,\text{ when }\abs{\xi}\geq M.
\end{equation*}
This completes the proof of~\eqref{EQ:Lmboundfrombelow} 
and thus of
\ref{LEM:tau-vabounds}.

\paragraph*{\ref{PROP:boundsonderivsoftau}:} We take $M>0$ 
so that for $\abs{\xi}\geq M$, the
roots $\tau_1(\xi),\dots,\tau_m(\xi)$ are distinct.

To prove the statement, we do induction on $\abs{\al}$.

First, assume $\abs{\al}=1$. Since $L(\tau_k(\xi),\xi)=0$ for each
$k=1,\dots,m$, we have, for each $i=1,\dots,n$,
\begin{equation*}
\frac{\pa L}{\pa \xi_i}(\tau_k(\xi),\xi) + \frac{\pa L}{\pa
\tau}(\tau_k(\xi),\xi)\frac{\pa\tau_k}{\pa \xi_i}(\xi)=0\,.
\end{equation*}
The first term is a symbol of order $m-1$ in
$(\tau_k(\xi),\xi)$, hence, by \ref{LEM:orderroot}, there exists a
constant $C$ such that, when $\abs{\xi}\ge M_1$ for some suitably
large constant $M_1\ge M$,
\begin{equation*}
\absBig{\frac{\pa L}{\pa \xi_i}(\tau_k(\xi),\xi)}\le
C\abs{\xi}^{m-1}\,.
\end{equation*}
The inequality~\eqref{EQ:rootisassymbol} for $\abs{\al}=1$
(i.e.~\eqref{EQ:gradtauisbounded}) then follows immediately from:
\begin{lem}\label{LEM:estfrombelowondP/dt}
There exists constants $C>0$, $M_2\ge M$ such that, for each
$k=1,\dots,m$,
\begin{equation*}
\absBig{\frac{\pa L}{\pa\tau}(\tau_k(\xi),\xi)}\ge
C\abs{\xi}^{m-1}\quad\text{when }\abs{\xi}\geq M_2\,.
\end{equation*}

\end{lem}
\begin{proof}
Note that
\begin{equation}\label{EQ:triineqondP/dt}
\absBig{\frac{\pa L}{\pa\tau}(\tau_k(\xi),\xi)}\ge \absBig{\frac{\pa
L_m}{\pa\tau}(\va_k(\xi),\xi)}-\absBig{\frac{\pa
L}{\pa\tau}(\tau_k(\xi),\xi) -\frac{\pa
L_m}{\pa\tau}(\va_k(\xi),\xi)}\,,
\end{equation}
where $L_m(\tau,\xi)$ is the principal symbol of~$L$ and
$\va_1(\xi),\dots,\va_m(\xi)$ are the corresponding characteristic
roots, ordered in the same way as in \ref{LEM:tau-vabounds}. We look
at each of the terms on the right-hand side in turn:
\begin{itemize}[leftmargin=*]
\item By strict hyperbolicity, $\frac{\pa
L_m}{\pa \tau}(\va_k(\xi),\xi)$ is non-zero for $\xi\ne0$.
Thus, for all $\xi\ne0$,
\begin{equation}\label{EQ:boundbelowdPm/dt}
\absBig{\frac{\pa L_m}{\pa\tau}(\va_k(\xi),\xi)}=
\abs{\xi}^{m-1}\absBig{\frac{\pa L_m}{\pa
\tau}\Big(\textstyle\frac{\xi}{\abs{\xi}},
\va\big(\frac{\xi}{\abs{\xi}}\big)\Big)} \ge C\abs{\xi}^{m-1}\,.
\end{equation}

\item Observe,
\begin{multline*}
\frac{\pa L}{\pa\tau}(\tau_k(\xi),\xi) -\frac{\pa
L_m}{\pa\tau}(\va_k(\xi),\xi) \\= \frac{\pa
L_m}{\pa\tau}(\tau_k(\xi),\xi) -\frac{\pa
L_m}{\pa\tau}(\va_k(\xi),\xi) + \sum_{r=0}^{m-1}\sum_{\abs{\al}+l=r}
l\;c_{\al,l}(\xi)\tau_k(\xi)^{l-1}\,.
\end{multline*}
Now,
\begin{multline*}
\frac{\pa L_m}{\pa\tau}(\tau_k(\xi),\xi) -\frac{\pa
L_m}{\pa\tau}(\va_k(\xi),\xi) \\=
m(\tau_k(\xi)^{m-1}-\va_k(\xi)^{m-1}) +\sum_{j=1}^m
(m-j)P_j(\xi)(\tau_k(\xi)^{m-j-1}-\va_k(\xi)^{m-j-1}),
\end{multline*}
and
\begin{multline*}
\abs{\tau_k(\xi)^r-\va_k(\xi)^r}=\abs{\tau_k(\xi)-\va_k(\xi)}
\abs{\tau_k(\xi)^{r-1}+\tau_k(\xi)^{r-2}\va_k(\xi)+
\dots+\va_k(\xi)^{r-1}}\,.
\end{multline*}
So, by \ref{LEM:orderroot} and \ref{LEM:tau-vabounds} (specifically
inequality~\eqref{EQ:tau-vabound}) and the fact that the $P_j(\xi)$
are homogeneous in $\xi$ of order $j$, we have, for some
suitably large $M_2\ge M$,
\begin{equation*}
\absBig{\frac{\pa L_m}{\pa\tau}(\tau_k(\xi),\xi) -\frac{\pa
L_m}{\pa\tau}(\va_k(\xi),\xi)}\le C\abs{\xi}^{m-2}\quad \text{when
}\abs{\xi}\geq M_2\,.
\end{equation*}
This, together with
\begin{equation*}
\absBig{\sum_{\abs{\al}+l=r}l\;
c_{\al,r}(\xi)\tau_k(\xi)^{l-1}} \le
C\abs{\xi}^{r-1} \leq
C\abs{\xi}^{m-2}\quad\text{when }\abs{\xi}\geq M_2,\;r=0,\dots,m-1\,,
\end{equation*}
which again follows straight from \ref{LEM:orderroot}, yields
\begin{equation}\label{EQ:dPm/dt(tau)-dP_m/dt(va)est}
\absBig{\frac{\pa L}{\pa\tau}(\tau_k(\xi),\xi) -\frac{\pa
L_m}{\pa\tau}(\va_k(\xi),\xi)}\le C\abs{\xi}^{m-2}\quad \text{for
}\abs{\xi}\geq M_2\,.
\end{equation}
\end{itemize}
The result now follows by combining~\eqref{EQ:triineqondP/dt},
\eqref{EQ:dPm/dt(tau)-dP_m/dt(va)est} and
\eqref{EQ:boundbelowdPm/dt}. The proof of Lemma
\ref{LEM:estfrombelowondP/dt} is complete.
\end{proof}
For $\abs{\al}=J>1$, assume inductively that,
\begin{equation*}
\absbig{\pa^\al_\xi\tau_k(\xi)}\le
C\abs{\xi}^{1-\abs{\al}}\,\quad\text{when
}\abs{\xi}\geq M,\;\abs{\al}\le J-1\,,
\end{equation*}
for some fixed $M\ge \max(M_1,M_2)$.

Then, for $\abs{\al}=J$, we use $\pa_\xi^\al[L(\tau_k(\xi),\xi)]=0$,
i.e.\ 
\begin{multline*}
\pa_\xi^\al\tau_k(\xi)\pa_\tau L(\tau_k(\xi),\xi)\\+
\sum_{\stackrel{\be^1+\dots+\be^r\le\al,}{
\be^j\ne0,\be^j\ne\al}}c_{\al,\be^1,\dots,\be^r}
\Big(\prod_{j=1}^r\pa_\xi^{\be^j}\tau_k(\xi)\Big)
\pa_\xi^{\al-\be^1-\dots-\be^r}\pa_\tau^{r} 
L(\tau_k(\xi),\xi)=0 \,.
\end{multline*}
By the inductive hypothesis and the fact that
$\pa_\xi^\be\pa_\tau^jL(\tau_k(\xi),\xi)$ is a symbol of order
$m-j-\abs{\be}$, we have, for all multi-indices
$\be^1,\dots,\be^r\ne0$ or $\al$ satisfying
$\be^1+\dots+\be^r\le\al$,
\begin{equation*}
\abslr{\Big(\prod_{j=1}^r\pa_\xi^{\be^j}\tau_k(\xi)\Big)
\pa_\xi^{\al-\be^1-\dots-\be^r}\pa_\tau^{r} L(\tau_k(\xi),\xi)}\le
C_{k,\al}\abs{\xi}^{m-\abs{\al}}\,\text{ when }\abs{\xi}\ge M.
\end{equation*}
Thus, using Lemma~\ref{LEM:estfrombelowondP/dt} again, we have
\begin{equation*}
\abs{\pa_\xi^\al\tau_k(\xi)}\le
\frac{C_\al\abs{\xi}^{m-\abs{\al}}}{\abs{\pa_\tau
L(\tau_k(\xi),\xi)}}\le C_{k,\al}\abs{\xi}^{1-\abs{\al}}\,\text{
when }\abs{\xi}\ge M,
\end{equation*}
which completes the proof of the induction step.

\paragraph*{\ref{PROP:derivoftau-derivofphi}:}
Once again, assume that the roots $\tau_k(\xi)$, 
$k=1,\dots,m$,
correspond to $\va_k(\xi)$, $k=1,\dots,m$, in the manner of
\ref{LEM:tau-vabounds}.

The proof of this part for general multi-index~$\al$ is quite
technical, so we first give the proof in the case $\abs{\al}=1$ to
demonstrate the main ideas required, and then show how it can be
extended when $\abs{\al}>1$.

From $L(\tau_k(\xi),\xi)=0=L_m(\va_k(\xi),\xi)$, we have for each
$i=1,\dots,n$,
\begin{gather*}
\frac{\pa L}{\pa \xi_i}(\tau_k(\xi),\xi) + \frac{\pa L}{\pa
\tau}(\tau_k(\xi),\xi)\frac{\pa\tau_k}{\pa \xi_i}(\xi)=0\,,\\
\frac{\pa L_m}{\pa \xi_i}(\va_k(\xi),\xi) + \frac{\pa L_m}{\pa
\tau}(\va_k(\xi),\xi)\frac{\pa\va_k}{\pa \xi_i}(\xi)=0\,.
\end{gather*}
Therefore,
\begin{multline}
\frac{\pa L}{\pa \tau}(\tau_k(\xi),\xi) \Big(\frac{\pa\tau_k}{\pa
\xi_i}(\xi)-\frac{\pa\va_k}{\pa\xi_i}(\xi)\Big) = \frac{\pa L_m}{\pa
\xi_i}(\va_k(\xi),\xi)-\frac{\pa L_m}{\pa \xi_i}(\tau_k(\xi),\xi)
\\+\frac{\pa\va_k}{\pa\xi_i}\Big[\frac{\pa
L_m}{\pa\tau}(\va_k(\xi),\xi) -\frac{\pa
L}{\pa\tau}(\tau_k(\xi),\xi)\Big]  -\frac{\pa
(L-L_m)}{\pa\xi_i}(\tau_k(\xi),\xi)\,.\label{EQ:findingdtau-dva}
\end{multline}
It suffices to show that the right-hand side is bounded absolutely
by $C\abs{\xi}^{m-2}$ when $\abs{\xi}\geq M_1$ 
for some suitably large
$M_1\ge M_0$; this is because an application of
Lemma~\ref{LEM:estfrombelowondP/dt} then yields
\begin{equation*}
\absBig{\frac{\pa\tau_k}{\pa\xi_i}(\xi)-
\frac{\pa\va_k}{\pa\xi_i}(\xi)} \le
\frac{C\abs{\xi}^{m-2}}{\absbig{\frac{\pa L}{\pa
\tau}(\tau_k(\xi),\xi)}}\le C \abs{\xi}^{-1}\quad\text{for
}\abs{\xi}\geq M\,,
\end{equation*}
where $M=\max(M_1,M_2)$.

Since $\pa_{\xi_i}(L-L_m)(\tau,\xi)$ is a symbol of order 
$\le m-2$ in $(\tau,\xi)$, 
it is immediately clear that the final term of
\eqref{EQ:findingdtau-dva} is bounded by
$C\abs{\xi}^{m-2}$; here we have also used \ref{LEM:orderroot}.
Also, noting that $\abs{\pa_{\xi_i}\va_k(\xi)}\le C$ by the
homogeneity of $\va_k(\xi)$, we have,
by~\eqref{EQ:dPm/dt(tau)-dP_m/dt(va)est},
\begin{equation*}
\absBig{\frac{\pa\va_k}{\pa\xi_i}(\xi)}\absBig{\frac{\pa L_m}{\pa
\tau}(\va_k(\xi),\xi)-
\frac{\pa L_m}{\pa \tau}(\tau_k(\xi),\xi)} \le
C\abs{\xi}^{m-2}\,.
\end{equation*}
Finally, by the Mean Value Theorem,
\begin{equation*}
\absBig{\frac{\pa L_m}{\pa \xi_i}(\va_k(\xi),\xi)-\frac{\pa L_m}{\pa
\xi_i}(\tau_k(\xi),\xi)} \le C\absBig{\frac{\pa^2
L_m}{\pa\tau\pa\xi_i}(\xi,\bar{\tau})}\abs{\va_k(\xi)-\tau_k(\xi)}
\,,
\end{equation*}
where $\bar{\tau}$ lies on the linear path between $\va_k(\xi)$ and
$\tau_k(\xi)$---which means that (using \ref{LEM:orderroot} once
more) $\abs{\bar{\tau}}\le C\abs{\xi}$ for $\abs{\xi}\ge M$. Since
$\pa_{\tau}\pa_{\xi_i}L_m(\tau,\xi)$ is a symbol of order $m-2$
in $(\tau,\xi)$, and $\abs{\va_k(\xi)-\tau_k(\xi)}\le C$ by
\ref{LEM:tau-vabounds}, this term is bounded by $C\abs{\xi}^{m-2}$,
completing the proof in the case $\abs{\al}=1$.

For $\abs{\al}=J>1$, we assume inductively that
\begin{equation*}
\absbig{\pa^\al_\xi\tau_k(\xi)-\pa_\xi^\al\va_k(\xi)} \le
C\abs{\xi}^{-\abs{\al}}\quad\text{for }\abs{\xi}\geq M\,,
\abs{\al}\le
J-1\,.
\end{equation*}
As in the proof of \ref{PROP:boundsonderivsoftau}, we have
\begin{multline*}
\pa_\xi^\al\tau_k(\xi)\pa_\tau L(\tau_k(\xi),\xi)\\+
\sum_{\stackrel{\be^1+\dots+\be^r\le\al,}{
\be^j\ne0,\be^j\ne\al}}c_{\al,\be^1,\dots,\be^r}
\Big(\prod_{j=1}^r\pa_\xi^{\be^j}\tau_k(\xi)\Big)
\pa_\xi^{\al-\be^1-\dots-\be^r}\pa_\tau^{r} 
L(\tau_k(\xi),\xi)=0 \,;
\end{multline*}
similarly,
\begin{multline*}
\pa_\xi^\al\va_k(\xi)\pa_\tau L_m(\va_k(\xi),\xi)\\+
\sum_{\stackrel{\be^1+\dots+\be^r\le\al,}{
\be^j\ne0,\be^j\ne\al}}c_{\al,\be^1,\dots,\be^r}
\Big(\prod_{j=1}^r\pa_\xi^{\be^j}\va_k(\xi)\Big)
\pa_\xi^{\al-\be^1-\dots-\be^r}\pa_\tau^{r} L_m(\va_k(\xi),\xi)=0
\,.
\end{multline*}
Thus,
\begin{multline*}
(\pa_\xi^\al\tau_k(\xi)-\pa_\xi^\al\va_k(\xi))\pa_\tau
L(\tau_k(\xi),\xi)= \\
\pa_\xi^\al\va_k(\xi)\big(\pa_\tau
L_m(\va_k(\xi),\xi)-\pa_\tau L(\tau_k(\xi),\xi)\big)\\
+\sum_{\stackrel{\be^1+\dots+\be^r\le\al,}{
\be^j\ne0,\be^j\ne\al}}c_{\al,\be^1,\dots,\be^r}
\Big(\prod_{j=1}^r\pa_\xi^{\be^j}\va_k(\xi)\Big)
[\pa_\xi^{\al-\be^1-\dots-\be^r}\pa_\tau^{r}L_m(\va_k(\xi),\xi)-\\
\pa_\xi^{\al-\be^1-\dots-\be^r}\pa_\tau^{r}L_m(\tau_k(\xi),\xi)]\\
+\sum_{\stackrel{\be^1+\dots+\be^r\le\al,}{
\be^j\ne0,\be^j\ne\al}}c_{\al,\be^1,\dots,\be^r} \Big(\prod_{j=1}^r
[\pa_\xi^{\be^j}\va_k(\xi)-\pa_\xi^{\be^j}\tau_k(\xi)]\Big)
\pa_\xi^{\al-\be^1-\dots-\be^r}\pa_\tau^{r} L_m(\tau_k(\xi),\xi)\\
-\sum_{\stackrel{\be^1+\dots+\be^r\le\al,}{
\be^j\ne0,\be^j\ne\al}}c_{\al,\be^1,\dots,\be^r}
\Big(\prod_{j=1}^r\pa_\xi^{\be^j}\tau_k(\xi)\Big)
\pa_\xi^{\al-\be^1-\dots-\be^r}\pa_\tau^{r}
(L-L_m)(\tau_k(\xi),\xi)\,.\\
\end{multline*}
We claim the right-hand side is then bounded absolutely by
$C_\al\abs{\xi}^{m-1-\abs{\al}}$, which, together with
Lemma~\ref{LEM:estfrombelowondP/dt}, yields the desired estimate.

To see this, let us look at each of the terms in turn:
\begin{itemize}
\item $\abs{\pa_\xi^\al\va_k(\xi)}\le C_\al\abs{\xi}^{1-\abs{\al}}$
by the homogeneity of $\va_k(\xi)$; using this with
\eqref{EQ:dPm/dt(tau)-dP_m/dt(va)est} gives the desired bound.
\item Using the Mean Value Theorem as in the case $\abs{\al}=1$, we
get
\begin{multline*}
\absbig{[\pa_\xi^{\al-\be^1-\dots-
\be^r}\pa_\tau^{r}L_m(\va_k(\xi),\xi)-
\pa_\xi^{\al-\be^1-\dots-\be^r}\pa_\tau^{r}L_m(\tau_k(\xi),\xi)]}\\
\le
C_\al\abs{\xi}^{m-\abs{\al}+\abs{\be^1}+\dots+\abs{\be^r}-r-1}\,;
\end{multline*}
coupled with $\abs{\pa_\xi^{\be}\va_k(\xi)}\le
C_\al\abs{\xi}^{1-\abs{\be}}$, this gives the correct bound.
\item By the inductive hypothesis,
\begin{equation*}
\abs{\pa_\xi^{\be^j}\va_k(\xi)-\pa_\xi^{\be^j}\tau_k(\xi)}\le
C_\be\abs{\xi}^{1-\abs{\be^j}}\,;
\end{equation*}
together with
\begin{equation*}
\abs{\pa_\xi^{\al-\be^1-\dots-\be^r}\pa_\tau^{r}
L_m(\tau_k(\xi),\xi)}\le C_\al
\abs{\xi}^{m-\abs{\al}+\abs{\be^1}+\dots+\abs{\be^r}-r}\,,
\end{equation*}
which follows from \ref{LEM:orderroot} and the homogeneity of
$L_m(\tau,\xi)$, this gives the correct estimate.

\item To show the final term is bounded absolutely by
$\abs{\xi}^{m-1-\abs{\al}}$, first note that
$$\pa_\xi^{\al-\be^1-\dots-\be^r}\pa_\tau^{r}
(L-L_m)(\tau_k(\xi),\xi)$$ is a symbol of order $\le
m-\abs{\al}+\abs{\be^1}+\dots+\abs{\be^r}-r-1$; applying
\ref{PROP:boundsonderivsoftau} to estimate the
$\pa_\xi^{\be^j}\tau_k(\xi)$ terms, we have the required result.
\end{itemize}
This completes the proof of~\eqref{EQ:derivsoftauandphisymbols};
\eqref{EQ:derivsoftau-phiforfewerlot} 
is proved in a similar way in
the proof using the set-up of the proof of \ref{LEM:tau-vabounds}.
The proof of Proposition \ref{PROP:perturbationresults}
is now complete.
\qed

We will now establish further symbolic properties of characteristic
roots.
A refinement of this proposition concerning real and
imaginary parts of
complex roots $\tau$ is given in Proposition \ref{Prop:asymptotic}.
\begin{prop}\label{prop:roots-r}
Suppose that the characteristic roots $\phi_k$,
$k=1,\ldots,m$, of the principal
part $L_m(\tau,\xi)$ of a strictly hyperbolic operator
$L(\tau,\xi)$ in \eqref{EQ:standardCP(repeat)} are non-zero
for all $\xi\not=0$. Then the roots $\tau(\xi)$ of the full
symbols satisfy the following properties: 
\begin{enumerate}[label=\textup{(}\textup{\roman*}\textup{)},leftmargin=*]
\item\label{HYP:realtauisasymbol} for all multi-indices $\al$ there
exists a constants $M, C_\al>0$ such that
\begin{equation*}
\abs{\pa_\xi^\al\tau(\xi)}\le
C_\al |\xi|^{1-\abs{\al}};
\end{equation*}
for all $|\xi|\geq M$.
\item\label{HYP:realtauboundedbelow} there exist constants $M,C>0$
such that for all $\abs{\xi}\ge M$ we have $\abs{\tau(\xi)}\ge
C\abs{\xi}$\textup{;}
\item\label{HYP:realderivativeoftaunonzero} there exists a constant
$C_0>0$ such that $\abs{\pa_\om\tau(\la\om)}\ge C_0$ for all $\om\in
\Snm$\textup{,} $\la>0$\textup{;} in particular,
$\abs{\grad\tau(\xi)}\ge C_0$ for all
$\xi\in\R^n\setminus\set{0}$\textup{;}
\item \label{HYP:reallimitofSi_la} there exists a constant $R_1>0$
such that\textup{,} for all $\la>0$\textup{,}
\[\frac{1}{\la}\Si_\la(\tau)\equiv \frac{1}{\la}
\{\xi\in\Rn:\;\tau(\xi)=\lambda\}\subset B_{R_1}(0)\,.\]
\end{enumerate}
\end{prop}
\begin{proof}
\begin{itemize}[leftmargin=*]
\item Property~\ref{HYP:realtauisasymbol}: by
Proposition~\ref{PROP:perturbationresults},
\ref{PROP:boundsonderivsoftau},
\begin{equation*}
\abs{\pa_\xi^\al\tau(\xi)}\le
C_\al\abs{\xi}^{1-\abs{\al}}\quad\text{for all }\abs{\xi}\ge M\,,
\end{equation*}
for all multi-indices $\al$.
\item Properties ~\ref{HYP:realtauboundedbelow} and
\ref{HYP:realderivativeoftaunonzero}: these follow by
using perturbation methods. By
Proposition~\ref{PROP:perturbationresults},
\ref{PROP:derivoftau-derivofphi}, there exists a homogeneous
function~$\va(\xi)$ of order~$1$ such that, 
for all $\abs{\xi}\ge M$
and $k=1,\dots,n$,
\begin{gather*}
\abs{\tau(\xi)-\va(\xi)}\le C_0\,\text{ and }\,
\abs{\pa_{\xi_k}\tau(\xi)-\pa_{\xi_k}\va(\xi)}\le
C_k\abs{\xi}^{-1}\,,
\end{gather*}
for some constants~$C_0,C_k>0$. Now, the homogeneity of~$\va(\xi)$
implies that $\va(\xi)=\abs{\xi}\va\big(\frac{\xi}{\abs{\xi}})$ and
$e_k\cdot\grad\va(e_k)=\va(e_k)$, where
$e_k=(\underbrace{0,\dots,0,1}_k,0,\dots,0)$, so
\begin{gather*}
\abs{\va(\xi)}\ge C'\abs{\xi}\text{ for all }
\xi\in\R^n\text{ and
} \abs{\pa_{\om}\va(\la\om)}\ge C'\text{ for all }\om\in
\Snm,\, \la>0\,,
\end{gather*}
for some constant $C'>0$. Thus,
\begin{equation}\label{EQ:rootgrowspoly}
\abs{\tau(\xi)}\ge \abs{\va(\xi)}-\abs{\tau(\xi)-\va(\xi)}\ge
C'\abs{\xi}-C_0\ge C\abs{\xi}\text{ for }\abs{\xi}\ge M\,,
\end{equation}
for some constants $M,C>0$, and
\begin{align*}
\abs{\pa_{\om}\tau(\la\om)} \ge \abs{\pa_{\om}\va(\la\om)}-
\abs{\pa_{\om}\va(\la\om)-\pa_{\om}\tau(\la\om)}\ge
C'-C_k\la^{-1}\ge C>0
\end{align*}
for all $\om\in \Snm$ and suitably large~$\la$; for
small~$\la>0$, $\pa_\om\tau(\la\om)$ is separated from~$0$
 by the
convexity condition, so $\abs{\pa_\om\tau(\la\om)}\ge C>0$ 
for all
$\om\in \Snm$, $\la>0$, as required.
\item Property~\ref{HYP:reallimitofSi_la}---there exists a
constant $R_1>0$ such that, for all $\la>0$,
$\frac{1}{\la}\Si_\la(\tau)\subset B_{R_1}(0)$---holds by
Proposition~\ref{PROP:perturbationresults},
\ref{LEM:tau-vabounds}, and the fact that
$\frac{1}{\la}\Si_\la(\va)=\Si_1(\va)$ 
for the characteristic root
of the principal symbol~$\va$ corresponding to~$\tau$. 
\end{itemize}
\end{proof}

\section{Oscillatory integrals with convexity}
\label{SEC:ch-convexity}
As discussed in Section~\ref{SEC:homogoperators}, in the case of
homogeneous $m^\text{th}$ order strictly hyperbolic operators,
geometric properties of the characteristic roots play the
fundamental role in determining the $L^p-L^q$ decay; in
particular, if the characteristic roots satisfy the convexity
condition of Definition~\ref{DEF:sugimotoconvexitycond}, then the
decay is, in general, more rapid than when they do not. We will
show that a similar improvement can be obtained for 
operators with
lower order terms when a suitable `convexity condition' 
holds. In Section~\ref{SUBSEC:convcondnreal}, 
we shall extend this notion of
the convexity condition to functions $\tau:\R^n\to\R$ 
and prove a
decay estimate for an oscillatory integral (related to the
solution representation for a strictly hyperbolic operator) with
phase function~$\tau$.

First, we give a general result for oscillatory 
integrals and show
how the concept of functions of ``convex type'' allows
its application to derive the time decay.

\subsection{Estimates for oscillatory integrals}
The following theorem is central in proving results involving
convexity conditions. In some sense, it bridges the gap between
the man der Corput Lemma and the method of stationary
phase, in that the
former is used when there is no convexity but gives a weaker
result, while the latter can be used when a stronger condition
than simply convexity holds and gives a better result. 
Here, we
state and prove a result that has no reference to convexity;
however, in the following section, we show how convexity 
(in some
sense) enables this result to be used in applications.
An earlier version of this result has appeared in
\cite{Ruzh07}, with applications to equations with time
dependent homogeneous symbols in \cite{MR07}. For
completeness we also
include a more detailed proof here.
\begin{thm}\label{THM:oscintthm}
Consider the oscillatory integral
\begin{equation}\label{EQ:genoscint}
I(\la,\Ivar)=\int_{\R^{N}}e^{i\la
\Phi(y,\Ivar)}A(y,\Ivar)g(y)\,dy\,,
\end{equation}
where $N\in\N$\textup{,} $I:[0,\infty)\times\IVar\to\C$\textup{,}
$\IVar$ is any set of parameters $\Ivar$, and
\begin{enumerate}[label=\textup{(I\arabic*)}]
\item\label{HYP:mainoscintgbdd} there exists a bounded open set
$U\subset\R^N$ such that $g\in C^\infty_0(U)$\textup{;}
\item\label{HYP:mainoscintImPhipos} $\Phi(y,\Ivar)$ is a
complex-valued function such that $\Im \Phi(y,\Ivar)\ge0$ for all
$y\in U$\textup{,} $\Ivar\in\IVar$\textup{;}
\item\label{HYP:mainoscintFconvexfn} for some fixed
$z\in\R^N$, some $\delta>0$,
and some $\ga\in\N$\textup{,} $\ga\ge2$\textup{,} the
function
\begin{equation*}
F(\rho,\om,\Ivar):=\Phi(\rho\om+z,\Ivar)
\end{equation*}
satisfies
\begin{gather*}
\abs{\pa_\rho F(\rho,\om,\Ivar)}\ge C\rho^{\ga-1}\text{ and }
\abs{\pa_\rho^mF(\rho,\om,\Ivar)}\le C_m\rho^{1-m}\abs{\pa_\rho
F(\rho,\om,\Ivar)}
\end{gather*}
for all $(\om,\Ivar)\in \SNm\times\IVar$\textup{,} 
all integers $1\leq m\leq \left[N/\gamma\right]+1$ and
all $\rho>0$, for which $\rho\omega+z\in U$\textup{;}
\item\label{HYP:mainoscintAderivsbdd} for each multi-index $\al$
such that $\abs{\al}\le \big[\frac{N}{\ga}\big]+1$\textup{,} there
exists a constant $C_\al>0$ such that $\abs{\pa_y^\al A(y,\Ivar)}\le
C_\al$ for all $y\in U$\textup{,} $\Ivar\in\IVar$.
\end{enumerate}
Then there exists a constant $C=C_{N,\ga}>0$ 
such that
\begin{equation}\label{EQ:oscintbound}
\abs{I(\la,\Ivar)}\le C(1+\la)^{-\frac{N}{\ga}}\quad\text{for all
}\; \la\in[0,\infty),\,\Ivar\in\IVar\,.
\end{equation}
\end{thm}
Constant $C$ in \eqref{EQ:oscintbound} is independent of
$\lambda$ and $\nu$.
\begin{rem}
This theorem extends to the case where $A(y,\Ivar)$ is replaced by
$A(y,\Ivar')$, where $\Ivar'$ may be 
independent of the variable $\Ivar$
appearing in the phase function $\Phi(y,\Ivar)$; these parameters do
not have to be related in any way, provided the estimates in
hypotheses~\ref{HYP:mainoscintImPhipos} and
\ref{HYP:mainoscintAderivsbdd} hold uniformly in the appropriate
parameters. We will simply unite both sets of parameters and
call this union $\nu$ again.
\end{rem}
\begin{proof}
It is clear that~\eqref{EQ:oscintbound} holds for $0\le\la\leq 1$
since $\abs{I(\la,\Ivar)}$ is bounded for such $\la$.

Now, consider the case where $\la\geq 1$. Set $y=\rho\om+z$, where
$\om\in \SNm$ (using the convention that ${\mathbb S}^0=\set{-1,1}$),
$\rho>0$ and $z\in\R^N$ is some fixed point; then
\begin{equation*}
I(\la,\Ivar)= \int_{\SNm}\int_0^\infty e^{i\la
\Phi(\rho\om+z,\Ivar)}A(\rho\om+z,\Ivar)g(\rho\om+z)
\rho^{N-1}\,d\rho\, d\om\,.
\end{equation*}
By the compactness of $\SNm$, it suffices to
prove~\eqref{EQ:oscintbound} for the inner integral.

Choose a function $\chi\in C_0^\infty([0,\infty))$, 
$0\le\chi(s)\le1$ for
all $s$, which is identically~$1$ on $0\le s\le\frac{1}{2}$ and is
zero when $s\ge1$; then, writing
$F(\rho,\om,\Ivar)=\Phi(\rho\om+z,\Ivar)$, we split the inner
integral into the sum of the two integrals
\begin{gather*}
I_1(\la,\Ivar,\om,z)=\int_0^\infty e^{i\la
F(\rho,\om,\Ivar)}A(\rho\om+z,\Ivar)g(\rho\om+z)
\chi(\la^{\frac{1}{\ga}}\rho)\rho^{N-1}\,d\rho\,, \\
I_2(\la,\Ivar,\om,z)=\int_0^\infty e^{i\la
F(\rho,\om,\Ivar)}A(\rho\om+z,\Ivar)g(\rho\om+z)
(1-\chi)(\la^{\frac{1}{\ga}}\rho)\rho^{N-1}\,d\rho\,.
\end{gather*}

Let us first look at $I_1=I_1(\la,\Ivar,\om,z)$; since
$\chi(\la^{\frac{1}{\ga}}\rho)$ is zero for
$\la^{\frac{1}{\ga}}\rho\ge1$, we have, by the change of variables
$\widetilde{\rho}=\la^{\frac{1}{\ga}}\rho$,
\begin{align*}
\abs{I_1} &\le
C\int_0^\infty\chi(\la^{\frac{1}{\ga}}\rho)\rho^{N-1}\,d\rho
=C\int_0^\infty(\widetilde{\rho})^{N-1}\la^{-\frac{N-1}{\ga}}
\chi(\widetilde{\rho})\la^{-\frac{1}{\ga}}\,d\widetilde{\rho}\\
& \le C\la^{-\frac{N}{\ga}}\int_0^1(\widetilde{\rho})^{N-1}
\,d\widetilde{\rho} = C\la^{-\frac{N}{\ga}}\,,
\end{align*}
where we have used $\abs{e^{i\la F(\rho,\om,\Ivar)}}\le1$ since $\Im
F(\rho,\om,\Ivar)\ge0$ for all $\rho,\om,\Ivar$ by
hypothesis~\ref{HYP:mainoscintImPhipos}; this is the desired
estimate for $\abs{I_1}$.

In order to estimate $I_2=I_2(\la,\Ivar,\om,z)$, let us first define
the operator $L:=(i\la\pa_\rho F(\rho,\om,\Ivar))^{-1}\frac{\pa}{\pa
\rho}$ and observe that
\begin{equation*}
L(e^{i\la F(\rho,\om,\Ivar)})=e^{i\la F(\rho,\om,\Ivar)}\,.
\end{equation*}
Denoting the adjoint of $L$ by $L^*$, we have, for each $l\in\NUO$,
\begin{equation*}
I_2=\int_0^\infty e^{i\la
F(\rho,\om,\Ivar)}(L^*)^l[A(\rho\om+z,\Ivar)g(\rho\om+z)
(1-\chi)(\la^{\frac{1}{\ga}}\rho)\rho^{N-1}]\,d\rho\,.
\end{equation*}

Now,
\begin{equation*}
(L^*)^l=\Big(\frac{i}{\la}\Big)^l\sum C_{s_1,\dots,s_p,p,r,l}
\frac{\pa_\rho^{s_1}F\dots\pa_\rho^{s_p}F}{(\pa_\rho
F)^{l+p}}(\rho,\om,\Ivar)\frac{\pa^r}{\pa\rho ^r}\,,
\end{equation*}
where the sum is over all integers $s_1,\dots,s_p,p,r\ge0$ such that
$s_1+\dots+s_p+r-p=l$. By Hypothesis~\ref{HYP:mainoscintFconvexfn},
\begin{equation*}
\absBig{\frac{\pa_\rho^{s_1}F\dots\pa_\rho^{s_p}F}{(\pa_\rho
F)^{l+p}}(\rho,\om,\Ivar)}\le
C\rho^{p-s_1-\dots-s_p-l\ga+l}=C\rho^{r-l\ga}\,.
\end{equation*}
Also, we claim that, for $r\le[\frac{N}{\ga}]+1$,
\begin{equation}\label{EQ:derivsofintegrandwrtrho}
\absBig{\frac{\pa^r}{\pa \rho^r}[A(\rho\om+z,\Ivar)g(\rho\om+z)
(1-\chi)(\la^{\frac{1}{\ga}}\rho)\rho^{N-1}]} \le C_{N}\rho^{N-1-r}
\widetilde{\chi}(\la,\rho)\,,
\end{equation}
where $\widetilde{\chi}(\la,\rho)$ is a smooth function in~$\rho$ which
is zero for $\la^{\frac{1}{\ga}}\rho<\frac{1}{2}$. Assuming this is
true, we see that, for large enough $l$---it suffices to take
$l=[\frac{N}{\ga}]+1$, i.e.\ $N-l\ga<0$---we have,
\begin{align*}
\abs{I_2}\le& C_{N}\la^{-l}\int_0^\infty\sum C_{s_1,\dots,s_p,p,r,l}
\rho^{r-l\ga}[\rho^{N-1-r}]\widetilde{\chi}(\la,\rho)\,d\rho\\
\le& C_{N}\la^{-l}\int_{\frac{1}{2}\la^{-\frac{1}{\ga}}}^\infty
\rho^{N-1-l\ga}\,d\rho = C_{N}\la^{-l}
\Big[\frac{\rho^{N-l\ga}}{N-l\ga}\Big]^{\infty}_{\frac{1}{2}
\la^{-\frac{1}{\ga}}} =C_{N,\ga}\la^{-\frac{N}{\ga}};
\end{align*}
together with the estimate for $\abs{I_1}$, this yields the desired
estimate~\eqref{EQ:oscintbound}. Here we need $l>N/\gamma$,
which means an application of $(L^*)^l$, or estimates on
$\partial_\rho^\alpha F$ for $|\alpha|\leq l$. This gives
a restriction on the number $m$ of derivatives in (I3).

Finally, let us check~\eqref{EQ:derivsofintegrandwrtrho}. It holds
because:
\begin{enumerate}[label=(\roman*)]
\item $\abs{\pa_\rho^r(\rho^{N-1})}\le C_{r,N}\rho^{N-1-r}$ for all
$r\in\N$.
\item For each $r\in\N$,
$\pa_\rho^r[(1-\chi)(\la^{\frac{1}{\ga}}\rho)]=
-\la^{\frac{r}{\ga}}(\pa^r_s\chi)(\la^{\frac{1}{\ga}}\rho)$; now,
$(\pa_s\chi)(\la^{\frac{1}{\ga}}\rho)$ is supported on the set
$\set{(\la,\rho)\in(0,\infty)\times(0,\infty):
\frac{1}{2}<\la^{\frac{1}{\ga}}\rho<1}$, so, in particular, on its
support $\la^{\frac{1}{\ga}}<\rho^{-1}$; therefore,
\begin{equation*}
\abs{\pa_\rho^r[(1-\chi)(\la^{\frac{1}{\ga}}\rho)]}\le
C\rho^{-r}(\pa^r_s\chi)(\la^{\frac{1}{\ga}}\rho)\quad\text{for all
}r\in\N\,,
\end{equation*}
and $(\pa^r_s\chi)(\la^{\frac{1}{\ga}}\rho)$ is smooth in $\rho$ and
zero for $\la^{\frac{1}{\ga}}\rho\le\frac{1}{2}$.

\item By hypothesis~\ref{HYP:mainoscintAderivsbdd}, $\abs{\pa_\rho^r
A(\rho\om+z,\Ivar)}\le C_r$ for each $r\le[\frac{N}{\ga}]+1$ (this
can be seen for $r=1$ by noting that $\pa_\rho A(\rho\om+z,\Ivar)=
\om\cdot\grad_yA(y,\Ivar)\big\vert_{y=\rho\om+z}$, and then for
$r\ge2$ by calculating the higher derivatives). Also, $g$ is smooth
in $U$, so, $\abs{\pa_\rho^r [A(\rho\om+z,\Ivar)g(\rho\om+z)]}\le
C_r$ for $r\le[\textstyle\frac{N}{\ga}]+1$. Furthermore, by
hypothesis~\ref{HYP:mainoscintgbdd}, there exists a constant
$\rho_0>0$ so that $g(\rho\om+z)=0$ for $\rho>\rho_0$; thus,
$\pa_\rho^r [A(\rho\om+z,\Ivar)g(\rho\om+z)]$ is zero for
$\rho>\rho_0$; hence, for $r\le[\textstyle\frac{N}{\ga}]+1$,
\begin{equation*}
\abs{\pa_\rho^r [A(\rho\om+z,\Ivar)g(\rho\om+z)]}\le
C_r\rho_0^r\rho^{-r}\,.
\end{equation*}
\end{enumerate}
This completes the proof of the claim, and thus the theorem.
\end{proof}

\subsection{Functions of convex type}
Hypothesis~\ref{HYP:mainoscintFconvexfn} of
Theorem~\ref{THM:oscintthm} is sufficient for the result of the
theorem to hold; however, it is often difficult to check. For this
reason, we now introduce the concept of a function of convex
type---a condition that is far simpler to verify---and show that for
such functions,~\ref{HYP:mainoscintFconvexfn} automatically holds.
\begin{defn}\label{DEF:phasefunctionofconvextype}
Let $F=F(\rho,\Fvar):[0,\infty)\times\FVar\to\C$ be a function that
is smooth in $\rho$ for each fixed $\Fvar\in\FVar$\textup{,} where
$\FVar$ is some parameter space. Write its $M^\text{th}$ order
Taylor expansion in $\rho$ about $0$ in the form
\begin{equation}\label{EQ:Fformwithremainder}
F(\rho,\Fvar)=\sum_{j=0}^M a_j(\Fvar)\rho^j + R_M(\rho,\Fvar)\,,
\end{equation}
where $R_M(\rho,\Fvar)= \int_0^\rho \pa^{M+1}_s
F(s,\Fvar)\frac{(\rho-s)^M}{M!}\,ds$ is the $M^{\text{th}}$
remainder term.

We say that $F$ is a \emph{function of convex type $\ga$} if\textup{,}
for some $\ga\in\N$\textup{,} $\ga\ge 2$\textup{,} and for some
$\de>0$, we have
\begin{enumerate}[leftmargin=*,label=\textup{(CT\arabic*)}]
\item\label{ITEM:AssumpF1} $a_0(\Fvar)=a_1(\Fvar)=0$ for all
$\Fvar\in\FVar$ {\rm(}i.e. the Taylor expansion of $F$ starts from
order $\geq 2${\rm)}\textup{;}
\item\label{ITEM:AssumpF2} there exists a constant $C>0$ such that
$\sum_{j=2}^\ga\abs{a_j(\Fvar)}\ge C$ for all
$\Fvar\in\FVar$\textup{;}
\item\label{ITEM:AssumpF3} for each
$\Fvar\in\FVar$\textup{,} $\abs{\pa_\rho F(\rho,\Fvar)}$ is
increasing in $\rho$ for $0<\rho<\de$\textup{;}
\item\label{ITEM:AssumpF4} for each $k\in\N$,
$\pa_\rho^kF(\rho,\Fvar)$ is bounded uniformly in
$0<\rho<\de$\textup{,} $\Fvar\in\FVar$.
\end{enumerate}
\end{defn}
\begin{rem}
Note that, if $F$ is \emph{real-valued}, then~\ref{ITEM:AssumpF3} 
implies that we have either $\pa_\rho^2 F(\rho,\Fvar)\ge 0$ for
all $0<\rho<\de$, or $\pa^2_\rho F(\rho,\Fvar)\le 0$ for all
$0<\rho<\de$---this is because $\pa_\rho F(0,\nu)=0$. This is the
connection with convexity, hence the name of such functions.
\end{rem}
Such functions have the following useful property:
\begin{lem}\label{lem:randolest}
Let $F(\rho,\Fvar)$ be a function of convex type $\ga$.
Then\textup{,} for each sufficiently small $0<\de\le1$ there exist
constants $C,C_m>0$ such that
\begin{gather}
\abs{\pa_\rho F(\rho,\Fvar)}\ge C\rho^{\ga-1}
\label{EQ:F'lowerbound}\\
\text{and }\abs{\pa_\rho^mF(\rho,\Fvar)}\le
C_m\rho^{1-m}\abs{\pa_\rho F(\rho,\Fvar)}\label{EQ:F^(m)upperbound}
\end{gather}
for all $0<\rho<\de$\textup{,} $\Fvar\in\FVar$ and $m\in\N$.
\end{lem}
\begin{rem}
A version of this lemma appeared in~\cite{sugi94} for
\emph{analytic} functions without dependence on~$\Fvar$ and is based
on Lemmas~3,~4 and~5 of Randol~\cite{rand69convex} (which also
appeared in Beals~\cite{beal82}, Lemmas~3.2,~3.3). 
Lemma \ref{lem:randolest}
extends it to functions that are only smooth and 
which depend on an additional parameter,
which will be necessary of our analysis.
A limited regularity version of this lemma appeared in 
\cite{Ruzh07}.
The proof of lemma given here is based on estimating
the remainder rather than on using the Cauchy's integral
formula for analytic functions.
\end{rem}
\begin{proof}
First, let us note that, for $0<\rho\le1$ we have, by
\ref{ITEM:AssumpF2},
\begin{equation}\label{EQ:pidefnandbound}
\pi(\rho,\Fvar):=\sum_{j=2}^\ga j\abs{a_j(\Fvar)}\rho^{j-1}\ge
C\rho^{\ga-1}\,.
\end{equation}
Thus, in order to prove~\eqref{EQ:F'lowerbound}, it suffices to show
\begin{equation}\label{EQ:F'bddbelowbypi}
\abs{\pa_\rho F(\rho,\Fvar)}\ge C\pi(\rho,\Fvar)\quad\text{for all
}0<\rho<\de,\,\Fvar\in\FVar\,;
\end{equation}

For $1\le m\le \ga$, we have, using~\eqref{EQ:Fformwithremainder},
\begin{equation}\label{EQ:TaylorexpfordmF}
\pa_\rho^mF(\rho,\Fvar)=\sum_{k=0}^{\ga-m}
\frac{(k+m)!}{k!}a_{k+m}(\Fvar)\rho^{k}+ R_{m,\ga-m}(\rho,\Fvar)\,,
\end{equation}
where $R_{m,\ga-m}(\rho,\Fvar)= \int_0^\rho
\pa_\rho^{\ga+1}F(s,\Fvar)\frac{(\rho-s)^{\ga-m}}{(\ga-m)!}\,ds$ is
the remainder term of the $(\ga-m)^{\text{th}}$ Taylor expansion of
$\pa_\rho^mF(\rho,\Fvar)$. By~\ref{ITEM:AssumpF4} and
\eqref{EQ:pidefnandbound}, we see
\begin{equation}\label{EQ:estremainderm<ga}
\abs{R_{m,\ga-m}(\rho,\Fvar)}\le C_{\ga,m}\rho^{\ga+1-m}\le
C_{\ga,m}\pi(\rho,\Fvar)\rho^{2-m}\quad\text{for }0<\rho<\de\,.
\end{equation}
Hence, for $0<\rho<\de$,
\begin{multline*}
\abs{\pa_\rho F(\rho,\Fvar)}=\absBig{\sum_{k=0}^{\ga-1}
(k+1)a_{k+1}(\Fvar)\rho^{k} +
R_{1,\ga-1}(\rho,\Fvar)}\\\ge\absBig{\sum_{j=2}^\ga
ja_j(\Fvar)\rho^{j-1}}-\absBig{ R_{1,\ga-1}(\rho,\Fvar)}
\ge\absBig{\sum_{j=2}^\ga ja_j(\Fvar)\rho^{j-1}}-
C_\ga\pi(\rho,\Fvar)\rho\,.
\end{multline*}

Now, by~\ref{ITEM:AssumpF3}, $\abs{\pa_\rho F(\rho,\Fvar)}$ is
increasing in $\rho$ for each $\Fvar\in\FVar$ and, by
\ref{ITEM:AssumpF1}, $\pa_\rho F(0,\Fvar)=0$; therefore,
\begin{align*}
\abs{\pa_\rho F(\rho,\Fvar)} &=\max_{0\le\si\le\rho}\abs{\pa_\rho
F(\si,\Fvar)}\\ \ge&  \max_{0\le\si\le\rho}\absBig{\sum_{j=2}^\ga
ja_j(\Fvar)\si^{j-1}}-
\max_{0\le\si\le\rho}C_\ga\pi(\si,\Fvar)\si\\
=&\max_{0\le\bar\si\le1}\absBig{\sum_{j=2}^\ga
ja_j(\Fvar)\rho^{j-1}\bar{\si}^{j-1}}- C_\ga\pi(\rho,\Fvar)\rho\,,
\end{align*}
since $\pi(\si,\Fvar)\si=\sum_{j=2}^\ga j\abs{a_j(\Fvar)}\si^{j}$
clearly achieves its maximum on $0\le\si\le\rho$ at $\si=\rho$.
Noting that
\begin{equation*}
\max_{0\le\bar\si\le1}\absBig{\sum_{j=1}^{L} 
z_j\bar{\si}^{j-1}}
\quad\text{and}\quad \sum_{j=1}^L \abs{z_j}
\end{equation*}
are norms on $\C^{L}$ and, hence, are equivalent, we immediately get
\begin{align*}
\abs{\pa_\rho F(\rho,\Fvar)} \ge &C\sum_{j=2}^\ga
j\abs{a_j(\Fvar)}\rho^{j-1}-
C_\ga\pi(\rho,\Fvar)\rho \\
\ge& (C-C_\ga\de)\pi(\rho,\Fvar)= C_{\ga,\de}\pi(\rho,\Fvar)\,,
\end{align*}
which completes the proof of~\eqref{EQ:F'bddbelowbypi}.

To prove~\eqref{EQ:F^(m)upperbound}, we consider the cases $1\le
m\le\ga$ and $m>\ga$ separately.

For $m>\ga$, we have, by~\ref{ITEM:AssumpF4},
\begin{equation*}
\abs{\pa_\rho^mF(\rho,\Fvar)}\le C_m\le
C_{m,\de}\rho^{\ga+1-m}\quad\text{for }0<\rho<\de\,,
\end{equation*}
since $\ga+1-m\le0$, and, thus, $\rho^{\ga+1-m}\ge \de^{\ga+1-m}>0$;
so, by~\eqref{EQ:F'lowerbound}, we have
\begin{equation}\label{EQ:estm>ga}
\abs{\pa_\rho^mF(\rho,\Fvar)}\le C_{m,\de}\rho^{2-m}\abs{\pa_\rho
F(\rho,\Fvar)}\quad\text{for }0<\rho<\de,\,m>\ga\,.
\end{equation}

For $1\le m\le\ga$, we have the representation
\eqref{EQ:TaylorexpfordmF}. It is clear that
\begin{equation*}
\absBig{\sum_{k=0}^{\ga-m}\frac{(k+m)!}{k!}\,a_{k+m}
(\Fvar)\rho^{k}} \le C_m\pi(\rho,\Fvar)\rho^{1-m}\,,
\end{equation*}
which, together with~\eqref{EQ:estremainderm<ga} and
\eqref{EQ:F'bddbelowbypi}, yields
\begin{equation*}
\abs{\pa_\rho^mF(\rho,\Fvar)}\le C_{m,\de}\rho^{1-m}\abs{\pa_\rho
F(\rho,\Fvar)}\quad\text{for }0<\rho<\de,\,1\le m\le\ga\,.
\end{equation*}
This, together with~\eqref{EQ:estm>ga}, completes the proof of
\eqref{EQ:F^(m)upperbound} and, thus, the lemma.
\end{proof}

This lemma means we have the following alternative version of
Theorem~\ref{THM:oscintthm}.
\begin{cor}\label{COR:weakerconvextypecondition}
Hypothesis~\ref{HYP:mainoscintFconvexfn} of
Theorem~\ref{THM:oscintthm} may be replaced by:
\begin{itemize}[leftmargin=*,label=\textup{(I3$^\prime$)}]
\item for some fixed
$z\in\R^N$\textup{,} the function
$F(\rho,\om,\Ivar):=\Phi(\rho\om+z,\Ivar)$ is a function of convex
type $\ga$\textup{,} for some $\ga\in\N$\textup{,} in the sense of
Definition~\ref{DEF:phasefunctionofconvextype} with $(\om,\Ivar)\in
\SNm\times\IVar\equiv\FVar$\textup{.}
\end{itemize}
\end{cor}

\subsection{Convexity condition for real-valued phase
functions}\label{SUBSEC:convcondnreal} Using the results of the
previous two sections, we can now prove a series of results for
which a so-called convexity condition holds; here we 
recall Definitions \ref{DEF:convexitycondition}
and \ref{DEF:ga(Si)}
from Section \ref{section:ests}
and prove the basic result for real-valued
functions.
We recall that a 
smooth function $\tau:\R^n\to\R$ is said to satisfy the
\emph{convexity condition} if~$\Si_\la$ is convex for each
$\la\in\R$ (and the empty set is considered to be convex).
The \emph{maximal order of
contact} of a hypersurface
$\Si$ is defined as follows. Let $\si\in\Si$\textup{,} and
denote the tangent plane at~$\si$ by~$T_\si$. Let~$P$ be a plane
containing the normal to~$\Si$ at~$\si$ and denote the order of the
contact between the line $T_\si\cap P$ 
and the curve $\Si\cap P$ by
$\ga(\Si;\si,P)$. Then we set
\begin{equation*}
\ga(\Si):=\sup_{\si\in\Si}\sup_P\ga(\Si;\si,P)\,.
\end{equation*}

In the proof of Theorem \ref{THM:convexsp} we will need a
Besov space version of the estimate for the kernel. For this,
let us introduce some useful notation for a family of cut-off
functions $g_R\in C_0^\infty(\R^n)$, $R\in[0,\infty)$: 
these functions will correspond to the cut-offs to annuli
in the frequency space and we need to trace the dependence
on the parameter $R$.
Suppose $g\in
C_0^\infty(\R^n)$ is such that, for some constants $c_0,c_1\ge0$, 
it is
supported in the set
\[\set{\xi:c_0<\abs{\xi}<c_1}\,,\] and let $g_0\in
C_0^\infty(\R^n\setminus\set{0})$ be another (arbitrary) compactly
supported function. Then, for $R\ge0$, set
\begin{equation}\label{EQ:eqnforg_R}
g_R(\xi):=
\begin{cases}g(\xi/R) & \text{if $R\ge1$}, \\
g_0(\xi) & \text{if $0\le R<1$}.
\end{cases}
\end{equation}

Now we can prove the main convexity theorem:
\begin{thm}\label{THM:sugimoto/randolargument}
Suppose $\tau:\R^n\to\R$ satisfies the convexity condition.
Set $\ga:=\sup_{\la>0}\ga(\Si_\la(\tau))$ and assume
this is finite. 
Let $a(\xi)$ be a symbol of order
$\frac{n-1}{\ga}-n$ of type $(1,0)$ on $\R^n$
\textup{;}
furthermore\textup{,} on $\supp a$, we assume\textup{:}
\begin{enumerate}[label=\textup{(}\textup{\roman*}\textup{)},leftmargin=*]
\item\label{HYP:realtauisasymbol} for all multi-indices $\al$ there
exists a constant $C_\al>0$ such that
\begin{equation*}
\abs{\pa_\xi^\al\tau(\xi)}\le
C_\al\brac{\xi}^{1-\abs{\al}};
\end{equation*}
\item\label{HYP:realtauboundedbelow} there exist constants $M,C>0$
such that for all $\abs{\xi}\ge M$ we have $\abs{\tau(\xi)}\ge
C\abs{\xi}$\textup{;}
\item\label{HYP:realderivativeoftaunonzero} there exists a constant
$C_0>0$ such that $\abs{\pa_\om\tau(\la\om)}\ge C_0$ for all $\om\in
\Snm$\textup{,} $\la>0$\textup{;} in particular,
$\abs{\grad\tau(\xi)}\ge C_0$ for all
$\xi\in\R^n\setminus\set{0}$\textup{;}
\item \label{HYP:reallimitofSi_la} there exists a constant $R_1>0$
such that\textup{,} for all $\la>0$\textup{,}
\[\frac{1}{\la}\Si_\la(\tau)\equiv \frac{1}{\la}
\{\xi\in\Rn:\;\tau(\xi)=\lambda\}\subset B_{R_1}(0)\,.\]
\end{enumerate}
 Then\textup{,} the
following estimate holds for all $R\ge0$\textup{,}
$x\in\R^n$\textup{,} $t>1$\textup{:}
\begin{equation}\label{EQ:sugimotothmest}
\absBig{\int_{\R^n}e^{i(x\cdot\xi+ \tau(\xi)t)}a(\xi)
g_R(\xi)\,d\xi} \le Ct^{-\frac{n-1}{\ga}}\,,
\end{equation}
where $g_R(\xi)$ is as given in~\eqref{EQ:eqnforg_R} and $C>0$ is
independent of $R$.
\end{thm}
\begin{rem}
For an integral of this type with some specific compactly supported
function, $\chi\in C_0^\infty(\R^n)$ say, in place of $g_R$, 
we can just use the result for $R=0$.
In this way we obtain Corollary \ref{cor:convex}.
\end{rem}
\begin{proof}
We may assume throughout, without loss of generality, that either
$\tau(\xi)\ge0$ for all $\xi\in\R^n$ or $\tau(\xi)\le0$ for all
$\xi\in\R^n$. Indeed, hypothesis~\ref{HYP:realtauboundedbelow} and
the continuity of~$\tau$ ensure that either~$\tau(\xi)$ is positive
for all $\abs{\xi}\ge M$ or negative for all $\abs{\xi}\ge M$. In
the case where~$\tau(\xi)$ is positive for all $\abs{\xi}\ge M$, set
\begin{equation*}
\tau_+(\xi):=\tau(\xi)+\min(0,\inf_{\abs{\xi}<M}\tau(\xi))\ge0\;
\text{ for all }\xi\in\R^n.
\end{equation*}
Now, $\tau(\xi)-\tau_+(\xi)$ is a constant (in particular, it is
independent of~$\xi$) and $\abs{e^{i[\tau(\xi)-\tau_+(\xi)]t}}=1$,
so it suffices to show
\begin{equation*}
\absBig{\int_{\R^n}e^{i(x\cdot\xi+\tau_+(\xi)t)}a(\xi)
g_R(\xi)\,d\xi} \le Ct^{-\frac{n-1}{\ga}}\,.
\end{equation*}
In the case where~$\tau(\xi)$ is negative for $\abs{\xi}\ge M$, set
$\widetilde\tau(\xi):=-\tau(\xi)$ and by similar reasoning to above, it
is sufficient to show
\begin{equation*}
\absBig{\int_{\R^n}e^{i(x\cdot\xi-\widetilde\tau_+(\xi)t)}a(\xi)
g_R(\xi)\,d\xi} \le Ct^{-\frac{n-1}{\ga}}\,,
\end{equation*}
where $-\widetilde\tau_+(\xi)\le0$ for all $\xi\in\R^n$.

We begin by dividing the integral into two parts: near to the
wave-front set, i.e.\ points where
$\grad_\xi[x\cdot\xi+\tau(\xi)t]=0$, and away from such points.
To this end, we introduce a cut-off function $\cutoffWF\in
C^\infty_0(\R^n)$, $0\le\cutoffWF(y)\le1$, which is identically~$1$
in the ball of radius $r>0$ (which will be fixed below) centred at
the origin, $B_r(0)$, and identically $0$ outside the ball of radius
$2r$, $B_{2r}(0)$. Then we estimate the following two integrals
separately:
\begin{gather*}
I_1(t,x):=\int_{\R^n}e^{i(x\cdot\xi+\tau(\xi)t)}a(\xi)g_R(\xi)
\cutoffWF\big(t^{-1}x+\grad\tau(\xi)\big)\,d\xi\,,\\
I_2(t,x):=\int_{\R^n}e^{i(x\cdot\xi+\tau(\xi)t)}a(\xi)g_R(\xi)
(1-\cutoffWF)\big(t^{-1}x+\grad\tau(\xi)\big)\,d\xi\,.
\end{gather*}
For $I_2(t,x)$ we have the following result:
\begin{lem}\label{LEM:estawayfromwavefront}
Suppose $a(\xi)$ is a symbol of order~$j\in\R$. Then, for each
$l\in\N$ with $l>n+j$\textup{,} we have\textup{,} for all
$t>0$\textup{,}
\begin{equation}\label{EQ:estawayfromwavefront}
\abs{I_2(t,x)}\le C_{r,l}t^{-l}\,,
\end{equation}
where the constants $C_{r,l}>0$ are independent of $R$.
\end{lem}
\begin{proof}
In the support of $(1-\cutoffWF)(t^{-1}x+\grad\tau(\xi))$,
we have
$\abs{x+t\grad\tau(\xi)}\ge rt>0$, so we can write
\begin{equation*}
\frac{(x+t\grad\tau(\xi))}{i\abs{x+t\grad\tau(\xi)}^2}\cdot
\grad_\xi( e^{i(x\cdot\xi+\tau(\xi)t)})=
e^{i(x\cdot\xi+\tau(\xi)t)}\,;
\end{equation*}
therefore, denoting the adjoint to
$P\equiv\frac{(x+t\grad\tau(\xi))}{i\abs{x+t\grad\tau(\xi)}^2}
\cdot\grad_\xi$ by $P^*$, we get
\begin{equation*}
I_2(t,x)=
\int_{\R^n}e^{i(x\cdot\xi+\tau(\xi)t)}(P^*)^l\big[a(\xi)g_R(\xi)
(1-\cutoffWF)\big(t^{-1}x+\grad\tau(\xi)\big)\big]\,d\xi
\end{equation*}
for each $l\in\N$. We claim that for each $l$ there exists some
constant $C_{r,l}>0$ independent of $R$ so that, when $t>1$, we have
\begin{equation}\label{EQ:P*lest}
(P^*)^l\big[a(\xi)g_R(\xi)(1-\cutoffWF)
\big(t^{-1}x+\grad\tau(\xi)\big)\big] \le
C_{r,l}t^{-l}\brac{\xi}^{j-l}\,;
\end{equation}
assuming this, we obtain,
\begin{equation*}
\abs{I_2(t,x)}\le
C_{r,l}t^{-l}\int_{\R^n}\frac{1}{\brac{\xi}^{l-j}}\,d\xi\,.
\end{equation*}
Noting that $\int_{\R^n}\frac{1}{\brac{\xi}^{l-j}}\,d\xi$
converges for $l-j>n$ yields the desired
estimate~\eqref{EQ:estawayfromwavefront}.

It remains to prove~\eqref{EQ:P*lest}. Let $f\equiv f(\xi;x,t)$ be a
function that is zero for $\abs{x+t\grad\tau(\xi)}\le rt$ and is
continuously differentiable with respect to~$\xi$; then,
\begin{multline}\label{EQ:expressionforP*f}
P^*f=\grad_\xi\cdot\Big[\frac{(x+t\grad\tau(\xi))}
{i\abs{x+t\grad\tau(\xi)}^2}f\Big]
=\frac{t\lap\tau(\xi)}{i\abs{x+t\grad\tau(\xi)}^2}f
+\frac{(x+t\grad\tau(\xi))}{i\abs{x+t\grad\tau(\xi)}^2}
\cdot\grad_\xi f\\
\qquad\qquad-\frac{2t(x+t\grad\tau(\xi))\cdot[\grad^2\tau(\xi)
\cdot(x+t\grad\tau(\xi))]}{i\abs{x+t\grad\tau(\xi)}^4}f\,.
\end{multline}
Hence, using $\abs{x+t\grad\tau(\xi)}\ge rt$ (hypothesis on $f$)
and $\abs{\pa^\al\tau(\xi)}\le C\brac{\xi}^{1-\abs{\al}}$
(hypothesis~\ref{HYP:realtauisasymbol}), we have
\begin{equation}\label{EQ:estforP*f}
\abs{P^*f}\le C_rt^{-1}[\brac{\xi}^{-1}\abs{f} +\abs{\grad_\xi
f}]\,.
\end{equation}
Now, for all multi-indices $\al$ and for all $\xi\in\R^n$, we get
\begin{itemize}
\item $\abs{\pa^\al a(\xi)}\le
C_\al\brac{\xi}^{j-\abs{\al}}$  for all $\xi\in\R^n$ as $a\in
S^{j}_{1,0}(\R^n)$;

\item $\abs{\pa_\xi^\al\big[(1-\cutoffWF)
\big(t^{-1}x+\grad\tau(\xi)\big)\big]}\le
C_{\al}\brac{\xi}^{-\abs{\al}}$, for all $\xi\in\R^n$---here we
have used hypothesis~\ref{HYP:realtauisasymbol} once more. Also,
it is zero for each $\al$ when $\abs{x+t\grad\tau(\xi)}\le rt$ by
the definition of $\ka$.
\end{itemize}
Furthermore, $\abs{\pa^\al g_R(\xi)}=\abs{\pa^\al g_0(\xi)}\le
C_\al\brac{\xi}^{-\abs{\al}}$ for $0\le R<1$, since
$C_0^\infty(\R^n\setminus\set{0})\subset S_{1,0}^0(\R^n)$. For $R\ge
1$, we have:
\begin{gather*}
\pa^\al g_R(\xi)=\pa^\al[g(\xi/R)]=R^{-\abs{\al}}(\pa^\al
g)(\xi/R)
\text{ and }g\in S_{1,0}^0(\R^n) \\
\implies \abs{\pa^\al g_R(\xi)}\le C_\al
R^{-\abs{\al}}\brac{\xi/R}^{-\abs{\al}} \le C_\al
\brac{\xi}^{-\abs{\al}}\,.
\end{gather*}
Therefore,
\begin{equation}\label{EQ:gRsymbolindepofR}
\abs{\pa^\al g_R(\xi)}\le C_{\al}\brac{\xi}^{-\abs{\al}}\text{ for
all }\xi\in\R^n\text{ and multi-indices }\al\,,
\end{equation}
where the $C_\al>0$ are independent of $R$.

Hence, by~\eqref{EQ:estforP*f}, we obtain
\begin{equation*}
\absbig{P^*[a(\xi)g_R(\xi)(1-\cutoffWF)
\big(t^{-1}x+\grad\tau(\xi)\big)]} \le
C_{r}t^{-1}\brac{\xi}^{j-1}\,.
\end{equation*}
To prove~\eqref{EQ:P*lest} for $l\ge2$ we do induction on $l$. Note
that
\begin{equation*}
\abs{(P^*)^lf}\le C_rt^{-1}[\brac{\xi}^{-1}\abs{(P^*)^{l-1}f}
+\abs{\grad_\xi\{(P^*)^{l-1}f\}}]\,.
\end{equation*}
The first term satisfies the desired estimate by the inductive
hypothesis. For the second term, repeated application of the
properties of $a(\xi)$, $g(\xi)$ and
$(1-\cutoffWF)(t^{-1}x+\grad\tau(\xi))$ noted above to inductively
estimate derivatives of $(P^*)^{l'}f$, $1\le l'\le l-2$ yields the
desired estimate. This completes the proof of the lemma.
\end{proof}
This lemma, with $j=\frac{n-1}{\ga}-n$, means that it suffices to
prove \eqref{EQ:sugimotothmest} for $I_1(t,x)$, where
$\abs{t^{-1}x+\grad\tau(\xi)}<2r$.

Let $\set{\cutoffcone_\ell(\xi)}_{\ell=1}^L$ be a partition of unity
in $\R^n$ where $\cutoffcone_\ell(\xi)\in C^\infty(\R^n)$ is
supported in a narrow (the breadth will be fixed below) open cone
$K_\ell$, $\ell=1,\dots,L$; let us assume that $K_1$
contains the point $e_n=(0,\dots,0,1)$ (if necessary, relabel the
cones to ensure this) and also that each $K_\ell$, $\ell=1,\dots,L$,
can be mapped onto $K_1$ by rotation. Then, it suffices to estimate
\begin{equation}\label{EQ:integralincone}
I_1'(t,x)=\int_{\R^n}e^{i(x\cdot\xi+\tau(\xi)t)}a(\xi)g_R(\xi)
\cutoffcone_1(\xi)\cutoffWF
\big(t^{-1}x+\grad\tau(\xi)\big)\,d\xi\,,
\end{equation}
since the properties of $\tau(\xi)$, $a(\xi)$, $g_R(\xi)$ and
$\cutoffWF(t^{-1}x+\grad\tau(\xi))$ used throughout are invariant
under rotation.

By hypothesis~\ref{HYP:realderivativeoftaunonzero}, the level sets
$\Si_\la=\set{\xi\in\R^n:\tau(\xi)=\la}$ are all non-degenerate (or
empty). Furthermore, the Implicit Function Theorem allows us to
parameterise the intersection of the surface
$\Si'_\la\equiv\frac{1}{\la}\Si_\la$ and the cone $K_1$:
\begin{equation*}
K_1\cap \Si'_\la=\set{(y,h_\la(y)):y\in U}\,;
\end{equation*}
here $U\subset\R^{n-1}$ is a bounded open set for which
$p(U)=\Snm\cap K_1$ where $p(y)=(y,\sqrt{1-\abs{y}^2})$, and
$h_\la:U\to\R$ is a smooth function for each $\la>0$; in particular,
each~$h_\la$ is concave due to~$\tau(\xi)$ satisfying the convexity
condition, i.e.\ $\Si_\la'$ is convex for each $\la\in\R$. Then, in
the case that $\tau(\xi)\ge0$ for all $\xi\in\R^n$, the cone $K_1$ is
parameterised by
\begin{equation*}
K_1=\set{(\la y,\la h_\la(y)):\la>0,\,y\in U}\,,
\end{equation*}
and when $\tau(\xi)\le0$ for all $\xi\in\R^n$,
\begin{equation*}
K_1=\set{(\la y,\la h_\la(y)):\la<0,\,y\in U}\,.
\end{equation*}

Now, let $\gauss:K_1\cap\Si'_\la\to \Snm$ be the Gauss map,
\begin{equation*}
\gauss(\zeta)=\frac{\grad\tau(\zeta)}{\abs{\grad\tau(\zeta)}}\,.
\end{equation*}
By the definition of $\cutoffWF(t^{-1}x+\grad\tau(\xi))$, we have
\begin{equation*}
\abs{t^{-1}x-(-\grad\tau(\xi_\la))}<2r
\end{equation*}
for each $\xi_\la\in K_1\cap\Si'_\la$ that is also in the support of
the integrand of~\eqref{EQ:integralincone}. Hence, provided $r>0$ is
taken sufficiently small, the convexity of $\Si'_\la$ ensures that
the points $t^{-1}x/\abs{t^{-1}x}$ and~$-\gauss(\xi_\la)$ are close
enough so that there exists $z(\la)\in U$ (for each~$\xi_\la\in
K_1\cap\Si'_\la$) satisfying
\begin{equation*}
\gauss\big(z(\la),h_\la(z(\la))\big)=
-t^{-1}x/\abs{t^{-1}x}=-x/\abs{x}\in \Snm\,.
\end{equation*}
Also, $(-\grad_yh_\la(y),1)$ is normal to $\Si'_\la$ at
$(y,h_\la(y))$, so, writing $x=(x',x_n)$, we have
\begin{multline*}
-\frac{x}{\abs{x}}=\frac{(-\grad_yh_\la(z(\la)),1)}
{\abs{(-\grad_yh_\la(z(\la)),1)}} \implies
-\frac{x_n}{\abs{x}}=\frac{1}{\abs{(-\grad_yh_\la(z(\la)),1)}}\\
\text{ and } -\frac{x'}{\abs{x}}=
\frac{-\grad_yh_\la(z(\la))}{\abs{(-\grad_yh_\la(z(\la)),1)}}
=\frac{x_n\grad_yh_\la(z(\la))}{\abs{x}}\,;
\end{multline*}
therefore, $-x'=x_n\grad_yh_\la(z(\la))$. We claim that~$x_n$ is
away from~$0$ provided the breadth of the cone~$K_1$ is chosen to be
sufficiently narrow, so
\begin{equation}\label{EQ:x'/xn=-gradh}
\frac{x'}{x_n}=-\grad_yh_\la(z(\la))\,.
\end{equation}

To prove this claim, first recall that $\Si_\la'\subset B_{R_1}(0)$
for all $\la>0$ (hypothesis~\ref{HYP:reallimitofSi_la}) and note
that $\pa_{\xi_n}\tau(\xi)$ is absolutely continuous on
$\clos{B_{R_1}(0)}$ (it is continuous in $\R^n$): taking $C_0>0$ as
in hypothesis~\ref{HYP:realderivativeoftaunonzero}, we get that
\begin{equation}\label{EQ:widthofcone}
\left.\begin{aligned}\text{there exists }\de>0\text{ so }&\text{that
}\abs{\eta^1-\eta^2}<\de,\;\text{ where }
\eta^1,\eta^2\in\clos{B_{R_1}(0)},\;\\\text{ implies }&
\abs{\pa_{\xi_n}\tau(\eta^1)-\pa_{\xi_n}\tau(\eta^2)}<C_0/4\,.
\end{aligned}\right.
\end{equation}
Then, fix the breadth of~$K_1$ so that the maximal shortest distance
from a point~$\xi\in K_1\cap(\bigcup_{\la>0}\Si'_\la)$ to the ray
$\set{\mu e_n:\mu>0}$ is less than this~$\de$, i.e.\
\begin{equation*}
\sup\set{\inf_{\mu>0}\abs{\xi-\mu e_n}:\xi\in
K_1\cap\big(\bigcup_{\la>0}\Si'_\la\big)}<\de\,.
\end{equation*}
Now, observe that for any $\xi^0\in\R^n$, $\mu>0$, we have
\begin{equation*}
\textstyle\absbig{\frac{x_n}{t}}\ge\abs{\pa_{\xi_n}\tau(\mu e_n)}
-\abs{\pa_{\xi_n}\tau(\xi^0)-\pa_{\xi_n}\tau(\mu
e_n)}-\abs{\frac{x_n}{t}+\pa_{\xi_n}\tau(\xi^0)}\,.
\end{equation*}
Choose $\xi^0\in
K_1\cap\Si_\la'\cap\supp[\cutoffWF(t^{-1}x+\grad\tau(\xi))]$ and
$\mu>0$ so that $\abs{\xi^0-\mu e_n}<\de$ and, hence,
\begin{equation*}
\abs{\pa_{\xi_n}\tau(\xi^0)-\pa_{\xi_n}\tau(\mu
e_n)}<C_0/4;
\end{equation*}
also, by hypothesis~\ref{HYP:realderivativeoftaunonzero},
$\abs{\pa_{\xi_n}\tau(\mu e_n)}\ge C_0$, so
\begin{equation*}
\abs{t^{-1}x_n}\ge 3C_0/4-2r.
\end{equation*}
Taking~$r$ sufficiently small, less than~$C_0/8$ say, (ensuring
$r>0$ satisfies the earlier condition also) we get
\begin{equation}\label{EQ:|xn|isnonzero}
\abs{x_n}\ge ct>0\,\end{equation}
proving the claim.

Before estimating~\eqref{EQ:integralincone}, we introduce some
useful notation: by the definition of $g_R(\xi)$,
\eqref{EQ:eqnforg_R}, when $R\ge1$
\begin{equation*}
\xi\in\supp g_R\;\implies Rc_0<\abs{\xi}<Rc_1;
\end{equation*}
also, if $0\le R<1$, then there exist constants
$\widetilde{c}_0,\widetilde{c}_1>0$ so that
$\widetilde{c}_0<\abs{\xi}<\widetilde{c}_1$ for $\xi\in\supp g_R$. Thus, by
hypotheses~\ref{HYP:realtauisasymbol}
and~\ref{HYP:realtauboundedbelow}, there exist constants
$c_0',c_1'>0$ such that
\begin{equation*}
\begin{cases}
Rc_0'<\abs{\tau(\xi)}<Rc_1'&\text{ if }R\ge 1\text{ and }\xi\in\supp
g_R,\\ c_0'<\abs{\tau(\xi)}<c_1'&\text{ if }0\le R< 1\text{ and
}\xi\in\supp g_R.
\end{cases}
\end{equation*}
Let $\cutoffadjust\in C_0^\infty(\R)$ be identically one on
the set $\set{s\in\R:c_0'<s<c_1'}$ and identically zero in a
neighbourhood of the origin; writing $\curlyR=\max(R,1)$, this then
satisfies
\begin{equation*}
g_R(\xi)= g_R(\xi) \cutoffadjust(\tau(\xi)/\curlyR)\,.
\end{equation*}
Also, for simplicity, write
\begin{equation}\label{EQ:defofatilde}
\widetilde{a}(\xi)\equiv
\widetilde{a}_R(\xi):=a(\xi)g_R(\xi)\cutoffcone_1(\xi)\,;
\end{equation}
this is a type (1,0) symbol of order $\frac{n-1}{\ga}-n$ supported
in the cone~$K_1$, and the constants in the symbolic estimates are
all independent of~$R$ as each~$g_R(\xi)$, $R\ge0$, is a symbol of
order~$0$ with constants independent of~$R$
(see~\eqref{EQ:gRsymbolindepofR}).

We now turn to estimating~\eqref{EQ:integralincone}. Using the
change of variables $\xi\mapsto(\la y,\la h_\la(y))$ and
equality~\eqref{EQ:x'/xn=-gradh}, it becomes
\begin{align}
\begin{split}
I_1'(t,x)=\int_0^{\infty}\int_{U} &e^{i[\la x'\cdot y +\la x_n
h_\la(y)+\tau(\la y,\la h_\la(y))t]}a(\la y,\la h_\la(y))\\
g_R(\la y,\la h_\la(y)) &\Psi_1(\la y,\la h_\la(y))
\ka\big(t^{-1}x+\grad\tau(\la y,\la
h_\la(y))\big)\frac{d\xi}{d(\la,y)}\,dy\,d\la
\end{split}\notag\\
\begin{split}
=\int_0^\infty\int_{U}&e^{i\la x_n[-\grad_yh_\la(z(\la))\cdot y
+h_\la(y)+tx_n^{-1}]} \widetilde{a}(\la y,\la h_\la(y))
\\ &\cutoffadjust(\la/\curlyR)\ka\big(t^{-1}x+\grad\tau(\la y,\la
h_\la(y))\big)\frac{d\xi}{d(\la,y)}\,dy\,d\la,
\end{split}\label{EQ:integralcov1}
\end{align}
where we have used $\tau(\la y,\la h_\la(y))=\la$ (definition
of~$\Si_\la$) in the last line. Here, note that
\begin{align*}
\frac{d\xi}{d(\la,y)}=
\begin{vmatrix}
\la I&y\\ \la\grad_y h_\la(y)&\pa_\la[\la h_\la(y)]
\end{vmatrix}=
\la^{n-1}(\pa_\la[\la h_\la(y)]-y\cdot\grad_y h_\la(y))\,,
\end{align*}
where $I$ is the identity matrix. Differentiating $\tau(\la y,\la
h_\la(y))=\la$ with respect to $\la$ in the first case and with
respect to $y$ in the second, gives
\begin{gather*}
y\cdot\grad_{\xi'}\tau(\la y,\la h_\la(y))+\pa_\la[\la
h_\la(y)]\pa_{\xi_n}\tau(\la y,\la h_\la(y))=1\,,\\
\la\grad_{\xi'}\tau(\la y,\la h_\la(y))+\la\grad_y
h_\la(y)\pa_{\xi_n}\tau(\la y,\la h_\la(y))=0\,.
\end{gather*}
Substituting the second of these equalities into the first yields
\begin{equation*}
\big(\pa_\la[\la h_\la(y)]-y\cdot\grad_y
h_\la(y)\big)\pa_{\xi_n}\tau(\la y,\la h_\la(y))=1\,.
\end{equation*}
We claim that
\begin{equation}\label{EQ:derivoftaubddbelow}
\abs{\pa_{\xi_n}\tau(\la y,\la h_\la(y))}\ge C>0\,.
\end{equation}
To see this, first note that
\begin{equation*}
\abs{\pa_{\xi_n}\tau(\la y,\la h_\la(y))}\ge
\abs{\pa_{\xi_n}\tau(\la\mu e_n)} -\absbig{\pa_{\xi_n}\tau(\la\mu
e_n)-\pa_{\xi_n}\tau(\la y,\la h_\la(y))}
\end{equation*}
where $\mu>0$ is chosen as above so that $\abs{\mu
e_n-(y,h_\la(y))}\le\de$; now, $\abs{\pa_{\xi_n}\tau(\la\mu e_n)}\ge
C_0$ by hypothesis~\ref{HYP:realderivativeoftaunonzero}. Also, by
the Mean Value Theorem, there exists $\bar\xi$ lying on the segment
between $(\la y,\la h_\la(y))$ and $\la\mu e_n$ such that
\begin{equation*}
\abs{\pa_{\xi_n}\tau(\la\mu e_n)-\pa_{\xi_n}\tau(\la y,\la
h_\la(y))}\le C\abs{\grad_\xi\pa_{\xi_n}\tau(\bar\xi)}\la\de \le
C\abs{\bar\xi}^{-1}\la\de \le C\de\,;
\end{equation*}
choosing $\de>0$ small enough (also ensuring it satisfies
condition~\eqref{EQ:widthofcone} above) completes the proof of the
claim. Hence,
\begin{equation}\label{EQ:jacobianofxi}
\absBig{\frac{d\xi}{d(\la,y)}}=
\absBig{\frac{\la^{n-1}}{\pa_{\xi_n}\tau(\la y,\la h_\la(y))}}\le
C\la^{n-1}\,.
\end{equation}
Also, note that this Jacobian is bounded below away from zero
because $\abs{\pa_{\xi_n}\tau(\xi)}\le C$ for all $\xi\in\R^n$
(hypothesis~\ref{HYP:realtauisasymbol}), which means that the
transformation above is valid in~$K_1$.

Next, using the change of variables $\latil=\la x_n=\la \xtil_nt$
in~\eqref{EQ:integralcov1}, writing $h(\la,y)\equiv h_\la(y)$ and
setting $\xtil:=t^{-1}x$ (so $\xtil_n=t^{-1}x_n$), we obtain
\begin{multline*}
\int_0^\infty\int_{U}e^{i\latil(-\grad_y
h\big(\frac{\latil}{\xtil_nt},
z\big(\frac{\latil}{\xtil_nt}\big)\big)\cdot y
+h\big(\frac{\latil}{\xtil_nt},y\big)+\xtil_n^{-1})}
\widetilde{a}\Big(\textstyle\frac{\latil}{\xtil_nt}
y,\frac{\latil}{\xtil_nt}
h\Big(\frac{\latil}{\xtil_nt},y\Big)\Big)\\
\cutoffadjust\Big(\textstyle\frac{\latil}{\curlyR\xtil_nt}\Big)
\cutoffWF\Big(\xtil+\grad\tau\Big(\frac{\latil}{\xtil_nt}
y,\frac{\latil}{\xtil_nt}
h\Big(\frac{\latil}{\xtil_nt},y\Big)\Big)\Big)
\displaystyle\frac{d\xi}{d(\la,y)}t^{-1}\xtil_n^{-1}\,dy\,d\latil\,.
\end{multline*}
Therefore, using $\absbig{\frac{d\xi}{d(\la,y)}}\le
C\latil^{n-1}\abs{\xtil_n}^{-(n-1)}t^{-(n-1)}$ (by
\eqref{EQ:jacobianofxi}) and recalling that $\abs{\ka(\eta)}\le1$,
we have,
\begin{equation}\label{EQ:I1'firstbound}
\abs{I_1'(t,x)}\le
Ct^{-\frac{n-1}{\ga}}\abs{\xtil_n}^{-\frac{n-1}{\ga}}
\int_0^\infty\absBig{I\Big(\latil,
\textstyle\frac{\latil}{\xtil_nt};
z\Big(\frac{\latil}{\xtil_nt}\Big)\Big)
\cutoffadjust\Big(\frac{\latil}{\curlyR\xtil_nt}\Big)
\latil^{\frac{n-1}{\ga}-1}}\,d\latil\,,
\end{equation}
where,
\begin{multline*}
I\Big(\latil, \textstyle\frac{\latil}{\xtil_nt};
z\Big(\frac{\latil}{\xtil_nt}\Big)\Big) =
\displaystyle\int_{U}e^{i\latil\big[h\big(\frac{\latil}{\xtil_nt},y\big)
-h\big(\frac{\latil}{\xtil_nt},z\big)-(y-z)\cdot\grad_y
h\big(\frac{\latil}{\xtil_nt},z\big)\big]}
\\\widetilde{a}\Big(\textstyle\frac{\latil}{\xtil_nt}
y,\frac{\latil}{\xtil_nt} h\Big(\frac{\latil}{\xtil_nt},y\Big)\Big)
\Big(\frac{\latil}{t\abs{\xtil_n}}\Big)^{n-\frac{n-1}{\ga}} \,dy\,.
\end{multline*}
With Theorem~\ref{THM:oscintthm} in mind, let us rewrite this in the
form of~\eqref{EQ:genoscint}:
\begin{equation*}
I(\la,\mu;z)=\int_{\R^{n-1}}e^{i\la \phasefn(y,\mu;z)} a_0(\mu y,\mu
h_\mu(y))b(y)\,dy\,,
\end{equation*}
with arbitrary $\la>0$, $\mu>0$ and $z\in\R^{n-1}$, where
\begin{itemize}
\item $\phasefn(y,\mu;z)=h_{\mu}(y)-h_\mu(z)-(y-z)\cdot\grad_y
h_{\mu}(z)$;
\item $a_0(\xi):=\widetilde{a}(\xi)\abs{\xi}^{n-\frac{n-1}{\ga}}$;
\item $b\in C_0^\infty(\R^{n-1})$ with support contained in $U$.
\end{itemize}
We shall show that the following conditions (numbered as in
Theorem~\ref{THM:oscintthm} and
Corollary~\ref{COR:weakerconvextypecondition}) are satisfied by
$I(\la,\mu;z)$:
\begin{enumerate}[label=(I\arabic*),leftmargin=*]
\item\label{COND:biscpctlysupp} there exists a bounded set
$U\subset\R^{n-1}$ such that $b\in C^\infty_0(U)$;
\item\label{COND:im>=0} $\Im \phasefn(y,\mu;z)\ge 0$ for all
$y\in U$, $\mu>0$;
\end{enumerate}
\begin{enumerate}[resume,label=(I\arabic*$^\prime$),leftmargin=*]
\item\label{COND:Fisconvexphasefn}
$F(\rho,\om,\mu;z)=\phasefn(\rho\om+z,\mu;z)$, $\om\in 
{\mathbb S}^{n-2}$,
$\rho>0$, is a function of convex type $\ga$ (see
Definition~\ref{DEF:phasefunctionofconvextype});
\end{enumerate}
\begin{enumerate}[resume,label=(I\arabic*),leftmargin=*]
\item\label{COND:a0derivsbdd} there exist constants $C_\al$ such
that $\abs{\pa_y^\al[ a_0(\mu y,\mu h_\mu(y))]}\le C_\al$ for all
$y\in U$, $\mu>0$ and $\abs{\al}\le [\frac{n-1}{\ga}]+1$.
\end{enumerate}
Assuming for now that these hold, Theorem~\ref{THM:oscintthm} (or,
more precisely, Corollary~\ref{COR:weakerconvextypecondition})
states that, for all $\la>0$, $\mu>0$,
\begin{equation*}
\abs{I(\la,\mu;z)}\le C\bract{\la}^{-\frac{n-1}{\ga}}\le
C\la^{-\frac{n-1}{\ga}}\,.
\end{equation*}
This, together with~\eqref{EQ:I1'firstbound}, gives
\begin{equation*}
\abs{I_1'(t,x)} \le
Ct^{-\frac{n-1}{\ga}}\abs{\xtil_n}^{-\frac{n-1}{\ga}}
\int_0^\infty\latil^{-\frac{n-1}{\ga}}
\cutoffadjust\Big(\textstyle\frac{\latil}{\curlyR\xtil_nt}\Big)
\latil^{\frac{n-1}{\ga}-1}\,d\latil\,;
\end{equation*}
then, setting $\nu=\frac{\latil}{\curlyR\xtil_nt}$, we have
\begin{align*}
\abs{I_1'(t,x)} &\le
Ct^{-\frac{n-1}{\ga}}\abs{\xtil_n}^{-\frac{n-1}{\ga}}\int_0^\infty
(\curlyR\xtil_nt\nu)^{-1}
\cutoffadjust(\nu) \curlyR\xtil_nt\,d\nu\\
&= Ct^{-\frac{n-1}{\ga}}\abs{\xtil_n}^{-\frac{n-1}{\ga}}
\int_0^\infty \nu^{-1} \cutoffadjust(\nu) \,d\nu \le
Ct^{-\frac{n-1}{\ga}}\quad\text{for all }t>1\,.
\end{align*}
Here we have used that~$G$ is identically zero in a neighbourhood of
the origin and that it is compactly supported and also
\eqref{EQ:|xn|isnonzero} ($\abs{\xtil_n}\ge C>0$); also, note the
constant here is independent of~$R$. Since this inequality holds for
$I_1'(t,x)$, it also holds for $I_1(t,x)$; thus, together with
Lemma~\ref{LEM:estawayfromwavefront}, this proves the desired
estimate~\eqref{EQ:sugimotothmest}, provided we show that the four
properties \ref{COND:biscpctlysupp}--\ref{COND:a0derivsbdd} above
hold.

Now, clearly~\ref{COND:biscpctlysupp} holds automatically and
\ref{COND:im>=0} is true since $h_\mu(y)$ is real-valued, so
$\Im\phasefn(y,\mu;z)=0$ for all $y\in U$, $\mu>0$.

For~\ref{COND:Fisconvexphasefn} and~\ref{COND:a0derivsbdd}, we need
an auxiliary result about the boundedness of the derivatives
of~$h_\la(y)$:
\begin{lem}\label{LEM:boundforderivsofh_la}
All derivatives of $h_\la(y)$  with respect to $y$ are bounded
uniformly in $y$. That is\textup{,} for each multi-index $\al$ there
exists a constant $C_\al>0$ such that
\begin{equation*}
\abs{\pa_y^\al h_\la(y)}\le C_\al\quad\text{for all }y\in
U,\,\la>0\,.
\end{equation*}
\end{lem}
\begin{proof}
By definition, $\tau(\la y,\la h_\la(y))=\la$. So,
\begin{multline*}
(\grad_{\xi'}\tau)(\la y,\la h_\la(y))+ (\pa_{\xi_n}\tau)(\la y,\la
h_\la(y))\grad_{y}h_\la(y) =\la^{-1}\grad_{y}[\tau(\la y,\la
h_\la(y))]=0\,,
\end{multline*}
or, equivalently,
\begin{equation}\label{EQ:eqnforh_laderiv}
\grad_{y}h_\la(y)=-\frac{(\grad_{\xi'}\tau)(\la y,\la h_\la(y))}
{(\pa_{\xi_n}\tau)(\la y,\la h_\la(y))}\,.
\end{equation}
Hypothesis~\ref{HYP:realtauisasymbol}
($\abs{\pa_\xi^\al\tau(\xi)}\le C_\al\brac{\xi}^{1-\abs{\al}}$ for
all $\xi\in\R^n$) and~\eqref{EQ:derivoftaubddbelow}
($\abs{\pa_{\xi_n}\tau(\la y,\la h_\la(y))}\ge C>0$) then ensure
that $\abs{\grad_y h_\la(y)}\le C$ for all $y\in U$, $\la>0$.

For higher derivatives, note that $\abs{(y,h_\la(y))}\le R_1$ by
hypothesis~\ref{HYP:reallimitofSi_la}; so, using
hypothesis~\ref{HYP:realtauisasymbol} once more, for all
multi-indices $\al$, there exists a constant $C_\al>0$ such that
\begin{equation*}
\abs{(\pa_\xi^\al\tau)(\la y,\la h_\la(y))}\le
C_\al\la^{1-\abs{\al}}\,.
\end{equation*}
Then, differentiating~\eqref{EQ:eqnforh_laderiv}, this ensures, by
an inductive argument, that the desired result for higher
derivatives of $h_\la(y)$ holds, proving the Lemma.
\end{proof}

Returning to the proof of~\ref{COND:a0derivsbdd}, note that,
\begin{equation*}
\abs{\pa_\xi^\al a_0(\xi)}\le C_\al\brac{\xi}^{-\abs{\al}}\text{ for
all }\xi\in\R^n\,,
\end{equation*}
since,~$\widetilde{a}(\xi)$ is a symbol of order $\frac{n-1}{\ga}-n$
(see~\eqref{EQ:defofatilde} for its definition). Together with
Lemma~\ref{LEM:boundforderivsofh_la}, this ensures that $\pa_y^\al[
a_0(\mu y,\mu h_\mu(y))$ is uniformly bounded for all $y\in U$,
$\mu>0$ and $\abs{\al}\le [\frac{n-1}{\ga}]+1$ as required.

Finally, we show~\ref{COND:Fisconvexphasefn}: observe that for
$\abs{\rho}<\de'$, some suitably small $\de'>0$,
\begin{align*}
F(\rho,\om,\mu;z)&=h_{\mu}(\rho\om+z)-h_\mu(z)-
\rho\om\cdot\grad_y h_{\mu}(z)\\
=&\sum_{k=2}^{\ga+1} \Big[
\sum_{\abs{\al}=k}\frac{1}{\al!}(\pa^\al_yh_\mu)(z)\om^\al\Big]\rho^k
+ R_{\ga+1}(\bar{\rho},\om,\mu;z)\rho^{\ga+2}\,.
\end{align*}
So, $F(\rho,\om,\mu;z)$ is a function of convex type~$\ga$ if (using
the numbering of Definition~\ref{DEF:phasefunctionofconvextype})
\begin{enumerate}[label=(CT\arabic*),leftmargin=*]
\addtocounter{enumi}{1}
\item\label{COND:convfn1inthm} $\sum_{k=2}^{\ga+1}
\absBig{\sum_{\abs{\al}=k}
\frac{1}{\al!}(\pa^\al_yh_\mu)(z)\om^\al}\ge C>0$ for all $\om\in
{\mathbb S}^{n-2}$, $\mu>0$, $z\in\R^{n-1}$.
\item\label{COND:convfn2inthm} 
$\abs{\pa_\rho F(\rho,\om,\mu;z)}$ is
increasing in $\rho$ for $0<\rho<\de$, for each $\om\in 
{\mathbb S}^{n-2}$,
$\mu>0$;
\item\label{COND:convfn3inthm} for each $k\in\N$,
$\pa_\rho^kF(\rho,\om,\mu;z)$ is bounded uniformly in
$0<\rho<\de'$, $\om\in {\mathbb S}^{n-2}$, $\mu>0$.
\end{enumerate}

Condition~\ref{COND:convfn3inthm}, follows straight from
Lemma~\ref{LEM:boundforderivsofh_la}.
The concavity of~$h_\mu(y)$ means that
\begin{align*}
\pa^2_\rho F(\rho,\om,\mu;z)= \pa^2_\rho[h_\mu(\rho\om+z)]
=\om^t\Hess h_\mu(\rho\om+z)\om\le 0
\end{align*}
for all $0<\rho<\de'$ and for each $\om\in {\mathbb S}^{n-2}$, $\mu>0$,
$z\in\R^{n-1}$; coupled with the fact that $\pa_\rho
F(0,\om,\mu;z)=0$, this ensures Condition~\ref{COND:convfn2inthm}
holds.

Lastly, recall that, by definition,~$\ga\ge\ga(\Si_\la)$ for all
$\la>0$, which is the maximal order of contact 
between~$\Si_\la$ and
its tangent plane; furthermore,~$\ga$ is assumed to be finite; thus,
for some $k\le\ga+1<\infty$, we have
\begin{equation*}
\pa_\rho^k[ h_\mu(z+\rho\om)]\big\vert_{\rho=0}\ne 0\,.
\end{equation*}
Now, $\pa_\rho^k[ h_\mu(z+\rho\om)]\big\vert_{\rho=0}=
\sum_{\abs{\al}=k}\frac{k!}{\al!}\pa_y^\al h_\mu(z)\om^\al$, so for
some $k\le\ga+1$, we have
$$k!\abslr{\sum_{\abs{\al}=k}\frac{1}{\al!}\pa_y^\al
h_\mu(z)\om^\al}\ge C>0$$ for all $\om\in {\mathbb S}^{n-2}.$
Thus, condition~\ref{COND:convfn1inthm} holds.

This completes the proof of conditions
\ref{COND:biscpctlysupp}--\ref{COND:a0derivsbdd}, and, hence,
Theorem \ref{THM:sugimoto/randolargument}.
\end{proof}

\section{Oscillatory integrals without convexity}
\label{SEC:ch-nonconvex}
Theorem~\ref{THM:sugimoto/randolargument} requires the phase
function to satisfy the convexity condition of
Definition~\ref{DEF:convexitycondition}; however, we will also
investigate solutions to hyperbolic equations for which the
characteristic roots do not necessarily satisfy 
such a condition. In
this section we state and prove a theorem for this case. First, we
give the key results that replaces Theorem~\ref{THM:oscintthm} in
the proof, the well-known van der Corput Lemma. We recall the 
standard van der Corput Lemma as given in, for
example,~\cite[Lemma 1.1.2]{sogg93}, or in
\cite[Proposition 2, Ch VIII]{stei93}:
\begin{lem}\label{LEM:VDCmain}
Let $\Phi\in C^\infty(\R)$ be real-valued\textup{,} 
$a\in C_0^\infty(\R)$ and
$m\ge2$ be an integer such that $\Phi^{(j)}(0)=0$ for $0\le j\le
m-1$ and $\Phi^{(m)}(0)\ne0$\textup{;} then
\begin{equation*}
\absBig{\int_0^\infty e^{i\la\Phi(x)}a(x)\,dx}\le 
C(1+\la)^{-1/m}\quad\text{for all}\quad\la\geq 0,
\end{equation*}
provided the support of~$a$ is sufficiently small. The constant on
the right-hand side is independent of~$\la$ and $\Phi$.
\end{lem}
If $m=1$, then the same result holds provided~$\Phi'(x)$ is
monotonic on the support of~$a$.

\subsection{Real-valued phase function}
In the case when the convexity condition holds the estimate of
Theorem~\ref{THM:sugimoto/randolargument} is given in terms of the
constant~$\ga$; as in the case of the homogeneous operators (see
Introduction, Section~\ref{SEC:homogoperators}) we introduce
an analog to this in the case where the convexity condition does
not hold.
Let~$\Si$ be a hypersurface in~$\R^n$; we set
\begin{equation*}
\ga_0(\Si):=\sup_{\si\in\Si}\inf_P\ga(\Si;\si,P)\le \ga(\Si)\,
\end{equation*}
where $\ga(\Si;\si,P)$ is as in Definition~\ref{DEF:ga(Si)}.

An important result for calculating this value is the following:
\begin{lem}[\cite{sugi96}]\label{LEM:propertyofgamma0} Suppose
$\Si=\set{(y,h(y)):y\in U}$, $h\in C^\infty(U)$, $U\subset\R^{n-1}$
is an open set, and let
\begin{equation*}
F(\rho)=h(\eta+\rho\om)-h(\eta)-\rho\grad h(\eta)\cdot\om
\end{equation*}
where $\eta\in U$, $\om\in {\mathbb S}^{n-2}$. Taking
$\si=(\eta,h(\eta))\in\Si$, $\om\in {\mathbb S}^{n-2}$ and
\begin{equation*}
P=\set{\si+s(\om,\grad h(\eta)\cdot\om)+t(-\grad
h(\eta),1)\in\R^n:s,t\in\R}\,,
\end{equation*}
then
\begin{equation*}
\ga(\Si;\si,P)=\min\set{k\in\N:F^{(k)}(0)\ne0}=:\ga(h;\eta,\om)\,.
\end{equation*}
Therefore,
\begin{gather*}
\ga(\Si)=\sup_\eta\sup_\om\ga(h;\eta,\om),\\
\ga_0(\Si)=\sup_\eta\inf_\om\ga(h;\eta,\om)\,.
\end{gather*}
\end{lem}
Now we are in a position to state and prove the result for
oscillatory integrals with a real-valued phase function that does
not satisfy the earlier introduced convexity condition. This is a
parameter dependent version of Corollary
\ref{cor:nonconvex}.
\begin{thm}\label{THM:noncovexargument}
Let $a(\xi)$ be a symbol of
order~$\frac{1}{\ga_0}-n$ of type $(1,0)$ on~$\R^n$.
Let $\tau:\R^n\to\R$ be smooth on $\supp a$, 
set $\ga_0:=\sup_{\la>0}\ga_0(\Si_\la(\tau))$ and assume it is
finite; furthermore, on $\supp a$, we also assume the following
conditions\textup{:}
\begin{enumerate}[label=\textup{(}\textup{\roman*}\textup{)}]
\item for all multi-indices $\al$ there exists a constant
$C_\al>0$ such that
\begin{equation*}
\abs{\pa_\xi^\al\tau(\xi)}\le
C_\al\brac{\xi}^{1-\abs{\al}};
\end{equation*}
\item there exist constants $M,C>0$ such that for all
$\abs{\xi}\ge M$ we have $\abs{\tau(\xi)}\ge C\abs{\xi}$\textup{;}
\item there exists a constant $C_0>0$ such that
$\abs{\pa_\om\tau(\la\om)}\ge C_0$ for all $\om\in
\Snm$\textup{,} $\la>0$\textup{;}
\item there exists a constant $R_1>0$ such that\textup{,} for all
$\la>0$\textup{,}
\[\frac{1}{\la}\Si_\la(\tau)\subset B_{R_1}(0)\,.\]
\end{enumerate}
Then\textup{,} the following estimate holds for all
$R\ge0$\textup{,} $x\in\R^n$\textup{,} $t>1$\textup{:}
\begin{equation*}
\absBig{\int_{\R^n}e^{i(x\cdot\xi+ \tau(\xi)t)}a(\xi)
g_R(\xi)\,d\xi} \le Ct^{-\frac{1}{\ga_0}}\,,
\end{equation*}
where $g_R(\xi)$ is as given in~\eqref{EQ:eqnforg_R} and $C>0$ is
independent of $R$.
\end{thm}
\begin{proof}
We follow the proof of Theorem~\ref{THM:sugimoto/randolargument}
as far as possible, and shall show how the absence of the
convexity condition affects the estimate. Thus, as in the proof
of  Theorem~\ref{THM:sugimoto/randolargument},
we may first assume, without loss of generality, that either
$\tau(\xi)\ge0$ for all $\xi\in\R^n$ or $\tau(\xi)\le0$ for all
$\xi\in\R^n$. We will always work on the support of $a$, so
by writing $\xi\in\Rn$ we will mean $\xi\in\supp a$.

Divide the integral into two parts:
\begin{gather*}
I_1(t,x):=\int_{\R^n}e^{i(x\cdot\xi+\tau(\xi)t)}a(\xi)g_R(\xi)
\cutoffWF\big(t^{-1}x+\grad\tau(\xi)\big)\,d\xi\,,\\
I_2(t,x):=\int_{\R^n}e^{i(x\cdot\xi+\tau(\xi)t)}a(\xi)g_R(\xi)
(1-\cutoffWF)\big(t^{-1}x+\grad\tau(\xi)\big)\,d\xi\,,
\end{gather*}
where $\cutoffWF\in C^\infty_0(\R^n)$, $0\le\cutoffWF(y)\le1$,
which is identically~$1$ in the ball of radius $r>0$ centred at
the origin, $B_r(0)$, and identically~$0$ outside the ball of
radius~$2r$, $B_{2r}(0)$. By Lemma~\ref{LEM:estawayfromwavefront}
(which does not require the phase function to satisfy the
convexity condition), we have
\begin{equation*}
\abs{I_2(t,x)}\le C_{r}t^{-1/\ga_0}\;\text{ for all }t>1.
\end{equation*}
To estimate $\abs{I_1(t,x)}$ we introduce, as before, a partition
of unity $\set{\cutoffcone_\ell(\xi)}_{\ell=1}^L$ and restrict
attention to
\begin{equation*}
I_1'(t,x)=\int_{\R^n}e^{i(x\cdot\xi+\tau(\xi)t)}a(\xi)g_R(\xi)
\cutoffcone_1(\xi)\cutoffWF
\big(t^{-1}x+\grad\tau(\xi)\big)\,d\xi\,,
\end{equation*}
where $\cutoffcone_1(\xi)$ is supported in a 
sufficiently narrow cone,~$K_1$, that
contains~$e_n=(0,\dots,0,1)$. Parameterise this cone in the same
way as above: with $U\subset\R^{n-1}$,
\begin{gather*}
K_1=\begin{cases}\set{(\la y,\la h_\la(y)):\la>0,\,y\in U}&\text{
if
}\tau(\xi)\ge0\text{ for all }\xi\in\R^n\\
\set{(\la y,\la h_\la(y)):\la<0,\,y\in U}&\text{ if
}\tau(\xi)\le0\text{ for all }\xi\in\R^n\,.
\end{cases}
\end{gather*}
Here the Implicit Function Theorem ensures the existence of a
smooth function $h_\la:U\to\R$ for each~$\la>0$, but there is one
major difference: the functions~$h_\la$ are not necessarily
concave, in contrast to the earlier proof. Using the change of
variables $\xi\mapsto(\la y,\la h_\la(y))$---note that
\begin{equation*}
0<C\le\absBig{\frac{d\xi}{d(\la, y)}}\le C\la^{n-1}
\end{equation*}
by the same argument as in the proof of
Theorem~\ref{THM:sugimoto/randolargument}, providing the width
of~$K_1$ is taken to be sufficiently small---gives
\begin{multline*}
I_1'(t,x)=\int_{0}^\infty\int_U e^{i[\la x'\cdot y+\la
x_nh_\la(y)+\tau(\la y,\la h_\la(y))t]}a(\la y,\la
h_\la(y))\\g_R(\la y,\la h_\la(y)) \cutoffcone_1(\la y,\la
h_\la(y))\cutoffWF \big(t^{-1}x+\grad\tau(\la y,\la
h_\la(y))\big)\frac{d\xi}{d(\la, y)}\,dy\,d\la\,.
\end{multline*}
Once again, let $\cutoffadjust\in C_0^\infty(\R)$ so that
$g_R(\xi)=g_R(\xi)\cutoffadjust(\tau(\xi)/\curlyR)$ (where
$\curlyR=\max(R,1)$) and
$\widetilde{a}(\xi)=a(\xi)g_R(\xi)\cutoffcone_1(\xi)$, which is a
symbol of order~$\frac{1}{\ga_0}-n$ supported in~$K_1$ and with
all the constants in the symbolic estimates independent of~$R$.
So, recalling that $\tau(\la y,\la h_\la(y))=\la$ and writing
$h(\la,y)\equiv h_\la(y)$, we get
\begin{align*}
I_1'(t,x)&=\int_{0}^\infty\int_U e^{i\la[x'\cdot
y+x_nh_\la(y)+t]}\widetilde{a}(\la y,\la h_\la(y))\\
&\qquad\qquad\qquad \cutoffadjust(\la/\curlyR)\cutoffWF
\big(t^{-1}x+\grad\tau(\la
y,\la h_\la(y))\big)\frac{d\xi}{d(\la,y)}\,dy\,d\la\\
=&\int_{0}^\infty\int_U e^{i\latil[\frac{\xtil'}{\xtil_n}\cdot
y+h\big(\frac{\latil}{\xtil_nt},y\big)+\xtil_n^{-1}]}
\widetilde{a}\Big(\textstyle\frac{\latil}{\xtil_nt}
y,\frac{\latil}{\xtil_nt} h\Big(\frac{\latil}{\xtil_nt},y\Big)\Big)\\
&\quad\cutoffadjust\Big(\textstyle\frac{\latil}{\curlyR\xtil_nt}\Big)
\cutoffWF\Big(\xtil+\grad\tau\Big(\frac{\latil}{\xtil_nt}
y,\frac{\latil}{\xtil_nt}
h\Big(\frac{\latil}{\xtil_nt},y\Big)\Big)\Big)
\displaystyle\frac{d\xi}{d(\la,y)}
\xtil_n^{-1}t^{-1}\,dy\,d\latil\,,
\end{align*}
where $x=t\widetilde{x}$ and
$\latil=\la x_n=\la\xtil_nt$. Thus, using
$\abs{\cutoffWF(\eta)}\le1$, we have
\begin{equation}\label{EQ:I'inrealnonconvex}
\abs{I_1'(t,x)}\le C\abs{\xtil_n}^{-1/\ga_0}t^{-1/\ga_0}
\int_0^\infty\absBig{
I\Big(\latil,\textstyle\frac{\latil}{\xtil_nt};\xtil_n^{-1}\xtil\Big)
\cutoffadjust\Big(\textstyle\frac{\latil}{\curlyR\xtil_nt}\Big)
\latil^{-1+(1/\ga_0)}}\,d\latil
\end{equation}
where
%\begin{multline*}
$$
I\Big(\latil,\textstyle\frac{\latil}{\xtil_nt};
\xtil_n^{-1}\xtil'\Big)\\
=\displaystyle\int_{U}e^{i\latil\big[\xtil_n^{-1}\xtil'\cdot
y+h\big(\frac{\latil}{\xtil_nt},y\big)\big]}
\widetilde{a}\Big(\textstyle\frac{\latil}{\xtil_nt}
y,\frac{\latil}{\xtil_nt}
h\Big(\frac{\latil}{\xtil_nt},y\Big)\Big)
\Big(\frac{\latil}{\abs{\widetilde{x}_n}t}\Big)^{n-\frac{1}{\ga_0}}\,dy\,.
$$
%\end{multline*}
At this point, we diverge from the proof of the earlier theorem
since we cannot apply Theorem~\ref{THM:oscintthm}; instead, note
that, for some $b\in C_0^\infty(\R^{n-1})$ with support 
contained in $U$, we have
\begin{multline*}
\absBig{I\Big(\latil,\textstyle\frac{\latil}{\xtil_nt};
\xtil_n^{-1}\xtil'\Big)} \le
\displaystyle\int_{\R^{n-2}}\absBig{\int_{\R}
e^{i\latil\big[\xtil_n^{-1}\xtil'\cdot
y+h\big(\frac{\latil}{\xtil_nt},y\big)\big]}\\
\widetilde{a}\Big(\textstyle\frac{\latil}{\xtil_nt}
y,\frac{\latil}{\xtil_nt}
h\Big(\frac{\latil}{\xtil_nt},y\Big)\Big)
\Big(\frac{\latil}{\abs{\widetilde{x}_n}t}\Big)^{n-\frac{1}{\ga_0}}
b(y)\,dy_1}\,dy'\,.
\end{multline*}
We wish to apply the van der Corput Lemma, Lemma~\ref{LEM:VDCmain},
to the inner integral. Set $\phasefn(y,\mu;z):=z\cdot y+h_\mu(y)$,
which is real-valued, and consider the integral
\begin{equation*}
\int_{\R}e^{i\la\phasefn(y,\mu;z)}a_0(y,\mu)b(y)\,dy_1
\end{equation*}
where $a_0(y,\mu):=\mu^{n-(1/\ga_0)}\widetilde{a}(\mu y,\mu h_\mu(y))$.
Recall that
\begin{equation*}
\Si_\mu=\set{(y,h_\mu(y)):y\in U}\,,
\end{equation*}
so by Lemma~\ref{LEM:propertyofgamma0},
\begin{equation*}
\min\set{k\in\N:\pa_{y_1}^{k}\phasefn(y,\mu;z)\big|_{y_1=0}\ne0}
=\ga(h_\mu;0,(1,0,\dots,0))=:m\,.
\end{equation*}
Fixing the size of~$U$ so that
$\abs{\pa_{y_1}^{(m)}\phasefn(y,\mu;z)}\ge\ep>0$ for all $y\in U$
ensures that the hypotheses of Lemma~\ref{LEM:VDCmain} are
satisfied. Thus, since the support of~$b$ is compact in~$\R^{n-1}$,
is contained in $U$,
and~$a_0$ is smooth, we obtain
\begin{equation*}
\absBig{\int_{\R}e^{i\la\phasefn(y,\mu;z)}a_0(y,\mu)b(y)\,dy_1}\le
C\la^{-1/m}\,.
\end{equation*}
Carry out a suitable change of coordinates so that
$m=\inf_\om\ga(h_\mu;0,\om)$ (this is possible due to the rotational
invariance of all properties used); then, since $m\le\ga_0$ by
definition, we have
\begin{equation*}
\absBig{I\Big(\latil,\textstyle\frac{\latil}{\xtil_nt};
\xtil_n^{-1}\xtil'\Big)} \le C\latil^{-1/\ga_0}\,,
\end{equation*}
for all~$\latil$ such that $\frac{\latil}{\curlyR\xtil_nt}\in\supp
G$ (this is to ensure $\latil$ is away from the origin). Combining
this with~\eqref{EQ:I'inrealnonconvex} then gives the required
estimate:
\begin{align*}
\abs{I_1'(t,x)}&\le
C\abs{\xtil_n}^{-{1}/{\ga_0}}t^{-{1}/{\ga_0}}
\int_0^\infty\absBig{\latil^{-1}
G\Big(\textstyle\frac{\latil}{\curlyR\xtil_nt}\Big)}\,d\latil\\
=&C\abs{\xtil_n}^{-{1}/{\ga_0}}t^{-{1}/{\ga_0}}
\int_0^\infty (\nu\curlyR \xtil_nt)^{-1}
G(\nu)\curlyR\xtil_nt\,d\nu\le Ct^{-\frac{1}{\ga_0}}\,.\qedhere
\end{align*}
\end{proof}

\section{Decay of solutions to the Cauchy problem}
\label{SEC:Cauchy}

Recall that we begin with the Cauchy problem with solution
$u=u(t,x)$ 
\begin{equation}\label{EQ:standardCP(ch4)}
\left\{\begin{aligned}& \pat^m
u+\sum_{j=1}^{m}P_{j}(\pax)\pat^{m-j}u+
\sum_{l=0}^{m-1}\sum_{\abs{\al}+r=l}
c_{\al,r}\pax^\al\pat^ru=0,\quad t>0,\\
&\pat^lu(0,x)=f_l(x)\in C_0^{\infty}(\R^n),\quad l=0,\dots,m-1,\;
x\in\R^n\,,
\end{aligned}\right.
\end{equation}
where $P_j(\xi)$, the polynomial obtained from the operator
$P_j(\pax)$ by replacing each derivative
$D_{x_k}=\frac{1}{i}\partial_{x_k}$ 
by~$\xi_k$, is a constant
coefficient homogeneous polynomial of order~$j$, and
the~$c_{\al,r}$ are constants. In this section we will
prove different parts of Theorem \ref{THM:overallmainthm}.

\subsection{Representation
of the solution}

Applying the partial Fourier transform with respect to~$x$ 
yields an
ordinary differential equation for $\widehat{u}=
\widehat{u}(t,\xi):=\int_{\R^n}e^{-ix\cdot\xi}u(t,x)\,dx$:
\begin{subequations}
\begin{align}\label{EQ:ftcauchyprob}
\pat^m\widehat{u}+\sum_{j=1}^mP_j(\xi)\pat^{m-j}\widehat{u}+
\sum_{l=0}^{m-1}\sum_{\abs{\al}+r=l}c_{\al,r}
\xi^\al\pat^r\widehat{u}=0\,,
\\
\pat^l\widehat{u}(0,\xi)=\widehat{f_l}(\xi),\quad
l=0,\dots,m-1,\label{EQ:ftcauchydata}
\end{align}
where $(t,\xi)\in [0,\infty)\times\R^n$ and $P_j(\xi)$
are symbols of $P_j(D_x)$. Let $E_j=E_j(t,\xi)$,
$j=0,\dots,m-1$, be the solutions to~\eqref{EQ:ftcauchyprob} with
initial data
\begin{equation}\label{EQ:initialdataforEj}
\pat^lE_j(0,\xi)=\begin{cases}1\quad\text{ if }l=j,\\
0\quad\text{ if }l\ne j.\end{cases}
\end{equation}
\end{subequations}
Then the solution~$u$ of~\eqref{EQ:standardCP(ch4)} can be written
in the form
\begin{equation}\label{EQ:repforu}
u(t,x)=\sum_{j=0}^{m-1}(\FT^{-1}E_j\FT f_j)(t,x),
\end{equation}
where $\FT$ and $\FT^{-1}$ represent the partial Fourier transform
with respect to $x$ and its inverse, respectively.

Now, as~\eqref{EQ:ftcauchyprob},~\eqref{EQ:ftcauchydata} is
the Cauchy problem for a linear ordinary differential equation, we
can write, denoting the characteristic roots
of~\eqref{EQ:standardCP(ch4)} by $\tau_1(\xi),\dots,\tau_m(\xi)$,
\begin{equation*}
E_j(t,\xi)=\sum_{k=1}^mA^k_j(t,\xi)e^{i\tau_k(\xi)t},
\end{equation*}
where $A^k_j(t,\xi)$ are polynomials in~$t$ whose coefficients
depend on~$\xi$. Moreover, for each $k=1,\dots,m$ and
$j=0\dots,m-1$, the $A_j^k(t,\xi)$ are independent of~$t$ at
points of the (open) set $\set{\xi\in\R^n:\tau_k(\xi)\ne
\tau_l(\xi)\,\forall\,l\ne k}$; when this is the case, we write
$A_j^k(t,\xi)\equiv A_j^k(\xi)$. In particular,  
there exists $M>0$ such that if
$\abs{\xi}\geq M$, the roots are pairwise distinct. For $A_j^k(\xi)$,
we have the following properties:
\begin{lem}\label{LEM:ordercoeff}
Suppose $\xi\in S_k:=\set{\xi\in\R^n:\tau_k(\xi)\ne
\tau_l(\xi)\,\forall\,l\ne k}$\textup{;} then we have the
following formula\textup{:}
\begin{equation}\label{EQ:Ajkformula}
A_j^k(\xi)=\frac{(-1)^j\displaystyle\sideset{}{^k}\sum_{1\le
s_1<\dots<s_{m-j-1}\le
m}\prod_{q=1}^{m-j-1}\tau_{s_q}(\xi)}{\displaystyle\prod_{l=1,l\ne
k}^m(\tau_l(\xi)-\tau_k(\xi))}\;,
\end{equation}
where $\sum^k$ means sum over the range indicated excluding $k$.
Furthermore\textup{,} we have\textup{,} for each $j=0,\dots,m-1$
and $k=1,\dots,m$\textup{,}
\begin{enumerate}[label=\textup{(}\roman*\textup{)},leftmargin=*]
\item $A_j^k(\xi)$ is smooth in~$S_k$\textup{;}
\item $A_j^k(\xi)=O(\abs{\xi}^{-j})$ as $\abs{\xi}\to\infty$.
\end{enumerate}
\end{lem}
\begin{proof}
The representation~\eqref{EQ:Ajkformula} 
follows from Cramer's rule
(and is done explicitly in~\cite{klin67}): $A_j^k(\xi)=\frac{\det
V_j^k}{\det V}$, where 
$V:=\big(\tau_i^{l-1}(\xi)\big)_{i,l=1}^m$ is
the Vandermonde matrix and~$V_j^k$ is the matrix obtained by
taking~$V$ and replacing the $k^{\text{th}}$ column by
$(\underbrace{0\ \dots\ 0\ 1}_j\ 0\ \dots\ 0)^{\mathrm T}$.

Smoothness of $A_j^k(\xi)$ then follows by
Proposition~\ref{PROP:ctyofroots} and the asymptotic 
behaviour is
a consequence of \ref{LEM:orderroot} of
Proposition~\ref{PROP:perturbationresults}
since~\eqref{EQ:Ajkformula} holds for all $\abs{\xi}\geq M$.
\end{proof}

\subsection{Division of the
integral}\label{SEC:step2}

We choose $M>0$ so that all roots
$\tau_k(\xi)$, $k=1,\dots,n$, are
distinct for $\abs{\xi}\geq M$. 
Let $\chi=\cutoffN(\xi)\in C_0^\infty(\R^n)$,
$0\le\cutoffN(\xi)\le1$, be a cut-off function that is
identically~$1$ for $\abs{\xi}<M$ and identically zero for
$\abs{\xi}>2M$. Then \eqref{EQ:repforu} can be rewritten as:
\begin{equation}\label{EQ:divideduprepforu}
u(t,x)=\sum_{j=0}^{m-1}\FT^{-1}(E_j\cutoffN\FT f_j)(t,x)
+\sum_{j=0}^{m-1}\FT^{-1}(E_j(1-\cutoffN)\FT f_j)(t,x)\,.
\end{equation}

\paragraph{Large~$\abs{\xi}$:} The second term
of~\eqref{EQ:divideduprepforu} is the most straightforward to
study: by the choice of~$M$, we have
\begin{equation*}
E_j(t,\xi)(1-\cutoffN)(\xi)=\sum_{k=1}^m
e^{i\tau_k(\xi)t} A_j^k(\xi)(1-\cutoffN)(\xi)
\,;
\end{equation*}
therefore, since each summand is smooth in~$\R^n$, we can write
\begin{multline*}
\sum_{j=0}^{m-1}\FT^{-1}(E_j(1-\cutoffN)\FT
f_j)(t,x)\\=\frac{1}{(2\pi)^{n}}\sum_{j=0}^{m-1}\sum_{k=1}^m
\int_{\R^n} e^{i(x\cdot\xi+\tau_k(\xi)t)}
A_j^k(\xi)(1-\cutoffN)(\xi)\widehat{f}_j(\xi)\,d\xi\,.
\end{multline*}
Each of these integrals may be studied separately. Note that, unlike
in the cases of the wave equation, Brenner~\cite{bren75}, and the
general $m^{\text{th}}$ order homogeneous strictly hyperbolic
equations, Sugimoto~\cite{sugi94}, we may not assume that $t=1$.
The $L^p-L^q$
estimates obtained under different conditions on the phase function
for operators of this type are given in
Section~\ref{SEC:largexistep} below.

\paragraph{Bounded~$\abs{\xi}$:} We turn our attention to the
terms of the first sum in \eqref{EQ:divideduprepforu}, the case of
bounded frequencies,
\begin{equation}\label{EQ:bddxiintegral}
\FT^{-1}(E_j\cutoffN\FT f)(t,x)=
\frac{1}{(2\pi)^n}\int_{\R^n}e^{ix\cdot\xi}\Big(\sum_{k=1}^m
e^{i\tau_k(\xi)t}
A_j^k(t,\xi)\Big)\cutoffN(\xi)\widehat{f}(\xi)\,d\xi\,.
\end{equation}
Unlike in the case above, here the characteristic roots
$\tau_1(\xi),\dots,\tau_m(\xi)$ are not necessarily distinct at
all points in the support of the integrand (which is contained in
the ball of radius~$2M$ about the origin); in particular, this
means that the $A_j^k(t,\xi)$ may genuinely depend on~$t$ and we have
no simple formula valid for them in the whole region.

For this reason, we begin by systematically separating
neighbourhoods of points where roots meet---referred to henceforth
as multiplicities---from the rest of the region, and then
considering the two cases separately. In
Section~\ref{SEC:bddxiawayfrommults} we find $L^p-L^q$ estimates
in the region away from multiplicities under various conditions;
in Section~\ref{SEC:bddxiaroundmults} we show how these differ in
the neighbourhoods of singularities.
First, we need to understand in what type of sets the roots
$\tau_k(\xi)$ can intersect:
\begin{lem}\label{LEM:multiplerootssetisnice}
The complement of the set of multiplicities of a linear strictly
hyperbolic constant coefficient partial differential
operator~$L(D_t,D_x)$\textup{,}
\begin{equation*}
S:=\set{\xi\in\R^n:\tau_j(\xi)\ne\tau_k(\xi)\text{ for all }j\ne
k}\,,
\end{equation*}
is dense in~$\R^n$.
\end{lem}
\begin{proof}
First note
\begin{equation*}
S=\set{\xi\in\R^n:\De_L(\xi)\ne0}\,,
\end{equation*}
where~$\De_L$ is the discriminant of~$L(\tau,\xi)$ 
(see the proof of
Proposition \ref{PROP:ctyofroots}
for definition and some properties). Now,
by Sylvester's Formula (see~\cite{gelf+kapr+zele94}), 
$\De_L$ is a polynomial in the coefficients of
$L(\tau,\xi)$, which are themselves polynomials in~$\xi$.
Hence,~$\De_L$ is a polynomial in~$\xi$; as it is not identically
zero (for large~$\abs{\xi}$, the characteristic roots are distinct,
and hence it is non-zero at such points), it cannot be zero on an
open set, and hence its complement is dense in~$\R^n$.
\end{proof}

\begin{cor}\label{COR:multiplerootsetisnice}
Let $L(D_t,D_x)$ be a linear strictly hyperbolic constant
coefficient partial differential operator with characteristic
roots $\tau_1(\xi),\dots,\tau_m(\xi)$. Suppose,
for $k\ne l$, that
$\curlyM_{kl}\subset\R^n$ is the set of all $\xi$ such that
$\tau_{k}(\xi)=\tau_{l}(\xi)$. For $\ep>0$, define
\begin{equation*}
\curlyM_{kl}^\ep:=\set{\xi\in\R^n:\dist(\xi,\curlyM_{kl})<\ep}\,;
\end{equation*}
denote the largest $\nu\in\N$ such that
$\meas(\curlyM_{kl}^{\ep})\le C\ep^{\nu}$ for all sufficiently small
$\ep>0$ by $\codim\curlyM_{kl}$. Then $\codim\curlyM_{kl}\ge 1$.
\end{cor}
\begin{proof}
Follows straight from Lemma~\ref{LEM:multiplerootssetisnice}: the
fact that $\curlyM_{kl}$ has non-empty interior 
(it is an algebraic set)
ensures that its
$\ep$-neighbourhood is bounded by~$C\ep$ in at least one dimension
for all small $\ep>0$.
\end{proof}
We can note that if $L(D_t,D_x)$ is not differential,
but pseudo-differential in $D_x$, the rest of the analysis
goes through in a similar way, but we may need to assume
that $\codim\curlyM_{kl}\ge 1$.

With this in mind, we shall subdivide the
integral~\eqref{EQ:bddxiintegral}: suppose~$L$ roots meet in a
set~$\curlyM$ with $\codim\curlyM=\ell$; without loss of
generality, by relabelling, assume the coinciding roots are
$\tau_1(\xi),\dots,\tau_L(\xi)$. By continuity, there exists an
$\ep>0$ such that 
they do not intersect other roots $\tau_{L+1},\ldots,\tau_m$
in $\curlyM^{\ep}$. Furthermore, we may assume
that $\pa\curlyM^{\ep}\in C^1$: for each $\ep>0$ there exists a
set~$S_\ep$ with~$C^1$ boundary such that $\curlyM^{\ep}\subset
S_\ep$ and $\meas(S_\ep\backslash\curlyM^{\ep})\to 0$ as $\ep\to0$.
Then:
\begin{enumerate}[leftmargin=*]
\item Let $\cutoffM_{\curlyM,\ep}\in C^\infty(\R^n)$ be a smooth
function identically~$1$ on~$\curlyM^{\ep}$ and identically zero
outside $\curlyM^{2\ep}$; now consider the subdivision
of~\eqref{EQ:bddxiintegral}:
\begin{multline*}
\int_{B_{2M}(0)}e^{ix\cdot\xi}E_j(t,\xi)\widehat{f}(\xi)\,d\xi
=\int_{B_{2M}(0)}e^{ix\cdot\xi}E_j(t,\xi)\cutoffM_{\curlyM,\ep}(\xi)
\widehat{f}(\xi)\,d\xi\\+
\int_{B_{2M}(0)}e^{ix\cdot\xi}E_j(t,\xi)(1-\cutoffM_{\curlyM,\ep})(\xi)
\widehat{f}(\xi)\,d\xi\,;
\end{multline*}
for the second integral, simply repeat the above procedure around
any root multiplicities in $B_{2M}(0)\setminus\curlyM^{\ep}$.
\item For the first integral, the case where the integrand is
supported on $\curlyM^{\ep}$, split off the coinciding roots from
the others:
\begin{multline}\label{EQ:intdivisionsplittingoffmults}
\int_{B_{2M}(0)}e^{ix\cdot\xi}E_j(t,\xi)\cutoffM_{\curlyM,\ep}(\xi)
\widehat{f}(\xi)\,d\xi
\\=\int_{B_{2M}(0)}e^{ix\cdot\xi}\Big(\sum_{k=1}^L
e^{i\tau_k(\xi)t}A_j^k(t,\xi)\Big)
\cutoffM_{\curlyM,\ep}(\xi)\widehat{f}(\xi)\,d\xi
\\+\int_{B_{2M}(0)}e^{ix\cdot\xi} \Big(\sum_{k=L+1}^m
e^{i\tau_k(\xi)t}A_j^k(t,\xi)\Big)
\cutoffM_{\curlyM,\ep}(\xi)\widehat{f}(\xi)\,d\xi.
\end{multline}
\item For the first integral, we use techniques discussed in
Section~\ref{SEC:bddxiaroundmults} below to estimate it.

\item For the second there are two possibilities: firstly, two or
more of the characteristic 
roots $\tau_{L+1}(\xi),\dots,\tau_m(\xi)$ coincide in
$B_{2M}(0)$---in this case, repeat the procedure above for
this integral. Alternatively, these roots are all distinct in
$B_{2M}(0)\backslash \curlyM^{\ep}$---in 
this case, it suffices to study each
integral separately as the $A_k^j(t,\xi)$ are independent of~$t$,
and thus the expression~\eqref{EQ:Ajkformula} is valid and we can
write
\begin{multline*}
\int_{B_{2M}(0)}e^{ix\cdot\xi} \Big(\sum_{k=L+1}^m
e^{i\tau_k(\xi)t}A_j^k(t,\xi)\Big)
\cutoffM_{\curlyM,\ep}(\xi)\widehat{f}(\xi)\,d\xi\\=
\sum_{k=L+1}^m\int_{B_{2M}(0)}e^{i[x\cdot\xi+\tau_k(\xi)t]}
A_j^k(\xi)\cutoffM_{\curlyM,\ep}(\xi)\widehat{f}(\xi)\,d\xi\;;
\end{multline*}
estimates for integrals of the type on the right-hand side are found
in Section~\ref{SEC:bddxiawayfrommults}---
note that in this case we
may use that the region is bounded to ensure that all continuous
functions are also bounded.
\end{enumerate}
Continue this procedure until all multiplicities are accounted for
in this way.

\bigskip
\noindent
Finally, let us recall
the following result that can be found in \cite[Theorem
6.4.5]{berg+lofs76}:
\begin{thm}\label{thm:maininterpolationthm}
Suppose~$T$ is a linear map such that it maps
\begin{gather*}
T: W^{s_0}_{p_0}\to L^{q_0}\,,\quad T:W^{s_1}_{p_1}\to L^{q_1}\,,
\end{gather*}
where $s_0\ne s_1$, $1\le p_0,p_1<\infty$; then $T$ also maps:
\begin{equation*}
T:W^{s_\theta}_{p_\theta} \to L^{q_\theta}\,,
\end{equation*}
where $0\leq\theta\leq 1$ and
\begin{gather*}
\frac{1}{p_\theta}=\frac{1-\theta}{p_0}+\frac{\theta}{p_1}\,,\quad
\frac{1}{q_\theta}=\frac{1-\theta}{q_0}+\frac{\theta}{q_1}\,,\quad
s_\theta=(1-\theta)s_0+\theta s_1\,.
\end{gather*}
That is, $\norm{Tf}_{L^{q_\theta}}\le
C\norm{f}_{W_{p_\theta}^{s_\theta}}$ and $C$ is independent of
$f\in W_{p_\theta}^{s_\theta}$.
\end{thm}

In particular, this means that if we have estimates
\begin{equation*}
\norm{Tf}_{L^\infty}\le Ct^{d_0}\norm{f}_{W_1^{N_0}}\,,\quad
\norm{Tf}_{L^2}\le Ct^{d_1}\norm{f}_{W_2^{N_1}}\,,
\end{equation*}
then
\begin{equation*}
\norm{Tf}_{L^q}\le C\bract{t}^{d_p}\norm{f}_{W_p^{N_p}}
\end{equation*}
where $\frac1p+\frac1q=1$, $1\leq p\leq 2$,
$N_p=N_0\big(\frac{1}{p}-\frac{1}{q}\big)
+\frac{2}{q}N_1$ and $d_p=d_0\big(\frac{1}{p}-\frac{1}{q}\big)
+\frac{2}{q}d_1$.
As usual, this 
reduces our task to finding $L^1-L^\infty$ and $L^2-L^2$
estimates in each case.

\subsection{Estimates for
large frequencies}\label{SEC:largexistep}

Via the division of the integral above, it suffices to find
$L^p-L^q$ estimates for integrals of the form
\begin{equation*}
\int_{\R^n}e^{i(x\cdot\xi+\tau(\xi)t)} a_j(\xi)\widehat{f}(\xi)\,d\xi\,,
\end{equation*}
where $a_j(\xi)=O(\abs{\xi}^{-j})$ as $|\xi|\to\infty$ 
is smooth (or
is zero in a neighbourhood of $0$), and $\tau(\xi)$ is a
complex-valued, smooth function which is
$O(\abs{\xi})$ as $|\xi|\to\infty$ and $\Im\tau(\xi)\ge0$ 
for all
$\xi\in\R^n$. Note that $\tau(\xi)$ does not have to be
homogeneous.

By further judicious use of cut-off functions, we
can split the considerations into the following cases of
Theorem \ref{THM:overallmainthm}:
\begin{enumerate}[leftmargin=*]
\item $\tau(\xi)$ is separated from the real axis, i.e.~there
exists $\de>0$ such that $\Im\tau(\xi)\ge\de$ for all
$\abs{\xi}\geq M$ (Theorem \ref{THM:expdecay});
\item $\tau(\xi)$ lies on the real axis
(this case is contained in
Theorems \ref{THM:nondeghess}--\ref{THM:noncovexargument-sp}
since $\tau$ is real valued);
%\item $\tau(\xi)$ tends asymptotically to the real axis as
%$\abs{\xi}\to\infty$ (covered by
%Theorem \ref{THM:noncovexargument-sp}).
\end{enumerate}
Let us look at each of these in turn. We will not consider
the case of $\tau(\xi)$ tending asymptotically to the real 
axis as $\abs{\xi}\to\infty$ since it is not part of
Theorem \ref{THM:overallmainthm} and since we do not
have at present any examples of such behaviour.

\subsection{Phase separated from the real axis:
Theorem \ref{THM:expdecay}}
\label{SEC:pfth21}
In this section, we consider the case where characteristic root
$\tau(\xi)$ is separated from the real axis for large~$\abs{\xi}$;
let us define $\de>0$ to be a constant such that
$\Im\tau(\xi)\ge\de$ for all $\abs{\xi}\ge M$.
Again, $\chi$ is a cut-off to the region (which may be unbounded)
where these properties hold.

We claim that, for all $t\geq 0$, we have
\begin{gather*}
\normBig{D^r_tD_x^\al\Big(\int_{\R^n}e^{i(x\cdot\xi+\tau(\xi)t)}
a_j(\xi)\chi(\xi)\widehat{f}(\xi)\,dx\Big)}_{L^\infty} \le Ce^{-\de
t}\norm{f}_{W_1^{N_1+\abs{\al}+r-j}}\,,\\
\normBig{D^r_tD_x^\al\Big(\int_{\R^n}e^{i(x\cdot\xi+\tau(\xi)t)}
a_j(\xi)\chi(\xi)\widehat{f}(\xi)\,dx\Big)}_{L^2} \le Ce^{-\de
t}\norm{f}_{W_2^{\abs{\al}+r-j}}\,,
\end{gather*}
where $N_1>n$, $r\ge0$, $\al$ multi-index. Indeed, these follow
immediately from:
\begin{prop}\label{PROP:rootsawayfromaxis}
Let $\tau:U\to\C$ be a smooth function, $U\subset\R^n$ open, and
$a_j=a_j(\xi)\in S^{-j}_{1,0}(U)$. Assume\textup{:}
\begin{enumerate}[leftmargin=*,label=\textup{(\roman*)}]
\item there exists $\de>0$ such that $\Im\tau(\xi)\ge\de$ for all
$\xi\in U$\textup{;}
\item $\abs{\tau(\xi)}\le C\brac{\xi}$ for all $\xi\in U$.
\end{enumerate}
Then\textup{,}
\begin{gather*}
\normBig{\int_{U}e^{i(x\cdot\xi+\tau(\xi)t)}a_j(\xi)
\xi^{\al}\tau(\xi)^r \widehat{f}(\xi)\,d\xi}_{L^\infty(\R^n_x)} \le
Ce^{-\de
t}\norm{f}_{W^{N_1+\abs{\al}+r-j}_1}\\
\intertext{and}
\normBig{\int_{U}e^{i(x\cdot\xi+\tau(\xi)t)}a_j(\xi)
\xi^{\al}\tau(\xi)^r \widehat{f}(\xi)\,d\xi}_{L^2(\R^n_x)} \le
Ce^{-\de t}\norm{f}_{W^{\abs{\al}+r-j}_2}
\end{gather*}
for all $t\geq 0$\textup{,} $N_1>n$\textup{,} multi-indices~$\al$,
$r\in\R$ and $\widehat{f}\in C_0^\infty(U)$.
\end{prop}
Note that in the case of $r=0$, condition (ii) may be omitted.
\begin{proof}
By the hypotheses on $\tau(\xi)$ and $a_j(\xi)$, we can estimate
\begin{align*}
\absBig{\int_{U}&e^{i(x\cdot\xi+\tau(\xi)t)}a_j(\xi)
\xi^{\al}\tau(\xi)^r\widehat{f}(\xi)\,d\xi}\le
\int_{U}\abs{e^{i\tau(\xi)t}a_j(\xi)}
\abs{\xi}^{\abs{\al}}\abs{\tau(\xi)}^r
\abs{\widehat{f}(\xi)}d\xi\\
&=\int_{U}e^{-\Im\tau(\xi)t}\abs{a_j(\xi)}
\abs{\xi}^{\abs{\al}}\abs{\tau(\xi)}^r \abs{\widehat{f}(\xi)}d\xi \le
Ce^{-\de t}\int_{U}\bra{\xi}^{\abs{\al}+r-j}
\abs{\widehat{f}(\xi)}\,d\xi\\
& \le Ce^{-\de t}\int_{U} \jp{\xi}^{-N_1}d\xi\;
\normbig{\bra{\xi}^{N_1+\abs{\al}+r-j}
\abs{\widehat{f}(\xi)}}_{L^\infty} \le Ce^{-\de
t}\norm{f}_{W^{N_1+\abs{\al}+r-j}_1}\,.
\end{align*}
This proves the first inequality. For the second, note that
Plancherel's theorem implies
\begin{equation*}
\normBig{\int_{U}e^{i(x\cdot\xi+\tau(\xi)t)}a_j(\xi)
\xi^{\al}\tau(\xi)^r \widehat{f}(\xi)\,d\xi}_{L^2(\R^n_x)}
=\normbig{e^{i\tau(\xi)t}a_j(\xi) \xi^{\al}\tau(\xi)^r
\widehat{f}(\xi)}_{L^2(U)};
\end{equation*}
then,
\begin{align*}
\int_{U}\absbig{e^{i\tau(\xi)t}a_j(\xi)
\xi^{\al}&\tau(\xi)^r\widehat{f}(\xi)}^2\,d\xi\\
&\le \int_{U}e^{-2\Im\tau(\xi)t}\abs{a_j(\xi)}^2
\abs{\xi}^{2\abs{\al}}\abs{\tau(\xi)}^{2r}
\abs{\widehat{f}(\xi)}^2d\xi\\
&\le Ce^{-2\de t}\int_{U}\bra{\xi}^{2(\abs{\al}+r-j)}
\abs{\widehat{f}(\xi)}^2\,d\xi\le Ce^{-2\de
t}\norm{f}_{W^{\abs{\al}+r-j}_2}^2\,.
\end{align*}
This completes the proof of the proposition.
\end{proof}
We note that there may be different version of the 
$L^\infty$-estimate for the integral in Proposition 
\ref{PROP:rootsawayfromaxis}. For example, applying 
Cauchy--Schwartz inequality to the estimate
$$
\absBig{\int_{U} e^{i(x\cdot\xi+\tau(\xi)t)}a_j(\xi)
\xi^{\al}\tau(\xi)^r\widehat{f}(\xi)\,d\xi} \leq
Ce^{-\de t}\int_{U}\bra{\xi}^{\abs{\al}+r-j}
\abs{\widehat{f}(\xi)}\,d\xi
$$
established in the proof, we get
$$
\int_{U}\bra{\xi}^{\abs{\al}+r-j}
\abs{\widehat{f}(\xi)}\,d\xi \leq 
\left(\int_{U} \jp{\xi}^{-2N_1^\prime}d\xi\right)^{1/2}\;
\left({\int_U \bra{\xi}^{2N_1^\prime+2\abs{\al}+2r-2j}
\abs{\widehat{f}(\xi)}^2}d\xi\right)^{1/2}, 
$$
from which we obtain the estimate
\begin{equation}\label{eq:Matsumura}
\absBig{\int_{U} e^{i(x\cdot\xi+\tau(\xi)t)}a_j(\xi)
\xi^{\al}\tau(\xi)^r\widehat{f}(\xi)\,d\xi}   \leq
Ce^{-\de
t}\norm{f}_{W^{N_1^\prime+\abs{\al}+r-j}_2},
\end{equation}
with\footnote{Here $N_1^\prime$ does not have to be an integer.}
 $N_1^\prime>\frac{n}{2}$. Interpolating with the
$L^2$-estimate from Proposition \ref{PROP:rootsawayfromaxis}
yields estimate \eqref{EQ:expdecayu-l2} in Section
\ref{SEC:21}.

From Proposition \ref{PROP:rootsawayfromaxis}, 
by the interpolation Theorem~\ref{thm:maininterpolationthm},
we get
\begin{equation*}
\normBig{D^r_tD_x^\al\Big(\int_{\R^n}e^{i(x\cdot\xi+\tau(\xi)t)}
a_j(\xi)\chi(\xi)\widehat{f}(\xi)\,dx\Big)}_{L^q} \le Ce^{-\de
t}\norm{f}_{W_p^{N_p+\abs{\al}+r-j}}\,,
\end{equation*}
where $\frac1p+\frac1q=1$, $1<p\le 2$, $N_p\ge
n\big(\frac{1}{p}-\frac{1}{q}\big)$, $r\ge0$,~$\al$ a multi-index
and $f\in C_0^\infty(\R^n)$. Thus, in this case we have
exponential decay of the solution.
This proves the first part of Theorem \ref{THM:expdecay}.
The second part of the statement of Theorem \ref{THM:expdecay}
is a straightforward consequence.

\subsection{Non-degenerate phase: Theorems
\ref{THM:nondeghess} and \ref{THM:nondeghess2}}
\label{SEC:mainstatphasesection}
In this section, we will prove Theorems
\ref{THM:nondeghess} and \ref{THM:nondeghess2}
and discuss the behavior of critical points of the phase.
In fact, we will prove Theorem \ref{THM:nondeghess} since
the proof of Theorem \ref{THM:nondeghess2} can be
given in the same way after restricting to a subset of
variables on which the non-degenerate matrix $A(\xi^0)$ is
attained (possibly after a coordinate change).
We will not write a further cut-off function $\chi$
to a set $U$ as in
Theorems \ref{THM:nondeghess} and \ref{THM:nondeghess2}
to ensure that the results that we obtain are uniform over
the positions of such sets $U$. However, we will keep in
mind that we are only interested in the local in frequency
region here, so all the integrals are convergent.
So, we first consider the case where we have
\begin{equation*}
\int_{\R^n}
e^{i(\xtil\cdot\xi+\tau(\xi))t}a(\xi)\widehat{f}(\xi)\,d\xi\,,
\end{equation*}
and $\det\Hess\tau(\xi)\ne0$ for all $\xi\in\supp a$.
Here we denote $\xtil=t^{-1}x$.
 To estimate
this, we first consider the oscillatory integral
\begin{equation*}
\int_{\R^n} e^{i(\xtil\cdot\xi+\tau(\xi))t}a(\xi)\,d\xi\,,
\end{equation*}
where $a=a(\xi)\in S_{1,0}^{-\mu}$, some $\mu\in\R$,
$\Im\tau(\xi)\ge0$ for all $\xi\in\R^n$, and, for some
$\xi^0\in\R^n$, $\xtil+\grad_\xi\tau(\xi^0)=0$ and
$\det\Hess\tau(\xi^0)\ne0$; we refer to~$\xi^0$ as a (non-degenerate)
critical point and we
microlocalise around it. 
Let us assume that~$\xi^0$ is the only such critical
point---if there are more than one, we use suitable cut-off
functions to localise around each separately (we assume the set of
critical points has no accumulation points). Indeed, let 
$\vartheta\in
C_0^\infty(\R^n)$ be supported in a neighbourhood~$V$
of~$\xi^0$ so that there are no other critical points in~$V$. Then
consider separately
\begin{equation*}
\int_{\R^n}
e^{i(\xtil\cdot\xi+\tau(\xi))t}a(\xi)\vartheta(\xi)\,d\xi
\;\;\textrm{ and } \int_{\R^n}
e^{i(\xtil\cdot\xi+\tau(\xi))t}a(\xi)(1-\vartheta)(\xi)\,d\xi\,.
\end{equation*}
The second integral, which we may assume contains no critical points
in its support (otherwise introduce further cut-off functions around
those), can be shown to decay faster than any power of $t$: 
note that away from the
critical points, we can use the equality
\begin{equation*}
e^{i(\xtil\cdot\xi+\tau(\xi))t}=
\frac{\xtil+\grad\tau(\xi)}{it\abs{\xtil+\grad\tau(\xi)}^2}\cdot
\grad_\xi[e^{i(\xtil\cdot\xi+\tau(\xi))t}]\,;
\end{equation*}
so, integrating by parts repeatedly shows that for any $N\in\N$
sufficiently large,
\begin{equation*}
\absBig{\int_{\R^n}
e^{i(\xtil\cdot\xi+\tau(\xi))t}a(\xi)(1-\vartheta)(\xi)\,d\xi} \le
C_N t^{-N}\,.
\end{equation*}
Let us return to the case when there is a critical point.
We may assume that $\Im\tau(\xi^0)=0$ since otherwise
$\Im\tau(\xi^0)>0$ in view of \eqref{EQ:imtau>=0}, and
then Theorem \ref{THM:expdecay} would actually give the
exponential decay rate.
We now claim that
\begin{align}
\absBig{\int_{\R^n}e^{i(\xtil\cdot\xi+\tau(\xi))t}
a(\xi)\vartheta(\xi)\,d\xi}&\le
Ct^{-n/2}\abs{\det\Hess(\xi^0)}^{-1/2}\abs{a(\xi^0)
\cutoffSP(\xi^0)}\notag\\
\le &Ct^{-n/2}\abs{\det\Hess(\xi^0)}^{-1/2}\brac{\xi^0}^{-\mu}\,.
\label{EQ:methodofstatphaseapplied}
\end{align}
This is a consequence of the following theorem, 
see e.g.~\cite[Theorem 7.7.12, p.~228]{horm83I}:
\begin{thm}\label{THM:hormanderstatphase}
Suppose $\Phi=\Phi(x,y)\in C^\infty(\R^n\times\R^p)$ 
is a complex-valued smooth
function in a neighbourhood of the origin $(0,0)\in\R^n\times\R^p$
such that\textup{:}
\begin{itemize}
\item $\Im \Phi\ge 0$\textup{;} \item $\Im \Phi(0,0)=0$\textup{;}
\item $\Phi'_x(0,0)=0$\textup{;}
\item $\det\Phi^{\prime\prime}_{xx}(0,0)\ne0$.
\end{itemize}
Also\textup{,} suppose $u\in C_0^\infty(K)$ where $K$ is a small
neighbourhood of $(0,0)$. Then
\begin{multline*}
\absBig{\int_{\R^n}e^{i\om \Phi(x,y)}u(x,y)\,dx- \\\big((\det(\om
\Phi^{\prime\prime}_{xx}/2\pi i))^0\big)^{-1/2}e^{i\om \Phi^0}
\sum_{j=0}^{N-1}(L_{\Phi,j}u)^0\om^{-j}}\le C_N\om^{-N-n/2}\,,
\end{multline*}
for some choice of operators $L_{\Phi,j}$,
where the notation $G^0(y)$ \textup{(}where $G(x,y)$ is the
function\textup{)} means the function of $y$ 
only which is in the
same residue class modulo the ideal generated by 
$\pa \Phi/\pa x_j$, $j=1,\dots,n$.
\end{thm}
The proof of this result uses the method of stationary phase;
similar results (with slightly differing conditions and conclusions)
can be found in~\cite[(1.1.20), p.~49]{sogg93}, \cite[Ch.~VIII, 2.3,
Proposition 6, p.~344]{stei93}, \cite[Proposition 1.2.4,
p.~14]{duis96} and \cite[p.~432, Ch.~VIII, (2.15)--(2.16)]{trev802},
for example.

So, we have~\eqref{EQ:methodofstatphaseapplied} as a simple
consequence of this theorem; now, in order to show that
\begin{equation}\label{EQ:stphase}
\absBig{\int_{\R^n}e^{i(\xtil\cdot\xi+\tau(\xi))t}
a(\xi)\vartheta(\xi)\,d\xi}\le Ct^{-n/2}\,,
\end{equation}
we must choose $\mu\in\R$ suitably. 
In the sequel we may assume that $M$ is even; if $M$ is odd,
the result follows by a standard interpolation argument
taking the geometric mean.

Assume that
$\abs{\det\Hess\tau(\xi)}\ge C\brac{\xi}^{-M}$ for some $M\in\R$;
then taking $\mu=M/2$, we have this estimate. 
This extends the
case of Klein--Gordon equation (which is done in~\cite{horm97}
pp.146--155) where $\det\Hess\tau(\xi)=\brac{\xi}^{-n-2}$, so
$M=n+2$.

Let us now apply this result to our situation. We have
\begin{equation*}
\int_{\R^n}e^{i(x\cdot\xi+\tau(\xi)t)} a_j(\xi)
\vartheta(\xi)\widehat{f}(\xi)\,d\xi\,,
\end{equation*}
where we may now think of $\vartheta$ as
$\vartheta\in S^0_{1,0}$ to ensure uniformity, and 
$a_j(\xi)=O(\abs{\xi}^{-j})$ as $|\xi|\to\infty$; we assume
$\abs{\det\Hess\tau(\xi)}\ge C\brac{\xi}^{-M}$. Now, for each
$\nu\in\N$, we have
\begin{align*}
a_j(\xi)&=(1+|\xi|^2)^{-\nu}(1+|\xi|^2)^{\nu}a_j(\xi)\\=&
\sum_{\abs{\al}\le\nu}c_\alpha
(1+|\xi|^2)^{-\nu}\xi^\al a_j(\xi)\xi^\al=
\sum_{\abs{\al}\le\nu}a_{j,\al}(\xi)\xi^\al\,,
\end{align*}
where $a_{j,\al}(\xi)=c_\alpha
(1+|\xi|^2)^{-\nu}\xi^\al a_j(\xi)$ is of order
${-j-2\nu+\abs{\al}}$. 
Moreover, $a_{j,\al}\vartheta$ is of order
${-j-2\nu+\abs{\al}}$ uniformly over $\vartheta$
(satisfying the necessary uniform symbolic estimates).
Taking $\nu=M/2-j$ and using that
$|\alpha|\leq\nu$, we can 
ensure that the worst order of any of these symbols is~$-M/2$. Then,
\begin{align*}
\int_{\R^n}e^{i(x\cdot\xi+\tau(\xi)t)} a_j(\xi)
\vartheta(\xi)\widehat{f}(\xi)\,d\xi&=
\sum_{\abs{\al}\le\nu}\int
e^{i(\xtil\cdot\xi+\tau(\xi))t}a_{j,\al}(\xi) \vartheta(\xi)
\widehat{D^\al f}(\xi)\,d\xi\\&=\sum_{\abs{\al}\le\nu}
\left(\int
e^{i(\xtil\cdot\xi+\tau(\xi))t}a_{j,\al}(\xi)\vartheta(\xi)\,d\xi\conv
D^\al f\right)(x)\,.
\end{align*}
Then
\begin{multline*}
\normBig{\sum_{\abs{\al}\le\nu}\int
e^{i(\xtil\cdot\xi+\tau(\xi))t}a_{j,\al}(\xi)\vartheta(\xi)\,d\xi\conv
D^\al f(x)}_{L^\infty} \\ \le \sum_{\abs{\al}\le \nu}\normBig{\int
e^{i(\xtil\cdot\xi+\tau(\xi))t}
a_{j,\al}(\xi)\vartheta(\xi)\,d\xi}_{L^\infty}\norm{D^\al f}_{L^1} \le
Ct^{-n/2}\norm{f}_{W^{M/2-j}_1},
\end{multline*}
where we used estimate \eqref{EQ:stphase}.
Thus, we have an $L^1-L^\infty$ estimate in this case. To find an
$L^2-L^2$ estimate is simpler: by the Plancherel's theorem, we have
\begin{align*}
\normBig{\int_{\R^n}e^{i(x\cdot\xi+\tau(\xi)t)}
a_j(\xi)\vartheta(\xi)
\widehat{f}(\xi)\,d\xi}_{L^2(\R^n_x)}=C\normbig{e^{i\tau(\xi)t}
a_j(\xi)\vartheta(\xi)\widehat{f}(\xi)}_{L^2(\R^n_\xi)}\\
\le C\normbig{\bra{\xi}^{-j}\widehat{f}(\xi)}_{L^2}\le
C\norm{f}_{W_2^{-j}}\,.
\end{align*}
Using the interpolation Theorem~\ref{thm:maininterpolationthm} and
noting that all integrals are bounded for small $t$,
we obtain Theorem~\ref{THM:nondeghess}.

\paragraph{Behaviour of Critical Points:} Above, we assumed that
$\xi^0$ was the only critical point of the phase function; this is
not such an unreasonable assumption as the following 
observation shows:
\begin{lem}\label{lemma:critical-points}
If the matrix of second order derivatives
$\Hess\tau(\xi)$ is positive definite for all~$\xi$\textup{,}
then the integral
\begin{equation*}
\int_{\R^n} e^{i(\xtil\cdot\xi+\tau(\xi))t}a(\xi)\,d\xi
\end{equation*}
has only one critical point.
\end{lem}
\begin{proof}
Suppose $\xi^1,\xi^2\in\R^n$ are two such critical points. So
$\xtil+\grad_\xi\tau(\xi^1)=\xtil+\grad_\xi\tau(\xi^2)$, or
$\partial_{\xi_j}\tau(\xi^1)=\partial_{\xi_j}\tau(\xi^2)$ for each
$j=1,\dots,n$. Thus, by the fundamental theorem of calculus, 
for all $j=1,\dots,n$, we have
\begin{equation*}
0=\pa_{\xi_j}\tau(\xi^1)-\pa_{\xi_j}\tau(\xi^2) =\int_0^1
(\xi^1-\xi^2)\cdot\grad_\xi(\pa_{\xi_j})
\tau(\xi^1+s(\xi^2-\xi^1)\,ds\,.
\end{equation*}
But this means that
$(\xi^1-\xi^2)\Hess\tau(\xi^1+s(\xi^2-\xi^1))(\xi^1-\xi^2)=0$ for
all $s$ since the Hessian is positive definite; and since it is
never zero, we have that $\xi^1-\xi^2=0$, 
which shows that there is
at most one critical point.
\end{proof}
An example of such an operator is the Klein--Gordon equation.
\begin{rem}
In general, 
another consequence of $\Hess\tau(\xi)$ being positive definite is
that the level sets $S_\lambda=\set{\xi\in\R^n:\tau(\xi)=\la}$, 
$\la\in\R$ are
all strictly convex; indeed, if we 
take a smooth curve $\xi(s)\in S_\lambda$,
$s\geq 0$, where $\xi(0)=\xi^0$ and, by assumption,
$\dot\xi(s)\ne0$, then $\grad\tau(\xi(s))\cdot\dot\xi(s)=0$
(differentiate $\tau(\xi(s))=\la$), and (differentiating again)
\begin{equation*}
\dot\xi(s)^T\cdot\Hess\tau(\xi(s))\cdot\dot\xi(s)+
\grad\tau(\xi(s))\cdot\ddot\xi(s)=0.
\end{equation*}
Then, since $\Hess\tau(\xi)$ is positive definite, 
the first term in
this sum is positive, hence the second is negative---which means
that the angle between $\grad\tau(\xi(s))$, that is, the normal to
the level set, and~$\ddot\xi(s)$ is strictly greater than~$\pi/2$,
so the level set is strictly convex. In particular, this shows that
imposing the condition $\Hess\tau(\xi)$ positive definite is
stronger than imposing the convexity condition of
Definition~\ref{DEF:convexitycondition}, and making it clear why we
get a faster rate of decay in this case (see the next section for
that case).
\end{rem}

\begin{rem}
If $\rank\Hess\tau(\xi)=n-1$, 
then a similar argument can be used to
prove the corresponding part of Theorem~\ref{THM:overallmainthm},
i.e.\ that there is decay of order $-\frac{n-1}{2}$. This is a
consequence of an extension to
Theorem~\ref{THM:hormanderstatphase}---see
H\"ormander \cite[Section 7.7]{horm83I}.
\end{rem}

\subsection{Phase satisfies the
convexity condition: Theorem \ref{THM:convexsp}}
\label{SEC:convexity}

The case of real roots and real-valued phase functions
subdivides into the following subcases, each of which
yields a different decay rate:
\begin{enumerate}[leftmargin=*,label=(\roman*)]
\item $\det\Hess\tau(\xi)\ne0$; in this case we use the method of
stationary phase in the same way as in Section
\ref{SEC:mainstatphasesection}, with same result;
\item $\det\Hess\tau(\xi)=0$ and $\tau(\xi)$ satisfies the convexity
condition of Definition~\ref{DEF:convexitycondition}; in this case
we use Theorem~\ref{THM:sugimoto/randolargument};
\item the general case when $\det\Hess\tau(\xi)=0$ (i.\ e.\
$\tau(\xi)$ does not satisfy the convexity condition); in this
case, we use Theorem~\ref{THM:noncovexargument}.
\end{enumerate}
We assume throughout that $\tau(\xi)\ge0$ for all $\xi\in\R^n$ or
$\tau(\xi)\le0$ for all $\xi\in\R^n$. This is valid because for
the characteristic roots lying on the real axis, there exists
a linear function $\widetilde\tau(\xi)$ such that
$\widetilde\tau_k(\xi):=\tau_k(\xi)-\widetilde\tau(\xi)$ is either
everywhere non-negative or everywhere non-positive, and, if
$\tau_k(\xi)$ satisfies the convexity condition, so does
$\widetilde\tau_k(\xi)$. 
A proof for this in the case of homogeneous symbols is
given in \cite{sugi94} and we recall this result here for
completeness:
\begin{prop}
Let $\va_k(\xi)$, $k=1,\ldots,m$, be the characteristic 
roots of
a strictly hyperbolic operator with homogeneous
symbol of order $m$, ordered as
$\va_1(\xi)>\va_2(\xi)>\cdots>\va_m(\xi)$ for $\xi\not=0$.
Suppose that all the Hessians $\va_k^{\prime\prime}(\xi)$ are
semi-definite for $\xi\not=0$. Then there exists a polynomial
$\alpha(\xi)$ of order one such that
$\va_{m/2}(\xi)>\alpha(\xi)>\va_{m/2+1}$ 
{\rm (}if $m$ is even{\rm )} or
$\alpha(\xi)=\va_{(m+1)/2}(\xi)$ {\rm (}if $m$ is odd{\rm )}. 
Moreover,
the hypersurfaces $\Sigma_k=\{\xi\in\Rn;\, 
\widetilde{\va}_k=\pm 1\}$
with $\widetilde{\va}_k(\xi)=\va_k(\xi)-\alpha(\xi)$ 
$(k\not=(m+1)/2)$ are convex and $\gamma(\Sigma_k)\leq 2[m/2].$
\label{PROP:limitingphases}
\end{prop}

The generalisation of this proposition to the case of
non-homogeneous
symbols follows using the perturbation results in
Section~\ref{CHAP3}.
%\marginpar{\bf Prove}

Assume that~$\tau(\xi)$ satisfies the convexity condition of
Definition~\ref{DEF:convexitycondition}. Set
$\ga\equiv\ga(\tau):=\sup_{\la>0}\ga(\Si_\la(\tau))$, where, as
before,
\begin{equation*}
\Si_\la(\tau)=\set{\xi\in\R^n:\tau(\xi)=\la}\,.
\end{equation*}
and
\begin{equation*}
\ga(\Si_\la(\tau)):=
\sup_{\si\in\Si_\la(\tau)}\sup_P\ga(\Si_\la(\tau);\si,P)
\end{equation*}
where the second supremum is over planes~$P$ containing the normal
to~$\Si_\la(\tau)$ at~$\si$ and $\ga(\Si_\la(\tau);\si,P)$ denotes
the order of the contact between the line $T_\si\cap P$---$T_\si$
is the tangent plane at~$\si$---and the curve $\Si_\la(\tau)\cap
P$.

We have the following results which ensures that this is finite:
\begin{lem}\label{LEM:convexityconstantconv}
Suppose~$\tau:\R^n\to\R$ is a characteristic root of a linear
$m^{\text{th}}$ order constant coefficient strictly hyperbolic
partial differential operator. Then, there exists a
homogeneous function of order~$1$\textup{,}~$\va(\xi)$\textup{,} a
characteristic root of the principal symbol\textup{,} such that
\begin{equation*}
\ga(\Si_\la(\tau))\to\ga(\Si_1(\va))\text{ as }\la\to\infty\,.
\end{equation*}
If we assume that
$\gamma(\Sigma_\lambda(\tau))<\infty$ for all 
$\lambda>0$, then we have $\ga(\tau)<\infty$.
\end{lem}
\begin{proof}
This is true because:
\begin{enumerate}[leftmargin=*,label=(\alph*)]
\item  by Proposition~\ref{PROP:perturbationresults},
\ref{LEM:tau-vabounds}, $\Si_\la(\tau)$ is near to $\Si_\la(\va)$
for large~$\la$ in a suitable metric;
\item by the homogeneity of~$\va$, if $\abs{\la-\la'}$ is
sufficiently small, then $\Si_\la(\va)$ is near to $\Si_{\la'}(\va)$
for large~$\la$ in the same metric;
\item Proposition~\ref{PROP:perturbationresults},
\ref{PROP:derivoftau-derivofphi}, ensures that $T_\si(\tau)$ is near
to $T_\si(\va)$ (because derivatives of~$\tau$ tend to those
of~$\va$) for large~$\la$;
\item so, with $\Si_\la(\tau)$ and $T_\si(\tau)$ near to (in a
suitable sense) the corresponding data of~$\va$ 
for large~$\la$, it
is clear that the $\ga(\Si_\la(\tau);\si,P)$ is near to
$\ga(\Si_\la(\va);\si,P)$, 
and hence $\ga(\Si_\la(\tau))$ is near to
$\ga(\Si_\la(\va))$;
\item finally, $\ga(\Si_1(\va))=\ga(\Si_\la(\va))$ by homogeneity.
\end{enumerate}
\end{proof}

In order to prove Theorem \ref{THM:convexsp}, we shall show
that if $a_j\in S^{-j}_{1,0}$ 
is a symbol of order $-j$, then we have the estimate
\begin{equation}\label{EQ:resultforconvonaxis}
\normBig{\int_{\R^n}e^{i(x\cdot\xi+\tau(\xi)t)}
a_j(\xi)\widehat{f}(\xi)\,d\xi}_{L^q}\le
C(1+t)^{-\frac{n-1}{\ga}\big(\frac{1}{p}-\frac{1}{q}\big)}
\norm{f}_{W_p^{N_{p,j}}}\,,
\end{equation}
for all $t\geq 0$, where $\frac1p+\frac1q=1$, $1<p\le 2$, 
and $f\in C_0^\infty(\R^n)$.
The Sobolev order $N_{p,j}$ 
(which does not have to be an integer here)
is worse for small times,
being
$N_{p,j}\ge n(\frac1p-\frac1q)-j$.
It can be actually improved for large times, which
will be done in estimate \eqref{EQ:convestL1Linftylarget}.

\paragraph{Besov Space Reduction:}
We begin by following Brenner~\cite{bren75} and also
Sugimoto \cite{sugi94} in using the theory of Besov spaces 
and Paley decomposition to
reduce this to showing, for all $t\geq 0$, the estimate
\begin{equation}\label{EQ:partitioned}
\normBig{\FT^{-1}(e^{i\tau(\xi)t}a_j(\xi)\cutoffbes_l(\xi)
\widehat{f}(\xi))}_{L^q}\le
C(1+t)^{-\frac{n-1}{\ga}\big(\frac{1}{p}-\frac{1}{q}\big)}
\norm{f}_{W^{N_{p,j}}_p}\,;
\end{equation}
here $\set{\cutoffbes_l(\xi)}_{l=0}^\infty$ is a Hardy--Littlewood
partition: let $\cutoffbes\in C_0^\infty(\R^n)$ be such that
\begin{gather*}
\supp\cutoffbes=\set{\xi\in\R^n:\frac{1}{2}\le\abs{\xi}\le2}\,,\quad
\cutoffbes(\xi)>0\text{ for }\frac{1}{2}<\abs{\xi}<2\,,\\
\text{and }\sum_{k=-\infty}^\infty\cutoffbes(2^{-k}\xi)=1\text{
for }\xi\ne0\,,
\end{gather*}
and set
\begin{equation*}
\cutoffbes_0(\xi)=1-\sum_{l=1}^\infty\cutoffbes(2^{-l}\xi)\,,\quad
\cutoffbes_l(\xi):=\cutoffbes(2^{-l}\xi)\,,\,l\in\N\,.
\end{equation*}
Now, recall the definition of a Besov space, as given in, for
example, Bergh and L\"{o}fstr\"{o}m~\cite{berg+lofs76}:
\begin{defn}
For suitable $p,q,s\in\R$ define the Besov norm by
\begin{equation*}
\norm{f}_{B^s_{p,q}}:=
\norm{\FT^{-1}(\cutoffbes_0(\xi)\widehat{f}(\xi))}_{L^p}
+\Big(\sum_{l=1}^\infty(2^{sl} \norm{\FT^{-1}
(\cutoffbes_l(\xi)\widehat{f}(\xi))}_{L^p})^p\Big)^{1/q}\,;
\end{equation*}
the \emph{Besov space $B_{p,q}^s$} is the space of functions in
$\S^\prime(\R^n)$ for which this norm is finite.
\end{defn}
This result is the main one we shall need:
\begin{thm}[\cite{berg+lofs76}, Theorem 6.4.4]\label{THM:berg+l6.4.4}
The following inclusions hold\textup{:}
\begin{gather*}
B_{p,p}^s\subset W^s_p\subset B_{p,2}^s\text{ and }B_{q,2}^s\subset
W^s_q\subset B_{q,q}^s
\end{gather*}
for all $s\in\R$\textup{,} $1<p\le 2$\textup{,} $2\le q<\infty$.
\end{thm}
There are some weaker versions of these embeddings for $p=1$.
%We can note that for $p=1$ the embedding are true with a small
%loss of regularity.
Using this theorem, we have
\begin{align*}
\normBig{\int_{\R^n}e^{i(x\cdot\xi+\tau(\xi)t)}
&a_j(\xi)\widehat{f}(\xi)\,d\xi}_{L^q(\R^n)}=(2\pi)^n\normbig{
\FT^{-1}(e^{i\tau(\xi)t}a_j(\xi)\widehat{f}(\xi))(t,x)}_{L^q}\\
\le C&\normbig{\FT^{-1}(e^{i\tau(\xi)t}a_j(\xi)
\widehat{f}(\xi))}_{B^0_{q,2}}\\
= C&\Big(\sum_{l=0}^\infty\normbig{\FT^{-1}(
e^{i\tau(\xi)t}a_j(\xi)\cutoffbes_l(\xi)
\widehat{f}(\xi))}_{L^q}^2\Big)^{1/2}\\
= C&\Big(\sum_{l=0}^\infty\normBig{\FT^{-1}
(e^{i\tau(\xi)t}a_j(\xi)\cutoffbes_l(\xi) \sum_{r=l-1}^{l+1}
\cutoffbes_r(\xi)\widehat f(\xi))}_{L^q}^2\Big)^{1/2}\,;
\end{align*}
in the final line we have used that
$\sum_{r=l-1}^{l+1}\cutoffbes_r(\xi)=1$ on
$\supp\cutoffbes_l(\xi)$ by the structure of the partition of
unity. Now, assuming that~\eqref{EQ:partitioned} holds, this can
be further estimated:
\begin{align*}
\Big(\sum_{l=0}^\infty\normBig{\FT^{-1}
(&e^{i\tau(\xi)t}a_j(\xi)\cutoffbes_l(\xi) \sum_{r=l-1}^{l+1}
\cutoffbes_r(\xi)\widehat f(\xi))}_{L^q}^2\Big)^{1/2}\\
&\le C{t}^{-\frac{n-1}{\ga}
\big(\frac{1}{p}-\frac{1}{q}\big)} \Big(\sum_{l=0}^\infty
\Big(\sum_{r=l-1}^{l+1} \norm{\FT^{-1}(\cutoffbes_r(\xi)\widehat
f(\xi))}_{W^{N_{p,j}}_p}\Big)^2\Big)^{1/2}\\
& \le C{t}^{-\frac{n-1}{\ga}
\big(\frac{1}{p}-\frac{1}{q}\big)} \Big(\sum_{l=0}^\infty
\sum_{r=l-1}^{l+1} \norm{\FT^{-1}(\cutoffbes_r(\xi)\widehat
f(\xi))}_{W^{N_{p,j}}_p}^2\Big)^{1/2}\\
&\le C{t}^{-\frac{n-1}{\ga}
\big(\frac{1}{p}-\frac{1}{q}\big)}
\Big(\sum_{l=0}^\infty\norm{\FT^{-1}(\cutoffbes_l(\xi) \widehat
f(\xi))}_{W^{N_{p,j}}_p}^2\Big)^{1/2}\,.
\end{align*}
Finally, using Theorem~\ref{THM:berg+l6.4.4} once again, we get
\begin{align*}
\Big(\sum_{l=0}^\infty\norm{\FT^{-1}(\cutoffbes_l(\xi) \widehat
f(\xi))}_{W^{N_{p,j}}_p}^2\Big)^{\frac{1}{2}} &\le
C\Big(\sum_{l=0}^\infty\sum_{\abs{\al}\le N_{p,j}}
\norm{D_x^\al[\FT^{-1}(\cutoffbes_l(\xi) \widehat
f(\xi))]}_{L^p}^2\Big)^{\frac{1}{2}}\\
=C&\sum_{\abs{\al}\le N_{p,j}}
\Big(\sum_{l=0}^\infty\norm{\FT^{-1}(\cutoffbes_l(\xi)
\widehat{D^\al f}(\xi))]}_{L^p}^2\Big)^{1/2}\\
=C&\sum_{\abs{\al}\le N_{p,j}}\norm{D^\al f}_{B^0_{p,2}}\le
C\norm{f}_{W^{N_{p,j}}_p}\,.
\end{align*}
Combining these estimates shows that \eqref{EQ:partitioned}
implies \eqref{EQ:resultforconvonaxis}
as desired. So, it suffices to prove~\eqref{EQ:partitioned};
moreover, as shown above, this requires us to show two estimates and
then interpolate---Theorem~\ref{thm:maininterpolationthm} yields:
\begin{gather}
\normbig{\FT^{-1}(e^{i\tau(\xi)t}a_j(\xi)\cutoffbes_l(\xi)
\widehat{f}(\xi))(t,x)}_{L^\infty}\le
C{(1+t)}^{-\frac{n-1}{\ga}}\norm{f}_{W^{N_1-j}_1}\,,
\label{EQ:convL1Linftyest}\\
\normbig{\FT^{-1}(e^{i\tau(\xi)t}a_j(\xi)\cutoffbes_l(\xi)
\widehat{f}(\xi))(t,x)}_{L^2}\le
C\norm{f}_{W_2^{-j}}\,,\label{EQ:convL2est}
\end{gather}
where $N_1>n$.

\paragraph{$L^2-L^2$ estimate:}

Since $\tau(\xi)$ is real-valued and
$a_j(\xi)=O(\abs{\xi}^{-j})$ as $\abs{\xi}\to\infty$,
by Plancherel's theorem we get
\begin{align*}
\normbig{\FT^{-1}(e^{i\tau(\xi)t}a_j(\xi)\cutoffbes_l(\xi)
\widehat{f}(\xi))}_{L^2} & =
\int_{\R^n}\abs{e^{i\tau(\xi)t}a_j(\xi)\cutoffbes_l(\xi)
\widehat{f}(\xi)}^2\,d\xi \\
& \le C\int_{\abs{\xi}\ge
M}\abs{\xi}^{-2j}\abs{\widehat{f}(\xi)}^2\,d\xi \le
C\norm{f}_{W_2^{-j}}\,.
\end{align*}
Note that~$C$ is independent of~$l$ because
$a_j(\xi)\abs{\xi}^{j}$ is uniformly bounded in~$\R^n$. This
proves the required estimate~\eqref{EQ:convL2est}.

\paragraph{$L^1-L^\infty$ estimate:}
First, suppose $0\leq t<1$; then
\begin{align}
\normBig{\int_{\R^n}e^{i(x\cdot\xi+\tau(\xi)t)}
a_j(\xi)\cutoffbes_l(\xi)\widehat{f}(\xi)\,d\xi}_{L^\infty}
&\le C\int_{\abs{\xi}\ge M}
\abs{\xi}^{-j}\abs{\widehat{f}(\xi)}\,d\xi\notag\\
\le C&\int_{\abs{\xi}\ge M}\abs{\xi}^{-N_1}\,d\xi\;
\normbig{\bra{\xi}^{N_1-j}\widehat{f}(\xi)}_{L^\infty}\notag\\
\le C&\norm{f}_{W_1^{N_1-j}}\,,\label{EQ:convestL1Linftysmallt}
\end{align}
where $N_1>n$.

For $t\ge1$, we show
\begin{equation}\label{EQ:convestL1Linftylarget}
\normBig{\int_{\R^n}e^{i(x\cdot\xi+\tau(\xi)t)}
a_j(\xi)\cutoffbes_l(\xi)\widehat{f}(\xi)\,d\xi}_{L^\infty}\le
Ct^{-\frac{n-1}{\ga}} \norm{f}_{W^{n-\frac{n-1}{\ga}-j}_1}\,.
\end{equation}
Together \eqref{EQ:convestL1Linftysmallt} and
\eqref{EQ:convestL1Linftylarget} will imply
\eqref{EQ:convL1Linftyest}.
We claim now that it suffices to prove that there exists a constant~$C>0$
which is independent of~$l$ such that, for all $t\ge1$,
\begin{equation}\label{EQ:L1Linftyestforrealconvphase}
\normlr{\int_{\R^n}e^{i(x\cdot\xi+\tau(\xi)t)} a_j(\xi)
\bra{\xi}^{\frac{n-1}{\ga}-n+j}
\cutoffbes_l(\xi)\,d\xi}_{L^\infty}\le Ct^{-\frac{n-1}{\ga}}\,.
\end{equation}
Indeed,
\begin{align*}
\int_{\R^n}e^{i(x\cdot\xi+\tau(\xi)t)}a_j(\xi)\cutoffbes_l(\xi)
\widehat{f}(\xi)\,&d\xi
=(2\pi)^n\FT^{-1}(e^{i\tau(\xi)t}a_j(\xi)\cutoffbes_l(\xi)
\widehat{f}(\xi))\\
&=(2\pi)^n\FT_{\xi\to x}^{-1}
[e^{i\tau(\xi)t}a_j(\xi)\cutoffbes_l(\xi)]\conv f(x)\\
&=\Big(\int_{\R^n}e^{i(x\cdot\xi+\tau(\xi)t)}
a_j(\xi)\cutoffbes_l(\xi)\,d\xi\Big)\conv f(x)\,,
\end{align*}
and, by the definition of the symbol of $\bra{D_x}$, we have
\begin{align*}
\Big(\int_{\R^n}&e^{i(x\cdot\xi+\tau(\xi)t)}
a_j(\xi)\cutoffbes_l(\xi)\,d\xi\Big)\conv f(x)\\
&= \Big(\int_{\R^n}\bra{D_x}^{n-\frac{n-1}{\ga}-j}
e^{i(x\cdot\xi+\tau(\xi)t)} a_j(\xi)\cutoffbes_l(\xi)
\bra{\xi}^{\frac{n-1}{\ga}-n+j}\,d\xi\Big)\conv f(x) \\
&=\bra{D_x}^{n-\frac{n-1}{\ga}-j}\Big(\int_{\R^n}
e^{i(x\cdot\xi+\tau(\xi)t)} a_j(\xi)\cutoffbes_l(\xi)
\bra{\xi}^{\frac{n-1}{\ga}-n+j}\,d\xi\Big)\conv f(x)\\
&=\Big(\int_{\R^n} e^{i(x\cdot\xi+\tau(\xi)t)}
a_j(\xi)\cutoffbes_l(\xi)
\bra{\xi}^{\frac{n-1}{\ga}-n+j}\,d\xi\Big)\conv
\bra{D_x}^{n-\frac{n-1}{\ga}-j}f(x)\,;
\end{align*}
also,
\begin{equation*}
\norm{g*h}_{L^\infty} \le\norm{g}_{L^\infty}\norm{h}_{L^1}\,,
\end{equation*}
for all $g\in L^\infty(\R^n)$, $h\in L^1(\R^n)$. Combining all
these shows that \eqref{EQ:L1Linftyestforrealconvphase}
implies \eqref{EQ:convestL1Linftylarget}.

In order to prove~\eqref{EQ:L1Linftyestforrealconvphase}, 
we can use
Theorem~\ref{THM:sugimoto/randolargument} as 
$\tau:\R^n\to\R$ is
assumed to satisfy the convexity condition; let us check that each
hypothesis holds. 
In addition to properties ensured by Proposition \ref{prop:roots-r},
we have:
\begin{itemize}[leftmargin=*]
\item Property (i) suffices for the hypothesis (i) 
of Theorem~\ref{THM:sugimoto/randolargument} to
hold since~$a_j(\xi)$ is supported away from the origin.
\item $a_j(\xi)\bra{\xi}^{\frac{n-1}{\ga}-n+j}$ 
is a symbol of order
$\frac{n-1}{\ga}-n$ since $a\in S^{-j}$ 
and because it is zero in a neighbourhood of
the origin.
\item the partition of unity $\set{\cutoffbes_l(\xi)}_{l=1}^\infty$
is in the form of $g_R(\xi)$ as required by
Theorem~\ref{THM:sugimoto/randolargument}.
\end{itemize}
Also,
$\ga<\infty$ by Lemma~\ref{LEM:convexityconstantconv} above.
Therefore, for $t\ge1$, we get
\begin{equation*}
\absBig{\int_{\R^n}e^{i(x\cdot\xi+\tau(\xi)t)} a_j(\xi)
\abs{\xi}^{\frac{n-1}{\ga}-n+j} \cutoffbes_l(\xi)\,d\xi}\le
Ct^{-\frac{n-1}{\ga}}\,.
\end{equation*}
Hence, we have~\eqref{EQ:convestL1Linftylarget}, which, together
with~\eqref{EQ:convestL1Linftysmallt},
proves~\eqref{EQ:convL1Linftyest}; this completes the proof of
Theorem~\ref{THM:convexsp} on real axis with convexity
condition~$\ga$.

\subsection{Results without convexity:
Theorem \ref{THM:noncovexargument-sp}}
\label{SEC:nonconvex}
The general case depends upon Theorem~\ref{THM:noncovexargument},
just as the case where the convexity condition holds depends upon
Theorem~\ref{THM:sugimoto/randolargument}.
Here we assume that $\tau$ is real valued.
We introduce
$\ga_0\equiv\ga_0(\tau):=\sup_{\la>0}\ga_0(\Si_\la(\tau))$, where,
\begin{equation*}
\ga_0(\Si_\la(\tau)):=
\sup_{\si\in\Si_\la(\tau)}\inf_P\ga(\Si_\la(\tau);\si,P)
\end{equation*}
(all notation as before). For this quantity we have the analogous
result to Lemma~\ref{LEM:convexityconstantconv}, which can be
proved in the same way:
\begin{lem}\label{LEM:ga0conv}
If~$\tau:\R^n\to\R$ is a characteristic root of a linear
$m^{\text{th}}$ order constant coefficient strictly hyperbolic
partial differential operator, then, there exists a homogeneous
function of order~$1$\textup{,}~$\va(\xi)$\textup{,} a
characteristic root of the principal symbol\textup{,} such that
\begin{equation*}
\ga_0(\Si_\la(\tau))\to\ga_0(\Si_1(\va))\text{ as }\la\to\infty\,.
\end{equation*}
If we assume that
$\gamma_0(\Sigma_\lambda(\tau))<\infty$ for all 
$\lambda>0$, then we have $\ga_0(\tau)<\infty$.
\end{lem}
We shall show
\begin{equation*}
\normBig{\int_{\R^n}e^{i(x\cdot\xi+\tau(\xi)t)}
a_j(\xi)\widehat{f}(\xi)\,d\xi}_{L^q}\le
C{(1+t)}^{-\frac{1}{\ga_0}\big(\frac{1}{p}-\frac{1}{q}\big)}
\norm{f}_{W_p^{N_{p,j}}}\,,
\end{equation*}
for all $t\geq 0$, where $\frac1p+\frac1q=1$, $1\le p\le 2$,
$f\in C_0^\infty(\R^n)$,
$N_{p,j}\ge n(\frac1p-\frac1q)-j$ and $N_{1,j}>n-j$.
Similarly to \eqref{EQ:convestL1Linftylarget}, the Sobolev
order $N_{p,j}$ can be improved for large times.

As in the case of Section
\ref{SEC:convexity}, this can be reduced, via a Besov space
reduction the interpolation
Theorem~\ref{thm:maininterpolationthm}, to showing
\begin{gather*}
\normbig{\FT^{-1}(e^{i\tau(\xi)t}a_j(\xi)\cutoffbes_l(\xi)
\widehat{f}(\xi))(t,x)}_{L^\infty}\le
C(1+t)^{-\frac{1}{\ga_0}}\norm{f}_{W^{N_1-j}_1}\,,\\
\normbig{\FT^{-1}(e^{i\tau(\xi)t}a_j(\xi)\cutoffbes_l(\xi)
\widehat{f}(\xi))(t,x)}_{L^2}\le C\norm{f}_{W_2^{-j}}\,,
\end{gather*}
where the partition of unity
$\set{\cutoffbes_l(\xi)}_{l=1}^\infty$ is as above and $N_1>n$.

The~$L^2$ estimate follows by the Plancherel's theorem in the same
way as before.

For the $L^1-L^\infty$ estimate, the case $0\le t<1$ is as
in~\eqref{EQ:convestL1Linftysmallt}; for $t\ge1$ it suffices to
show (see the earlier argument),
\begin{equation*}
\normlr{\int_{\R^n}e^{i(x\cdot\xi+\tau(\xi)t)} a_j(\xi)
\bra{\xi}^{\frac{1}{\ga_0}-n+j}
\cutoffbes_l(\xi)\,d\xi}_{L^\infty}\le Ct^{-1/\ga_0}\,.
\end{equation*}
This follows by Theorem~\ref{THM:noncovexargument}: the hypotheses
of this hold by the same arguments as above
(see Proposition \ref{prop:roots-r})---the convexity
condition is not required for the perturbation methods
employed---and Lemma~\ref{LEM:ga0conv}.
This completes the proof of 
\ref{THM:noncovexargument-sp}.

\subsection{Asymptotic properties of complex phase functions}
\label{SEC:asymptotic}
Here we consider what happens when the phase function $\tau(\xi)$
is complex valued and look at its behaviour for large frequencies.
In particular, this is related to the case
\begin{equation*}
\Im\tau(\xi)\to0\text{ as }\abs{\xi}\to\infty\,.
\end{equation*}
Unlike in the case of the phase function~$\tau(\xi)$ 
lying on the
real axis, here we do not consider a case where the phase function
satisfies a ``convexity condition''. The reason for this is
twofold: firstly, there is no straightforward analog of the
convexity condition for real-valued phase functions as the
presence of the non-zero imaginary part causes problems; secondly,
there are no common examples of this situation, 
and hence it does
not seem worthwhile developing a complicated theory for this
situation.

If $\det\Hess\tau(\xi)\ne0$, the analysis
can be done in exactly the same way as that in
Section~\ref{SEC:mainstatphasesection}, since
Theorem~\ref{THM:hormanderstatphase} holds for integrals with
complex phase functions.

In general, we can derive certain properties of real and
imaginary parts of $\tau(\xi)$ using perturbation arguments
of Section \ref{CHAP3}.
For example, for the index
$\ga_0=\ga_0(\Re\tau)=\sup_{\la>0}\ga_0(\Si_\la(\Re\tau))$ we
can note the following:

\begin{lem}\label{LEM:cplxga0conv}
If~$\tau:\R^n\to\C$ is a characteristic root of a linear
$m^{\text{th}}$ order constant coefficient strictly hyperbolic
partial differential operator such that $\Im\tau(\xi)\to0$ as
$\abs{\xi}\to\infty$, then, 
there exists a homogeneous function of
order~$1$\textup{,}~$\va(\xi)$\textup{,} 
a characteristic root of
the principal symbol\textup{,} such that
\begin{equation*}
\ga_0(\Si_\la(\Re\tau))\to\ga_0(\Si_1(\va))\text{ as
}\la\to\infty\,.
\end{equation*}
In particular\textup{,} $\ga_0(\Re\tau)<\infty$.
\end{lem}
\begin{proof}
The hypothesis that the imaginary part goes to zero as
$\abs{\xi}\to\infty$ implies that
$\abs{\tau(\xi)-\Re\tau(\xi)}\to0$ as $\abs{\xi}\to\infty$. With
this additional observation, the proof of
Lemma~\ref{LEM:convexityconstantconv} can then be used once more.
\end{proof}

In addition to  Proposition \ref{prop:roots-r},
we will now prove the following refined perturbation properties:

\begin{prop}\label{Prop:asymptotic}
Suppose $\tau:\R^n\to\C$ is a characteristic root of
the strictly hyperbolic Cauchy problem 
\eqref{EQ:standardCauchyproblem}. Assume that it is a
smooth function satisfying
$\Im\tau(\xi)\ge 0$. 
Assume also that the roots $\phi_k(\xi)$, $k=1,\ldots,m$, of
the principal part $L_m$ are non-zero for all $\xi\not=0$.
Then we have the following properties:
\begin{enumerate}[label=\textup{(}\textup{\roman*}\textup{)}]
\item\label{HYP:nonconcplxtauisasymbol} for all multi-indices $\al$
there exist constants $M, C_\al, C_\al'>0$ such that
\begin{equation*}
\abs{\pa_\xi^\al\Re\tau(\xi)}\le C_\al\brac{\xi}^{1-\abs{\al}}
\end{equation*}
and
\begin{equation*}
\abs{\pa_\xi^\al\Im\tau(\xi)}\le C_\al'\brac{\xi}^{-\abs{\al}};
\end{equation*}
for all $|\xi|\geq M$.
\item\label{HYP:nonconcplxtauboundedbelow} there exist constants
$M,C>0$ such that for all $\abs{\xi}\ge M$ we have
$\abs{\Re\tau(\xi)}\ge C\abs{\xi}$\textup{;}
\item\label{HYP:nonconcplxderivativeoftaunonzero} there exists a
constant $C_0>0$ such that 
$\abs{\pa_\om\Re\tau(\la\om)}\ge C_0$ for
all $\om\in \Snm$ and sufficiently large $\la>0$\textup{;}
\item\label{HYP:nonconcplxlimitofSi_la} there exists a constant
$R_1>0$ such that\textup{,} for all sufficiently large
 $\la>0$\textup{,}
\begin{equation*}
\frac{1}{\la}\set{\xi\in\R^n:\Re\tau(\xi)=\la}\subset
B_{R_1}(0)\,.
\end{equation*}
\end{enumerate}
\end{prop}
\begin{proof}

\begin{itemize}[leftmargin=*]
\item[(i)] The statements follow by
Proposition~\ref{PROP:perturbationresults}:
\ref{PROP:boundsonderivsoftau} implies that for all $\abs{\xi}\ge
N$ and multi-indices~$\al$,
\begin{equation*}
\abs{\pa_\xi^\al\Re\tau(\xi)}\le \abs{\pa_\xi^\al\tau(\xi)}\le
C\abs{\xi}^{1-\abs{\al}}\,,
\end{equation*}
which suffices for the first part of (i).
Furthermore, \ref{PROP:derivoftau-derivofphi} tells us that for
all $\abs{\xi}\ge N$ and multi-indices~$\al$,
\begin{equation*}
\abs{\pa_\xi^\al[\Re\tau(\xi)-\va(\xi)]+i\pa^\al_\xi\Im\tau(\xi)}
=\abs{\pa_\xi^\al\tau(\xi)-\pa_\xi^\al\va(\xi)} \le
C\abs{\xi}^{-\abs{\al}}\,,
\end{equation*}
where~$\va(\xi)$ is a characteristic root of the principal part
(and is thus real-valued by definition of hyperbolicity); this
implies that, for all $\abs{\xi}\ge N$ and multi-indices~$\al$,
\begin{equation}\label{EQ:Impartofrootsymbol}
\abs{\pa_\xi^\al[\Re\tau(\xi)-\va(\xi)]}\le
C\abs{\xi}^{-\abs{\al}}\,\text{ and }
\abs{\pa^\al_\xi\Im\tau(\xi)}\le C\abs{\xi}^{-\abs{\al}}\,.
\end{equation}
The second of these gives us the second part of (i).
\item[(ii)] We
note that there exist constants $C,C',C'',M>0$ such that, for all
$\abs{\xi}\ge M$,
\begin{equation*}
\abs{\Re\tau(\xi)}\ge \abs{\tau(\xi)}-\abs{\Im\tau(\xi)}\ge
C'\abs{\xi}-C''\ge C\abs{\xi}\,.
\end{equation*}
Here we have used~\eqref{EQ:rootgrowspoly}, which did not
require~$\tau$ to be real-valued (nor to satisfy the convexity
condition), simply to be a characteristic root of a linear
constant coefficient strictly hyperbolic partial differential
equation, and the second part of~\eqref{EQ:Impartofrootsymbol}.
\item[(iii)] This follows 
in a similar way: using 
\eqref{EQ:Impartofrootsymbol}, we have, for
$\la\ge M$, some $M>0$, that 
\begin{equation*}
\abs{\pa_\om\Re\tau(\la\om)}\ge
\abs{\pa_\om\tau(\la\om)}-\abs{\pa_\om\Im\tau(\la\om)} \ge
C'-C''\la^{-1}\ge C\,.
\end{equation*}
\item[(iv)] This follows from
$\abs{\Re\tau(\xi)-\va(\xi)}\le C$ for all $\xi\in\R^n$
which holds in all~$\R^n$ by 
\ref{LEM:tau-vabounds} of
Proposition~\ref{PROP:perturbationresults}. 
\end{itemize}
\end{proof}

\subsection{Estimates for bounded 
frequencies
%$\abs{\xi}$ 
away from multiplicities}
\label{SEC:bddxiawayfrommults}

In the following sections we find $L^p-L^q$ 
estimates for integrals of the kind
\begin{equation*}
\int_{\Om}e^{i(x\cdot\xi+\tau(\xi)t)}a(\xi)\widehat{f}(\xi)\,d\xi\,,
\end{equation*}
where $\Om\subset\R^n$ is open and bounded, $f\in
C^\infty_0(\R^n)$, $a\in C^\infty_0(\Om)$, $\tau\in C^\infty(\Om)$
and $\Im\tau(\xi)\ge 0$ for
all $\xi\in\Om$.

As in the case of large~$\abs{\xi}$, we can further split this into
three main cases by using suitable cut-off functions:
\begin{enumerate}[leftmargin=*]
\item $\tau(\xi)$ is separated from the real axis for all
$\xi\in\Om$ (Theorem \ref{THM:expdecay});
\item $\tau(\xi)$ meets the real axis with order~$s<\infty$ at a
point $\xi^0\in\Om$ (Theorem \ref{THM:dissipative});
\item $\tau(\xi)$ lies on the real axis for all $\xi\in\Om$.
\end{enumerate}
We look at each in turn.

\subsection{Phase separated from the real
axis: Theorem \ref{THM:expdecay} again
}\label{SEC:bddxiawayfromaxis} Similarly to the case for
large~$\abs{\xi}$, we show that when the phase function~$\tau(\xi)$
is separated from the real axis (here, for $\xi\in\Om$,
$\Om$ is a bounded set),
\begin{equation}\label{EQ:estforrootawayfromaxisbddxi}
\normBig{D^r_tD_x^\al\Big(\int_{\Om}e^{i(x\cdot\xi+\tau(\xi)t)}
a(\xi)\widehat{f}(\xi)\,d\xi\Big)}_{L^q} \le Ce^{-\de t}\norm{f}_{L^p}\,,
\end{equation}
where $\frac1p+\frac1q=1$, $1\leq p\le 2$, $r\ge0$,~$\al$ 
a multi-index,
$f\in C_0^\infty(\R^n)$, $\de>0$ is a constant such that
$\Im\tau(\xi)\ge\de$ for all $\xi\in\Om$ and $C\equiv
C_{\Om,r,\al,p}>0$. So, in this case we also have exponential
decay of the solution.

By interpolating (Theorem~\ref{thm:maininterpolationthm}), it
suffices to show for such~$\tau(\xi)$
\begin{gather*}
\normBig{D^r_tD_x^\al\Big(\int_{\Om}e^{i(x\cdot\xi+\tau(\xi)t)}
a(\xi)\widehat{f}(\xi)\,d\xi\Big)}_{L^\infty} \le Ce^{-\de
t}\norm{f}_{L^1}\,,\\
\normBig{D^r_tD_x^\al\Big(\int_{\Om}e^{i(x\cdot\xi+\tau(\xi)t)}
a(\xi)\widehat{f}(\xi)\,d\xi\Big)}_{L^2} \le Ce^{-\de t}\norm{f}_{L^2}\,,
\end{gather*}
for $t\geq 0$, where $r\ge0$ and~$\al$ is a multi-index.

These are proved in a similar way to
Proposition~\ref{PROP:rootsawayfromaxis}, but noting that the
boundedness of~$\Om$ and the continuity in~$\Om$ of
$\tau(\xi)^ra(\xi)$ ensure there exists a constant
$C_{\Om,r,\al}\equiv C>0$ such that
$\abs{\tau(\xi)}^r\abs{a(\xi)}\abs{\xi}^{\abs{\al}}\le C$ for all
$\xi\in\Om$. Then, for all $t\geq 0$ and $r,\al$ as above,
we can estimate
\begin{align*}
\absBig{D^r_tD_x^\al\Big(\int_{\Om}&e^{i(x\cdot\xi+\tau(\xi)t)}
a(\xi)\widehat{f}(\xi)\,d\xi\Big)}
=\absBig{\int_{\Om}e^{i(x\cdot\xi+\tau(\xi)t)}
a(\xi)\xi^\al\tau(\xi)^r\widehat{f}(\xi)\,d\xi} \\
&\le C\int_{\Om}e^{-\Im\tau(\xi)t}
\abs{a(\xi)}\abs{\xi}^{\abs{\al}}\abs{\tau(\xi)}^r
\abs{\widehat{f}(\xi)}\,d\xi\\ &\le C\int_{\Om}e^{-\Im\tau(\xi)t}
\abs{\widehat{f}(\xi)}\,d\xi \le Ce^{-\de
t}\norm{\widehat{f}}_{L^\infty(\Om)}\le Ce^{-\de t}\norm{f}_{L^1}\,,
\end{align*}
and
\begin{align*}
\normBig{D^r_tD_x^\al\Big(\int_{\Om}&e^{i(x\cdot\xi+\tau(\xi)t)}
a(\xi)\widehat{f}(\xi)\,d\xi\Big)}_{L^2(\R^n_x)} =\normbig{e^{i\tau(\xi)t}
a(\xi)\xi^\al\tau(\xi)^r\widehat{f}(\xi)}_{L^2(\Om)} \\
&= \Big(\int_{\Om}e^{-2\Im\tau(\xi)t}
\abs{a(\xi)}^2\abs{\xi^{\al}}^{2}\abs{\tau(\xi)}^{2r}
\abs{\widehat{f}(\xi)}^2\,d\xi\Big)^{1/2}\\
&\le Ce^{-\de t}\norm{\widehat{f}}_{L^2(\Om)}\le Ce^{-\de
t}\norm{f}_{L^2}\,.
\end{align*}
We have now completed the proof of Theorem~\ref{THM:expdecay}.

\subsection{Roots meeting the real axis: Theorem
\ref{THM:dissipative}}\label{SEC:rootmeetingreal-sec}

In the case of bounded~$\abs{\xi}$, we must also consider the
situation where the phase function~$\tau(\xi)$ meets the real
axis. Suppose~$\xi^0\in\Om$ is such a point, i.e.\
$\Im\tau(\xi^0)=0$, while in each punctured ball around~$\xi^0$,
$B'_\ep(\xi^0)\subset\Om$, $\ep>0$, we have
$\Im\tau(\xi)>0$. Then,
$\xi^0$ is a root of~$\Im\tau(\xi)$ of some finite order~$s$:
indeed, if~$\xi^0$ were a zero of $\Im\tau(\xi)$ of infinite
order, then, by the analyticity of $\Im\tau(\xi)$ at~$\xi^0$
(which follows straight from the analyticity of~$\tau(\xi)$
at~$\xi^0$) it would be identically zero in a neighbourhood
of~$\xi^0$, contradicting the assumption.

Furthermore, we claim that $s\ge2$,~$s$ is even, and that there exist
constants $c_0,c_1>0$ such that, for all~$\xi$ sufficiently close
to~$\xi^0$, we have
\begin{equation*}%\label{EQ:estforimtauonaxis}
c_0\abs{\xi-\xi^0}^{s}\le\abs{\Im\tau(\xi)}\le
c_1\abs{\xi-\xi^0}^2\,.
\end{equation*}
Indeed, the Taylor expansion of $\Im\tau(\xi)$ around~$\xi^0$,
\begin{equation*}
\Im\tau(\xi)=\sum_{i=1}^n
\pa_{\xi_i}\Im\tau(\xi^0)(\xi_i-(\xi^0)_i)
+O(\abs{\xi-\xi^0}^2)\,,
\end{equation*}
is valid for~$\xi\in B_\ep(\xi^0)\subset\Om$ for some small
$\ep>0$. Now, if $\xi\in B_\ep(\xi^0)$, then $-\xi+2\xi^0\in
B_\ep(\xi^0)$ also. However,
\begin{equation*}
\Im\tau(-\xi+2\xi^0)=-\sum_{i=1}^n
\pa_{\xi_i}\Im\tau(\xi^0)(\xi_i-(\xi^0)_i)+O(\abs{\xi-\xi^0}^2)\,;
\end{equation*}
thus, for $\ep>0$ chosen small enough, this means that either
$\Im\tau(\xi)\le0$ or $\Im\tau(-\xi+2\xi^0)\le0$.
In view of 
the hypothesis that $\Im\tau(\xi)\ge0$ for all $\xi\in\Om$,
we must have
$\pa_{\xi_i}\Im\tau(\xi^0)=0$ for each $i=1,\dots,n$. In
conclusion, $\Im\tau(\xi)=O(\abs{\xi-\xi^0}^2)$ for all $\xi\in
B_\ep(\xi^0)$, which means that the zero is of order $s\ge2$, and
a similar argument shows that~$s$ must be even; also, this means
that there exist $c_0,c_1>0$ so that the above inequality holds
for $\xi\in B_\ep(\xi^0)$, proving the claim.

Now, we need the following result, which  
will be useful in the sequel.
Moreover, we will give its further extension in
Proposition \ref{PROP:generalrootsmeetingaxis1} to deal
with the setting of Theorem \ref{TH:multonaxis}.
\begin{prop}\label{PROP:generalrootsmeetingaxis}
Let $\phi:U\to\R$\textup{,} $U\subset\R^n$ open\textup{,} be a
continuous function and suppose $\xi^0\in U$ is such that
$\phi(\xi^0)=0$ and such that $\phi(\xi)>0$ in a punctured open
neighbourhood of~$\xi^0$\textup{,} $V\setminus\set{\xi^0}$.
Furthermore\textup{,} assume that\textup{,} for some $s>0$,
there exists a constant $c_0>0$ such that\textup{,} for all
$\xi\in V$\textup{,}
\begin{equation*}
\phi(\xi)\ge c_0\abs{\xi-\xi^0}^{s}\,.
\end{equation*}
Then\textup{,} for any function~$a(\xi)$ that is bounded and
compactly supported in~$U$\textup{,} and for all $t\geq 0$, $f\in
C_0^\infty(\R^n)$\textup{,} and $r\in\R$\textup{,} we have
\begin{gather}
\int_{V}e^{-\phi(\xi)t}\abs{\xi-\xi^0}^r\abs{a(\xi)}
\abs{\widehat{f}(\xi)}\,d\xi \le
C\bract{t}^{-(n+r)/s}\norm{f}_{L^{1}}\,,\label{EQ:matsL1est}\\
\intertext{and} \normbig{e^{-\phi(\xi)t}\abs{\xi-\xi^0}^r a(\xi)
\widehat{f}(\xi)}_{L^2(V)} \le
C\bract{t}^{-r/s}\norm{f}_{L^2}\,.\label{EQ:matsL2est}
\end{gather}
The constant $C$ depends on $U, V, c_0$ and $||a||_{L^\infty}$,
but not on the position of $\xi_0$.
\end{prop}

First, we establish
a straightforward result that is useful in proving
each of these estimates:
\begin{lem}\label{LEM:simplelemma}
For each~$\rho,M\geq 0$ and $\sigma,c>0$ there exists $C\equiv
C_{\rho,\sigma,M,c}\ge0$ such that\textup{,} for all
$t\ge0$\textup{,} we have
\begin{gather*}
\int_0^{M}x^\rho e^{-cx^\sigma t}\,dx\le
C\bract{t}^{-(\rho+1)/\sigma} \text{ and } \sup_{0\le x\le
M}x^\rho e^{-cx^\sigma t}\le C\bract{t}^{-\rho/\sigma}\,.
\end{gather*}
\end{lem}
\begin{proof}
For $0\le t\le 1$, each is clearly bounded: the first by
$\frac{M^{\rho+1}}{\rho+1}$ and the second by~$M^\rho$. For $t>1$,
set $y=xt^{1/\sigma}$; with this substitution, the first becomes
\begin{equation*}
\int_0^{M t^{1/\sigma}}y^\rho
t^{-\rho/\sigma}e^{-cy^\sigma}t^{-1/\sigma}\,dy \le
t^{-(\rho+1)/\sigma}\int_0^\infty y^\rho e^{-cy^\sigma}\,dy\,,
\end{equation*}
while the second becomes
\begin{equation*}
\sup_{0\le y\le Mt^{1/\sigma}}y^\rho
t^{-\rho/\sigma}e^{-cy^\sigma}\le
t^{-\rho/\sigma}\sup_{y\ge0}y^\rho e^{-cy^\sigma}\,;
\end{equation*}
These estimates imply those of Lemma \ref{LEM:simplelemma}
since both the integral and the supremum in the right hand
sides are bounded.
\end{proof}
\begin{proof}[Proof of Proposition
\ref{PROP:generalrootsmeetingaxis}]
As for the proof of
~\eqref{EQ:matsL1est},
since~$a(\xi)$ is
bounded in~$U$ by assumption, we have
\begin{equation*}
\int_{V}e^{-\phi(\xi)t}\abs{\xi-\xi^0}^r\abs{a(\xi)}
\abs{\widehat{f}(\xi)}\,d\xi \le C\int_{V'}
e^{-\phi(\xi)t}\abs{\xi-\xi^0}^r \abs{\widehat{f}(\xi)}\,d\xi\,,
\end{equation*}
where $V'=V\cap\supp a$; this, in turn, can be estimated in the
following manner using the hypothesis on $\phi(\xi)$ and
H\"older's inequality:
\begin{align*}
\int_{V'} e^{-\phi(\xi)t}\abs{\xi-\xi^0}^r
\abs{\widehat{f}(\xi)}&\,d\xi \le C\int_{V'}
e^{-c_0\abs{\xi-\xi^0}^{s}t}\abs{\xi-\xi^0}^r
\abs{\widehat{f}(\xi)}\,d\xi\\
&\le C\int_{V'} e^{-c_0\abs{\xi-\xi^0}^{s}t}\abs{\xi-\xi^0}^r\,d\xi
\;\norm{\widehat{f}}_{L^{\infty}(V')}\,.
\end{align*}
Then, transforming to polar coordinates and using the
Hausdorff--Young inequality, we find that, for some $R>0$ (chosen
so that $V'\subset B_R(\xi^0)$, which is possible since~$a(\xi)$ is
compactly supported), we have
\begin{equation*}
\int_{V'} e^{-c_0\abs{\xi-\xi^0}^{s}t}\abs{\xi-\xi^0}^r\,d\xi
\norm{\widehat{f}}_{L^{\infty}(V')} \le
C\int_{\Snm}\int_0^R\abs{\eta}^{r+n-1}
e^{-c_0\abs{\eta}^{s}t}\,d\abs{\eta} d\omega 
\norm{f}_{L^{1}(\R^n)}\,.
\end{equation*}
Finally, by the first part of Lemma~\ref{LEM:simplelemma}, we find
\begin{align*}
\int_{V}e^{-\phi(\xi)t}\abs{\xi-\xi^0}^r\abs{a(\xi)}
\abs{\widehat{f}(\xi)}\,d\xi &\le C\int_0^R y^{r+n-1}
e^{-c_0y^{s}t}\,dy\norm{f}_{L^{1}(\R^n)}\\&\le
C\bract{t}^{-(n+r)/s}\norm{f}_{L^{1}}\,.
\end{align*}
This completes the proof of the first part.

Now let us look at the second part. By the second part of
Lemma~\ref{LEM:simplelemma}, we get
\begin{multline*}
\normbig{e^{-\phi(\xi)t}\abs{\xi-\xi^0}^ra(\xi)
\widehat{f}(\xi)}_{L^2(V)}^2\le \int_{V'}
e^{-2c_0\abs{\xi-\xi^0}^{s}t}\abs{\xi-\xi^0}^{2r}
\abs{\widehat{f}(\xi)}^2\,d\xi\\ \le
C\bract{t}^{-2r/s}\int_{V'}e^{-c_0\abs{\xi-\xi^0}^{s}t}
\abs{\widehat{f}(\xi)}^2\,d\xi\,.
\end{multline*}
Now, it follows that
\begin{equation*}
\int_{V'} e^{-c_0\abs{\xi-\xi^0}^{s}t}
\abs{\widehat{f}(\xi)}^2\,d\xi\le\sup_{V'}
\absbig{e^{-c_0\abs{\xi-\xi^0}^{s}t}}
\norm{\widehat{f}}_{L^{2}(V')}^2\le C\norm{f}_{L^{2}}^2\,.
\end{equation*}
Together these give the required estimate~\eqref{EQ:matsL2est}.
\end{proof}

So, using this proposition, we have, for all $t\geq 0$, and
sufficiently small $\ep>0$,
\begin{multline*}
\normBig{D^r_tD^\al_x\int_{B_\ep(\xi^0)}e^{i(x\cdot\xi+\tau(\xi)t)}
a(\xi) \widehat{f}(\xi)\,d\xi}_{L^\infty(\R^n_x)}\\
\le\int_{B_\ep(\xi^0)}e^{-\Im\tau(\xi)t}\abs{a(\xi)}\abs{\tau(\xi)}^r
\abs{\xi}^{\al} \abs{\widehat{f}(\xi)}\,d\xi \le
C{(1+t)}^{-n/s}\norm{f}_{L^{1}}\,,
\end{multline*}
and, using the Plancherel's theorem, we get
\begin{multline*}
\normBig{D^r_tD^\al_x\int_{B_\ep(\xi^0)}e^{i(x\cdot\xi+\tau(\xi)t)}
a(\xi) \widehat{f}(\xi)\,d\xi}_{L^2(\R^n_x)}\\ =
C\normbig{e^{i\tau(\xi)t}\tau(\xi)^r\xi^\al a(\xi)
\widehat{f}(\xi)}_{L^2(B_\ep(\xi^0))} \le C\norm{f}_{L^2}\,;
\end{multline*}
here we have used 
that~$\abs{\xi}^{\abs{\al}}\abs{\tau(\xi)}^r\le C$ 
on $B_\epsilon(\xi^0)$ for $r\in\N$, $\al$ a multi-index.

Thus, by Theorem~\ref{thm:maininterpolationthm}, for all $t\geq 0$,
we get
\begin{equation}\label{EQ:estforrootmeetingaxis}
\normBig{D^r_tD^\al_x\int_{B_\ep(\xi^0)}e^{i(x\cdot\xi+\tau(\xi)t)}a(\xi)
\widehat{f}(\xi)\,d\xi}_{L^p(\R^n_x)} \le
C{(1+t)}^{-\frac{n}{s}\big(\frac{1}{p}-\frac{1}{q}\big)}
\norm{f}_{L^{q}}\,,
\end{equation}
where $1\leq p\le2$, 
$\frac1p+\frac1q=1$. This completes the proof of
Theorem~\ref{THM:dissipative} for
roots meeting the axis with finite order and no multiplicities.

\begin{rem}\label{rem:mats}
If $\xi^0=0$, then Proposition~\ref{PROP:generalrootsmeetingaxis}
further tells us that
\begin{equation*}
\normBig{D^r_tD^\al_x
\int_{B_\ep(0)}e^{i(x\cdot\xi+\tau(\xi)t)}a(\xi)
\widehat{f}(\xi)\,d\xi}_{L^q(\R^n_x)}\le
C\bract{t}^{-\frac{n}{s}\big(\frac{1}{p}-\frac{1}{q}\big)
-\frac{\abs{\al}}{s}}
\norm{f}_{L^{p}}\,.
\end{equation*}
If, in addition, we have $|\tau(\xi)|\leq c_1|\xi-\xi^0|^{s_1}$, 
for~$\xi$ near~$\xi^0$, then we also get
\begin{equation*}
\normBig{D^r_tD^\al_x
\int_{B_\ep(\xi^0)}e^{i(x\cdot\xi+\tau(\xi)t)}a(\xi)
\widehat{f}(\xi)\,d\xi}_{L^q(\R^n_x)} \le
C\bract{t}^{-\frac{n}{s}\big(\frac{1}{p}-\frac{1}{q}\big)
-\frac{r s_1}{s}}
\norm{f}_{L^{p}}\,.
\end{equation*}
If both assumptions hold, we get the improvement from both 
cases, which is the estimate by
$C\bract{t}^{-\frac{n}{s}
\big(\frac{1}{p}-\frac{1}{q}\big)-
\frac{|\alpha|}{s}-\frac{r s_1}{s}}.$
\end{rem}
From this, we obtain the statement of Theorem \ref{THM:dissipative}
in the frequency region $B_\epsilon(\xi^0)$. Since there are
only finitely many such points by hypothesis (H2) of Theorem
\ref{THM:dissipative}, hypothesis (H1) guarantees that on the
complement of their neighborhoods we have $\Im\tau_k>0$.
There we can apply Theorems \ref{THM:expdecay} and
\ref{THM:expdecay2} to get the exponential decay. In may happen
that the roots are multiple, but Theorem \ref{THM:expdecay2}
provides the required estimate in such cases as well. 
The Sobolev orders in Theorem \ref{THM:dissipative} come
from large frequencies as given in Theorem \ref{THM:expdecay}.
This completes the proof of Theorem \ref{THM:dissipative}
and of Remark \ref{REM:dismore}.

\subsection{Phase function lies on the real axis}

As in the case of large~$\abs{\xi}$, we can subdivide into several
subcases:
\begin{enumerate}[leftmargin=*,label=(\roman*)]
\item $\det\Hess\tau(\xi)\ne0$;
\item $\det\Hess\tau(\xi)=0$ and $\tau(\xi)$ satisfies the convexity
condition;
\item the general case when $\det\Hess\tau(\xi)=0$.
\end{enumerate}
For the first case, the approach used in
Section~\ref{SEC:mainstatphasesection} can be used here also, since
there we do not use that~$\abs{\xi}$ is large other than to ensure
that $\tau(\xi)$ was smooth; here, we are away from multiplicities,
so that still holds. 
Therefore, the conclusion is the same, giving
Theorem~\ref{THM:nondeghess}.

The other two cases are considered in the next section 
alongside the
case where there are multiplicities since it is important precisely
how the integral is split up for such cases.

\section{Estimates for bounded
frequencies
%$\abs{\xi}$ 
around multiplicities}
\label{SEC:bddxiaroundmults}

Finally, let us turn to finding estimates for the first term
of~\eqref{EQ:intdivisionsplittingoffmults}, which we may write in
the form
\begin{equation*}
\int_{\Om}e^{ix\cdot\xi}\Big(\sum_{k=1}^L
e^{i\tau_k(\xi)t}A_j^k(t,\xi)\Big)
\cutoffM(\xi)\widehat{f}(\xi)\,d\xi\,,
\end{equation*}
where the characteristic roots $\tau_1(\xi),\dots,\tau_L(\xi)$
coincide in a set~$\curlyM\subset\Om$ of codimension~$\ell$ (in the
sense of Section \ref{SEC:21}),
$\Om\subset\R^n$ is a bounded open set and $\cutoffM\in
C_0^\infty(\Om)$.

As before, we must consider the
cases where the image of the phase function(s) either lie on the
real axis, are separated from the real axis or meet the real axis.
One additional thing to note in this case is that 
in principle the order of
contact at points of multiplicity may be infinite as the roots are
not necessarily analytic at such points; we have no examples of
such a situation occurring, so it is not worth studying too deeply
unless such an example can be found---for now, we can use the same
technique as if the point(s) were points where the roots lie
entirely on the real axis, and the results in these two situations
are given together in {Theorem~\ref{THM:overallmainthm}.}
We study this very briefly nevertheless to ensure
the completeness of the obtained results.

Unlike in the case away from multiplicities of characteristic
roots, we have no explicit representation for the
coefficients~$A_j^k(t,\xi)$
(as we have in Lemma \ref{LEM:ordercoeff} away from
the multiplicities), which in turn means we cannot split
this into~$L$ separate integrals. To overcome this, we first show,
in Section~\ref{SEC:resofroots}, that a useful representation for
the above integral does exist that allows us to use techniques
from earlier. Using this alternative representation, it is a
simple matter to find estimates in the case where the image of the
set~$\curlyM$ mapped by the characteristic roots
is separated from the real axis (this is Theorem
\ref{THM:expdecay2}) and when it arises
on the real axis as a result of all the roots meeting the axis
with finite order, and these are done in
Sections~\ref{SEC:phasefnsepfromaxis}
and~\ref{SEC:phasefnmeetsfinorder}, respectively.

The situations where the roots meet on the real axis and at least
one has a zero of infinite order there (either because it fully
lies on the axis, or because it meets the axis with infinite
order) is slightly more complicated; this is discussed in
Section~\ref{SEC:phasefnliesonaxisbddxi}.

\subsection{Resolution of multiple roots}\label{SEC:resofroots}
In this section, we find estimates for
\begin{equation*}
\sum_{k=1}^L e^{i\tau_k(\xi)t}A_j^k(t,\xi)\,,
\end{equation*}
corresponding to \eqref{EQ:solnpart},
where $\tau_1(\xi),\dots,\tau_L(\xi)$ coincide in a set~$\curlyM$ of
codimension~$\ell$. For simplicity, first consider the simplest
case of two roots intersecting at a single point,
so that we have 
$L=2$ and $\curlyM=\set{\xi^0}$; the general case works in a
similar way, and we shall show how it differs below. So, assume
\begin{equation*}
\tau_1(\xi^0)=\tau_2(\xi^0)\text{ and }
\tau_k(\xi^0)\ne\tau_1(\xi^0)\text{ for }k=3,\dots,m\,;
\end{equation*}
by continuity, there exists a ball of radius $\ep>0$
about~$\xi^0$, $B_\ep(\xi^0)$, in which the only root which
coincides with~$\tau_1(\xi)$ is~$\tau_2(\xi)$. Then:
\begin{lem}\label{LEM:2rootsmeetingresolution}
For all $t\ge0$ and $\xi\in B_\ep(\xi^0)$, we have
\begin{equation}\label{EQ:intersectingrootsbound}
\absBig{\sum_{k=1}^2e^{i\tau_k(\xi)t}A_j^k(t,\xi)}\le
C(1+t)e^{-\min(\Im\tau_1(\xi),\Im\tau_2(\xi))t}\,,
\end{equation}
where the minimum is taken over $\xi\in B_\ep(\xi^0)$.
\end{lem}
\begin{proof}
First, note that in the set
\begin{equation*}
S:=\{\xi\in\R^n:\tau_1(\xi)\ne\tau_k(\xi)\;\forall k=2,\dots,m\text{
and}\,\tau_2(\xi)\ne\tau_l(\xi)\quad\forall l=3,\dots,m\}
\end{equation*}
the formula~\eqref{EQ:Ajkformula} is valid for $A_j^1(\xi)$ and
$A_j^2(\xi)$. Now, recall that the sum
$E_j(t,\xi)=\sum_{k=1}^me^{i\tau_k(\xi)t}A^k_j(t,\xi)$ is the
solution to the Cauchy
problem~\eqref{EQ:ftcauchyprob},~\eqref{EQ:initialdataforEj}, and
thus is continuous; therefore, for all~$\eta\in\R^n$ such that
$\tau_1(\eta)\ne\tau_k(\eta)$ and $\tau_2(\eta)\ne\tau_k(\eta)$ for
$k=3,\dots,m$ (but allow $\tau_1(\eta)=\tau_2(\eta)$), we have
\begin{align*}
\sum_{k=1}^2e^{i\tau_k(\eta)t}A_j^k(t,\eta)
&=\lim_{\xi\to\eta}\big(e^{i\tau_1(\xi)t}A_j^1(\xi)+
e^{i\tau_2(\xi)t}A_j^2(\xi)\big)\,,
\end{align*}
provided~$\xi$ varies in the set~$S$ (thus, ensuring
$e^{i\tau_1(\xi)t}A_j^1(\xi)+ e^{i\tau_2(\xi)t}A_j^2(\xi)$ is
well-defined). Hence, to obtain~\eqref{EQ:intersectingrootsbound}
for all $\xi\in B_\ep(\xi^0)$, it suffices to show
\begin{equation*}
\absbig{e^{i\tau_1(\xi)t}A_j^1(\xi)+ e^{i\tau_2(\xi)t}A_j^2(\xi)}
\le C(1+t)e^{-\min(\Im\tau_1(\xi),\Im\tau_2(\xi))t}
\end{equation*}
for all $t\ge0$, $\xi\in
B'_\ep(\xi^0)=B_\ep(\xi^0)\setminus\set{\xi^0}$.

Now, note the following trivial equality:
\begin{align*}
K_1e^{iy_1}+&\,K_2e^{iy_2} = K_1e^{iy_2}e^{i(y_1-y_2)}
+K_2e^{iy_1}e^{-i(y_1-y_2)}\\
=&\, \frac{e^{i(y_1-y_2)}-e^{-i(y_1-y_2)}}{2}K_1e^{iy_2}
+\frac{e^{i(y_1-y_2)}+e^{-i(y_1-y_2)}}{2}K_1e^{iy_2} \\
&+\frac{e^{-i(y_1-y_2)}-e^{i(y_1-y_2)}}{2}K_2e^{iy_1}+
\frac{e^{-i(y_1-y_2)}+e^{i(y_1-y_2)}}{2}K_2e^{iy_1}\\
=&\,\sinh(y_1-y_2)[K_1e^{iy_2}-K_2e^{iy_1}]
+\cosh(y_1-y_2)[K_1e^{iy_2}+K_2e^{iy_1}]\,.
\end{align*}
Using this, we have, for all $\xi\in B'_\ep(\xi^0)$, $t\ge0$,
\begin{multline}\label{EQ:splittingfor2roots}
e^{i\tau_1(\xi)t}A_j^1(\xi)+e^{i\tau_2(\xi)t}A_j^2(\xi)
\\=\sinh[(\tau_1(\xi)-\tau_2(\xi))t]
(e^{i\tau_2(\xi)t}A^1_j(\xi)-e^{i\tau_1(\xi)t}A^2_j(\xi))\\
+\cosh[(\tau_1(\xi)-\tau_2(\xi))t]
(e^{i\tau_2(\xi)t}A^1_j(\xi)+e^{i\tau_1(\xi)t}A^2_j(\xi))\,.
\end{multline}
We estimate each of these terms:
\begin{enumerate}[leftmargin=*,label=(\alph*)]
\item\underline{``$\sinh$'' term}: The first term is simple to
estimate: since
\begin{equation*}
\frac{\sinh[(\tau_1(\xi)-\tau_2(\xi))t]}{(\tau_1(\xi)-\tau_2(\xi))}
\to t\;\text{ as }\,(\tau_1(\xi)-\tau_2(\xi))\to0\,,
\end{equation*}
or, equivalently, as $\xi\to\xi^0$ through $S$, and
$A^k_j(\xi)(\tau_1(\xi)-\tau_2(\xi))$ is continuous in
$B_\ep(\xi^0)$ for $k=1,2$, it follows that, for all $\xi\in
B'_\ep(\xi^0)$, $t\ge0$, we have
\begin{multline}\label{EQ:sinhest}
\absbig{\sinh[(\tau_1(\xi)-\tau_2(\xi))t]
(A^1_j(\xi)e^{i\tau_2(\xi)t}- A^2_j(\xi)e^{i\tau_1(\xi)t})}\\ \le
Ct[\abs{e^{i\tau_2(\xi)t}}+\abs{e^{i\tau_1(\xi)t}}] \le
Cte^{-\min(\Im\tau_1(\xi),\Im\tau_2(\xi))t}\,.
\end{multline}

\item\underline{``$\cosh$'' term}: Estimating the second term is
slightly more complicated. First, recall the explicit
representation~\eqref{EQ:Ajkformula} for the~$A_j^k(\xi)$ at
points away from multiplicities of~$\tau_k(\xi)$

\begin{equation*}
A_j^k(\xi)=\frac{(-1)^j\displaystyle\sideset{}{^k}\sum_{1\le
s_1<\dots<s_{m-j-1}\le m}
\prod_{q=1}^{m-j-1}\tau_{s_q}(\xi)}{\displaystyle\prod_{l=1,l\ne
k}^m(\tau_l(\xi)-\tau_k(\xi))}\;.
\end{equation*}
So, we can write
\begin{align*}
&\cosh[(\tau_1(\xi)-\tau_2(\xi))t](A^1_j(\xi)e^{i\tau_2(\xi)t}
+A^2_j(\xi)e^{i\tau_1(\xi)t})\\
&=\frac{\cosh[(\tau_1(\xi)-\tau_2(\xi))t]}
{\prod_{k=3}^m(\tau_k(\xi)-\tau_1 (\xi))
(\tau_k(\xi)-\tau_2(\xi))}\frac{e^{i\tau_2(\xi)t}F_{j+1}^{1,2}(\xi)
-e^{i\tau_1(\xi)t}F_{j+1}^{2,1}(\xi)}{\tau_1(\xi)-\tau_2(\xi)}\,,
\end{align*}
where
\begin{equation*}
F_i^{\rho,\sigma}(\xi):=\left(\sideset{}{^{\rho}}\sum_{1\le
s_1<\dots<s_{m-i}\le m
}\prod_{q=1}^{m-i}\tau_{s_q}(\xi)\right)\prod_{k=1,k\ne\rho,
\sigma}^m(\tau_k(\xi)-\tau_\sigma(\xi)).
\end{equation*}
Now, $\big(\cosh[(\tau_1(\xi)-\tau_2(\xi))t]\big)\big/
\big(\prod_{k=3}^m(\tau_k(\xi)-
\tau_1(\xi))(\tau_k(\xi)-\tau_2(\xi))\big)$ is continuous in~$S$,
hence it is bounded there, and, thus, absolutely converges to a
constant,~$C\ge0$ say, as $\xi\to\xi^0$ through~$S$. This leaves
the $[e^{i\tau_2(\xi)t}F_{j+1}^{1,2}(\xi)
-e^{i\tau_1(\xi)t}F_{j+1}^{2,1}(\xi)]/(\tau_1(\xi)-\tau_2(\xi))$
term.

For this, write
\begin{equation*}
F_i^{\rho,\sigma}(\xi)=\sum_{\ka=0}^{m-1}
Q_{\ka,i}^{\rho,\sigma}(\xi)\tau_\sigma(\xi)^\ka,
\end{equation*}
where the $Q_{\ka,i}^{\rho,\sigma}(\xi)$ are polynomials in the
$\tau_k(\xi)$ for $k\ne\rho,\sigma$ (which depend on~$i$); also,
note $Q_{\ka,i}^{\rho,\sigma}(\xi)=Q_{\ka,i}^{\sigma,\rho}(\xi)$.
Then, we have
\begin{multline}\label{EQ:coshpartrearrangement}
\frac{e^{i\tau_2(\xi)t}F_{j+1}^{1,2}(\xi)-
e^{i\tau_1(\xi)t}F_{j+1}^{2,1}(\xi)}{\tau_1(\xi)-\tau_2(\xi)}\\
=\frac{\sum_{\ka=0}^{m-1}\big[Q^{1,2}_{\ka,j+1}(\xi)
(\tau_2(\xi)^\ka e^{i\tau_2(\xi)t}- \tau_1(\xi)^\ka
e^{i\tau_1(\xi)t})\big]} {\tau_1(\xi)-\tau_2(\xi)}\,.
\end{multline}
Let us show that this is continuous in $B_\ep(\xi^0)$ and is
bounded absolutely by $Cte^{-\min\set{\la_1,\la_2}}$: for $y_1\ne
y_2$, and for all $r,s\in\N$, $t\ge0$, we have
\begin{multline*}
\frac{y_2^sy_1^re^{iy_2t}-y_1^sy_2^re^{iy_1t}} {y_1-y_2}=\\
\frac{y_2^sy_1^r(e^{iy_2t}-e^{iy_1t})}{y_1-y_2}+
\frac{y_2^se^{iy_1t}(y_1^r-y_2^r)}{y_1-y_2}
+\frac{e^{iy_1t}y_2^r(y_2^s-y_1^s)}{y_1-y_2}\,.
\end{multline*}
Furthermore, for all $y_1,y_2\in\C$, $t\in[0,\infty)$, $s\in\N$,
\begin{gather*}
\absBig{\frac{e^{iy_2t}-e^{iy_1t}}{y_1-y_2}}\le C_0te^{-\min(\Im
y_1,\Im y_2)t}\quad \text{ and }\quad
\absBig{\frac{y_1^s-y_2^s}{y_1-y_2}}\le C_s\,,
\end{gather*}
for some constants $C_0,C_s$. Using these with $y_1=\tau_1(\xi)$,
$y_2=\tau_2(\xi)$, $r=\ka$, and~$s$ chosen appropriately for
$Q^{1,2}_{\ka,j+1}(\xi)$, the continuity and upper bound follow
immediately. Thus, for all $\xi\in B'_\ep(\xi^0)$, $t\ge0$,
\begin{multline}\label{EQ:coshbound}
\abs{\cosh[(\tau_1(\xi)-\tau_2(\xi))t](A^1_j(\xi)e^{i\tau_2(\xi)t}
+A^2_j(\xi)e^{i\tau_1(\xi)t})}\\ \le
Cte^{-\min(\Im\tau_1(\xi),\Im\tau_2(\xi))t}\,.
\end{multline}
\end{enumerate}
Combining~\eqref{EQ:splittingfor2roots},~\eqref{EQ:sinhest} and
\eqref{EQ:coshbound} we have~\eqref{EQ:intersectingrootsbound},
which completes the proof of the lemma.
\end{proof}

Now we show that a similar result holds in the general case: suppose
the characteristic roots $\tau_1(\xi),\dots,\tau_L(\xi)$, $2\le L\le
m$, coincide in a set $\curlyM$, and that
$\tau_1(\xi)\ne\tau_k(\xi)$ for all $\xi\in\curlyM$ when
$k=L+1,\dots,m$. By continuity, we may take $\ep>0$ so that the set
$\curlyM^\ep=\set{\xi\in\R^n: \dist(\xi,\curlyM)<\ep}$ contains no
points~$\eta$ at which
$\tau_1(\eta),\dots,\tau_L(\eta)=\tau_k(\eta)$ for $k=L+1,\dots,m$.
With this notation, we have:
\begin{lem}\label{LEM:Lrootsmeetingest}
For all $t\ge0$ and $\xi\in\curlyM^\ep$\textup{,} we have
the estimate
\begin{equation}\label{EQ:multiplerootsrepbound}
\absBig{\sum_{k=1}^Le^{i\tau_k(\xi)t} A_j^k(t,\xi)} \le
C(1+t)^{L-1}e^{-t\min_{k=1,\dots,L}\Im\tau_k(\xi)}\,,
\end{equation}
where the minimum is taken over $\xi\in \curlyM^\ep$.
\end{lem}
Note that this estimate does not depend on the codimension
of~$\curlyM$.
\begin{proof}
First note that, just as in the previous proof, for
all~$\eta\in\R^n$ such that
$\tau_1(\eta)\dots,\tau_L(\eta)\ne\tau_k(\eta)$ when $k=L+1,\dots,m$
(but allowing any or all of $\tau_1(\eta),\dots,\tau_L(\eta)$ 
to be equal),
\begin{align*}
\sum_{k=1}^Le^{i\tau_k(\eta)t}A_j^k(t,\eta)
&=\lim_{\xi\to\eta}\big(e^{i\tau_1(\xi)t}A_j^1(\xi)+\dots+
e^{i\tau_L(\xi)t}A_j^L(\xi)\big)\,,
\end{align*}
provided~$\xi$ to varies the set $S:=\bigcup_{l=1}^LS_l$, where
\begin{equation*}
S_l:=\{\xi\in\R^n:\tau_l(\xi)\ne\tau_k(\xi)\;\forall k\ne l\},
\end{equation*}
to ensure that each term of the sum on the right-hand side is
well-defined. Note that Lemma~\ref{LEM:multiplerootssetisnice}
ensures every point in~$\curlyM$ is the limit of a sequence of
points in~$S$ in the case of
differential operators. 
Thus, we must simply show, for all $t\ge0$, $\xi\in
(\curlyM^\ep)'=\curlyM^\ep\setminus\curlyM$, that we have the
estimate
\begin{equation*}
\absbig{e^{i\tau_1(\xi)t}A_j^1(\xi)+\dots+
e^{i\tau_L(\xi)t}A_j^L(\xi)} \le
C(1+t)^{L-1}e^{-t\min_{k=1,\dots,L}\Im\tau_k(\xi)}\,.
\end{equation*}

Now, we claim that we can write
$\sum_{k=1}^{L}e^{i\tau_k(\xi)t}A_j^k(t,\xi)$, for
$\xi\in(\curlyM^\ep)'$ and $t\ge0$, as a sum of terms involving
products of $\frac{(L-1)L}{2}$ $\sinh$ and $\cosh$ terms of
differences of coinciding roots; to
clarify,~\eqref{EQ:splittingfor2roots} is this kind of
representation for $L=2$, while for $L=3$, we want sums of terms
such as
\begin{equation*}
\sinh[\al_1(\tau_1(\xi)-\tau_2(\xi))t]
\cosh[\al_2(\tau_1(\xi)-\tau_3(\xi))t]
\sinh[\al_3(\tau_2(\xi)-\tau_3(\xi))t]\,,
\end{equation*}
where the~$\al_i$ are appropriately chosen constants; incidentally,
a comparison to the $L=2$ case suggests that the term above is
multiplied by $$\big(A^1_j(\xi)e^{i\tau_2(\xi)t}
-A^2_j(\xi)e^{i\tau_1(\xi)t}\big)$$ in the full representation.

To show this, we do induction on~$L$; Lemma 
\ref{LEM:2rootsmeetingresolution} gives us
the case $L=2$ (note that the proof holds with~$\xi^0$
and~$B_\ep(\xi^0)$ replaced throughout by~$\curlyM$ and
$\curlyM^\ep$, respectively). Assume there is such a representation
for $L=K\le m-1$. Observe,
\begin{multline*}
\sum_{k=1}^{K+1}e^{i\tau_k(\xi)t}A_j^k(\xi) =\frac{1}{K}
\sum_{k=1}^{K}e^{i\tau_k(\xi)t}A_j^k(\xi) +\frac{1}{K}\sum_{k=1,k\ne
K}^{K+1}e^{i\tau_k(\xi)t}A_j^k(\xi)\\+\dots
+\frac{1}{K}\sum_{k=2}^{K+1}e^{i\tau_k(\xi)t}A_j^k(\xi)\,;
\end{multline*}
by the induction hypothesis, there is a representation for each of
these terms by means of products of $\frac{(K-1)K}{2}$
\begin{equation*}
\sinh[\al_{k,l}(\tau_k(\xi)-\tau_l(\xi))t]\text{ and }
\cosh[\be_{k,l}(\tau_k(\xi)-\tau_l(\xi))t]\text{ terms,}
\end{equation*}
where $1\le k,l\le K+1$ and the $\al_{k,l},\be_{k,l}$ are some
non-zero constants. Next, note that we can write
$(\tau_1(\xi)-\tau_2(\xi))$ (or, indeed, the difference of any pair
of roots from $\tau_1(\xi),\dots,\tau_{K+1}(\xi)$) as a linear
combination of the $\frac{K(K+1)}{2}$ differences
$\tau_k(\xi)-\tau_l(\xi)$ such that $1\le k<l\le K+1$; that is
\begin{equation*}
\sinh[\al_{1,2}(\tau_1(\xi)-\tau_2(\xi))t] =\sinh\Big[\sum_{1\le
k<l\le K+1}\al'_{k,l}(\tau_k(\xi)-\tau_l(\xi))t\Big]\,,
\end{equation*}
for some non-zero constants $\al'_{k,l}$; similarly, there is such a
representation for $\cosh[\be_{1,2}(\tau_1(\xi)-\tau_2(\xi))t]$.
Lastly, repeated application of the double angle formulae
\begin{gather*}
\sinh(a\pm b)=\sinh a\cosh b \pm \cosh a\sinh b\,,\\
\cosh(a\pm b)=\cosh a\cosh b \pm \sinh a\sinh b\,,
\end{gather*}
yields products of $\frac{K(K+1)}{2}$ terms, which completes the
induction step.

Now, as in the previous proof, each of these terms must be
estimated. The key fact to observe is that
\begin{equation*}
A_j^k(\xi)\prod_{l=1,l\ne k}^L (\tau_l(\xi)-\tau_k(\xi))
\end{equation*}
is continuous in $\curlyM^\ep$ for all $k=1,\dots,L$. Then, using
the same arguments as for each of the terms in the earlier proof,
and observing that the exponent of~$t$ is determined by the products
involving either
\begin{enumerate}[label=(\alph*),leftmargin=*]
\item $(\sinh[\al_{k,l}(\tau_k(\xi)-\tau_l(\xi)t)])/
(\tau_k(\xi)-\tau_l(\xi))$ terms, or
\item $(e^{i\tau_k(\xi)t}-e^{i\tau_l(\xi)t})/
(\tau_k(\xi)-\tau_l(\xi))$ terms
(see~\eqref{EQ:coshpartrearrangement}),
\end{enumerate}
the estimate~\eqref{EQ:multiplerootsrepbound} is immediately
obtained.
\end{proof}

\subsection{Phase separated from the real
axis: Theorem \ref{THM:expdecay2}}
\label{SEC:phasefnsepfromaxis}

We now turn back to finding $L^p-L^q$ estimates for
\begin{equation*}
\int_{\Om}e^{ix\cdot\xi}\Big(\sum_{k=1}^L
e^{i\tau_k(\xi)t}A_j^k(t,\xi)\Big)
\cutoffM(\xi)\widehat{f}(\xi)\,d\xi\,,
\end{equation*}
when $\tau_1(\xi),\dots,\tau_L(\xi)$ 
coincide in a set~$\curlyM$ of
codimension~$\ell$; choose $\ep>0$ so that these roots do not
intersect with any of the roots $\tau_{L+1}(\xi),\dots,\tau_m(\xi)$
in $\curlyM^\ep$. The set $\Omega$ is bounded, and we may take
$\chi\in C_0^\infty(\curlyM^\epsilon)$.

In this section (under assumptions of Theorem \ref{THM:expdecay2}), 
we assume that there exists $\de>0$ such that
$\Im\tau_k(\xi)\ge\de$ for all $\xi\in\curlyM^\ep$---so,
$\min_k\Im\tau_k(\xi)\ge\de>0$. For this, we use the same approach
as in Section~\ref{SEC:bddxiawayfromaxis}, but using
Lemma~\ref{LEM:Lrootsmeetingest} to estimate the sum. Firstly, the
$L^1-L^\infty$ estimate:
\begin{align*}
\normBig{D^r_tD_x^\al\Big(\int_{\Om}& e^{ix\cdot\xi}
\Big(\sum_{k=1}^L e^{i\tau_k(\xi)t}A_j^k(t,\xi)\Big)
\cutoffM(\xi)\widehat{f}(\xi)\,dx\Big)}_{L^\infty(\R^n_x)}\\
&=\normBig{\int_{\Om}e^{ix\cdot\xi} \Big(\sum_{k=1}^L
e^{i\tau_k(\xi)t}A_j^k(t,\xi)\tau_k(\xi)^r\Big)
\xi^\al\cutoffM(\xi)\widehat{f}(\xi)\,dx}_{L^\infty(\R^n_x)} \\
&\le \max_k\sup_\Om\abs{\tau_k(\xi)}^r
\int_{\curlyM^\ep}\absBig{\sum_{k=1}^L
e^{i\tau_k(\xi)t}A_j^k(t,\xi)}\abs{\xi}^{\abs{\al}}
\abs{\widehat{f}(\xi)}\,dx\\
&\le C(1+t)^{L-1}e^{-\de t}\norm{\widehat{f}}_{L^\infty
(\curlyM^\ep)}\le
C(1+t)^{L-1}e^{-\de t}\norm{f}_{L^1}\,.
\end{align*}
Similarly, the $L^2-L^2$ estimate:
\begin{align*}
\normBig{D^r_tD_x^\al\Big(\int_{\Om}& e^{ix\cdot\xi}
\Big(\sum_{k=1}^L e^{i\tau_k(\xi)t}A_j^k(t,\xi)\Big)
\cutoffM(\xi)\widehat{f}(\xi)\,dx\Big)}_{L^2(\R^n_x)}\\
&=\normBig{\Big(\sum_{k=1}^L
e^{i\tau_k(\xi)t}A_j^k(t,\xi)\tau_k(\xi)^r\Big)
\xi^\al\cutoffM(\xi)\widehat{f}(\xi)}_{L^2(\Om)} \\
&\le C(1+t)^{L-1}e^{-\de t}\norm{\widehat{f}}_{L^2(\Om)}\le
C(1+t)^{L-1}e^{-\de t}\norm{f}_{L^2}\,.
\end{align*}
Then, Theorem~\ref{thm:maininterpolationthm} yields
\begin{multline*}
\normBig{D^r_tD_x^\al\Big(\int_{\Om} e^{ix\cdot\xi}
\Big(\sum_{k=1}^L e^{i\tau_k(\xi)t}A_j^k(t,\xi)\Big)
\cutoffM(\xi)\widehat{f}(\xi)\,dx\Big)}_{L^q(\R^n_x)}\\ \le
C(1+t)^{L-1}e^{-\de t}\norm{f}_{L^p}\,,
\end{multline*}
where $\frac1p+\frac1q=1$, $1\le p\le2$. Once again, we have
exponential decay. This, together
with~\eqref{EQ:estforrootawayfromaxisbddxi} gives the statement when
there are multiplicities away from the axis and completes
the proof of 
Theorem~\ref{THM:expdecay2}.

\subsection{Phase meeting the real axis:
Theorem \ref{TH:multonaxis}}\label{SEC:phasefnmeetsfinorder}

We next look at the case where the characteristic roots
$\tau_1(\xi),\dots,\tau_L(\xi)$ that coincide in the 
$C^1$ set~$\curlyM$
of codimension~$\ell$ meet the real axis in~$\curlyM$ with finite
orders.  If there are more points
in~$\curlyM$ at which the above roots meet the axis with finite
order (or even with infinite order/lying on the axis), they may be
considered separately in the same way (or using the method below
where necessary), while away from such points, the roots are
separated from the axis, and the previous arguments 
and results of Section \ref{SEC:21} may be used.

Since the characteristic roots are not necessarily analytic (or
even differentiable) in~$\curlyM$, we must look at each branch of
the roots as they approach the real axis; set~$s_k$ to be the
maximal order of the contact with the real axis for $\tau_k(\xi)$,
that is, the maximal value for which there exist constant
$c_0>0$ such that
\begin{equation*}
c_{0}\dist (\xi, Z_k)^{s_k}\le{\Im\tau_k(\xi)}\,,
\end{equation*}
for all~$\xi$ sufficiently near~$Z_k$,
where 
$Z_k=\set{\xi\in\Rn: \Im\tau_k(\xi)=0}$.
By assumptions of Theorem \ref{TH:multonaxis}, we have
the estimate
\begin{equation*}
c_{0}\dist (\xi, \curlyM)^s\le{\Im\tau_k(\xi)}\,,
\end{equation*}
for some $c_0>0$ and $s\geq\max(s_1,\dots,s_L)$, for
$\xi$ close to $\curlyM$.
We will need the following extension of
Proposition~\ref{PROP:generalrootsmeetingaxis}. Its proof is
similar to the proof of Proposition~\ref{PROP:generalrootsmeetingaxis}
if we consider the $C^1$ coordinate system associated to $\curlyM$.
As usual $\curlyM^\epsilon=\{\xi\in\Rn: \dist(\xi,\curlyM)<
\epsilon\}.$
\begin{prop}\label{PROP:generalrootsmeetingaxis1}
Let $U\subset\R^n$ be open and let
$\phi:U\to\R$ be a
continuous function. Suppose $\curlyM\subset U$ is a
$C^1$ set of codimension $\ell$ such that
\begin{equation*}
c_{0}\dist (\xi, \curlyM)^s\le{\phi(\xi)}\,,
\end{equation*}
for some $c_0>0$, and all $\xi\in\curlyM^\epsilon$ 
for sufficiently small $\epsilon>0$. 
Then\textup{,} for any function~$a(\xi)$ that is bounded and
compactly supported in~$U$\textup{,} and for all $t\geq 0$, $f\in
C_0^\infty(\R^n)$\textup{,} and $r\in\R$\textup{,} we have
\begin{gather*}
\int_{\curlyM^\epsilon}
e^{-\phi(\xi)t}\dist (\xi, \curlyM)^r\abs{a(\xi)}
\abs{\widehat{f}(\xi)}\,d\xi \le
C\bract{t}^{-(\ell+r)/s}\norm{f}_{L^{1}}\,,\label{EQ:matsL1est1}\\
\intertext{and} \normbig{e^{-\phi(\xi)t}\dist (\xi, \curlyM)^r 
a(\xi)
\widehat{f}(\xi)}_{L^2(\curlyM^\epsilon)} \le
C\bract{t}^{-r/s}\norm{f}_{L^2}\,.\label{EQ:matsL2est1}
\end{gather*}
\end{prop}
The proof of this proposition is similar to the proof
of Proposition~\ref{PROP:generalrootsmeetingaxis} and is
omitted. Theorem \ref{TH:multonaxis} states that we must have
the estimate \eqref{EQ:est-mult}, which is
\begin{multline*}
\normBig{D^r_tD_x^\al\Big(\int_{\curlyM^\ep} e^{ix\cdot\xi}
\Big(\sum_{k=1}^L e^{i\tau_k(\xi)t}A_j^k(t,\xi)\Big)
\cutoffM(\xi)\widehat{f}(\xi)\,d\xi\Big)}_{L^q(\R^n_x)}\\ \le
C\bract{t}^{-\frac{\ell}{s}\big(\frac{1}{p}-\frac{1}{q}\big)
+L-1}\norm{f}_{L^p}\,.
\end{multline*}

By Lemma~\ref{LEM:Lrootsmeetingest} and
Proposition~\ref{PROP:generalrootsmeetingaxis1}, 
to estimate the sum in the
amplitude, for all $t\geq 0$, we have
\begin{align*}
\normBig{D^r_tD_x^\al&\Big(\int_{\curlyM^\ep} e^{ix\cdot\xi}
\Big(\sum_{k=1}^L e^{i\tau_k(\xi)t}A_j^k(t,\xi)\Big)
\cutoffM(\xi)\widehat{f}(\xi)\,d\xi\Big)}_{L^\infty(\R^n_x)}\\
&\leq C\normBig{\int_{\curlyM^\ep}e^{ix\cdot\xi} \Big(\sum_{k=1}^L
e^{i\tau_k(\xi)t}A_j^k(t,\xi)\tau_k(\xi)^r\Big)
\xi^\al\cutoffM(\xi)\widehat{f}(\xi)\,d\xi}_{L^\infty(\R^n_x)}\\
&\le C\int_{\curlyM^\ep}
(1+t)^{L-1}e^{-t\min_{k=1,\dots,L}\Im\tau_k(\xi)}\abs{\cutoffM(\xi)}
\abs{\widehat{f}(\xi)}\,d\xi\\
&\le C\bract{t}^{L-1-(\ell/s)}\norm{f}_{L^1}\,.
\end{align*}
Also, using the Plancherel's theorem, we have
\begin{align*}
\normBig{D^r_tD_x^\al&\Big(\int_{\curlyM^\ep} e^{ix\cdot\xi}
\Big(\sum_{k=1}^L e^{i\tau_k(\xi)t}A_j^k(t,\xi)\Big)
\cutoffM(\xi)\widehat{f}(\xi)\,d\xi\Big)}_{L^2(\R^n_x)}\\
&=\normBig{\int_{\curlyM^\ep}e^{ix\cdot\xi} \Big(\sum_{k=1}^L
e^{i\tau_k(\xi)t}A_j^k(t,\xi)\tau_k(\xi)^r\Big)
\xi^\al\cutoffM(\xi)\widehat{f}(\xi)\,d\xi}_{L^2(\R^n_x)}\\
&=\normBig{\Big(\sum_{k=1}^L
e^{i\tau_k(\xi)t}A_j^k(t,\xi)\tau_k(\xi)^r\Big)
\xi^\al\cutoffM(\xi)\widehat{f}(\xi)}_{L^2(\curlyM^\ep)}\\
&\le C(1+t)^{L-1} \normbig{e^{-t\min_{k=1,\dots,L}\Im\tau_k(\xi)}
\abs{\cutoffM(\xi)}\abs{\widehat{f}(\xi)}}_{L^2(\curlyM^\ep)}\\
&\le C\bract{t}^{L-1}\norm{f}_{L^2}\,.
\end{align*}
Therefore, interpolation Theorem~\ref{thm:maininterpolationthm}
yields, for all $t\geq 0$,
\begin{multline*}
\normBig{D^r_tD_x^\al\Big(\int_{\curlyM^\ep} e^{ix\cdot\xi}
\Big(\sum_{k=1}^L e^{i\tau_k(\xi)t}A_j^k(t,\xi)\Big)
\cutoffM(\xi)\widehat{f}(\xi)\,d\xi\Big)}_{L^q(\R^n_x)}\\ \le
C\bract{t}^{-\frac{\ell}{s}\big(\frac{1}{p}-\frac{1}{q}\big)
+L-1}\norm{f}_{L^p}\,,
\end{multline*}
where $\frac1p+\frac1q=1$, $1\le p\le2$; this, together
with~\eqref{EQ:estforrootmeetingaxis} proves
Theorem~\ref{TH:multonaxis} for
roots meeting the axis with finite order.

\subsection{Phase function on the real axis for
bounded
%~$\abs{\xi}$
frequencies
}\label{SEC:phasefnliesonaxisbddxi}

Recall that in the division of the integral in
Section~\ref{SEC:step2}, we have
\begin{equation*}
\int_{B_{2M}(0)}e^{ix\cdot\xi} \Big(\sum_{k=1}^m
e^{i\tau_k(\xi)t}A_j^k(t,\xi)\Big) \widehat{f}(\xi)\,d\xi \,,
\end{equation*}
which we then subdivide around and away from multiplicities. The
cases where the root or roots are either separated from the real
axis or meet it with finite order have already been discussed;
here we shall complete the analysis by proving estimates for the
situation where a root or roots lie on the real axis.
These results can be also applied to the
case of multiple roots.

We note that in the case of nonhomogeneous symbols this analysis
is essential since time genuinely interacts with frequencies.
Unlike in the case of homogeneous symbols in Section
\ref{SEC:homogoperators}, where one could eliminate time 
completely from estimates by rescaling, here it is present
in phases and amplitude and causes them to blow up even
for low frequencies. Thus, we must carry out a detailed 
investigation of the structure of solutions for low frequencies,
and it will be done in this section.

A number of estimates can be already obtained using our results on
multiple roots from Section \ref{SEC:resofroots}.
To have any possibility of obtaining better estimates, we must impose
additional conditions on the characteristic roots at low
frequencies---for large~$\abs{\xi}$, these properties were
obtained by using perturbation results, but naturally such results
are no longer valid for $\abs{\xi}\leq M$. Also, we can impose the
convexity condition on the roots to obtain a better result than
the general case. We will give different formulation of
possible results in this section.

Again, throughout we assume that either $\tau(\xi)\ge0$ for all
$\xi$ or $\tau(\xi)\le 0$ for all $\xi$. 
The key point is to
use a carefully chosen cut-off
function to isolate the multiplicities and then use
Theorem~\ref{THM:sugimoto/randolargument} or
Theorem~\ref{THM:noncovexargument} to estimate the integrals where
there are no multiplicities (and hence the coefficients
$A_j^k(t,\xi)$ are independent of~$t$) and use suitable
adjustments around the singularities. For these purposes, let us
first assume that the only multiplicity is at a point $\xi^0\in
B_{2M}(0)$ and $\tau_1(\xi^0)=\tau_2(\xi^0)$ are the only
coinciding roots, and let 
$\chi$ be a cut-off function
around $\xi^0$. Then, we must consider the sum of the first two
roots, where we have a multiplicity at~$\xi^0$,
\begin{equation}\label{EQ:twoint}
I=\int_{B_{2M}(0)}e^{ix\cdot\xi}
\Big(\sum_{k=1}^2e^{i\tau_k(\xi)t}A^k_j(t,\xi)\Big) \cutoffM(\xi)
\widehat{f}(\xi)\,d\xi\,,
\end{equation}
and terms involving the remaining roots, which are all distinct,
\begin{equation*}
II=\sum_{k=3}^m\int_{B_{2M}(0)}e^{i(x\cdot\xi+\tau_k(\xi))t}A^k_j(t,\xi)
\cutoffM(\xi) \widehat{f}(\xi)\,d\xi\,.
\end{equation*}

\subsubsection{Case of no multiplicities: Theorem
\ref{THM:convexsp}} 
For the second of these integrals $II$, we
wish to apply Theorem~\ref{THM:sugimoto/randolargument} if
$\tau_k(\xi)$ satisfies the convexity condition, and
Theorem \ref{THM:noncovexargument} otherwise.

In order to ensure the hypotheses of these theorems are satisfied,
however, we need to impose an additional regularity 
condition on the
behaviour of the characteristic roots for the
relevant frequencies (i.e.\
$\xi\in B_{2M}(0)$) to avoid pathological situations:
\begin{equation}\label{EQ:smallxicondn}
\text{Assume $\abs{\pa_\om\tau_k(\la\om)}\ge C_0$ for all $\om\in
\Snm$, $2M\geq \la>0$.}
\end{equation}
Since this is satisfied for large~$\abs{\xi}$ 
(see Proposition \ref{prop:roots-r}) and always
satisfied for roots of
operators with 
homogeneous symbols, it is quite a natural extra assumption.

The other hypotheses of these theorems hold:
hypothesis~\ref{HYP:realtauisasymbol} is satisfied because
$\abs{\pa_\xi^\al\tau_k(\xi)}\le C_\al$ for all $\xi$ since the
characteristic roots are smooth in~$\R^n$;
hypothesis~\ref{HYP:realtauboundedbelow} only requires information
about high frequencies; and hypotheses~\ref{HYP:reallimitofSi_la}
holds by the same argument as for large~$\abs{\xi}$, where only
\ref{LEM:tau-vabounds} of Proposition~\ref{PROP:perturbationresults}
is needed, and that holds for all~$\xi\in\R^n$. Also, the
coefficients $A_k^j(\xi)$ are smooth away from multiplicities, so
the symbolic behaviour (i.e.
decay, or bounded for small frequencies) holds.

Now $L^1-L^\infty$ and $L^2-L^2$ estimates can be found as in the
case for large~$\abs{\xi}$, and the interpolation theorem used to
give the desired results.
Thus, with condition~\eqref{EQ:smallxicondn}, 
we have proved the on
axis, no multiplicities case of Theorem~\ref{THM:convexsp}.

\subsubsection{Multiplicities: shrinking neighborhoods}
\label{sec:shrinking}
Now we can turn to the other integral given by \eqref{EQ:twoint}. 
Here we will analyse what happens in certain shrinking
neighborhoods of multiplicities. First we will assume that
only two roots intersect at an isolated point, and then we will
indicate what happens in the general situation.

To continue the analysis of an isolated point of multiplicity
as in \eqref{EQ:twoint}, we
introduce a cut-off
function $\cutoffsing\in C_0^\infty([0,\infty))$,
$0\le\cutoffsing(s)\le 1$, which is identically~$0$ for $s>1$
and~$1$ for $s<\frac{3}{4}$; then \eqref{EQ:twoint} can be rewritten 
as the sum of two integrals $I=I_1+I_2$, where
\begin{gather*}
I_1=(2\pi)^{-n}\int_{\R^n}e^{ix\cdot\xi}
\cutoffsing(t\abs{\xi-\xi^0})
\cutoffM(\xi)\sum_{k=1}^2A^k_j(t,\xi)
e^{i\tau_k(\xi)t}\widehat{f}(\xi)\,d\xi\,,\\
I_2=(2\pi)^{-n}\int_{\R^n}e^{ix\cdot\xi}
(1-\cutoffsing)(t\abs{\xi-\xi^0})\cutoffM(\xi)
\sum_{k=1}^2A^k_j(t,\xi)e^{i\tau_k(\xi)t}\widehat{f}(\xi)\,d\xi\,.
\end{gather*}
We study $L^1-L^\infty$ estimates for $I_1$ and
$L^2-L^2$ estimates for both $I_1$ and $I_2$ in this section.

\paragraph{$L^1-L^\infty$ estimates:}
For this, we use the resolution of multiplicities technique
of Section \ref{SEC:resofroots}.
By Lemma~\ref{LEM:2rootsmeetingresolution}, we have,
in particular,
\begin{equation*}
\absBig{\sum_{k=1}^2A^k_j(t,\xi)e^{i\tau_k(\xi)t}}\le C(1+t),
\end{equation*}
for $\abs{\xi-\xi^0}<t^{-1}$. Now, we may estimate the integral
using the compactness of the support of~$\cutoffsing(s)$: for
$0\le t\le1$,~$I_1$ is clearly bounded; for $t>1$, we have
\begin{align*}
\abs{I_1}&\le Ct\int_{\R^n}
\abs{\cutoffsing(t\abs{\xi-\xi^0})}\abs{\widehat{f}(\xi)}\,d\xi\\&=
Ct^{1-n}\norm{\widehat{f}}_{L^\infty}
\int_{\R^n}\cutoffsing(\abs{\eta})\,d\eta \le
C\bract{t}^{1-n}\norm{f}_{L^1}.
\end{align*}
This argument can be extended  
to the case when $L$
roots meet on a set of codimension $\ell$. In the following
proposition we will change the notation for the cut-off
function to avoid any confusion with point multiplicities
in the case above.

\begin{prop}\label{Prop:aroundmult}
Suppose that $L$ roots intersect in a set $\curlyM$ of
codimension $\ell$. Let $\curlyM^\epsilon=\{\xi\in\Rn:
\dist(\xi,\curlyM)<\epsilon\}$, and let 
$\theta\in C_0^\infty(\curlyM^\epsilon)$ for sufficiently
small $\epsilon>0$. Then we have the estimate
\begin{equation}\label{EQ:estaroundmult}
\left|\int_{\R^n}e^{ix\cdot\xi}
\theta(t\dist(\xi,\curlyM))
\sum_{k=1}^L A^k_j(t,\xi)
e^{i\tau_k(\xi)t}\widehat{f}(\xi)\,d\xi\right|\leq 
C (1+t)^{L-1-\ell}.
\end{equation}
\end{prop}
\begin{proof}
By using 
Lemma~\ref{LEM:Lrootsmeetingest} in the (bounded) neighborhood
$\curlyM^\epsilon$ of $\curlyM$, we obtain
$$
\absBig{\sum_{k=1}^Le^{i\tau_k(\xi)t} A_j^k(t,\xi)} \le
C(1+t)^{L-1}\,.
$$
The size of the support of $\theta(t\dist(\xi,\curlyM))$
can be bounded by $(1+t)^{-\ell}$, which implies
estimate \eqref{EQ:estaroundmult}.
\end{proof}

\paragraph{$L^2-L^2$ estimates:}
Let us now analyse the $L^2$-estimate. This analysis will
apply not only in a shrinking, but in a fixed neighborhood
of the set of multiplicities.
We will discuss first the case of two roots intersecting at
a point in
more detail, thus analysing mainly integral $I$ 
in \eqref{EQ:twoint}.
We can have several
versions of $L^2$-estimates dependent on conditions on 
multiplicities and on the Cauchy data that we can impose.
For example, by Lemma~\ref{LEM:2rootsmeetingresolution} 
and Plancherel's theorem we get 
\begin{equation}\label{EQ:singi1}
\norm{I}_{L^2}\le C(1+t)\norm{f}_{L^2}.
\end{equation}
On the other hand we can improve the time behaviour of
the $L^2$-estimate \eqref{EQ:singi1} if we make
additional regularity assumptions for the data. For example, we can
eliminate time from estimate \eqref{EQ:singi1}
if we work in suitable Sobolev type spaces taking the
singularity into account.
Let us rewrite
\begin{multline*}
I=(2\pi)^{-n}\int_{\R^n}e^{ix\cdot\xi}
\cutoffM(\xi)\sum_{k=1}^2A^k_j(t,\xi)
e^{i\tau_k(\xi)t}\widehat{f}(\xi)\,d\xi\, \\
=(2\pi)^{-n}\int_{\R^n}e^{ix\cdot\xi}
\cutoffM(\xi)
\left[
(\tau_1(\xi)-\tau_2(\xi))\sum_{k=1}^2A^k_j(t,\xi)
e^{i\tau_k(\xi)t}\right] 
% \\
(\tau_1(\xi)-\tau_2(\xi))^{-1}\widehat{f}(\xi)\,d\xi.
\end{multline*}
Using the representation from Lemma \ref{LEM:ordercoeff}
we see that the
expression in the square brackets is bounded. Hence by the
Plancherel's theorem we get that
\begin{equation}\label{EQ:singi2}
\norm{I}_{L^2}\le 
\norm{(\tau_1(\xi)-\tau_2(\xi))^{-1}\chi(\xi)\widehat{f}(\xi)}_{L^2}
=\norm{(\tau_1(D)-\tau_2(D))^{-1}\chi(D)f}_{L^2}.
\end{equation}
An example of this is the appearance of homogeneous Sobolev
spaces for small frequencies in the analysis of the wave
equations, or more general equations with homogeneous symbols.
For example, in the case of the wave equation we have
$\tau_1(\xi)=|\xi|$ and $\tau_2(\xi)=-|\xi|$, so that
\eqref{EQ:singi2} means that we have the low frequency
estimate for the solution of the form
$$
\norm{I}_{L^2}\le 
\norm{f}_{\dot{H}^{-1}},
$$
with the homogeneous Sobolev space $\dot{H}^{-1}.$

In the case of several roots intersecting
in a set $\curlyM$, we have similarly:

\begin{prop}\label{Prop:aroundmultl2}
Suppose that $L$ roots intersect in a set $\curlyM$.
Let $\curlyM^\epsilon=\{\xi\in\Rn:
\dist(\xi,\curlyM)<\epsilon\}$, and let 
$\theta\in C_0^\infty(\curlyM^\epsilon)$ for sufficiently
small $\epsilon>0$. 
Let $J$ denote the part of solution corresponding to these
roots microlocalised near the set $\curlyM$ of multiplicities:
$$
J(t,x)=\int_{\R^n}e^{ix\cdot\xi}
\theta(\xi)
\sum_{k=1}^L A^k_j(t,\xi)
e^{i\tau_k(\xi)t}\widehat{f}(\xi)\,d\xi.
$$
Then we have the estimate
\begin{equation}\label{EQ:estaroundmultl2}
\left|\left| J \right|\right|_
{L^2(\R^n_x)}
\leq C (1+t)^{L-1}||f||_{L^2(\R^n_x)}.
\end{equation}
Moreover, let us assume without loss of generality that 
intersecting $L$ roots are labeled by $\tau_1, \cdots, \tau_L$.
Then we also have
\begin{equation}\label{EQ:estaroundmultl22}
\left|\left| \prod_{1\leq l<k\leq L} (\tau_l(D)-\tau_k(D))^{-1}
J\right|\right|_
{L^2(\R^n_x)}
\leq C ||f||_{L^2(\R^n_x)}.
\end{equation}
\end{prop}
Estimate \eqref{EQ:estaroundmultl2} follows from
Lemma~\ref{LEM:Lrootsmeetingest} and Plancherel's theorem.
Estimate \eqref{EQ:estaroundmultl22} follows from
Plancherel's theorem and formula
\eqref{EQ:Ajkformula}.

Interpolating between Propositions 
\ref{EQ:estaroundmult} and \ref{EQ:estaroundmultl2},
we can obtain different versions of the dispersive estimate
in a region shrinking around $\curlyM$,
depending on whether we use \eqref{EQ:estaroundmultl2}
or \eqref{EQ:estaroundmultl22}.

\subsubsection{Multiplicities: fixed neighborhoods}
Here, for simplicity, we will concentrate on the case
of two roots $\tau_1$ and $\tau_2$ intersecting at an
isolated point $\xi^0$.
We will discuss both $L^1-L^\infty$ and $L^2-L^2$ estimates
under additional assumptions on the roots $\tau_1$ and $\tau_2$.

\paragraph{$L^1-L^\infty$ estimates:}
For~$I_2$ we are away from the singularity, so we can use that
\begin{equation*}
\sum_{k=1}^2A^k_j(t,\xi)e^{i\tau_k(\xi)t}=
A^1_j(\xi)e^{i\tau_1(\xi)t}+A^2_j(\xi)e^{i\tau_2(\xi)t}\,.
\end{equation*}
Now, we would like to apply
Theorem~\ref{THM:sugimoto/randolargument} 
(for the case where the
root satisfies the convexity condition)
and Theorem~\ref{THM:noncovexargument} 
(for the general case), as in the
case of simple roots; however, the proximity of the multiplicity
brings the additional cut-off function,
$(1-\cutoffsing)(t\abs{\xi-\xi^0})$, into play, and this depends
on~$t$. Therefore, the aforementioned results cannot be used
directly. However, a similar result does hold, provided we impose
some additional conditions, producing analogues of Theorems
\ref{THM:sugimoto/randolargument} and
\ref{THM:noncovexargument} in this case.
\begin{prop}\label{prop:sings}
Let $\chi\in C_0^\infty(\Rn).$
Suppose $\tau_k(\xi)$, $k=1,2$, satisfy the following assumptions
on $\supp\chi$:
\begin{enumerate}[label=\textup{(}\textup{\roman*}\textup{)}]
\item\label{EQ:conditionontaunearsing} for each
multi-index $\al$ there exists a constant $C_\al>0$ 
such that, for some $\de>0$,
\begin{equation*}
\abs{\pa_\eta^\al[(\grad_\xi\tau_k)(\xi^0+s\eta)]}\le C_\al
\brac{\eta}^{-\abs{\al}}\,,\text{ for small
}s>0\;\text{and}\;\abs{\eta}>\de\,;
\end{equation*}
\item\label{HYP:singlevelsetsnondeg} there exists a constant
$C_0>0$ such that 
$\abs{\pa_\om\tau_k(\xi^0+\la\om)}\ge C>0$
for all $\om\in \Snm$ and $\lambda>0$\textup{;} 
in particular\textup{,} each of
the level sets
\begin{equation*}
\la\Si_\la'\equiv\Si_\la=\set{\eta\in\R^n:
\tau_k(\xi^0+\eta)=\la}
\end{equation*}
is non-degenerate\textup{;}
\item\label{HYP:singlevelsetsinbddball} there exists a constant
$R_1>0$ such that\textup{,} for all $\la>0$\textup{,}
\[\Si'_\la:=\frac{1}{\la}\Si_\la(\tau_k)\subset B_{R_1}(0)\,.\]
\end{enumerate}
Furthermore, assume that $A_j^k(\xi)$ satisfies the following
condition: for each multi-index~$\al$ there exists a constant
$C_\al>0$ such that
\begin{enumerate}[resume, label=\textup{(}\textup{\roman*}\textup{)}]
\item\label{EQ:conditiononAnearsing}
we have the estimate
\begin{equation*}
\abs{\pa_\eta^\al[A_j^k(\xi^0+s\eta)]}\le C_\al
s^{-j}\brac{\eta}^{-j-\abs{\al}}\,,\text{ for small
}s>0\;\text{ and }\;\abs{\eta}>\de\,.
\end{equation*}
\end{enumerate}
Finally, assume that $\psi\in C_0^\infty((-\delta,\delta))$
is such that $\psi(\sigma)=1$ for $|\sigma|\leq \delta/2$.
Then\textup{,} the following estimate holds for all
$x\in\R^n$\textup{,} $t\geq 0$\textup{:}
\begin{equation}\label{EQ:finalsingpropest}
\absBig{\sum_{k=1}^2\int_{\R^n}e^{i(x\cdot\xi+\tau_k(\xi)t)}
A^k_j(\xi)(1-\cutoffsing)(t\abs{\xi-\xi^0})\cutoffM(\xi)\\
\,d\xi}\le C(1+t)^{j-n}\,,
\end{equation}
for $j\ge n-\frac{n-1}{\ga}$, where
$\ga:=\sup_{\la>0}\ga(\Si_\la(\tau_k))$, if $\tau_k(\xi)$
satisfies the convexity condition; and for $j\ge
n-\frac{1}{\ga_0}$, where
$\ga_0:=\sup_{\la>0}\ga_0(\Si_\la(\tau_k))$, if it does not.
\end{prop}
\begin{rem}
Conditions~\ref{EQ:conditionontaunearsing},~\ref{HYP:singlevelsetsnondeg}
and \ref{EQ:conditiononAnearsing} appear and are
satisfied naturally
when roots $\tau_k(\xi)$ are homogeneous functions of order one---for
example, the wave equation, or for homogeneous equations.
\end{rem}
\begin{rem}
Assumption~\ref{EQ:conditiononAnearsing} is needed because
$A_j^k(\xi)$ has a singularity at~$\xi^0$, 
so we must ensure we
are away from that---this is the role of the cut-off function
$(1-\cutoffsing)(\abs{\eta})$ in this proposition; 
\end{rem}
\begin{rem}
As usual, for example in the convex case, 
taking $j=n-\frac{n-1}{\gamma}$, we get the time
decay estimate
$$
\left| \text{ Left hand side of \eqref{EQ:finalsingpropest} }
\right| \leq C (1+t)^{-\frac{n-1}{\gamma}}.
$$
\end{rem}
\begin{proof}
As before, cut-off near the wave front: let $\cutoffWF\in
C^\infty_0(\R^n)$ be a cut-off function supported in $B(0,r)$.
Then, consider
%\begin{multline*}
$$
I_{1}(t,x):=\sum_{k=1}^2\int_{\R^n}e^{i(x\cdot\xi+\tau_k(\xi)t)}
A^k_j(\xi)(1-\cutoffsing)(t\abs{\xi-\xi^0})\cutoffM(\xi)\\
\cutoffWF \big(t^{-1}x+\grad\tau_k(\xi)\big)\,d\xi,
$$
%\end{multline*}
and
\begin{multline*}
I_{2}(t,x):=\sum_{k=1}^2\int_{\R^n}e^{i(x\cdot\xi+\tau_k(\xi)t)}
A^k_j(\xi)(1-\cutoffsing)(t\abs{\xi-\xi^0})\cutoffM(\xi)\\
(1-\cutoffWF) \big(t^{-1}x+\grad\tau_k(\xi)\big)\,d\xi.
\end{multline*}

\paragraph{Away from the wave front set:}
First, we estimate~$I_2(t,x)$; we claim that
\begin{equation}\label{EQ:singestforwavefront}
\abs{I_{2}(t,x)}\le C_r(1+t)^{j-n}\,\text{ for all }t>0\,,x\in\R^n\,.
\end{equation}
In order to show this, we consider each term of the sum
separately,
\begin{equation*}
I_2^k(t,x)=\int_{\R^n}e^{i(x\cdot\xi+\tau_k(\xi)t)}
A^k_j(\xi)(1-\cutoffsing)(t\abs{\xi-\xi^0})\cutoffM(\xi)
(1-\cutoffWF)
\big(\textstyle\frac{x}{t}+\grad\tau_k(\xi)\big)\,d\xi\,,
\end{equation*}
and imitate the proof of Lemma~\ref{LEM:estawayfromwavefront} (in
which the corresponding term was estimated in
Theorem~\ref{THM:sugimoto/randolargument}), but noting that in
place of $g_R(\xi)\in C_0^\infty(\R^n)$ we have
$(1-\cutoffsing)(t(\xi-\xi^0))$, which depends also on~$t$; in
particular, this means that care must be taken when carrying out
the integration by parts when derivatives fall on
$(1-\cutoffsing)(t\abs{\xi-\xi^0})$. To take this into account,
use the change of variables $\xi=\xi^0+t^{-1}\eta$:
\begin{multline*}
I_{2}^k(t,x)=e^{ix\cdot\xi^0}\int_{\R^n}e^{i(t^{-1}x\cdot \eta
+\tau_k(\xi^0+t^{-1}\eta)t)}
A^k_j(\xi^0+t^{-1}\eta)(1-\cutoffsing)(\abs{\eta})
\\\cutoffM(\xi^0+t^{-1}\eta)(1-\cutoffWF)
\big(t^{-1}x+(\grad_\xi\tau_k)(\xi^0+t^{-1}\eta)\big)t^{-n}\,d\eta.
\end{multline*}
Integrating by parts, with respect to~$\eta$ gives
\begin{multline*}
I_{2}^k(t,x)=e^{ix\cdot\xi^0}t^{-n}\int_{\R^n}e^{i(t^{-1}x\cdot
\eta +\tau_k(\xi^0+t^{-1}\eta)t)}P^*\big[A^k_j(\xi^0+t^{-1}\eta)
(1-\cutoffsing)(\abs{\eta})\\
\cutoffM(\xi^0+t^{-1}\eta)
(1-\cutoffWF)\big(t^{-1}x+(\grad_\xi\tau_k)(\xi^0+t^{-1}\eta)\big)\big]
\,d\eta\,,
\end{multline*}
where $P^*$ is the adjoint operator to $P=
\frac{t^{-1}x+(\grad_\xi\tau_k)(\xi^0+t^{-1}\eta)}
{i\abs{t^{-1}x+(\grad_{\xi}\tau_k)(\xi^0+t^{-1}\eta)}^2}
\cdot\grad_\eta$; this integration by parts is valid as
$\abs{t^{-1}x+(\grad_\xi\tau_k)(\xi^0+t^{-1}\eta)}\ge r>0$, 
in the
support of
$(1-\cutoffWF)\big(t^{-1}x+\grad\tau_k(\xi^0+t^{-1}\eta)\big)$.
For suitable functions~$f\equiv f(\eta;x,t)$, 
and $\xi=\xi^0+t^{-1}\eta$, we have
\begin{align*}
P^*f=&\grad_\eta\cdot \Big[\frac{t^{-1}x+(\grad_\xi\tau_k)(\xi)}
{i\abs{t^{-1}x+(\grad_\xi\tau_k)(\xi)}^2}f\Big]\\
=&\frac{\grad_\eta\cdot(\grad_\xi\tau_k)(\xi)}
{i\abs{t^{-1}x+(\grad_\xi\tau_k)(\xi)}^2}f
+\frac{t^{-1}x+(\grad_\xi\tau_k)(\xi)}
{i\abs{t^{-1}x+(\grad_\xi\tau_k)(\xi)}^2}
\cdot\grad_\eta f\\
&\qquad-\frac{2(t^{-1}x+(\grad_\xi\tau_k)(\xi))
\cdot[\grad_\eta[(\grad_\xi\tau_k)(\xi)]
\cdot(t^{-1}x+(\grad_\xi\tau_k)(\xi) )]}
{i\abs{t^{-1}x+(\grad_\xi\tau_k)(\xi)}^4}f.
\end{align*}
Comparing this to~\eqref{EQ:expressionforP*f}, observe that the
first and third terms have one power of~$t$ fewer in the
denominator due to the transformation; this is critical in this
case where we are approaching a singularity in
$A_j^k(\xi^0+t^{-1}\eta)$ when $t\to\infty$. By
hypothesis~\ref{EQ:conditionontaunearsing}, for $\eta$ in the
support of the integrand of $I_{2}^k(t,x)$, we get
\begin{equation*}
\frac{\grad_\eta\cdotp[(\grad_\xi\tau_k)(\xi^0+t^{-1}\eta)]}
{\abs{t^{-1}x+(\grad_{\xi}\tau_k)(\xi^0+t^{-1}\eta)}^2}\le
C_r\brac{\eta}^{-1}\,;
\end{equation*}
thus, we have
\begin{equation*}
\abs{P^*f}\le C_r[\brac{\eta}^{-1}\abs{f} +\abs{\grad_\eta f}]\,.
\end{equation*}

In Lemma~\ref{LEM:estawayfromwavefront}, we carried out this
integration by parts repeatedly in order to estimate the integral.
Here, however, note that differentiating
$(1-\cutoffsing)(\abs{\eta})$ once is sufficient: by definition
of~$\cutoffsing(s)$,
\begin{equation*}
\pa_{\eta_j}[(1-\cutoffsing)(\abs{\eta})]
=-\frac{\eta_j}{\abs{\eta}} (\pa_s\cutoffsing)(\abs{\eta})
\end{equation*}
is supported in $\frac{3}{4}\le \abs{\eta}\le 1$, so
\begin{equation*}
\abs{\pa_{\eta_j}[(1-\cutoffsing)(\abs{\eta})]}\le C
\indicat_{1\ge\abs{\eta}\ge3/4}(\eta) \,,
\end{equation*}
where $\indicat_{1\ge\abs{\eta}\ge3/4}(\eta)$ denotes the
characteristic function of $\set{\eta\in\R^n:1\ge\abs{\eta}\ge3/4}$;
hence, by hypothesis~\ref{EQ:conditiononAnearsing},
for large $t$ we have
\begin{align}
\begin{split}
&\int_{\R^n}\absBig{\frac{t^{-1}x+(\grad_\xi\tau_k)(\xi^0+t^{-1}\eta)}
{i\abs{t^{-1}x+(\grad_{\xi}\tau_k)(\xi^0+t^{-1}\eta)}^2}}
\abs{A^k_j(\xi^0+t^{-1}\eta)} \abs{\pa_{\eta_j}
[(1-\cutoffsing)(\abs{\eta})]}\\
&\mspace{150mu}\abs{\cutoffM(\xi^0+t^{-1}\eta)}\abs{(1-\cutoffWF)
\big(t^{-1}x+\grad\tau_k(\xi^0+t^{-1}\eta)\big)}t^{-n}\,d\eta
\end{split}\notag\\
&\mspace{120mu}\le C_r\int_{\frac{3}{4}\le \abs{\eta}\le
1}\abs{A^k_j(\xi^0+t^{-1}\eta)}t^{-n} \,d\eta\notag\\
& \mspace{120mu}\le C_rt^{j}\int_{\frac{3}{4}\le \abs{\eta}\le
1}\frac{1}{\brac{\eta}^j}t^{-n}\,d\eta\le
C_rt^{j-n}\,,\label{EQ:derivonpsi}
\end{align}
which is the desired estimate~\eqref{EQ:singestforwavefront}.

On the other hand, if, when integrating by parts, the derivative
does not fall on~$\cutoffsing(\abs{\eta})$, we use a similar
argument to that in the earlier proof; let us look at the effect
of differentiating each of the other terms: in the support of
$\cutoffsing(\abs{\eta})$, for each multi-index~$\al$ and $t>0$,
\begin{itemize}
\item $\abs{\pa_\eta^\al[A_j^k(\xi^0+t^{-1}\eta)]}\le C_\al
t^{j}\brac{\eta}^{-j-\abs{\al}}$ by
hypothesis~\ref{EQ:conditiononAnearsing};
\item $\abs{\pa_\eta^\al[\cutoffM(\xi^0+t^{-1}\eta)]}\le
C_\al\brac{\eta}^{-\abs{\al}}$: for $\al=0$, take $C_\al=1$; for
$\abs{\al}\ge1$, note that
\begin{equation*}
\pa_\eta^\al[\cutoffM(\xi^0+t^{-1}\eta)]=
t^{-\abs{\al}}(\pa^\al_\xi\cutoffM)(\xi^0+t^{-1}\eta)\,,
\end{equation*}
and that $(\pa^\al_\xi\cutoffM)(\xi^0+t^{-1}\eta)$ is supported in
$N\le\abs{\xi^0+t^{-1}\eta}\le 2N$, so $t^{-1}\le
C_{N,\xi^0}\abs{\eta}^{-1}$;
\item $\abs{\pa_\eta^\al[(1-\cutoffWF)\big(t^{-1}x+
(\grad_\xi\tau_k)(\xi^0+t^{-1}\eta)\big)]}\le
C_\al\brac{\eta}^{-\abs{\al}}$: obvious for $\al=0$; for
$\abs{\al}\ge1$, note
\begin{multline*}
\pa_\eta^\al[(1-\cutoffWF)(t^{-1}x+
(\grad_\xi\tau_k)(\xi^0+t^{-1}\eta))]\\
=-(\pa_\xi^\al\cutoffWF)(t^{-1}x+\grad_\xi\tau_k(\xi))
\pa_\eta^\al[(\grad_\xi\tau_k)(\xi^0+t^{-1}\eta)]\,,
\end{multline*}
which yields the desired estimate by
hypothesis~\ref{EQ:conditionontaunearsing}.
\end{itemize}
Summarising, this means
\begin{multline*}
\absbig{(1-\cutoffsing)(\abs{\eta})
\pa^\al_\eta\big[A^k_j(\xi^0+t^{-1}\eta)\cutoffM(\xi^0+t^{-1}\eta)
(1-\cutoffWF)
\big(t^{-1}x+(\grad_\xi\tau_k)(\xi^0+t^{-1}\eta)\big)\big]}
\\
\le C_r \brac{\eta}^{-j-\abs{\al}}
t^{j}\indicat_{\abs{\eta}\ge\frac{3}{4}}(\eta)\,.
\end{multline*}
So, repeatedly integrating by parts we find that either a
derivative falls on $(1-\cutoffsing)(\abs{\eta})$ 
(in which case a
similar argument to that in~\eqref{EQ:derivonpsi} 
above works) or
we eventually get the integrable function
$Ct^j \brac{\eta}^{-n-1}\indicat_{\abs{\eta}\ge3/4}(\eta)$ 
as an upper
bound; in either case, we have~\eqref{EQ:singestforwavefront}.

\paragraph{On the wave front set:}
Next, we look at the term supported around the wave front set,
$I_1(t,x)$. As in the case away from the wave front, set
$\xi=\xi^0+t^{-1}\eta$: consider, for $k=1,2$,
\begin{multline*}
I_{1}^k(t,x):=e^{ix\cdot\xi^0}\int_{\R^n}e^{i(t^{-1}x\cdot \eta
+\tau_k(\xi^0+t^{-1}\eta)t)}
A^k_j(\xi^0+t^{-1}\eta)(1-\cutoffsing)(\abs{\eta})
\\\cutoffM(\xi^0+t^{-1}\eta)\cutoffWF
\big(t^{-1}x+(\grad_\xi\tau_k)(\xi^0+t^{-1}\eta)\big)t^{-n}\,d\eta\,.
\end{multline*}
As in the proof of Theorems~\ref{THM:sugimoto/randolargument}
and~\ref{THM:noncovexargument}, let
$\set{\cutoffcone_\ell(\eta)}_{\ell=1}^L$ 
be a conic partition of
unity, where the support of~$\cutoffcone_\ell(\eta)$ is a
cone~$K_\ell$, and each cone can be mapped by rotation onto~$K_1$,
which contains $e_n=(0,\dots,0,1)$. Then, it suffices to estimate
\begin{multline*}
t^{-n}\int_{\R^n}e^{i(t^{-1}x\cdot \eta
+\tau_k(\xi^0+t^{-1}\eta)t)}
A^k_j(\xi^0+t^{-1}\eta)(1-\cutoffsing)(\abs{\eta})
\\\cutoffcone_1(\eta)\cutoffM(\xi^0+t^{-1}\eta)\cutoffWF
\big(t^{-1}x+(\grad_\xi\tau_k)(\xi^0+t^{-1}\eta)\big)\,d\eta\,,
\end{multline*}
for $k=1,2$.

Let us parameterise the cone $K_1$: it follows from 
hypothesis~\ref{HYP:singlevelsetsnondeg} that 
each of the level sets
\begin{equation*}
\Sigma_{\la,t} \equiv
\set{\eta\in\R^n:\tau_k(\xi^0+t^{-1}\eta)=t^{-1}\la}
\end{equation*}
is non-degenerate; so, for some $U\subset\R^{n-1}$, and smooth
function $h_k(t,\la,\cdot):U\to\R$,
\begin{equation*}
K_1=\set{(\la y,\la h_k(t,\la,y)):\la>0,y\in U}\,.
\end{equation*}
If~$\tau_k(\xi)$ satisfies the convexity condition, 
then~$h_k$ is
also a concave function in~$y$. Now, we change variables
$\eta\mapsto(\la y,\la h_k(t,\la,y))$ and will 
often omit $t$ from
the notation of $h_k$ since the dependence on $t$ will be
uniform. We obtain:
\begin{multline}\label{EQ:basicsingcaseintegralaftertrans}
t^{-n}\int_{0}^\infty\int_U e^{i\la (t^{-1}x'\cdot y+
t^{-1}x_nh_k(\la,y)+1)} A^k_j(\xi^0+t^{-1}\la(y,
h_k(\la,y)))\\
(1-\cutoffsing)(\la\abs{(y,h_k(\la,y))})
\cutoffcone_1(\la(y,h_k(\la,y)))\cutoffM(\xi^0+t^{-1}\la(y,
h_k(\la,y)))\\\cutoffWF
\big(t^{-1}x+(\grad_\xi\tau_k)(\xi^0+t^{-1}\la(y,h_k(\la,y)))\big)
\frac{d\eta}{d(\la,y)} \,d\la dy,
\end{multline}
where we have used $\tau_k(\xi^0+t^{-1}(\la y,\la
h_k(\la,y)))=t^{-1}\la$. As in the earlier proofs, 
we ensure~$x_n$ is
away from zero in the cone---this requires
hypotheses~\ref{EQ:conditionontaunearsing}
and~\ref{HYP:singlevelsetsinbddball}). 
So, in the general case, we
can write this as, with $\xtil=t^{-1}x$, $\latil=\la\xtil_n=\la
t^{-1}x_n$,
\begin{multline*}
t^{-n}\int_{0}^\infty\int_U e^{i\la x_n(t^{-1}x_n^{-1} 
x'\cdot y+
t^{-1}h_k(\la,y)+\xtil_n^{-1})} A^k_j(\xi^0+t^{-1}\la(y,
h_k(\la,y)))\\
(1-\cutoffsing)(\la\abs{(y,h_k(\la,y))})
\cutoffcone_1(\la(y,h_k(\la,y)))\cutoffM(\xi^0+t^{-1}\la(y,
h_k(\la,y)))\\\cutoffWF
\big(t^{-1}x+(\grad_\xi\tau_k)(\xi^0+t^{-1}\la(y,h_k(\la,y)))\big)
\frac{d\eta}{d(\la,y)} \,d\la dy\,.
\end{multline*}

If the convexity condition holds, then, as in the proof of
Theorem~\ref{THM:sugimoto/randolargument}, we have the Gauss map
\begin{equation*}
\gauss_k:K_1\cap \Si_\la'\to S^{n-1},\; \gauss_k(\zeta)=
\frac{\grad_\zeta[\tau_k(\xi^0+t^{-1}\zeta)]}
{\abs{\grad_\zeta[\tau_k(\xi^0+t^{-1}\zeta)]}}=
\frac{(\grad_\xi\tau_k)(\xi^0+t^{-1}\zeta)}
{\abs{(\grad_\xi\tau_k)(\xi^0+t^{-1}\zeta)}}\,,
\end{equation*}
and, as before, can define $z_k(\la)\in U$ so that
\begin{equation*}
\gauss_k(z_k(\la),h_k(\la,z(\la)))=-x/\abs{x}\,.
\end{equation*}
Then,
\begin{equation*}
\frac{x'}{x_n}=-\grad_y h_k(\la,z(\la))\,.
\end{equation*}
So, in this case,~\eqref{EQ:basicsingcaseintegralaftertrans}
becomes:
\begin{multline*}
(I_1^k)':=t^{-n}\int_{0}^\infty\int_U e^{i\la x_n[-t^{-1}\grad_y
h_k(\la,z(\la))\cdot y+t^{-1}h_k(\la,y)+\xtil_n^{-1}]}\\
A^k_j(\xi^0+t^{-1}\la(y, h_k(\la,y)))
(1-\cutoffsing)(\la\abs{(y,h_k(\la,y))})
\cutoffcone_1(\la(y,h_k(\la,y)))\\
\cutoffM(\xi^0+t^{-1}\la(y,h_k(\la,y))) \cutoffWF
\big(\xtil+(\grad_\xi\tau_k)(\xi^0+t^{-1}\la(y,h_k(\la,y)))\big)
\frac{d\eta}{d(\la,y)}\,d\la dy,
\end{multline*}
Let us estimate this integral in the case where the convexity
condition holds. We have:
\begin{itemize}[leftmargin=*]
\item The same argument as in the earlier proof (which
uses hypothesis~\ref{HYP:singlevelsetsnondeg}), shows
\begin{equation*}
\absBig{\frac{d\eta}{d(\la,y)}}\le C\la^{n-1}\,.
\end{equation*}
The constant $C$ here is independent of $t$;
\item Now, with $\widetilde{A}_k^j(\nu)=A_k^j(\nu)
\cutoffM(\nu)\cutoffWF \big(\xtil+(\grad_\xi\tau_k)(\nu)\big)
\cutoffcone_1(\la(y,h_k(\la,y)))$, where $\nu=\xi^0+t^{-1}\la(y,
h_k(\la,y))$, we have
\begin{multline*}
\abs{(I_1^k)'}\le t^{j-n}\int_{0}^\infty\Big|\int_U e^{i\la
\xtil_n[-(y-z(\la))\cdot\grad_y h_k(\la,z(\la))
+h_k(\la,y)+h_k(\la,z(\la))]}\\
t^{-j}\la^j\widetilde{A}^k_j(\xi^0+t^{-1}\la(y, h_k(\la,y)))
(1-\cutoffsing)(\la\abs{(y,h_k(\la,y))})\,dy\Big|
\la^{n-1-j}\,d\la\,.
\end{multline*}
\item Now, applying Theorem~\ref{THM:oscintthm}---
this may be used
due to the properties of $A_j^k(\xi)$ and $
\tau_k(\xi)$ stated in
hypotheses~\ref{EQ:conditiononAnearsing}
and~\ref{EQ:conditionontaunearsing}---we find that
\begin{multline*}
\Big|\int_U e^{i\la\xtil_n[-(y-z(\la))\cdot\grad_y h_k(\la,z(\la))
+h_k(\la,y)+h_k(\la,z(\la))]}\\
t^{-j}\la^j\widetilde{A}^k_j(\xi^0+t^{-1}\la(y, h_k(\la,y)))
(1-\cutoffsing)(\la\abs{(y,h_k(\la,y))})\,dy\Big|\le
C\la^{j-n}\widetilde{\chi}(\la)\,,
\end{multline*}
where $\widetilde\chi(\la)$ is a compactly 
supported smooth function that is
zero in a neighbourhood of the origin.
\item Hence,
\begin{equation*}
\abs{(I_1^k)'}\le
t^{j-n}\int_0^\infty\widetilde\chi(\la)\la^{-1}\,d\la\le Ct^{j-n}\,.
\end{equation*}
\end{itemize}
Finally, the general case without convexity
can be estimated in a similar way, 
with the necessary changes used in the proof of
Theorem~\ref{THM:noncovexargument} 
to account for the change in
the phase function---in particular, the use of the Van der Corput
Lemma, Lemma~\ref{LEM:VDCmain}, in place of
Theorem~\ref{THM:oscintthm}. This completes the proof of
\eqref{EQ:finalsingpropest}.
\end{proof}

Using Proposition \ref{prop:sings}, it is clear that
\begin{multline*}
\normBig{\int_{\R^n}e^{ix\cdot\xi}
(1-\cutoffsing)(t\abs{\xi-\xi^0})\cutoffM(\xi)
\sum_{k=1}^2A^k_j(t,\xi)e^{i\tau_k(\xi)t}\widehat{f}(\xi)\,d\xi}_
{L^\infty(\R^n_x)}\\
\le C
\bract{t}^{-\frac{n-1}{\ga}}\norm{f}_{L^1}
\end{multline*}
if the roots satisfy the convexity condition, and
\begin{multline*}
\normBig{\int_{\R^n}e^{ix\cdot\xi}
(1-\cutoffsing)(t\abs{\xi-\xi^0})\cutoffM(\xi)
\sum_{k=1}^2A^k_j(t,\xi)e^{i\tau_k(\xi)t}\widehat{f}(\xi)\,d\xi}_
{L^\infty(\R^n_x)}\\
\le C
\bract{t}^{-\frac{1}{\ga_0}}\norm{f}_{L^1}
\end{multline*}
otherwise.  In comparison to \eqref{EQ:convestL1Linftylarget},
here we have $L^1$-norms on the right hand sides, since
$\chi$ is a cut-off function to bounded frequencies.

Finally, we must consider the case where~$L$ 
roots intersect; 
the above proof can easily be adapted for
such a case, giving corresponding results. 

\paragraph{$L^2-L^2$ estimates:} For the $L^2$-estimates on
the support of $(1-\cutoffsing)(t\abs{\xi-\xi^0})\cutoffM(\xi)$
we only need assumption \ref{EQ:conditiononAnearsing}
of Proposition \ref{prop:sings} with
$\alpha=0$ for the amplitude, namely that
\begin{equation}\label{EQ:ass4}
\abs{A_j^k(\xi^0+s\eta)}\le C_\al
s^{-j}\brac{\eta}^{-j}\,,\text{ for small
}s>0\;\text{ and }\;\abs{\eta}>\de\,.
\end{equation}
Then, for the left hand side of \eqref{EQ:finalsingpropest},
we have 
\begin{align*}
& \left|\left|\sum_{k=1}^2\int_{\R^n}e^{i(x\cdot\xi+\tau_k(\xi)t)}
A^k_j(\xi)(1-\cutoffsing)(t\abs{\xi-\xi^0})\cutoffM(\xi)
\widehat{f}(\xi) \,d\xi\right|\right|_{L^2(\R^n_x)} \\
 & = 
\left|\left|\sum_{k=1}^2e^{i\tau_k(\xi)t}
A^k_j(\xi)(1-\cutoffsing)(t\abs{\xi-\xi^0})\cutoffM(\xi)
\widehat{f}(\xi) \right|\right|_{L^2(\R^n_\xi)} \\
& \leq ||t^j(1+|\eta|)^{-j}\widehat{f}(\xi^0+t^{-1}\eta)||_
{L^2(\R^n_\eta)},
\end{align*}
where we used Plancherel's theorem, 
\eqref{EQ:ass4}, and the notation $s=t^{-1}$,
$\xi=\xi^0+t^{-1}\eta$, so that $\eta=t(\xi-\xi^0)$.
Then we can easily estimate

\begin{align*}
||t^j(1+|\eta|)^{-j}\widehat{f}(\xi^0+t^{-1}\eta)||_
{L^2(\R^n_\eta)} & =
||t^j(1+t|\xi-\xi^0|)^{-j}\widehat{f}(\xi)||_
{L^2(\R^n_\xi)} \\
& = ||(t^{-1}+|\xi-\xi^0|)^{-j}\widehat{f}(\xi)||_
{L^2(\R^n_\xi)} \\
& \leq ||\,|\xi-\xi^0|^{-j}\widehat{f}(\xi)||_
{L^2(\R^n_\xi)} \\
& = ||\,|D-D_0|^{-j}f||_
{L^2(\R^n_x)},
\end{align*}
where $D-D_0$ is a Fourier multiplier with symbol
$\xi-\xi^0$. So, we finally obtain the estimate
\begin{multline*}
\left|\left|\sum_{k=1}^2\int_{\R^n}e^{i(x\cdot\xi+\tau_k(\xi)t)}
A^k_j(\xi)(1-\cutoffsing)(t\abs{\xi-\xi^0})\cutoffM(\xi)
\widehat{f}(\xi) \,d\xi\right|\right|_{L^2(\R^n_x)} \\ \leq
C ||\,|D-D_0|^{-j}f||_
{L^2(\R^n_x)}.
\end{multline*}
In the case of equations with homogeneous symbols
(like for the wave equation), when roots
are homogeneous, we have $\xi^0=0$, so that the right
hand side becomes just the norm in 
the corresponding homogeneous Sobolev space.

Due to the earlier bound
near the multiplicity, we can combine the results with the
interpolation Theorem~\ref{thm:maininterpolationthm}.

\section{Examples and extensions\label{CHAP5}}

Theorem~\ref{THM:overallmainthm} gives estimates for operators
provided the characteristic roots satisfy certain hypotheses.
However, in order to test the validity of such an
 estimate for an
arbitrary linear, constant coefficient $m^{\text{th}}$ order
strictly hyperbolic operator with lower order terms, it is desirable
to find conditions on the structure of the lower order 
terms under
which certain conditions for the characteristic roots hold.
For the case $m=2$, a complete
characterisation can be given, and some extension of this
is discussed in 
Section~\ref{SEC:analysisofm=2}. However, for large~$m$, it is
difficult to do such an analysis, as no explicit
formulae are available in general; nevertheless, certain conditions
can be found that do make the task of checking the 
conditions of the
characteristic roots, and these are discussed in
Section~\ref{SEC:condonlowerorderterms}, where a method is also
given that can be used to find many examples. Finally, in
Section~\ref{SEC:FokkerPlanck}, we give a few applications of 
these results.

\subsection{Wave equation with mass and
dissipation}\label{SEC:analysisofm=2} 

As an example of how to use Theorem \ref{THM:overallmainthm},
here we will show that we can still have time decay of
solutions if we allow the negative mass but exclude certain
low frequencies for Cauchy data. This is given in 
\eqref{EQ:wave} below. In the case of the negative mass
and positive dissipation, there is an interplay between them with
frequencies that we are going to exhibit. The
usual non-negative and also 
time dependent mass and
dissipation with oscillations have been considered before,
even with oscillations.
See, for example, \cite{HR03} and references therein.
%In the critical case we
%get the decay rate of $L^p-L^q$ norm of
%$(1+t)^{-\frac12\left(\frac1p-\frac1q\right)}.$

Let us consider second order equations of the following 
form 
\begin{equation*}
\left\{
\begin{aligned} \pa_t^2u-c^2\lap u+\diss \pa_tu+\mass u=0\,,\\
u(0,x)=0,\;u_t(0,x)=g(x)\,.
\end{aligned}
\right.
\end{equation*}
Here $\diss$ is the dissipation and $\mass$ is the mass.
For simplicity, the first Cauchy data is taken to be zero.
The general case when both Cauchy data are
non-zero can be treated in a similar way.
Let us now apply Theorem \ref{THM:overallmainthm} to the
analysis of this equation.
The associated characteristic polynomial is
\begin{equation*}
\tau^2-c^2\abs{\xi}^2-i\diss\tau-\mass=0\,,
\end{equation*}
and it has roots
\begin{equation*}
\tau_\pm(\xi)=\frac{i\diss}{2}\pm\sqrt{c^2\abs{\xi}^2+\mass-\diss^2/4}\,.
\end{equation*}
Now, we have the following well-known 
cases, which also correspond to different
cases of Theorem \ref{THM:overallmainthm}:
\begin{itemize}
\item $\diss=\mass=0$. This is the wave equation.
\item $\diss=0$, $\mass>0$. This is the Klein--Gordon equation.
\item $\mass=0$, $\diss>0$. This is the dissipative wave equation.
\item $\diss<0$. In this case, 
$\Im\tau_-(\xi)\le\frac{\diss}{2}<0$ for
all $\xi$, hence we cannot expect any decay in general.
\item $\diss>0$, $\mass>0$. 
In this case the discriminant is always strictly
greater than $-\diss^2/4$, and thus the roots always 
lie in the 
upper half plane and are separated from the real axis.
So we have exponential decay.

Here is the main case for us, where we can show an interesting
interplay between negative mass $\mu<0$ and how it is compensated
by positive dissipation $\delta>0$ for different frequencies:
\item dissipation $\diss\ge0$, mass $\mass<0$. 
In this case, note that
$\Im\tau_-(\xi)\ge0$ if and only if $c^2\abs{\xi}^2+\mass\ge0$,
i.e. $\Im\tau_-(\xi)=0$ for
$\abs{\xi}=\sqrt{-\mass}/{c}$. Therefore, the answer depends
on the Cauchy data $g$. In particular, if
$\supp \widehat{g}$ is  contained in 
$\{c^2\abs{\xi}^2+\mass\geq 0\}$,
then we may get decay of some type. More precisely,
let $B(0,r)$ denote the open ball with radius $r$ centred 
at the origin.
Then we have:
\begin{itemize}
\item if $g$ is such that $\supp\widehat{g}\cap
B(0,\frac{\sqrt{-\mass}}{c})\ne\varnothing$, then we
have no decay;
\item if there is some
$\epsilon>0$ such that $\supp\widehat{g}\subset \R^n\setminus
B(0,\frac{\sqrt{-\mass}}{c}+\ep)$, then the
roots are either separated from the real axis (if $\de>0$), 
and we
get exponential decay, or lie on the real axis 
(if $\de=0$), and we
get Klein--Gordon type behaviour 
(since the Hessian of $\tau$ is nonsingular).
\item if, for all $g$, $\supp\widehat{g}\subset\R^n\setminus
B(0,\frac{\sqrt{-\mass}}{c})
=\left\{|\xi|\geq \frac{\sqrt{-\mass}}{c}\right\}$ , 
then again we must consider
$\de=0$ and $\de>0$ separately.

If $\de=0$, then the roots lie completely on the real axis, 
and they meet on the sphere
$|\xi|=\sqrt{-\mu}/{c}$. It follows from \eqref{EQ:around}
(which is justified in Proposition \ref{Prop:aroundmult})
with $L=2$ and $\ell=1$ that, although the representation
of solution as a sum of Fourier integrals breaks down at the
sphere, the solution is still bounded
in a ($1/t$)-neighbourhood of the sphere. In its complement we
can get the decay.

If $\de>0$, then the root
$\tau_{-}$ comes to
the real axis at $\abs{\xi}=\frac{\sqrt{-\mass}}{c}$, in
which case we get the decay
\begin{equation}\label{EQ:wave}
||u(t,\cdot)||_{L^q}\leq C (1+t)^{-\left(\frac1p-\frac1q
\right)} ||g||_{L^p}.
\end{equation}
Indeed, in this case
the order of the root $\tau_{-}$ at the axis is one,
i.e. estimate \eqref{EQ:disttau} holds with $s=1$.
Here $1/p+1/q=1$ and $1\leq p\leq 2$. 
Note also that compared to the case of 
no mass when $\ell=n$, now the codimension of the sphere
$\set{\xi\in\R^n:\abs{\xi}=\frac{\sqrt{-\mass}}{c}}$ is
$\ell=1$. We can apply the last case of Part II of
Theorem \ref{THM:overallmainthm} with 
$L=1$ and $s=\ell=1$ which 
gives estimate \eqref{EQ:wave}.
\end{itemize}

\end{itemize}

\subsection{Higher order equations}\label{SEC:condonlowerorderterms}
%\subsection{Coefficient of $\pat^{m-1}u$}
Let us now derive a simple consequence of the stability
condition of 
$\Im\tau_k(\xi)\ge0$, for all $k=1,\dots,m$ and 
$\xi\in\R^n$, for
the coefficient of the $\pat^{m-1}u$ term
in~\eqref{EQ:standardCauchyproblem}. In fact, this coefficient
plays an important role for higher order equations and can be
compared with the dissipation term in the dissipative
wave equation.

Let $L=L(\pat,\pax)$ be an $m^{\text{th}}$ order constant
coefficient, linear strictly hyperbolic operator such that
$\Im\tau_k(\xi)\ge0$ for all $k=1,\dots,m$ and for all $\xi\in\R^n$.
Recall that the characteristic polynomial corresponding to the
principal part of $L$ is of the form
\begin{equation*}
L_m=L_m(\tau,\xi)=\tau^m+\sum_{k=1}^mP_k(\xi)\tau^{m-k}=0\,,
\end{equation*}
where the $P_k(\xi)$ are homogeneous polynomials of order $k$.
Then,
by the strict hyperbolicity of~$L$, $L_m$ has real roots
$\va_1(\xi)\le\va_2(\xi)\le\dots\le\va_m(\xi)$ (where the
inequalities are strict when $\xi\ne0$). By the Vi\`{e}ta formulae,
observe that
\begin{equation}\label{EQ:condonprincpart}
P_1(\xi)=-\sum_{k=1}^m\va_k(\xi)\in \R\,.
\end{equation}
On the other hand, the characteristic polynomial of the full
operator is
\begin{equation}\label{EQ:fullcharpoly}
L(\tau,\xi)=\tau^m+\sum_{k=1}^mP_k(\xi)\tau^{m-k}+
\sum_{j=0}^{m-1}\sum_{\abs{\al}+l=j}c_{\al,l}\xi^{\al}\tau^l=0\,.
\end{equation}
In particular, the coefficient of the $\tau^{m-1}$ term is
\begin{equation}\label{EQ:condnonm-1coeff}
P_1(\xi)+c_{0,m-1}=-\sum_{k=1}^m\tau_k(\xi),
\end{equation}
where the $\tau_k(\xi)$, $k=1,\dots,m$ are the roots
of~\eqref{EQ:fullcharpoly}. 
Comparing~\eqref{EQ:condonprincpart} and
\eqref{EQ:condnonm-1coeff}, we see that
$\Im\big(\sum_{k=1}^m\tau_k(\xi)\big)=-\Im c_{0,m-1}$. Therefore,
since $\Im\tau_k(\xi)\ge0$ for all $k=1,\dots,m$ and $\xi\in\R^n$,
it follows that $\Im c_{0,m-1}\le0$, or, equivalently,
$\Re ic_{0,m-1}\ge0$.
 Furthermore, if
$\Im c_{0,m-1}=0$ then it must be the case that $\Im\tau_k(\xi)=0$
for all $\xi\in\R^n$ and $k=1,\dots,m$ since the characteristic
roots are continuous. Hence we have shown the following:
\begin{prop}\label{PROP:coeffofm-1term}
Let $L=L(\pat,\pax)$ be an $m^{\text{th}}$ order linear constant
coefficient strictly hyperbolic operator such that all the
characteristic roots~$\tau_k(\xi)$\textup{}, $k=1,\dots,m$\textup{,}
satisfy $\Im \tau_k(\xi)\ge0$ for all $\xi\in\R^n$. Then the
imaginary part of the coefficient of $\pat^{m-1}u$ is non-positive.
Furthermore, if in addition the \textup{(}imaginary part of
the\textup{)} coefficient of $\pat^{m-1}u$ 
is zero then each of the
characteristic roots lie completely on the real axis. 
\end{prop}

If we transform our operator back to the form $L(\pa_t,\pa_x)$,
this result tells us that in order for the characteristic 
polynomial to
be stable, that is for $\Im\tau_k(\xi)\ge0$ for all $k=1,\dots,m$,
$\xi\in\R^n$, it is necessary for the coefficient of $\pa_t^{m-1}u$
to be non-negative; this is the case for the dissipative wave
equation. In some sense this may be interpreted
as a \emph{higher order dissipation},
since it is necessary for the characteristic roots to behave
geometrically like those of the wave equation with a dissipative
term, where they lie in the half-plane $\Im z\ge0$ 
and lie away from
$\Im z=0$ for large $\abs{\xi}$.

In the next section, we look at the case where 
characteristic roots must lie completely on 
the real
axis. First, though, let us consider one case where a root lies
completely on the real axis but the
coefficient $c_{0,m-1}$ is nonzero,~$c_{0,m-1}\ne0$.

Consider a constant coefficient strictly hyperbolic 
operator of the form
\begin{equation}\label{EQ:3termoperator}
L_m(\pa_t,\pa_x)+L_{m-1}(\pa_t,\pa_x)+L_{m-2}(\pa_t,\pa_x)=0,
\end{equation}
where $L_r=L_r(\pa_t,\pa_x)$ denotes a homogeneous operator of
degree $r$ with real coefficients. 
This is an example of a hyperbolic triple, which will be
discussed in more generality in Section \ref{SEC:hyptriples}.
Furthermore, assume that
$L_{m-1}$ is not identically zero. Let $\tau(\xi)\in\R$ be a
characteristic root of \eqref{EQ:3termoperator} which lies
completely on the real axis. So, denoting 
as usual $D_{x_j}=-i\pa_{x_j}$,
$D_t=-i\pa_t$, we have that $\tau(\xi)$ is a root of
\begin{equation*}
L_m(\tau,\xi)-iL_{m-1}(\tau,\xi)-L_{m-2}(\tau,\xi)=0.
\end{equation*}
This means that $L_{m-1}(\xi,\tau(\xi))=0$, 
and so $\tau(\xi)$ is
homogeneous of order~$1$, and thus for such roots Theorem
\ref{THM:overallmainthm} applies to yield results similar
to those described in Section \ref{SEC:homogoperators}.

\subsection{Hyperbolic triples}
\label{SEC:hyptriples}
We now turn to the case when
all the characteristic roots
lie completely on the real axis. 
This section is devoted to showing some more examples of
appearances of real valued non-homogeneous roots and some
sufficient conditions for this.
In order to study this case we first recall
some results of Volevich--Radkevich~\cite{vole+radk03} 
on hyperbolic pairs and triples. Throughout this section only,
$L_r(\tau,\xi)$ denotes a homogeneous polynomial in $\tau$ and
$\xi=(\xi_1,\dots,\xi_n)$ of order $r$ such that $L_r(\tau,i\xi)$
has real coefficients.

\begin{defn}\label{DEF:hyppair}
Suppose $L_m=L_m(\tau,\xi)$ and $L_{m-1}=L_{m-1}(\tau,\xi)$ are
homogeneous polynomials as above. Furthermore,
assume that the roots of $L_m$, $\tau_1(\xi),\dots,\tau_m(\xi)$,
and those of $L_{m-1}$, $\sigma_1(\xi),\dots,\sigma_{m-1}(\xi)$,
are real-valued \textup{(\textit{in which case we say $L_m$ and
$L_{m-1}$ are \emph{hyperbolic polynomials}})}. Then,
$(L_m,L_{m-1})$ is called a \emph{hyperbolic pair} if
\textup{(\textit{possibly after reordering})}
\begin{equation}\label{EQ:hyppaircond}
\tau_1(\xi)\le \sigma_1(\xi)\le \tau_2(\xi)\le\dots\le
\tau_{m-1}(\xi)\le \sigma_{m-1}(\xi)\le \tau_m(\xi).
\end{equation}
If, in addition, the roots of $L_m,L_{m-1}$ are pairwise distinct
for $\xi\ne0$ \textup{(\textit{in which case they are called
\emph{strictly hyperbolic polynomials}})} 
and the inequalities in
\eqref{EQ:hyppaircond} are all strict, then we say $(L_m,L_{m-1})$
is a \emph{strictly hyperbolic pair}.
\end{defn}

\begin{defn}\label{DEF:hyptriple}
Let \[L_m=L_m(\tau,\xi)\,,\; L_{m-1}=L_{m-1}(\tau,\xi)\,,\;
L_{m-2}=L_{m-2}(\tau,\xi)\] be \textup{(\textit{homogeneous})}
hyperbolic polynomials. If $(L_m,L_{m-1})$ and $(L_{m-1},L_{m-2})$
are both hyperbolic pairs then we say that $(L_m,L_{m-1},L_{m-2})$
is a \emph{hyperbolic triple}. If, in addition, all the polynomials
and all the pairs are strictly hyperbolic 
\textup{(\textit{in the
sense of Definition}~\ref{DEF:hyppair})} then
$(L_m,L_{m-1},L_{m-2})$ is called a \emph{strictly hyperbolic
triple}.
\end{defn}

\begin{thm}[\cite{vole+radk03}]\label{THM:voleradk}
Suppose that $(L_m,L_{m-1},L_{m-2})$ is a strictly hyperbolic
triple. Then $L_m(\tau,\xi)+L_{m-1}(\tau,\xi)+L_{m-2}(\tau,\xi)\ne0$
for all $\;\Im\tau\le0$. Furthermore, 
any two of the polynomials $L_m,L_{m-1},L_{m-2}$
have no common purely imaginary zeros.
\end{thm}

We also recall a theorem of Hermite (see, for example,
\cite{nish00}):
\begin{thm}\label{THM:hermitethm}
Suppose $p_m(z)$, $p_{m-1}(z)$ are real polynomials of degree
$m,m-1$, respectively, and that all the zeros of
$p(z)=p_m(z)-ip_{m-1}(z)$ lie in the upper half-plane
\textup{(\textit{that is, if $p(z)=0$ then $\Im z>0$})}. 
Then all
the zeros of $p_m(z)$ and $p_{m-1}(z)$ are real and distinct.
\end{thm}
Now we will give some rather constructive 
examples of how non-homogeneous
real roots may arise, and some sufficient conditions for this.

Assume that $L$ is of the form $L_m(\pat,D_x)+L_{m-2}(\pat,D_x)$,
where the $L_r$ are as in Definition~\ref{DEF:hyptriple} 
and neither
is identically zero. Suppose that there exists a homogeneous
operator of order $m-1$, $L_{m-1}(\pat,D_x)$, such that the
characteristic polynomials $L_m(\tau,\xi)$, 
$L_{m-1}(\tau,\xi)$ and
$L_{m-2}(\tau,\xi)$ form a strictly hyperbolic triple. 
Then, by
Theorem~\ref{THM:voleradk}, we have
\begin{equation*}
L_m(\tau,\xi)+L_{m-1}(\tau,\xi) +L_{m-2}(\tau,\xi)\ne0
\text{ for }
\Im\tau\le0\,.
\end{equation*}
Thus, by Theorem~\ref{THM:hermitethm}, all the zeros of
$L_m(\tau,\xi)+L_{m-2}(\tau,\xi)$ are real, but clearly
non-homogeneous if $L_{m-2}\not\equiv 0$. 
So, using this construction, we
can obtain examples of operators for which all the characteristic
roots lie completely on the imaginary axis
(so that $i\tau(\xi)$ are real, which would be the notation
for the rest of this paper), 
but for which we cannot
automatically expect the standard decay for homogeneous
symbols to hold.

\subsection{Strictly hyperbolic systems}
\label{SEC:systems}
Our results can also be used to 
establish $L^p-L^q$ decay rates for strictly
hyperbolic systems. Let us briefly sketch the reduction of
systems to the situation covered by results of
this paper.
Let
\begin{equation*}
iU_t=A(D)U\,,\quad U(0)=U_0\,,
\end{equation*}
be an $m\times m$ first order strictly hyperbolic system of 
partial differential equations.
That is, the associated system of polynomials may be 
written as
$A(\xi)=A_1(\xi)+A_0(\xi)$, with 
$A_1$ being positively homogeneous of order one in $\xi$ and
$A_0(\xi)\in S_{1,0}^0(\R^n)$. If $A(\xi)$ is a matrix of
first order polynomials, then $A_0$ is constant. It is known
that $A(D)$ is hyperbolic if and only if $\det A(D)$ is
hyperbolic (see e.g. Atiyah, Bott and 
G{\aa}rding \cite{ABG}). Moreover, if
$\det A_1(D)$ is strictly hyperbolic, then $A(D)$ is strongly
hyperbolic. 

Now, the strict hyperbolicity of
$A(D)$ means that the roots
$\va_1(\xi),\dots,\va_m(\xi)$ of equation $\det(\va I-A_1(\xi))=0$ 
are all
real and distinct away from the origin.  
Denote the roots of the equation $\det(\tau
I-A(\xi))=0$ 
(which is an $m^{\text{th}}$ order polynomial in $\tau$ with
smooth coefficients) by $\tau_1(\xi),\dots,\tau_m(\xi)$. 
Now, by
analogy to the case of the $m^{\text{th}}$ order scalar equation, 
we can,
via perturbation methods, show that for large~$\abs{\xi}$
the~$\tau_k(\xi)$ behave similarly to the~$\va_k(\xi)$, 
in that they
are distinct, analytic and belong to $S_{1,0}^1(\R^n)$. For
bounded~$\abs{\xi}$ we will need similar regularity 
assumptions on
the characteristic roots~$\tau_k(\xi)$ as for the 
scalar equations.
Furthermore, we assume that there exists $Q\in S^0_{1,0}(\R^n)$
such that $\abs{\det Q(\xi)}\ge C>0$ and such that
$$Q^{-1}AQ=\diag(\tau_1(\xi),\dots,\tau_m(\xi))=:T\,.$$
The existence of such $Q$ is a very interesting question
itself, especially in the presence of
variable multiplicities, but we will not go into such details here.
Now, we use the transformation $U=Q(D)V$, so that
\begin{equation*}
U_t=QV_t\implies iQV_t=A(D)QV \implies iV_t=TV\,; \; U(0)=QV(0).
\end{equation*}
This systems decouples into $m$ independent scalar equations:
\begin{equation*}
\pa_t V_k=\tau_k(D)V_k,\quad k=1,\dots,m,\quad V_k(0)=(Q^{-1}U(0))_k
\end{equation*}
each of which is solved by
\begin{equation*}
V_k(t,x)=\int_\Rn
e^{i(x\cdot\xi+\tau_k(\xi)t)}\widehat{V}_k(0,\xi)\,d\xi\,.
\end{equation*}
Now, $Q\in S^0(\R^n)$, 
so it is a bounded map $L^q\to L^q$, $1<q<\infty$, and
we can get our estimates for~$V_k$ as in the case of 
$m^{\text{th}}$
order scalar equations; thus, we can conclude that
\begin{multline*}
\norm{U}_{L^q}=\norm{Q V}_{L^q}\le C\norm{V}_{L^q}\le
C K(t)\norm{V}_{L^p}=
C K(t)\norm{Q^{-1}U}_{L^p}\le C K(t)\norm{U}_{L^p}\,,
\end{multline*}
where $K(t)$ is as in Theorem~\ref{THM:overallmainthm}.

\subsection{Application to Fokker--Planck equation}
\label{SEC:FokkerPlanck}
The classical Boltzmann equation for the particle 
distribution 
function $f=f(t,x,c)$, where $x,\mathbf{c}\in\R^n$, 
$n=1,2,3$, is 
\begin{equation*}
(\pa_t+\mathbf{c}\cdot\grad_x)f=S(f),
\end{equation*}
where $S(f)$ is the so-called integral of collisions.
The important special case of this equation is the
Fokker--Planck equation for the distribution function of 
particles in Brownian motion, 
when the integral of collisions is linear and
is given by 
$$S(f)=\grad_\mathbf{c}\cdot(\mathbf{c}+\grad_{\mathbf{c}})f=
\sum_{k=1}^n \partial_{c_k}(c_k+\partial_{c_k})f.$$
In this case the kinetic Fokker--Planck equations takes the form
$$
\left(\pa_t+\sum_{k=1}^n c_k \partial_{x_k}\right)f(t,x,c)=
\sum_{k=1}^n \partial_{c_k}(c_k+\partial_{c_k})f.$$
The Hermite--Grad method of dealing with Fokker--Planck equation
consists in decomposing $f(t,x,\cdot)$ in the Hermite basis, i.e.
writing
$$f(t,x,c)=\sum_{|\alpha|\geq 0} \frac{1}{\alpha !}
m_\alpha(t,x) \psi^\alpha(c),$$ 
where
$\psi^\alpha(c)=(2\pi)^{-n/2}(-\partial_c)^\alpha
\exp(-\frac{|c|^2}{2})$ are Hermite functions. 
They are derivatives of the Maxwell distribution
$\psi^0$ which annihilates the integral of collisions
and form a complete orthonormal basis in the weighted
Hilbert space
$L^2_w(\R^n)$ with weight $w=1/\psi^0.$
This decomposition
\footnote{Thus, 
the convergence of the series of such decomposition is
understood as a convergence of the decomposition with respect
to a basis in a Hilbert space.}
yields the infinite system
$$
\pa_tm_\be(t,x)+\be_k\pa_{x_k}m_{\be-e_k}(t,x)+\\
\pa_{x_k}m_{\be+e_k}(t,x)+\abs{\be}m_\be(t,x)=0.
$$
The Galerkin approximation $f^N$ of the solution $f$ is
\begin{equation*}\mspace{-20mu}
f^N(t,x,c)=\sum_{0\le\abs{\al}\le N}\frac{1}{\al!}m_\al(t,x)
\psi^\alpha(c)\,,
\end{equation*}
with $m(t,x)=\{ m_\beta(t,x):\; 0\leq |\beta|\leq N\}$ being
the unknown function of coefficients.
For $m(t,x)$ one obtains the following system
of equations
\begin{equation*}
\pat m(t,x)+\sum_jA_j\paxj m(t,x)-iBm(t,x)=0,
\end{equation*}
where $B$ is a diagonal matrix, 
$B_{\alpha,\beta}=|\alpha|\delta_{\alpha,\beta},$
and the only non-zero elements of the matrix $A_j$ are
$a_j^{\alpha-e_j,\alpha}=\alpha_j$,
$a_j^{\alpha+e_j,\alpha}=1$. 
Hence, the dispersion equation for the system is
\begin{equation}\label{EQ:NthFP}
P(\tau,\xi)\equiv\det(\tau I+\sum_j A_j\xi_j-iB)=0,
\end{equation}
which we will call the $N^{th}$ Fokker--Planck polynomial,
and we have, in particular,
\begin{equation}\label{EQ:FPpol}
P(\tau,0)=\det(\tau I-iB)=\tau\prod_{j=1}^{N}(\tau-ji)^
{\gamma_j}, 
\end{equation}
for some powers $\gamma_j\geq 0$.
Properties of this polynomial $P(\tau,\xi)$
have been extensively studied 
by Volevich and Radkevich in
\cite{vole+radk04}, who gave conditions and examples of
situations
when $\Im\tau_j(\xi)\ge 0$, for all $\xi\not=0$.
They also
described more general (necessary) conditions in terms of
coefficients of $P$. See also \cite{vole+radk03,ZR04}.
In our situation here we have to take additional care of
possible multiple roots, as is done in Theorem
\ref{THM:dissipative}.

From formula \eqref{EQ:FPpol} it follows in particular
that there is a single characteristic
root at the origin. Let $M=\prod_{j=1}^N j^{\gamma_j}.$

Let us examine the structure of the operator $P(\tau,\xi)$.
It is a polynomial of 
order $m$ which can be written in the form
$$P(\tau,\xi)=\sum_{j=0}^{m} (-i)^{m-j} P_j(\tau,\xi),$$ 
with
$P_j$ being a homogeneous polynomial of order $j$.
Moreover, we have
$$P_0=0,\; P_1=M\tau, \;
P_2=M\sum_{k=2}^m\frac{1}{k-1}\tau^2-M|\xi|^2.$$ 
The case $n=1$ was considered in \cite{vole+radk03}, where
one has $M=N!$

Let $P(\tau(\xi),\xi)=0$, where $\tau(0)=0$ 
is the simple root at the origin. 
Differentiation with respect to $\tau$ yields 
$\frac{\partial\tau}{\partial\xi}(0)=0$.
Differentiating again we get
$${\frac{\partial^2\tau}{\partial\xi^2}(0)=2i I_n}.$$
So, for small frequencies we obtain the decomposition
$$\Im\tau(\xi)=2|\xi|^2+\ldots+c(\log m)||\xi||^4 +\ldots,$$
where 
$$m=1+\gamma_1+\ldots \gamma_N\approx c_n N,$$ 
and 
$||\xi||^4$ denotes a fourth order polynomial in $\xi$.
We also easily have a rough
estimate for $M$ of the form
$$N^N \preceq M\preceq (N!)^N, \;(n\geq 2).$$
It follows then that for {\em small} frequencies we get
the estimate
$${|m(t,x)|\leq C (1+t)^{-n/2}+C e^{-\varepsilon(N) t}},$$
where, in general,
 it may be that $\varepsilon(N)\to 0$ as $N\to\infty$. 
For {\em medium} frequencies we get exponential decay in
view of the result of 
Theorem \ref{THM:expdecay}, also in the case when there are
multiple characteristics, where we can use Theorem
\ref{THM:expdecay2}. Here, there is 
an additional polynomial growth with
respect to time caused by the resolution procedure
of Section \ref{SEC:resofroots}, but this
is compensated by the exponential decay given by 
characteristics with strictly positive imaginary part
(see Theorem \ref{THM:expdecay2}).
 
Let us discuss the situation with {\em large} frequencies.
For operators of general form,
away from points where roots coincide, 
the roots are analytic.
For large $\abs{\xi}$, perturbation arguments 
of Section \ref{CHAP3} give 
properties of roots
$\tau_k(\xi)$ related to $\va_k(\xi)$, 
the characteristics of the
principal part. Here $\tau_k(\xi)$ and $\phi_k(\xi)$ 
are defined as roots of equations $P(\tau,\xi)=0$ and
its principal part  $P_m(\va,\xi)=0$, respectively.
Let $K$ be the maximal order of lower order terms.
Then we can summarise the following properties of $P$
established in Section \ref{CHAP3}:

\begin{itemize}
\item there are no multiple roots for large $\xi$;

\item $\abs{\pa^\al_\xi\tau_k(\xi)}\le
C\p{1+|\xi|}^{1-\abs{\al}}$, i.e. ${\tau_k\in S^1}$;

\item the exits $\va_k$ such that
$\abs{\pa^\al\tau_k(\xi)-\pa^\al\va_k(\xi)}\le
C\p{1+|\xi|}^{K+1-m-\abs{\al}}$,
for all $\xi\in\R^n$ and all multi-indices $\al$;

\item Since $\phi_k$ are real-valued, we get 
${\Im\tau_k\in S^{K+1-m}}$. In particular,
$\Im\tau_k\in S^0$.
\end{itemize}

The statements above are obtained by perturbation arguments
and rely on the strict hyperbolicity of the principal part.
However, this does not have to be the case for polynomials
$P$ that we obtain in the Galerkin approximation. Moreover,
in general, it might happen that $\Im\tau_k(\xi)\to 0$ as
$|\xi|\to\infty$, the case which is discussed
in Section \ref{SEC:asymptotic}. 
To avoid these problems we impose the
condition of strong stability. 
First, we 
will say that $P(\tau,\xi)$ is a {\em stable} polynomial if
its roots $\tau(\xi)$ satisfy $\Im\tau(\xi)\geq 0$ for all
$\xi\in\R^n$, and if
$\Im\tau(\xi)=0$ implies $\xi=0$. Then we will say that $P(\tau,\xi)$
is {\em strongly stable} if, moreover, 
$\Im\tau(\xi)=0$ implies $\xi=0$ and $\Re\tau(\xi)=0$, and if
its roots $\tau(\xi)$ satisfy
$\liminf_{|\xi|\to\infty}\Im\tau(\xi)>0$. Thus, the condition of
strong stability means that the roots $\tau(\xi)$ may become
real only at the origin of the complex plane at $\xi=0$,
and that they do not approach the real axis asymptotically
for large $\xi$.

In Section \ref{SEC:hyptriples}, as well as in
\cite{vole+radk03, vole+radk04}, there
are several sufficient conditions for the
stability of hyperbolic polynomials.
In this case we have a consequence of Theorem
\ref{THM:dissipative} and Remark \ref{REM:dismore} in the
form of estimate \eqref{EQ:bestdispest}:

\begin{cor}\label{TH:Cons1}
Let $P$ be a strongly stable polynomial with
characteristic roots with non-negative imaginary parts. 
Let $1\le p\le2$ and
$2\leq q\leq\infty$ be such that $\frac{1}{p}+\frac{1}{q}=1$. 
Then the solution to 
Cauchy problem \eqref{EQ:standardCP(repeat)} satisfies
dispersive estimate \eqref{EQ:bestdispest}, i.e. we have
\begin{equation*}
\normBig{D^r_tD^\al_x
u(t,\cdot)}_{L^q(\R^n_x)}\le
C\bract{t}^{-\frac{n}{s}
\big(\frac{1}{p}-\frac{1}{q}\big)-
\frac{\abs{\al}}{s}-\frac{r s_1}{s}}
\sum_{j=0}^{m-1}
\norm{f_j}_{W^{N_p+\abs{\al}+r-j}_p}
\,,
\end{equation*}
with $N_p\geq n(\frac1p-\frac1q)$ for $1<p\leq 2$ and
$N_1>n$ for $p=1$.
\end{cor}

From this, we can conclude the following estimates for
solution to the Galerkin approximations of Fokker--Planck
equation:

\begin{thm}\label{THM:FP}
If the 
$N^{th}$ Fokker--Planck polynomial $P$ in \eqref{EQ:NthFP}
is strongly stable, 
we have the estimate
$$||f_N(t,x,c)||_{L^\infty(\R^n_x)L^2_w(\R^n_c)}\leq
C(1+t)^{-n/2}+C_N e^{-\epsilon(N)t},$$ 
with $w=\exp(-|c|^2/2)$ and $\epsilon(N)>0$.
\end{thm}
Here the constant $C$ is independent of $N$, but, in general,
we may have asymptotically 
that $\epsilon(N)\to 0$ as $N\to\infty$.
The validity of the assumption of Theorem \ref{THM:FP}
for all $N$ is an open problem.

\def\cprime{$'$} \def\cprime{$'$}
\providecommand{\bysame}{\leavevmode\hbox to3em{\hrulefill}\thinspace}
\providecommand{\MR}{\relax\ifhmode\unskip\space\fi MR }
% \MRhref is called by the amsart/book/proc definition of \MR.
\providecommand{\MRhref}[2]{%
  \href{http://www.ams.org/mathscinet-getitem?mr=#1}{#2}
}
\providecommand{\href}[2]{#2}

\bigskip
\noindent
Department of Mathematics
  \\
  Imperial College London
  \\
  180 Queen's Gate, London SW7 2AZ \\
   United Kingdom
   
  \bigskip
  \noindent
  {\it E-mail address:} {\rm m.ruzhansky@imperial.ac.uk}

\end{document}